\documentclass[11pt,reqno]{amsart}

\usepackage{
    amsmath,
    amsfonts,
    amssymb,
    amsthm,
    amscd,
    comment,
    etoolbox,
    textcomp,
    gensymb,
    mathtools,
    mathdots,
    accents
}
\usepackage[dvipsnames]{xcolor}
\usepackage[all]{xy}
\usepackage{xparse}
\makeatletter
\@namedef{subjclassname@2020}{
  \textup{2020} Mathematics Subject Classification}
\makeatother

\usepackage[T1]{fontenc}
\usepackage{dsfont}
\usepackage[colorlinks=true, linkcolor=black, citecolor=black, urlcolor=blue, breaklinks=true]{hyperref}

\usepackage{enumitem, etoolbox}
\setlist[itemize]{leftmargin=1.8em}
\AtBeginEnvironment{theo}{
  \setlist[enumerate,1]{leftmargin=2.2em,topsep=0em,itemsep=.2em,label=\textnormal{(\roman*)}}
}
\AtBeginEnvironment{lem}{
  \setlist[enumerate,1]{leftmargin=2.2em,topsep=0em,itemsep=.2em,label=\textnormal{(\roman*)}}
}
\AtBeginEnvironment{cor}{
  \setlist[enumerate,1]{leftmargin=2.2em,topsep=0em,itemsep=.2em,label=\textnormal{(\roman*)}}
}
\AtBeginEnvironment{defin}{
  \setlist[enumerate,1]{leftmargin=2.2em,itemsep=.2em,topsep=0em}
}
\AtBeginEnvironment{eg}{
  \setlist[enumerate,1]{leftmargin=2.2em,topsep=0em,itemsep=.2em}
}

\usepackage{zref-clever}
\makeatletter
\let\orig@zcref\zcref
\renewcommand{\zcref}[1]{%
  \if\relax\detokenize{#1}\relax
    \textbf{??}
  \else
    \orig@zcref{#1}%
  \fi
}
\makeatother

\zcsetup{nameinlink=true,
countertype = {enumi=item, enumii=item, enumiii=item, enumiv=item},
counterresetby = {enumii=enumi, enumiii=enumii, enumiv=enumiii}}
    
\zcRefTypeSetup{item}{
    Name-sg = {} ,
    name-sg = {} ,
    Name-pl = {} ,
    name-pl = {} ,
}
\zcRefTypeSetup{cor}{
    Name-sg = Corollary ,
    name-sg = Corollary ,
    Name-pl = Corollaries ,
    name-pl = Corollaries ,
}
\zcRefTypeSetup{defin}{
    Name-sg = Definition ,
    name-sg = Definition ,
    Name-pl = Definitions ,
    name-pl = Definitions ,
}
\zcRefTypeSetup{eg}{
    Name-sg = Example ,
    name-sg = Example ,
    Name-pl = Examples ,
    name-pl = Examples ,
}
\zcRefTypeSetup{equation}{
    Name-sg = {} ,
    name-sg = {} ,
    Name-pl = {} ,
    name-pl = {} ,
}
\zcRefTypeSetup{lem}{
    Name-sg = Lemma ,
    name-sg = Lemma ,
    Name-pl = Lemmas ,
    name-pl = Lemmas ,
}
\zcRefTypeSetup{rem}{
    Name-sg = Remark ,
    name-sg = Remark ,
    Name-pl = Remarks ,
    name-pl = Remarks ,
}
\zcRefTypeSetup{conj}{
    Name-sg = Conjecture ,
    name-sg = Conjecture ,
    Name-pl = Conjectures ,
    name-pl = Conjectures ,
}
\zcsetup{countertype={subsection=subsection}}
\zcRefTypeSetup{subsection}{
    Name-sg = Subsection ,
    name-sg = subsection ,
    Name-pl = Subsections ,
    name-pl = subsections ,
}
\zcRefTypeSetup{table}{
    Name-sg = Table ,
    name-sg = Table ,
    Name-pl = Tables ,
    name-pl = Tables ,
}
\zcRefTypeSetup{theo}{
    Name-sg = Theorem ,
    name-sg = Theorem ,
    Name-pl = Theorems ,
    name-pl = Theorems ,
}
\zcsetup{countertype={alphatheorem=alphatheorem}}
\zcRefTypeSetup{alphatheorem}{
    Name-sg = Theorem ,
    name-sg = Theorem ,
    Name-pl = Theorems ,
    name-pl = Theorems ,
}

\newcommand{\cref}[1]{\zcref{#1}}
\newcommand{\Cref}[1]{\zcref[S]{#1}}

\AddToHook{env/conj/begin}{\zcsetup{countertype={theo=conj}}}
\AddToHook{env/cor/begin}{\zcsetup{countertype={theo=cor}}}
\AddToHook{env/defin/begin}{\zcsetup{countertype={theo=defin}}}
\AddToHook{env/eg/begin}{\zcsetup{countertype={theo=eg}}}
\AddToHook{env/lem/begin}{\zcsetup{countertype={theo=lem}}}
\AddToHook{env/rem/begin}{\zcsetup{countertype={theo=rem}}}
\AddToHook{env/table/begin}{\zcsetup{countertype={theo=table}}}

\newcommand\C{\mathbb{C}}
\newcommand\N{\mathbb{N}}

\newcommand\Q{\mathbb{Q}}
\newcommand\Z{\mathbb{Z}}
\newcommand\kk{\Bbbk}
\newcommand\one{\mathds{1}}
\newcommand{\tT}{\mathtt{T}}
\newcommand\tR{\mathtt{R}}

\newcommand{\p}{{\mathtt p}}
\newcommand\shape{\operatorname{S}}
\newcommand\basis{\operatorname{B}}
\newcommand\bk{\text{\boldmath{$k$}}}
\newcommand\bc{\text{\boldmath{$c$}}}
\newcommand\bd{\text{\boldmath{$d$}}}
\newcommand\bzero{\text{\boldmath{$0$}}}
\newcommand\bi{\text{\boldmath{$i$}}}
\newcommand\bI{\text{\boldmath{$I$}}}
\newcommand\bj{\text{\boldmath{$j$}}}
\newcommand\bJ{\text{\boldmath{$J$}}}
\newcommand\bK{\text{\boldmath{$K$}}}

\def\fg{{\operatorname{fg}}}
\def\Mod#1{{\operatorname{Mod}_{\operatorname{lfd}}(#1)}}

\newcommand\catB{\mathbf{B}}

\newcommand\Heis{\mathbf{Heis}}
\newcommand\catR{\mathbf{R}}

\newcommand{\G}{\Theta}
\newcommand{\gr}{\operatorname{gr}}
\newcommand\g{\Theta^+}
\newcommand{\Kar}{\operatorname{Kar}}

\newcommand\op{\mathrm{op}}

\newcommand{\End}{\operatorname{End}}
\newcommand{\Hom}{\operatorname{Hom}}

\newcommand\rev{\mathrm{rev}}

\newcommand\llbracket{[\![}
\newcommand\rrbracket{]\!]}
\newcommand\lround{(\!(}
\newcommand\rround{)\!)}
\def\0{{\bar 0}}
\def\1{{\bar 1}}

\def\del{\varsigma}

\newcommand{\U}{{\mathrm U}}
\newcommand\UU{\mathfrak{U}}
\newcommand\hatAB{\widehat{\AB}}
\newcommand\hatUU{\widehat{\mathfrak{U}}}
\newcommand\hatR{\widehat{R}}
\newcommand\hatH{\widehat{H}}
\newcommand\hatZ{\widehat{Z}}
\newcommand\cI{\mathbf{I}}

\newcommand\END{\mathbf{End}}
\newcommand\HOM{\mathbf{Hom}}

\newcommand{\AB}{\mathbf{A\hspace{-.4pt}B}}
\newcommand{\cB}{\mathbf{cB}}
\newcommand{\cC}{{\mathbf{cU}^\imath}}
\renewcommand{\Vec}{{\mathbf{V\hspace{-1.8pt}ec}_{\operatorname{fd}}}}
\newcommand{\SVec}{{\mathbf{S\hspace{-1.3pt}V\hspace{-1.8pt}ec}_{\operatorname{fd}}}}
\def\lo{{\operatorname{lo}}}
\def\hi{{\operatorname{hi}}}
\def\s{{\varepsilon}}

\def\t{{\delta}}
\def\iso{\stackrel{\smash{\raisebox{-1.2pt}{$\hspace{-2pt}\scriptstyle\sim$}}}{\rightarrow}}
\def\Iso{\stackrel{\smash{\raisebox{-1.2pt}{$\hspace{-2pt}\scriptstyle\sim$}}}{\Rightarrow}}
\newcommand{\Deltanabla}{{\mathchoice
{\begin{tikzpicture}[baseline=-1.2mm,scale=.5]
\draw[line width=1pt,line cap=butt] (-.249,0.02) to (.04,-0.37);
\draw[line width=1pt,line cap=butt] (.251,-0.02) to (-0.02,0.373);
\draw (0,-.4) to (.27,0) to (0,.4) to (-.27,0) to (0,-.4);
\end{tikzpicture}\hspace{-.08pt}}
{\begin{tikzpicture}[baseline=-1.2mm,scale=.5]
\draw[line width=1pt,line cap=butt] (-.249,0.02) to (.04,-0.37);
\draw[line width=1pt,line cap=butt] (.251,-0.02) to (-0.02,0.373);
\draw (0,-.4) to (.27,0) to (0,.4) to (-.27,0) to (0,-.4);
\end{tikzpicture}\hspace{-.08pt}}
{\begin{tikzpicture}[baseline=-1mm,scale=.4]
\draw[line cap=round] (0,-.4) to (.27,0);
\draw[line cap=round] (0,.4) to (-.27,0);
\draw[line width=1pt, line cap=butt] (-.2485,0.025) to (0.01,-0.37);
\draw[line width=1pt, line cap=butt] (.2485,-0.025) to (-.01,0.37);
\end{tikzpicture}\hspace{-.06pt}}
{\begin{tikzpicture}[baseline=-.8mm,scale=.33]
\draw[line cap=butt] (0,-.4) to (.27,0);
\draw[line cap=butt] (0,.4) to (-.27,0);
\draw[line width=1pt,line cap=butt] (-.24,0.03) to (0.04,-0.365);\draw[line width=1pt,line cap=butt] (.24,-.03) to (-.04,0.365);
\end{tikzpicture}\hspace{-.04pt}}}}

\DeclareMathOperator{\id}{id}
\DeclareMathOperator{\im}{im}

\DeclareMathOperator{\pr}{pr}
\DeclareMathOperator{\cf}{pr}

\DeclareMathOperator{\wt}{wt}
\def\ord{\operatorname{ord}}

\def\ibar{{\pmb{\overline{\phantom{\imath}}}\hspace{-3pt}\imath}}
\def\jbar{{\pmb{\overline{\phantom{\jmath}}}\hspace{-4pt}\jmath}}
\def\kbar{{\pmb{\overline{\phantom{k}}}\hspace{-5.3pt}k}}
\def\lbar{{\pmb{\overline{\phantom{l}}}\hspace{-3.3pt}l}}
\def\Ibar{{\pmb{\overline{\phantom{\imath}}}\hspace{-3.5pt}\imath}}
\def\Jbar{{\pmb{\overline{\phantom{\jmath}}}\hspace{-4.7pt}\jmath}}
\def\Kbar{{\pmb{\overline{\phantom{k}}}\hspace{-6pt}k}}
\NewDocumentCommand{\chiraised}{e{_^}}{\chi\IfNoValueTF{#2}{^{}}{^{#2}}\IfValueT{#1}{_{#1}}}

\usepackage{tikz}
\usetikzlibrary{arrows,patterns}
\usetikzlibrary{shapes.geometric}
\usetikzlibrary{decorations.markings,decorations.pathreplacing}
\usetikzlibrary{calc}
\usepackage{tikz-cd}

\tikzset{labl/.style={anchor=south, rotate=90, inner sep=.5mm}}

\newcommand*\getscale[1]{%
  \begingroup
    \pgfgettransformentries{\scaleA}{\scaleB}{\scaleC}{\scaleD}{\whatevs}{\whatevs}%
    \pgfmathsetmacro{#1}{sqrt(abs(\scaleA*\scaleD-\scaleB*\scaleC))}%
    \expandafter
  \endgroup
  \expandafter\def\expandafter#1\expandafter{#1}%
}

\tikzset{
    centerzero/.style={
        >=To,
        line cap = round, line join = round,
        baseline={([yshift=-0.5ex]#1)},
        },
    centerzero/.default={0,0},
    anchorbase/.style={centerzero=current bounding box.center},
}
\tikzset{wipe/.style={white,line width=4pt}}
\tikzset{rotateup/.style={anchor=south, rotate=90, inner sep=.8mm}}
\tikzset{rotatedown/.style={anchor=north, rotate=-90, inner sep=.8mm}}

\tikzset{lighty/.style=yellow!40!white}
\colorlet{lighty}{yellow!40!white}
\tikzset{lightr/.style=red!10!white}
\colorlet{lightr}{red!10!white}
\tikzset{darkr/.style=red!90!black}
\colorlet{darkr}{red!90!black}
\tikzset{lightb/.style=blue!10!white}
\colorlet{lightb}{blue!10!white}
\tikzset{darkg/.style=green!70!black}
\colorlet{darkg}{green!70!black}
\tikzset{lightg/.style=green!10!white}
\colorlet{lightg}{green!10!white}

\tikzset{pinhead/.style={thin,gray,fill=lighty}}
\tikzset{iQ/.style={thin,black}}
\tikzset{B/.style={semithick,magenta!80!black}}
\tikzset{bulb/.style={thin,fill}}
\tikzset{Bulb/.style={thin,fill=white}}
\tikzset{fakebubble/.style={semithick,darkr,fill=lightr}}

\newcommand\notch[2][0]{
    \draw[thin] (#2)++(#1:0.06) -- ++(#1:-0.12)
}

\newcommand\projcr[1]{
    \draw (#1) node[rotate=45] {$\scriptstyle{\Box}$}
}

\newcommand{\objpin}[3]{
    \path (#1) node[inner sep=1.6pt] (x) {} to (#2) node[rectangle,rounded corners,draw,thin,gray,fill=lightg,inner sep=2.5pt](y) {$\color{black}\scriptstyle#3$};
    \draw[Triangle Cap-,thick,gray] (x)--(y);
    \closeddot{#1}
}

\newcommand{\pin}[3]{
    \path (#1) node[inner sep=1.6pt] (x) {} to (#2) node[rectangle,rounded corners,draw,pinhead,inner sep=2.5pt](y) {$\color{black}\scriptstyle#3$};
    \draw[Triangle Cap-,thick,gray] (x)--(y);
    \closeddot{#1}
}

\newcommand{\pinpin}[4]{
\path (#1) node[inner sep=1.6pt] (x) {} to (#3) node[rectangle,rounded corners,draw,pinhead,inner sep=2.5pt](y) {$\color{black}\scriptstyle#4$};
    \draw[Triangle Cap-,thick,gray] (x)--(#2)--(y);
\closeddot{#2}; \closeddot{#1}
}

\newcommand{\pinpinpin}[5]{
\path (#1) node[inner sep=1.6pt] (x) {} to (#4) node[rectangle,rounded corners,draw,pinhead,inner sep=2.5pt](y) {$\color{black}\scriptstyle#5$};
    \draw[Triangle Cap-,thick,gray] (x)--(#2)--(#3)--(y);
\closeddot{#3}; \closeddot{#2}; \closeddot{#1}
}

\newcommand{\pinpinmid}[4]{
\path (#1) node[inner sep=1.6pt] (x) {} to (#2) node[rectangle,rounded corners,draw,pinhead,inner sep=2.5pt] (y) {$\color{black}\scriptstyle#4$} to (#3) node[inner sep=1.6pt] (z) {};
\draw[Triangle Cap-,thick,gray] (x)--(y);
\draw[Triangle Cap-,thick,gray] (z)--(y);
\closeddot{#3}; \closeddot{#1}
}

\newcommand{\pinpinpinmid}[5]{
\path (#1) node[inner sep=1.6pt] (x) {} to (#4) node[rectangle,rounded corners,draw,pinhead,inner sep=2.5pt] (w) {$\color{black}\scriptstyle#5$} to (#3) node[inner sep=1.6pt] (z) {} to (#2) node[inner sep=1.6pt] (y) {};
\draw[Triangle Cap-,thick,gray] (x)--(w);
\draw[Triangle Cap-,thick,gray] (y)--(w);
\draw[Triangle Cap-,thick,gray] (z)--(w);
\closeddot{#3}; \closeddot{#1}; \closeddot{#2}
}

\newcommand{\circledbar}[2]{
\getscale{\scalefactor}
\node[circle,draw,bulb,inner sep=.85pt] at (#1) {$\hspace{0.06pt}\color{white}\scriptscriptstyle \pmb#2$};\draw[line width=1.1pt,white] ($(#1)+({-.06/ \scalefactor},{.08 / \scalefactor})$) to ($(#1)+({.06/ \scalefactor},{.08 / \scalefactor})$)}

\newcommand{\circled}[2]{\node[circle,draw,bulb,inner sep=.85pt] at (#1) {$\hspace{0.06pt}\color{white}\scriptscriptstyle\pmb #2$}}

\newcommand\braidup{to[out=up,in=down]}
\newcommand\braiddown{to[out=down,in=up]}

\newcommand\dotlabel[1]{$\color{teal}\scriptstyle{#1}$}
\newcommand\regionlabel[1]{$\color{gray}\scriptstyle{#1}$}
\newcommand\strandlabel[1]{$\color{olive}\scriptstyle{#1}$}
\newcommand\stringlabel[2]{\node at (#1) {\strandlabel{#2}}} 
\newcommand\catlabel[2]{\node[inner sep=2pt,draw,darkg,fill=lightg,rounded corners]  at (#1) {$\color{black}\scriptstyle{#2}$}} 
\newcommand\botlabel[1]{node[inner sep=0.5pt,anchor=north] {\strandlabel{#1}}}
\newcommand\toplabel[1]{node[anchor=south,inner sep=0.5pt] {\strandlabel{#1}}}
\newcommand\dottybubblelabel[2]{\node at (#1) {$\color{teal}\scriptstyle{#2}$}} 
\newcommand\bubblelabel[3][black]{\node at (#2) {$\color{#1}\scriptstyle{#3}$}} 
\newcommand\region[2]{\node at (#1) {\regionlabel{#2}}} 
\newcommand\strand[2]{\node at (#1) {\strandlabel{#2}}} 

\newcommand\closeddot[1]{
    \node at (#1) {${\scriptstyle\bullet}$}
}

\newcommand\magdot[1]{
    \node[B] at (#1) {${\bullet}$}
}
\newcommand\magup[1]{
    \node[B] at (#1) {${\scriptstyle\blacktriangle}$}
}
\newcommand\magdown[1]{
    \node[B] at (#1) {${\scriptstyle\blacktriangledown}$}
}

\newcommand\magcr[1]{
    \draw[B] (#1) node {$\scriptstyle{\blacklozenge}$}
}

\newcommand\emptydot[1]{
    \node[circle,fill=white,draw=black,bulb,inner sep=.3pt] at (#1) {$\phantom{\scriptstyle o}$}
}
\newcommand\plusdot[1]{
    \node[circle,fill=white,draw=black,bulb,inner sep=.3pt] at (#1) {$\phantom{\scriptstyle o}$};
    \node at (#1) {{\boldmath$\scriptscriptstyle+$}}
}
\newcommand\minusdot[1]{
    \node[circle,fill=white,draw=black,bulb,inner sep=.3pt] at (#1) {$\phantom{\scriptstyle o}$};
    \node at (#1) {{\boldmath$\scriptscriptstyle-$}}
}

\newcommand\multcloseddot[3]{
    \closeddot{#1};
    \draw (#1) node[anchor=#2] {\dotlabel{#3}}
}

\tikzset{internalbubble/.style={fill=lightb}}

\newcommand\bubble[2][\phantom{B}]{
\node[rectangle,draw,darkg,fill=lightg,rounded corners,inner sep=2.5pt] at (#2)
{$\color{black}\scriptstyle #1$}
}

\NewDocumentCommand\westinternalbubbleplus{om}{
 \getscale{\scalefactor}
 \draw[-,internalbubble] ($(#2)+({-.2 / \scalefactor},0)$) arc(-180:180:{.2/\scalefactor});
 \bubblelabel{#2}{#1};
 \plusdot{$(#2)+({-.22 / \scalefactor},0)$}
}
\NewDocumentCommand\westinternalbubblePlus{om}{
 \getscale{\scalefactor}
 \draw[-,internalbubble] ($(#2)+({-.25 / \scalefactor},0)$) arc(-180:180:{.25/\scalefactor});
 \bubblelabel{#2}{#1};
 \plusdot{$(#2)+({-.27 / \scalefactor},0)$}
}

\NewDocumentCommand\westinternalbubbleminus{om}{
\getscale{\scalefactor}
 \draw[-,internalbubble] ($(#2)+({-.2 / \scalefactor},0)$) arc(-180:180:{.2/\scalefactor});
 \minusdot{$(#2)+({-.22 / \scalefactor},0)$};
 \bubblelabel{#2}{#1}
}
\NewDocumentCommand\westinternalbubbleMinus{om}{
\getscale{\scalefactor}
 \draw[-,internalbubble] ($(#2)+({-.25 / \scalefactor},0)$) arc(-180:180:{.25/\scalefactor});
 \minusdot{$(#2)+({-.27 / \scalefactor},0)$};
 \bubblelabel{#2}{#1}
}

\NewDocumentCommand\eastinternalbubbleplus{om}{
\getscale{\scalefactor}
 \draw[-,internalbubble] ($(#2)+({-.2 / \scalefactor},0)$) arc(-180:180:{.2/\scalefactor});
 \plusdot{$(#2)+({.22 / \scalefactor},0)$};
 \bubblelabel{#2}{#1}
}
\NewDocumentCommand\eastinternalbubblePlus{om}{
 \getscale{\scalefactor}
 \draw[-,internalbubble] ($(#2)+({-.25 / \scalefactor},0)$) arc(-180:180:{.25/\scalefactor});
 \plusdot{$(#2)+({.27 / \scalefactor},0)$};
 \bubblelabel{#2}{#1}
}

\NewDocumentCommand\eastinternalbubbleminus{om}{
 \getscale{\scalefactor}
 \draw[-,internalbubble] ($(#2)+({-.2 / \scalefactor},0)$) arc(-180:180:{.2/\scalefactor});
 \minusdot{$(#2)+({.22 / \scalefactor},0)$};
 \bubblelabel{#2}{#1}
}
\NewDocumentCommand\eastinternalbubbleMinus{om}{
 \getscale{\scalefactor}
 \draw[-,internalbubble] ($(#2)+({-.25 / \scalefactor},0)$) arc(-180:180:{.25/\scalefactor});
 \minusdot{$(#2)+({.27 / \scalefactor},0)$};
 \bubblelabel{#2}{#1}
}

\NewDocumentCommand\fatbubbleminus{m}{
 \getscale{\scalefactor}
\draw[internalbubble] 
($(#1)+({.2/\scalefactor},0)$) -- 
($(#1)+(0,{.24 /\scalefactor})$) --
($(#1)+({-.2/\scalefactor},0)$) --
($(#1)+(0,{-.24/\scalefactor})$) -- cycle;
 \node at (#1) {{\boldmath$\scriptscriptstyle-$}} 
}
\NewDocumentCommand\fatbubbleplus{m}{
 \getscale{\scalefactor}
\draw[internalbubble] 
($(#1)+({.2/\scalefactor},0)$) -- 
($(#1)+(0,{.24 /\scalefactor})$) --
($(#1)+({-.2/\scalefactor},0)$) --
($(#1)+(0,{-.24/\scalefactor})$) -- cycle;
 \node at (#1) {{\boldmath$\scriptscriptstyle+$}} 
}

\NewDocumentCommand\internalbubbleminus{O{0}m}{
 \getscale{\scalefactor}
\draw[internalbubble] 
($(#2)+({cos(#1)*.21/\scalefactor+sin(#1)*.13/\scalefactor},{-sin(#1)*.21/\scalefactor+cos(#1)*.13 /\scalefactor})$) -- 
($(#2)+({-cos(#1)*.21/\scalefactor+sin(#1)*.13/\scalefactor},{sin(#1)*.21/\scalefactor+cos(#1)*.13 /\scalefactor})$) --
($(#2)+({-sin(#1)*.21/\scalefactor},{-cos(#1)*.22 /\scalefactor})$) -- cycle;
 \node at (#2) {{\boldmath$\scriptscriptstyle-$}} 
}

\NewDocumentCommand\internalbubbleplus{O{0}m}{
 \getscale{\scalefactor}
\draw[internalbubble] 
($(#2)+({cos(#1)*.21/\scalefactor+sin(#1)*.13/\scalefactor},{-sin(#1)*.21/\scalefactor+cos(#1)*.13 /\scalefactor})$) -- 
($(#2)+({-cos(#1)*.21/\scalefactor+sin(#1)*.13/\scalefactor},{sin(#1)*.21/\scalefactor+cos(#1)*.13 /\scalefactor})$) --
($(#2)+({-sin(#1)*.21/\scalefactor},{-cos(#1)*.22 /\scalefactor})$) -- cycle;
 \node at (#2) {{\boldmath$\scriptscriptstyle+$}} 
}

\newtheorem{alphatheorem}{Theorem}

\newtheorem{theo}{Theorem}[section]

\newtheorem{lem}[theo]{Lemma}
\newtheorem{cor}[theo]{Corollary}

\theoremstyle{definition}
\newtheorem{defin}[theo]{Definition}
\newtheorem{rem}[theo]{Remark}
\newtheorem{eg}[theo]{Example}

\numberwithin{equation}{section}
\allowdisplaybreaks

\usepackage{comment}
\newif\ifdetails
\detailstrue
\ifdetails

\else
\excludecomment{details}
\fi

\begin{document}
\title{Brauer and iquantum categorification}
\author{Jonathan Brundan}
\address[J.B.]{Department of Mathematics, University of Oregon, Eugene, OR, USA}
\urladdr{\href{https://pages.uoregon.edu/brundan}{https://pages.uoregon.edu/brundan}, \textrm{\textit{ORCiD}:} \href{https://orcid.org/0009-0009-2793-216X}{0009-0009-2793-216X}}
\email{brundan@uoregon.edu}

\author{Alistair Savage}
\address[A.S.]{Department of Mathematics and Statistics, University of Ottawa, Ottawa, ON, Canada}
\urladdr{\href{https://alistairsavage.ca}{alistairsavage.ca}, \textrm{\textit{ORCiD}:} \href{https://orcid.org/0000-0002-2859-0239}{orcid.org/0000-0002-2859-0239}}
\email{alistair.savage@uottawa.ca}

\author{Ben Webster}
\address[B.W.]{Department of Pure Mathematics, University of Waterloo \& Perimeter Institute for Theoretical Physics,
Waterloo, ON, Canada}
\urladdr{\href{https://bwebste.github.io}{https://bwebste.github.io}, \textrm{\textit{ORCiD}:} \href{https://orcid.org/0000-0003-1896-5540}{0000-0003-1896-5540}}
\email{ben.webster@uwaterloo.ca}

\thanks{Research supported in part by NSF grant DMS-2348840 (B), NSERC grant RGPIN-2023-03842 (S) and NSERC grant RGPIN-2024-03760 (W)}



\begin{abstract}
We show that locally finite Abelian module categories over
the affine Brauer category
are the same thing as locally finite Abelian
2-representations of a certain 2-iquantum group satisfying some
additional nilpotency and finiteness properties.
Our results are valid over algebraically 
closed fields of characteristic other than 2 and for arbitrary spectra.
This extends recent work of Gao, Rui and Song, who treated characteristic 0 with half-integral spectrum.
The relevant Satake diagrams are
of quasi-split type AIII and of diagonal type A, including affine types in positive characteristic.
The key step in our approach identifies the affine Brauer category with a topologically dense monoidal subcategory of a completion of the 2-iquantum group.
As an application, we show that deformed cyclotomic quotients on the two sides are equivalent.
\end{abstract}

\maketitle

\tableofcontents
\setcounter{section}{0}

\section{Introduction}

The notion of a Kac--Moody categorification was introduced by Rouquier in \cite{Rou}, building on his earlier work with Chuang \cite{ChuangRouquier2008} treating the $\mathfrak{sl}_2$ case. It provides an elegant axiomatic framework encapsulating ideas of higher representation theory which emerged in the 1990s, extending the approach to representation theory of symmetric and general linear groups which can be traced back to work of Frobenius and Schur at the beginning of the 20th century and, later, Macdonald, Bernstein and Zelevinsky. A key ingredient in Rouquier's setup comes from the {\em quiver Hecke algebras} (a.k.a.\ Khovanov--Lauda--Rouquier algebras) discovered in \cite{Rou,KL1},
which are half of the {\em Kac--Moody 2-categories} $\UU(\mathfrak{g})$
introduced in \cite{Rou,KL3}. We also refer to Kac--Moody 2-categories as {\em 2-quantum groups} because $\UU(\mathfrak{g})$ categorifies Lusztig's modified integral form for the Drinfeld--Jimbo quantum group associated to the Kac--Moody algebra $\mathfrak{g}$. This is a theorem of Khovanov and Lauda \cite[Ths.~1.1--1.2]{KL3}, modulo the non-degeneracy of 2-quantum groups which was proved in full generality in \cite{Webster-2026}; see also \cite[Th.~5.6]{BWW}.

We work from now on over a fixed algebraically closed ground field $\kk$.
Let $\mathfrak{g}$ be a Kac-Moody algebra with weight lattice $X$.
A {\em Kac--Moody categorification} is a locally finite Abelian category
$$
\catR
= \bigoplus_{\lambda \in X} \catR_\lambda
$$ 
equipped with a categorical action of a 2-quantum group
$\UU(\mathfrak{g})$, i.e., it is a 2-representation of 
$\UU(\mathfrak{g})$,
such that for each $\lambda \in X$ and $V \in \catR_\lambda$
the set $\{i \in I\:|\:E_i V \neq 0\text{ or }F_i V \neq 0\}$
is finite.
We assume in addition that $\catR$ 
is {\em nilpotent}, which means that the natural transformations defined by dots act nilpotently on any object.
It follows that the Grothendieck group $K_0(\catR) = \bigoplus_{\lambda \in X} K_0(\catR_\lambda)$ is an integrable representation of the Kostant $\Z$-form of the universal enveloping algebra $U(\mathfrak{g})$. Moreover, the isomorphism classes of irreducible objects of $\catR$
equip $K_0(\catR)$ with a {\em perfect basis} in the sense of \cite[Def.~5.49]{BeK-2007}.
Aspects of the structure of the category $\catR$ can then be understood via the rich combinatorics of integrable representations, e.g., the action of the underlying braid group permuting weight spaces, and the theory of 
canonical/crystal bases of Lusztig and Kashiwara, both of which play an important role.

There are many naturally occurring examples of 
Kac--Moody categorifications arising in type A representation theory; see the first part of \cite{BSW-update} for a recent survey. In all of these examples, the underlying Kac--Moody algebra $\mathfrak{g}$ has simple roots indexed by the set $I \coloneq \kk$, and its Cartan matrix $(a_{i,j})_{i,j \in I}$ is
defined by $a_{i,j} \coloneq 2\delta_{i,j}-\delta_{i,j+1}-\delta_{i,j-1}$; the connected components are all of type A$_\infty$ if $p=0$ or A$_{p-1}^{(1)}$ if $p > 0$, where $p$ is the characteristic of $\kk$. 
To prove the existence of the categorical action of $\UU(\mathfrak{g})$ in the first place, one needs to construct explicit endofunctors and natural transformations satisfying the defining relations of $\UU(\mathfrak{g})$. The verification of Rouquier's {\em inversion relation} is particularly difficult.
In \cite[Th.~5.27]{Rou}, Rouquier developed a general approach to this which he called ``control by $K_0$''. 
In our work \cite{HKM}, we explained another approach which is more convenient in practice. Both approaches depend ultimately on the isomorphism between cyclotomic quotients of degenerate
affine Hecke algebras of type A and corresponding cyclotomic quiver Hecke algebras from \cite{BK-iso} or, equivalently, 
the isomorphism between completions of these algebras from \cite{Rou}.

\medskip

By a {\em Heisenberg categorification} we mean a locally finite Abelian category $\catR$ equipped with an action of a Heisenberg category $\Heis_\kappa$ for some central charge $\kappa \in \Z$, i.e., $\catR$ is a strict left $\Heis_\kappa$-module category.
By {\em Heisenberg category} here, we mean
the strict $\kk$-linear monoidal category
defined originally by Khovanov for central charge one
in \cite{Khovanov2014} and generalized to all central charges in \cite{BCNR,MS-2018,B-heis}.
The main idea in \cite{HKM} is that many of the Abelian categories of classical importance in type A representation theory are Heisenberg categorifications. Moreover, the action of $\Heis_\kappa$ is usually easy to construct.
Then we showed that any  
Heisenberg categorification is a nilpotent 
Kac--Moody categorification, and vice versa (assuming that the Cartan type of $\mathfrak{g}$ is as described above and the defining parameters of the 2-quantum group 
$\UU(\mathfrak{g})$ are chosen correctly).
Roughly speaking, we did this by upgrading the algebra isomorphisms between appropriate completions of degenerate affine Hecke algebras and quiver Hecke algebras arising from \cite{Rou,BK-iso} to
Heisenberg categories and 2-quantum groups, although we did not use this language explicitly.
In the setup of \cite{HKM},
the rightward cup and cap generators of $\Heis_\kappa$ 
map in an obvious way to the rightward cup and cap generators in the 2-quantum group.
However,
it is much harder to understand the relationship between the ``second adjunctions'' defined by the leftward cup and cap generators in these categories. We finally worked out how to do that explicitly in the second part of \cite{BSW-update}.

For later reference, we should also mention an application of the theory developed
in \cite{HKM}: it implies the existence of explicit
isomorphisms between {\em generalized cyclotomic quotients} (GCQs) of the Heisenberg category and GCQs of the corresponding 2-quantum group. These isomorphisms are a substantial generalization of 
the isomorphisms between ordinary cyclotomic quotients
from \cite{BK-iso}. For example, taking the ground field $\kk$ to be $\C$, 
the Deligne category $\underline{\operatorname{Re}}\!\operatorname{p}(GL_\delta)$, which by definition is the additive Karoubi envelope of the oriented Brauer category $\mathbf{OB}_\delta$, is a generalized cyclotomic quotient of level one.
This is a useful result:
for example, 
in \cite{BM-arc}, we use it to show for $\delta \in \Z$ that
$\underline{\operatorname{Re}}\!\operatorname{p}(GL_\delta)$
is Morita equivalent to a certain limiting version of Khovanov's arc algebra, as was originally conjectured in \cite{BS-brauer}.

\medskip

\begin{table}
$$
\begin{array}{|r|c|c|}
\hline
\text{A}_\infty&
p=0,\ k \not\sim 0, h\ &
\begin{tikzpicture}[anchorbase]
\node at (-2,-.3) {$\scriptstyle \phantom{k}$};
\node at (-2.8,0) {$\cdots$};
\node at (2.9,0) {$\cdots$};
\node at (-2,.3) {$\scriptstyle k-2$};
\node at (-1,.3) {$\scriptstyle k-1$};
\node at (0,.3) {$\scriptstyle k$};
\node at (1,.3) {$\scriptstyle k+1$};
\node at (2,.3) {$\scriptstyle k+2$};
\draw (-2.4,0) to (2.4,0);
\node at (-2,-1) {$\scriptstyle \phantom{k}$};
\node at (-2.8,-.7) {$\cdots$};
\node at (2.9,-0.7) {$\cdots$};
\node at (-2,-1) {$\scriptstyle -k+2$};
\node at (-1,-1) {$\scriptstyle -k+1$};
\node at (0,-1) {$\scriptstyle -k$};
\node at (1,-1) {$\scriptstyle -k-1$};
\node at (2,-1) {$\scriptstyle -k-2$};
\draw (-2.4,-.7) to (2.4,-.7);
\node at (1.5,0) {$\scriptstyle >$};
\node at (.5,0) {$\scriptstyle >$};
\node at (-.5,0) {$\scriptstyle >$};
\node at (-1.5,0) {$\scriptstyle >$};
\node at (1.5,-.7) {$\scriptstyle >$};
\node at (.5,-.7) {$\scriptstyle >$};
\node at (-.5,-.7) {$\scriptstyle >$};
\node at (-1.5,-.7) {$\scriptstyle >$};
\draw[densely dotted,red,<->] (2,-.55) to (2,-.15);
\draw[densely dotted,red,<->] (1,-.55) to (1,-.15);
\draw[densely dotted,red,<->] (0,-.55) to (0,-.15);
\draw[densely dotted,red,<->] (-1,-.55) to (-1,-.15);
\draw[densely dotted,red,<->] (-2,-.55) to (-2,-.15);
\minusdot{1,-0.7}; 
\plusdot{-2,0};
\plusdot{1,0};
\minusdot{-1,-.7};
\minusdot{2,-.7};
\plusdot{0,0};
\plusdot{2,0};
\minusdot{0,-.7};
\minusdot{-2,-.7};
\plusdot{-1,0};
\end{tikzpicture}
\\\hline
\text{AIII}_{2\infty+1}&
p=0,\ k \sim 0
&
\begin{tikzpicture}[anchorbase]
\node at (-1.8,0) {$\cdots$};
\node at (-1,.3) {$\scriptstyle 3$};
\node at (0,.3) {$\scriptstyle 2$};
\node at (1,.3) {$\scriptstyle 1$};
\draw (-1.4,0) to (1,0) to (1.8,-.35) to (1,-.7) to (-1.4,-.7);
\node at (-2,-1) {$\scriptstyle \phantom{k}$};
\node at (-1.8,-.7) {$\cdots$};
\node at (-1,-1) {$\scriptstyle -3$};
\node at (0,-1) {$\scriptstyle -2$};
\node at (1,-1) {$\scriptstyle -1$};
\node at (2.05,-.35) {$\scriptstyle 0$};
\node at (-.5,0) {$\scriptstyle >$};
\node at (.5,0) {$\scriptstyle >$};
\node at (-.5,-0.7) {$\scriptstyle <$};
\node at (.5,-0.7) {$\scriptstyle <$};
\node at (1.36,-.16) {\rotatebox[origin=c]{-24}{$\scriptstyle >$}};
\node at (1.36,-.54) {\rotatebox[origin=c]{24}{$\scriptstyle <\!\!<\!$}};
\draw[densely dotted,red,<->] (1,-.55) to (1,-.15);
\draw[densely dotted,red,<->] (0,-.55) to (0,-.15);
\draw[densely dotted,red,<->] (-1,-.55) to (-1,-.15);
\plusdot{-1,0};
\minusdot{0,0};
\plusdot{1,0};
\minusdot{-1,-.7};
\plusdot{0,-.7};
\minusdot{1,-.7};
\emptydot{1.8,-.35};
\end{tikzpicture}
\\\hline
\text{AIII}_{2\infty}&p=0,\ k \sim h&
\begin{tikzpicture}[anchorbase]
\node at (-2,-.35) {$\scriptstyle \phantom{k}$};
\node at (1.8,0) {$\cdots$};
\node at (0,.35) {$\scriptstyle -\frac{5}{2}$};
\node at (-2.23,.35) {$\scriptstyle -h\,=-\frac{1}{2}$};
\node at (-1,.35) {$\scriptstyle -\frac{3}{2}$};
\node at (1,.35) {$\scriptstyle -\frac{7}{2}$};
\node at (-2,-1) {$\scriptstyle \phantom{k}$};
\node at (1.8,-0.7) {$\cdots$};
\draw (-2,0) to (1.4,0);
\node at (-2.43,-1.05) {$\scriptstyle h\,=\frac{1}{2}$};
\node at (-1.1,-1.05) {$\scriptstyle \frac{3}{2}$};
\node at (-0.1,-1.05) {$\scriptstyle \frac{5}{2}$};
\node at (.9,-1.05) {$\scriptstyle \frac{7}{2}$};
\draw (-2,0) to [out=180,in=180,looseness=2] (-2,-.7) to (1.4,-.7);
\draw[densely dotted,red,<->] (1,-.55) to (1,-.15);
\draw[densely dotted,red,<->] (0,-.55) to (0,-.15);
\draw[densely dotted,red,<->] (-1,-.55) to (-1,-.15);
\draw[densely dotted,red,<->] (-2,-.55) to (-2,-.15);
\node at (-2.41,-.38) {\rotatebox[origin=c]{-90}{$\scriptstyle <$}};
\node at (-.5,0) {$\scriptstyle >$};
\node at (.5,0) {$\scriptstyle >$};
\node at (-1.5,0) {$\scriptstyle >$};
\node at (-1.5,-0.7) {$\scriptstyle <$};
\node at (-.5,-0.7) {$\scriptstyle <$};
\node at (.5,-0.7) {$\scriptstyle <$};
\minusdot{-2,0};
\plusdot{-1,0};
\minusdot{0,0};
\plusdot{1,0};
\plusdot{-2,-.7};
\minusdot{-1,-.7};
\plusdot{0,-.7};
\minusdot{1,-.7};
\end{tikzpicture}
\\\hline\hline
\text{A}_{p-1}^{(1)}&
p>2,\ k \not\sim 0\ &
\begin{tikzpicture}[anchorbase]
\node at (-1,1) {$\scriptstyle \phantom{\frac{p-3}{2}}$};
\node at (1.98,.7) {$\cdots$};
\draw (-1,.7) to (1.5,.7);
\draw (2.45,.7) to (3.95,.7);
\draw (-1,.7) to [out=60,in=120,looseness=.6] (3.95,.7);
\node at (-1.3,.7) {$\scriptstyle k$};
\node at (0,.95) {$\scriptstyle k+1$};
\node at (1,.95) {$\scriptstyle k+2$};
\node at (2.85,.95) {$\scriptstyle k+p-2$};
\node at (4.65,.7) {$\scriptstyle k+p-1$};
\node at (1.98,0) {$\cdots$};
\draw (-1,0) to (1.5,0);
\draw (2.45,0) to (3.95,0);
\draw (-1,0) to [out=-60,in=-120,looseness=.6] (3.95,0);
\node at (-1.4,0) {$\scriptstyle -k$};
\node at (-.1,-.25) {$\scriptstyle -k-1$};
\node at (.9,-.25) {$\scriptstyle -k-2$};
\node at (2.85,-.25) {$\scriptstyle -k-p+2$};
\node at (4.75,0) {$\scriptstyle -k-p+1$};
\draw[densely dotted,red,<->] (2.95,.55) to (2.95,0.15);
\draw[densely dotted,red,<->] (3.95,.55) to (3.95,0.15);
\draw[densely dotted,red,<->] (1,.55) to (1,0.15);
\draw[densely dotted,red,<->] (-1,.55) to (-1,0.15);
\draw[densely dotted,red,<->] (0,.55) to (0,0.15);
\node at (3.45,0) {$\scriptstyle >$};
\node at (.5,0) {$\scriptstyle >$};
\node at (-.5,0) {$\scriptstyle >$};
\node at (3.45,.7) {$\scriptstyle >$};
\node at (.5,.7) {$\scriptstyle >$};
\node at (-.5,.7) {$\scriptstyle >$};
\node at (1.4,1.45) {$\scriptstyle <$};
\node at (1.4,-.75) {$\scriptstyle <$};
\plusdot{-1,0.7};
\plusdot{0,0.7};
\plusdot{3.95,0.7};
\plusdot{2.95,0.7};
\plusdot{1,0.7};
\minusdot{-1,0};
\minusdot{0,0};
\minusdot{3.95,0};
\minusdot{2.95,0};
\minusdot{1,0};
\end{tikzpicture}
\\\hline
\text{AIII}^{(1)}_{p-1}&\ p \equiv 1\!\!\!\!\pmod{4},\ k \sim 0
\ &
\begin{tikzpicture}[anchorbase]
\node at (-2,-.35) {$\scriptstyle \phantom{k}$};
\node at (-2.23,.35) {$\scriptstyle \scriptstyle -h\,=\frac{p-1}{2}$};
\node at (-1,.35) {$\scriptstyle \frac{p-3}{2}$};
\node at (1,.3) {$\scriptstyle 2$};
\node at (2,.3) {$\scriptstyle 1$};
\draw (-2,0) to (-.5,0);
\draw (.5,0) to (2,0);
\node at (-2,-1.05) {$\scriptstyle \phantom{k}$};
\node at (-2.43,-1.01) {$\scriptstyle h\,=\frac{p+1}{2}$};
\node at (-1,-1.01) {$\scriptstyle \frac{p+3}{2}$};
\node at (1,-1.0) {$\scriptstyle p-2$};
\node at (2,-1.0) {$\scriptstyle p-1$};
\node at (3.05,-.35) {$\scriptstyle 0$};
\draw (-1.4,0) to (-.5,0);
\node at (0,0) {$\cdots$};
\node at (0,-.7) {$\cdots$};
\draw (.5,0) to (2,0) to (2.8,-.35) to (2,-.7) to (.5,-.7);
\draw (-2,0) to [out=180,in=180,looseness=2] (-2,-.7) to (-.5,-.7);
\draw[densely dotted,red,<->] (1,-.55) to (1,-.15);
\draw[densely dotted,red,<->] (2,-.55) to (2,-.15);
\draw[densely dotted,red,<->] (-1,-.55) to (-1,-.15);
\draw[densely dotted,red,<->] (-2,-.55) to (-2,-.15);
\node at (1.5,-.7) {$\scriptstyle <$};
\node at (-1.5,-.7) {$\scriptstyle <$};
\node at (1.5,0) {$\scriptstyle >$};
\node at (-1.5,0) {$\scriptstyle >$};
\node at (2.36,-0.16) {\rotatebox[origin=c]{-24}{$\scriptstyle >$}};
\node at (-2.41,-.38) {\rotatebox[origin=c]{-90}{$\scriptstyle <$}};
\node at (2.36,-.54) {\rotatebox[origin=c]{24}
{$\scriptstyle <\!\!<\!$}};
\minusdot{-2,0};
\plusdot{-1,0};
\minusdot{1,0};
\plusdot{2,0};
\plusdot{-2,-.7};
\minusdot{-1,-.7};
\plusdot{1,-.7};
\minusdot{2,-.7};
\emptydot{2.8,-.35};
\end{tikzpicture}
\\\hline
\text{AIII}^{(1)}_{p-1}&
\ p \equiv 3\!\!\!\!\pmod{4},\ k \sim 0
&
\begin{tikzpicture}[anchorbase]
\node at (-2,-.35) {$\scriptstyle \phantom{k}$};
\node at (-2.23,.35) {$\scriptstyle\scriptstyle -h\,= \frac{p-1}{2}$};
\node at (-1,.35) {$\scriptstyle \frac{p-3}{2}$};
\node at (1,.3) {$\scriptstyle 2$};
\node at (2,.3) {$\scriptstyle 1$};
\draw (-2,0) to (-.5,0);
\draw (.5,0) to (2,0);
\node at (-2,-1.05) {$\scriptstyle \phantom{k}$};
\node at (-2.43,-1.01) {$\scriptstyle \scriptstyle h\,=\frac{p+1}{2}$};
\node at (-1,-1.01) {$\scriptstyle \frac{p+3}{2}$};
\node at (1,-1.0) {$\scriptstyle p-2$};
\node at (2,-1.0) {$\scriptstyle p-1$};
\node at (3.05,-.35) {$\scriptstyle 0$};
\draw (-1.4,0) to (-.5,0);
\node at (0,0) {$\cdots$};
\node at (0,-.7) {$\cdots$};
\draw (.5,0) to (2,0) to (2.8,-.35) to (2,-.7) to (.5,-.7);
\draw (-2,0) to [out=180,in=180,looseness=2] (-2,-.7) to (-.5,-.7);
\draw[densely dotted,red,<->] (1,-.55) to (1,-.15);
\draw[densely dotted,red,<->] (2,-.55) to (2,-.15);
\draw[densely dotted,red,<->] (-1,-.55) to (-1,-.15);
\draw[densely dotted,red,<->] (-2,-.55) to (-2,-.15);
\node at (1.5,-.7) {$\scriptstyle <$};
\node at (-1.5,-.7) {$\scriptstyle <$};
\node at (1.5,0) {$\scriptstyle >$};
\node at (-1.5,0) {$\scriptstyle >$};
\node at (2.36,-0.16) {\rotatebox[origin=c]{-24}{$\scriptstyle>$}};
\node at (-2.41,-.38) {\rotatebox[origin=c]{-90}{$\scriptstyle <$}};
\node at (2.36,-.54) {\rotatebox[origin=c]{24}
{$\scriptstyle <\!\!<\!$}};
\minusdot{-2,0};
\plusdot{-1,0};
\plusdot{1,0};
\minusdot{2,0};
\plusdot{-2,-.7};
\minusdot{-1,-.7};
\minusdot{1,-.7};
\plusdot{2,-.7};
\emptydot{2.8,-.35};
\end{tikzpicture}\\
\hline
\end{array}
$$
\caption{Satake diagrams}\label{dynkintable}
\end{table}

In the last decade, there has been a great deal of work related to the Kolb--Letzter theory of quantum symmetric pairs (QSPs) \cite{letzter1,Kolb2014}.
We adopt here the terminology {\em iquantum group} introduced by 
Bao and Wang to refer to the coideal subalgebra in a QSP.
The influential papers \cite{BW18KL,ES} revealed that certain naturally occurring categories
in type B/C/D representation theory
are related from a combinatorial perspective to certain iquantum groups. 
These categories include blocks of the BGG category $\mathcal O$ for orthogonal and symplectic Lie algebras
(indeed, for orthosymplectic Lie superalgebras), the category of finite-dimensional representations of the orthosymplectic supergroup $OSp_{m|2n}(\C)$, 
and representations of the classical Brauer category.
The Satake diagrams of the relevant iquantum groups, displayed in \cref{dynkintable}, are of diagonal type A and of
{\em quasi-split type AIII}.
Wang refers to the latter as
``the new type A'' as there are analogs of many results, e.g., ones involving Schur--Weyl duality, which we used to think were peculiar to type A.

The existence of {\em icanonical bases} for iquantum groups established in \cite{BW18QSP} suggested that there might be a categorification of iquantum groups in the spirit of the Khovanov--Lauda--Rouquier categorification of quantum groups. For the Satake type AIII$_{2\infty}$,
the appropriate 2-category was constructed in \cite{BSWW}. Building on this, categorifications of all quasi-split iquantum groups whose underlying Cartan matrix is symmetric (and, conjecturally, most non-symmetric types too) were constructed in \cite{BWW}. These are denoted $\UU^\imath(\mathfrak{g})$, depending on an underlying Kac--Moody algebra together with an additional involution encoded combinatorially by the Satake diagram. For the ones in \cref{dynkintable}, this involution is indicated by the red arrows. Other decorations on the Satake diagrams there, such as the $+/-$ labels on the nodes, and the arrows on the edges, will be explained in \cref{iqg}---they are needed in order to specify the rather delicate 
parameter choices we use to define $\UU^\imath(\mathfrak{g})$.
We call $\UU^\imath(\mathfrak{g})$ a 
{\em 2-iquantum group}. This terminology is justified by the main result of \cite{BWW}, which shows that the graded 
Grothendieck ring of $\UU^\imath(\mathfrak{g})$ is isomorphic to the Bao--Wang modified integral form $\dot\U_\Z^\imath(\mathfrak{g})$
for $\U^\imath(\mathfrak{g})$.
Under the isomorphism, the class $[B_i \one_\lambda]$
of the generating 1-morphism $B_i \one_\lambda$ of the 2-iquantum group corresponds to the generator $b_i 1_\lambda$ of $\dot\U_\Z^\imath(\mathfrak{g})$.

In \cite{RS19},
Rui and Song introduced a strict monoidal category $\AB$
called the {\em affine Brauer category}, motivated by the observation that $\AB$ acts naturally on
many Abelian categories studied classically in type B/C/D representation theory.
We review the definition of $\AB$ 
in \cref{sabcat}. It is similar in spirit to the definition of $\Heis_\kappa$; actually, it is a little easier 
as there is only one generating object $B$---there is no orientation on the string diagrams used to represent morphisms in $\AB$---and there is no central charge $\kappa$ involved.
By a {\em Brauer categorification}, we mean a
locally finite Abelian category $\catR$ which is a strict left $\AB$-module category. By an {\em iquantum categorification}, we mean a locally finite Abelian category 
$$
\catR = \bigoplus_{\lambda \in X^\imath} \catR_\lambda
$$
which is a 2-representation of a 2-iquantum group $\UU^\imath(\mathfrak{g})$, where $X^\imath$ is the associated iweight lattice. We also require for each $\lambda \in X^\imath$
and $V \in \catR_\lambda$ that the set $\{i \in I\:|\:B_i V \neq 0\}$ is finite. We say that $\catR$ is {\em even} if $\catR_\lambda$ is trivial for all iweights of odd parity,
and {\em nilpotent} if the natural transformations defined by dots act nilpotently on any object.
The works \cite{BW18KL,ES} strongly suggested that there
should be a relationship similar to the one in \cite{HKM} connecting Brauer categorifications to iquantum categorifications for the appropriate Satake diagrams.

\medskip

In recent work \cite[Th.~A]{GRS-2026}, Gao, Rui and Song have constructed exactly such a bridge assuming that the ground field $\kk$ is of characteristic 0 and the spectrum of the Brauer categorification is $\frac{1}{2}+\Z$. This is the case that corresponds to the iquantum group of quasi-split type AIII$_{2\infty}$ in \cref{dynkintable}, exactly the one which was categorified in \cite{BSWW}. 
At the time that \cite{GRS-2026} appeared on the {\tt arXiv} (early September 2026), we had also been working on this problem for some time, and had independently developed a similar theory with no restrictions on the ground field (except $\operatorname{char}\ \kk \neq 2$) or on the spectrum. The relevant iquantum groups are all of the Satake types in \cref{dynkintable}. However, at that time, we had not yet checked all of the relations, so our proofs were incomplete. 
In fact, the relations interacting with the 0 node in these Satake diagrams turn out to be substantially more complicated.

In this article, we present our approach, which gives a complete solution to the problem in the general case.  
A key difference compared to Gao, Rui and Song's paper is that we have formulated the results using the unified general setup of \cite{BWW}, whereas Gao, Rui and Song used the older setup from 
\cite{BSWW} which only applies in type AIII$_{2\infty}$. This is essential in order to include all possible spectra 
in characteristic 0 and any case in positive characteristic. 
The principal difficulty in the Brauer/iquantum case 
compared to the Heisenberg/Kac--Moody case is that there is no counterpart of the
degenerate affine Hecke algebra or quiver Hecke algebra (``half''), so it is simply not possible to avoid the problem of relating (unoriented) cups and caps in the affine Brauer category to the ones in the 2-iquantum group.
Our formulae which do this explicitly involve natural transformations which we call {\em internal bubbles}; see \cref{bubbles}. They are closely related to the internal bubbles used to describe the relationship between leftward cups and caps in Heisenberg/Kac--Moody categorification which we worked out in \cite{BSW-update}.
In the type AIII$_{2\infty}$ case, Gao, Rui and Song discovered the same natural transformations, although their setup is somewhat different. Internal bubbles on strings labelled by $0$ 
of the Satake diagrams in \cref{dynkintable} are more difficult
to understand.

Our first main theorem, which is \cref{passportcor} in the main body of the text, is as follows:

\begin{alphatheorem}\label{thma}
Assuming that $\mathfrak{g}$ is of the Cartan type described above and
that the parameter choices for the 2-iquantum group 
$\UU^\imath(\mathfrak{g})$ are as 
specified in \cref{iqg},
any even 
nilpotent iquantum categorification $\catR = \bigoplus_{\lambda \in X^\imath} \catR_\lambda$
can be made into a Brauer categorification in a canonical way.
\end{alphatheorem}

We have not developed enough of the setup to be able to state \cref{thma} more precisely in this introduction, but of course that is done later in the paper.
In particular, we define $\AB$ in \cref{sabcat}, the relevant 2-iquantum group $\UU^\imath(\mathfrak{g})$ in \cref{iqg}, and develop some of the basic notions of  Brauer and iquantum categorifications in \cref{bcat,iqcat}, respectively; these results are also discussed in \cite{GRS-2026}.
The endofunctor $B$ on the Brauer side is related to the functors $B_i\:(i \in I)$ on the iquantum side by $B = \bigoplus_{i \in I} B_i$; specifically, 
$B_i$ is the generalized $i$-eigenspace of the endomorphism of $B$ defined by the dot generator of the affine Brauer category.
The explicit formulae relating the action of the generating morphisms of $\AB$ to the action of the generating 2-morphisms of $\UU^\imath(\mathfrak{g})$
are recorded in \cref{dotdef,capdef,cupdef,crossingdef}.
The generating function for dotted bubbles in $\AB$ 
acts via
\begin{equation*}
\begin{tikzpicture}[B,centerzero,scale=.9]
\draw (-0.25,0) arc(180:-180:0.25);
\node at (.55,0) {\color{black}$(u)$};
\end{tikzpicture}
=\ 
\prod_{i \in I}
\begin{tikzpicture}[iQ,centerzero,scale=.9]
\draw (-0.25,0) arc(180:-180:0.25);
\node at (-0.37,0) {\strandlabel{\ibar}};
\node at (.88,0) {\color{black}$(u-i)$};
\end{tikzpicture}.
\end{equation*}
We expected that this would be the case from the outset, and
this expectation led 
us to a key intermediate result, \cref{badbunny} below,
the proof of which is similar to the proof of
\cite[Th.~7.1]{BSW-update}.
This is how we discovered  the
appropriate definition of internal bubbles in this setting.

Our second main theorem, and the one which is most important from the point of view of applications to type B/C/D representation theory, shows that the construction in \cref{thma} is invertible, so that the notions of  Brauer categorification and even nilpotent iquantum categorification are equivalent. This is \cref{btoq-cor} in the main text.

\begin{alphatheorem}\label{thmb}
Any Brauer categorification $\catR$ can be made into an even nilpotent iquantum categorification in a canonical way. Moreover, this construction and the construction in \cref{thma} are inverse to each other.
\end{alphatheorem}

\cref{thmb} is deduced from a stronger result, \cref{btoq-thm},
which identifies $\AB$ with a topologically dense monoidal subcategory
of a topological 
strict monoidal category denoted $\hatAB$.
The latter is constructed 
by first passing from $\UU^\imath$ to
a completion $\hatUU^\imath$,
then collapsing $\hatUU^\imath$ to a monoidal category
by replacing its morphism categories with products over certain pairs of objects.
This construction is explained in \cref{completediq}.
Our approach to 
\cref{thma} also involves a
stronger result, \cref{passport}, which establishes the existence of a
strict $\kk$-linear 
monoidal functor $\Psi\colon\AB \rightarrow \hatAB$.
Then, to prove \cref{btoq-thm}, we show that 
$\Psi$ is faithful and topologically dense. Finally, the proof of \cref{thmb} is completed by showing that
the action of $\AB$ on any Brauer categorification
extends by continuity
to an action of $\hatAB$, hence, to $\UU^\imath$
from which $\hatAB$ was constructed in the first place.
This general strategy 
cuts down considerably on the number of relations that need to be checked explicitly compared to the approach taken in \cite{GRS-2026}:
we only need to make such calculations in one direction in order to prove \cref{passport}.

In the main body of the article, we refer to the categorifications discussed above as {\em Abelian} Brauer and iquantum categorifications. There are also analogous notions of
{\em Karoubian} Brauer and iquantum categorifications, 
in which the underlying category is merely required to be a
{\em locally finite Karoubian category}, that is, 
an additive idempotent-complete $\kk$-linear category whose morphism spaces are finite-dimensional.
On both the Brauer side and on the iquantum side,
there are universal families of examples of Karoubian categorifications:
for a monic polynomial $f(u) \in \kk[u]$ of degree $\beta \geq 0$, 
there is a corresponding
{\em cyclotomic quotient} of $\AB$ denoted $\cB(f)$
(\cref{brauercyclotomic});
for monomials $f_i(u) \coloneq u^{\beta_i}$ for $\beta_i \geq 0$ with $\beta_i=0$ for all but finitely many $i$,
there is a corresponding {\em cyclotomic quotient}
of $\UU^\imath(\mathfrak{g})$ denoted $\cC(f_i\,|\,i \in I)$
(\cref{iqcyclotomic}).
There are analogs of Theorems A and B in the Karoubian case.
In the last section of the article, we apply these
to prove that $\cB(f)$ and $\cC(f_i\,|\,i \in I)$ are equivalent categories when $f(u) = \prod_{i \in I} f_i(u-i)$, i.e., $\beta_i$ is the multiplicity of $i$ as a root of $f(u)$.
This result is the Brauer/iquantum counterpart of the isomorphism theorem between GCQs in \cite{HKM}; see \cref{cyciso}. 
Gao, Rui and Song also prove such a theorem in the AIII$_{2\infty}$ case in \cite[Th.~B]{GRS-2026}.

For example, this result implies the existence of an equivalence between the most important cyclotomic quotient of the affine Brauer category, the Deligne category $\underline{\operatorname{Re}}\!\operatorname{p} O_\delta$
at parameter $\delta \in \C$, 
and an iquantum cyclotomic quotient.
Indeed, 
$\underline{\operatorname{Re}}\!\operatorname{p} O_\delta$
is the additive Karoubi envelope of 
the classical Brauer category $\catB_\delta$ over $\C$ with 
$\begin{tikzpicture}[B,centerzero,scale=.9]
\draw (-0.25,0) arc(180:-180:0.25);
\end{tikzpicture} = \delta \id_\one$.
By Rui and Song's theory, the additive Karoubi envelope of $\catB_\delta$
is isomorphic to the level one cyclotomic quotient 
$\cB\big(u-(1-\delta)/2\big)$. It follows for any $\delta \in \C$ that
$\underline{\operatorname{Re}}\!\operatorname{p} O_\delta$
is equivalent to 
$\cC\big(u^{\delta_{i,(1-\delta)/2}}\,\big|\,i \in I\big)$, with the result for $\delta \in 2\Z$ following already from \cite{GRS-2026}.
These results solve a problem which has been asked by several different groups since the time of \cite{BK-iso}, e.g., see \cite{Li-2014}. The interesting $\Z$-grading on the 
Brauer category revealed by the equivalence was constructed 
in \cite{ES-2016} by a different approach.

We hope there will be further applications to the general theory of cyclotomic quotients of $\AB$ which has been developed in other work of Gao, Rui and Song, e.g., see \cite{GRS-2023,GRS-2023b}. The explicit connection to iquantum groups and their combinatorics which this paper and \cite{GRS-2026} establish should allow for a more conceptual interpretation of these earlier results, and it provides further impetus for the development of the theory of iquantum groups and their categorifications.

\medskip
\noindent
{\em Preliminaries.}
Throughout the paper we work over an algebraically closed field $\kk$ of characteristic $p \neq 2$. 
We fix a choice of a square root of $-1$, denoting it by
$\sqrt{-1}$.
As usual, ``$\kk$-linear'' indicates enrichment 
in the monoidal category of vector spaces. We use
the word ``graded'' to indicate enrichment 
in the monoidal category of $\Z$-graded vector spaces.

We let $I \coloneq \kk$. This set will also be the labelling set for simple roots in a certain root system as in \cref{dynkintable}; this will be explained at the start of \cref{iqg}. We reserve the letter
$h$ for $\frac{1}{2} \in \kk$. 
In string diagrams, we will often label strings by elements $i,j,k,\dots \in I$, and will use $\Ibar, \Jbar, \Kbar,\dots$ instead of $-i,-j,-k,\dots$ for clarity in diagrams.

We will often work with formal Laurent series $f(u)
= \sum_{n \in \Z} a_n u^n \in A\lround u^{-1}\rround$ 
with coefficients in some algebra $A$.
We use the notation $[f(u)]_{u:n}$ to denote the $u^n$-coefficient $a_n$, $[f(u)]_{u:\geq 0}$ for 
the polynomial part $\sum_{n \geq 0} a_n u^n$, and so on.
The formal power series
$$
\tfrac{1}{u-x} = 
u^{-1} + x u^{-2}+x^2 u^{-3}+\cdots \in A[x]\lround u^{-1}\rround
$$
will be particularly important.
Some basic tricks:
for any polynomial $f(x) \in A[x]$, we have that
\begin{align}\label{trick}
\frac{f(x)}{u-x} &= \left[\frac{f(u)}{u-x}\right]_{u:<0},&
f(x) &= \left[\frac{f(u)}{u-x}\right]_{u:-1}.
\end{align}

Let $\partial_{x,y}\colon A[x,y]\rightarrow A[x,y]$ be the
{\em Demazure operator}, the
linear map defined by
\begin{equation}\label{demazureoperator}
\partial_{x,y} f(x,y) \coloneq \frac{f(x,y)-f(y,x)}{x-y}.
\end{equation}
We also use the notation $\partial_x$ to denote the one-variable Demazure operator defined by
\begin{equation}
\partial_x f(x) \coloneq \frac{f(x)-f(-x)}{2x}
\end{equation}
for $f(x) \in A[x]$. 
Note that $\partial_{x,y} = -\partial_{y,x}$.
We extend $\partial_{x,y}$ to formal series in an auxiliary variable $u$ in the obvious way.
Using the second identity in \cref{trick} together with partial fractions, it follows that
\begin{align}\label{trick2}
\partial_{x,y}\Big(\frac{f(x)}{u-x}\Big)
&=\left[\frac{f(u)}{(u-x)(u-y)}\right]_{u:<0}\!\!\!\!,
&\partial_{x,y} f(x)
&=\left[\frac{f(u)}{(u-x)(u-y)}\right]_{u:-1}\!\!\!\!
\end{align}
in $A[x,y]\lround u^{-1}\rround$.
Later, we will also apply Demazure operators to formal power series; 
the extended versions are defined in the same way.

We will eventually need the following version of
Hensel's lemma:

\begin{lem}[{\cite[Cor.~2.4]{HKM}}]\label{hensel}
Let $A$ be a finite-dimensional commutative local $\kk$-algebra with unique 
maximal ideal $\mathfrak{m}$, so that $A/\mathfrak{m}=\kk$.
Let $f(u) \in A[u]$ be a monic polynomial 
and $\bar f(u) \in \kk[u]$ be its
reduction modulo $\mathfrak{m}$.
Let $\beta_i$ be the multiplicity of $i$ as a root 
of $\bar f(u)$.
Then there are unique
monic $f_i(u) \in A[u]$, all but finitely many of which are 1,
such that $\bar f_i(u) = u^{\beta_i}$ and
$f(u) = \prod_{i \in I} f_i(u-i)$.
\end{lem}

We call the polynomial $f_i(u)$ defined by \cref{hensel}
the {\em local factor} of $f(u)$ at $u=i$.
Similarly, we call the rational function $R_i(u)$ appearing in the following corollary
the {\em local factor} of $R(u)$ at $u=i$.

\begin{cor}\label{localfactors}
For $A$ and $\mathfrak{m}$ as in the preceding lemma,
let $R(u) = \frac{f(u)}{g(u)}\in A(u)$ be a rational function with $f(u), g(u) \in A[u]$ monic,
and 
$\bar R(u) \in \kk(u)$ be its reduction modulo $\mathfrak{m}$.
Letting $\ord_i(R) \in \Z$ be the order of $\bar R(u)$ at $u=i$ (the multiplicity of $i$ as a zero, or minus its multiplicity as a pole), 
there are unique $R_i(u) \in u^{\ord_i(R)} + u^{\ord_i(R)-1}\mathfrak{m}[u^{-1}]$ such that
$R_i(u) = 1$ for all but finitely many $i$ and
$R(u) = \prod_{i \in I} R_i(u-i)$.
\end{cor}

\begin{proof}
For existence, apply \cref{hensel} to $f(u)$ and $g(u)$
to obtain $f_i(u)$ and $g_i(u)$, then set
$R_i(u) \coloneq f_i(u) / g_i(u)$, which is in 
$u^{\ord_i(R)} + u^{\ord_i(R)-1}\mathfrak{m}[u^{-1}]$
as $g_i(u) = u^n(1 - k(u))$ for $n \coloneq \deg g_i(u)$ and $k(u) \in u^{-1}\mathfrak{m}[u^{-1}]$,
and $(1-k(u))^{-1} = 1 + k(u) + \cdots + k(u)^N$ for some $N$ as $\mathfrak{m}$ is nilpotent.
Uniqueness follows from the uniqueness assertion in \cref{hensel} by an argument which involves cross multiplying to clear denominators.
\end{proof}

\medskip
\noindent
{\em AI Statement.} The foundational material
about Brauer and iquantum categorifications
in \cref{bcat,iqcat} was worked out using BI; the first author gave a talk about this at a
\href{https://drive.google.com/file/d/1-iHMk7RxpZtZIJrqqRaxV7XSn1BgBEz_/view}{conference in Uppsala}
in August 2026. The definition of the internal bubbles in \cref{bubbles}, the formulae in \cref{passport}, and the choice of parameters for the 2-iquantum group in \cref{iqg} were 
also in place at that time, and we had almost finished checking the relations needed to prove \cref{passport}.
However, we were struggling to check the Reidemeister III relation in \cref{lem:riii}, especially 
when the 0 node was involved. We started using AI systematically in September, at which point ChatGPT suggested the amazingly short proof given 
below. We also used both ChatGPT and Anthropic's Claude extensively for the difficult convergence arguments in \cref{bccomplete}, although their suggestions were rewritten and expanded by the authors. 
All of the other writing was done by the authors, 
but ChatGPT and Claude were used for proofreading purposes.

\setcounter{section}{1}

\section{The affine Brauer category}\label{sabcat}

We begin by recalling Rui and Song's 
definition of the affine Brauer category.
We use string calculus with the standard conventions from the recent categorification literature.

\begin{defin}[{\cite[Def.~1.2]{RS19}}] \label{ABdef}
The \emph{affine Brauer} category $\AB$ is the strict $\kk$-linear monoidal category generated by an object $B$ with identity endomorphism 
$\begin{tikzpicture}[B,anchorbase]\draw(0,0) to (0,.4);\end{tikzpicture}$
and the morphisms listed in \cref{ABgens}.
\begin{table}
$$
\begin{array}{|rl|}
\hline
&\text{Generator}\\
\hline
\raisebox{3mm}{$\phantom{X}$}
\begin{tikzpicture}[B,anchorbase,scale=1.5]
\draw[-] (-0.15,-0.15) -- (-0.15,0) arc(180:0:0.15) -- (0.15,-0.15);
\end{tikzpicture}
&\hspace{-3mm}\colon B \otimes B \to \one\\
\raisebox{3.5mm}{$\phantom{X}$}
\begin{tikzpicture}[B,anchorbase,scale=1.5]
\draw[-] (-0.15,0.15) -- (-0.15,0) arc(180:360:0.15) -- (0.15,0.15);
\end{tikzpicture}
&\hspace{-3mm}\colon \one \to B \otimes B\\
\raisebox{3mm}{$\phantom{X}$}
\begin{tikzpicture}[B,centerzero,scale=1.3]
\draw[-] (0.2,-0.2) -- (-0.2,0.2);
\draw[-] (-0.2,-0.2) -- (0.2,0.2);
\end{tikzpicture} 
&\hspace{-3mm}\colon B \otimes B \to B \otimes B\ \ \\
\raisebox{-3mm}{$\phantom{X}$}\!\!\!\raisebox{3mm}{$\phantom{X}$}
\begin{tikzpicture}[B,centerzero,scale=1.5]
\draw (0,-0.2) -- (0,0.2);
\closeddot{0,0};
\end{tikzpicture} 
&\hspace{-3mm}\colon B \to B\\
\hline
\end{array}
$$
\caption{Generating morphisms for the affine Brauer category}\label{ABgens}
\end{table}
The generating morphisms are 
subject to the following relations:
\begin{align}
\label{zigzag}
\begin{tikzpicture}[B,centerzero,scale=1.4]
\draw (-0.3,0.4) -- (-0.3,0) arc(180:360:0.15) arc(180:0:0.15) -- (0.3,-0.4);
\end{tikzpicture}\ 
&=\ 
\begin{tikzpicture}[B,centerzero,scale=1.4]
\draw (0,-0.4) -- (0,0.4);
\end{tikzpicture}\ ,&
\begin{tikzpicture}[B,centerzero,scale=1.4]
\draw (-0.3,-0.4) -- (-0.3,0) arc(180:0:0.15) arc(180:360:0.15) -- (0.3,0.4);
\end{tikzpicture}&=\ \begin{tikzpicture}[B,centerzero,scale=1.4]
\draw (0,-0.4) -- (0,0.4);
\end{tikzpicture}\ ,\\
\label{dotslide}
\begin{tikzpicture}[B,anchorbase,scale=1.5]
\draw (-0.2,-0.3) -- (-0.2,-0.1) arc(180:0:0.2) -- (0.2,-0.3);
\closeddot{-0.2,-0.1};
\end{tikzpicture}\ 
&=\  -\ \begin{tikzpicture}[B,anchorbase,scale=1.5]
\draw (-0.2,-0.3) -- (-0.2,-0.1) arc(180:0:0.2) -- (0.2,-0.3);
\closeddot{0.2,-0.1};
\end{tikzpicture}\ ,&\begin{tikzpicture}[B,anchorbase,scale=1.5]
\draw (-0.2,0.3) -- (-0.2,0.1) arc(-180:0:0.2) -- (0.2,0.3);
\closeddot{-0.2,0.1};
\end{tikzpicture}\ 
&=\  -\ \begin{tikzpicture}[B,anchorbase,scale=1.5]
\draw (-0.2,0.3) -- (-0.2,0.1) arc(-180:0:0.2) -- (0.2,0.3);
\closeddot{0.2,0.1};
\end{tikzpicture}\ ,\\
\label{pitchfork}
\begin{tikzpicture}[B,baseline=-2.9mm,scale=1.5]
\draw (-0.2,-0.3) -- (-0.2,-0.1) arc(180:0:0.2) -- (0.2,-0.3);
\draw (-0.3,0.3) \braiddown (0,-0.3);
\end{tikzpicture}\ 
&=\ 
\begin{tikzpicture}[B,baseline=-2.9mm,scale=1.5]
\draw (-0.2,-0.3) -- (-0.2,-0.1) arc(180:0:0.2) -- (0.2,-0.3);
\draw (0.3,0.3) \braiddown (0,-0.3);
\end{tikzpicture}, &
\begin{tikzpicture}[B,baseline=1mm,scale=1.5]
\draw (-0.2,0.3) -- (-0.2,0.1) arc(-180:0:0.2) -- (0.2,0.3);
\draw (-0.3,-0.3) \braidup (0,0.3);
\end{tikzpicture}\ 
&=\ 
\begin{tikzpicture}[B,baseline=1mm,scale=1.5]
\draw (-0.2,0.3) -- (-0.2,0.1) arc(-180:0:0.2) -- (0.2,0.3);
\draw (0.3,-0.3) \braidup (0,0.3);
\end{tikzpicture}\ ,\\
\label{demazure}
\hspace{-25mm}\raisebox{3mm}{$\phantom{X}$}
\begin{tikzpicture}[B,centerzero,scale=1.2]
\draw (-0.35,-0.35) -- (0.35,0.35);
\draw (0.35,-0.35) -- (-0.35,0.35);
\closeddot{-0.17,0.17};
\end{tikzpicture}
-
\begin{tikzpicture}[B,centerzero,scale=1.2]
\draw (-0.35,-0.35) -- (0.35,0.35);
\draw (0.35,-0.35) -- (-0.35,0.35);
\closeddot{0.17,-0.17};
\end{tikzpicture}
&=\ 
\begin{tikzpicture}[B,centerzero,scale=1.2]
\draw (-0.2,-0.35) -- (-0.2,0.35);
\draw (0.2,-0.35) -- (0.2,0.35);
\end{tikzpicture}\ 
-\ \begin{tikzpicture}[B,centerzero,scale=1.2]
\draw (-0.2,-0.35) -- (-0.2,-0.3) arc(180:0:0.2) -- (0.2,-0.35);
\draw (-0.2,0.35) -- (-0.2,0.3) arc(180:360:0.2) -- (0.2,0.35);
\end{tikzpicture}\ ,&
\begin{tikzpicture}[B,centerzero,scale=1.2]
\draw (-0.35,0.35) -- (0.35,-0.35);
\draw (0.35,0.35) -- (-0.35,-0.35);
\closeddot{-0.17,-0.17};
\end{tikzpicture}
-
\begin{tikzpicture}[B,centerzero,scale=1.2]
\draw (-0.35,0.35) -- (0.35,-0.35);
\draw (0.35,0.35) -- (-0.35,-0.35);
\closeddot{0.17,0.17};
\end{tikzpicture}&=\ \begin{tikzpicture}[B,centerzero,scale=1.2]
\draw (-0.2,-0.35) -- (-0.2,0.35);
\draw (0.2,-0.35) -- (0.2,0.35);
\end{tikzpicture}\ 
-\ \begin{tikzpicture}[B,centerzero,scale=1.2]
\draw (-0.2,-0.35) -- (-0.2,-0.3) arc(180:0:0.2) -- (0.2,-0.35);
\draw (-0.2,0.35) -- (-0.2,0.3) arc(180:360:0.2) -- (0.2,0.35);
\end{tikzpicture}\ ,\\
\label{ri}
\raisebox{3mm}{\phantom{X}}\!\!
\begin{tikzpicture}[B,baseline=-3.9mm,scale=1.5]
\draw (-0.15,-0.4) to[out=60,in=-90] (0.15,0) arc(0:180:0.15) to[out=-90,in=120] (0.15,-0.4);
\end{tikzpicture}\ 
&\ =\  \begin{tikzpicture}[B,anchorbase,scale=1.8]
\draw[-] (-0.15,-0.15) -- (-0.15,0) arc(180:0:0.15) -- (0.15,-0.15);
\end{tikzpicture}\ ,&
\begin{tikzpicture}[B,baseline=2.2mm,scale=1.5]
\draw (-0.15,0.4) to[out=-60,in=90] (0.15,0) arc(0:-180:0.15) to[out=90,in=-120] (0.15,0.4);
\end{tikzpicture}\ 
&\ =\ \begin{tikzpicture}[B,anchorbase,scale=1.8]
\draw[-] (-0.15,0.15) -- (-0.15,0) arc(-180:0:0.15) -- (0.15,0.15);
\end{tikzpicture}\ ,\\
\label{rii-iii}
\begin{tikzpicture}[B,centerzero,scale=1.4]
\draw (0.2,-0.4) to[out=135,in=down] (-0.15,0) to[out=up,in=225] (0.2,0.4);
\draw (-0.2,-0.4) to[out=45,in=down] (0.15,0) to[out=up,in=-45] (-0.2,0.4);
\end{tikzpicture}\ 
&=\ 
\begin{tikzpicture}[B,centerzero,scale=1.4]
\draw (-0.2,-0.4) -- (-0.2,0.4);
\draw (0.2,-0.4) -- (0.2,0.4);
\end{tikzpicture}\ ,&
\begin{tikzpicture}[B,centerzero,scale=1.4]
\draw (0.4,-0.4) -- (-0.4,0.4);
\draw (0,-0.4) to[out=135,in=down] (-0.32,0) to[out=up,in=225] (0,0.4);
\draw (-0.4,-0.4) -- (0.4,0.4);
\end{tikzpicture}
&= \begin{tikzpicture}[B,centerzero,scale=1.4]
\draw (0.4,-0.4) -- (-0.4,0.4);
\draw (0,-0.4) to[out=45,in=down] (0.32,0) to[out=up,in=-45] (0,0.4);
\draw (-0.4,-0.4) -- (0.4,0.4);
\end{tikzpicture}\ .
\end{align}
\end{defin}

\begin{rem}
Our formulation of \cref{ABdef} is convenient but not minimal:
in \cref{dotslide,pitchfork,demazure,ri},
the second relation follows from the first one and the earlier relations.
Sketch proof for \cref{dotslide,pitchfork,ri}: attach a pair of cups at the bottom of the first relation, then simplify using the zig-zag relation; we refer to this sort of argument as rotation by 180$^\circ$.
Sketch proof for \cref{demazure}: attach a cup at the bottom and a cap at the top of one of the strings in the first relation, then use the zig-zag relation; this is rotation by 90$^\circ$.
See \cite{RS19} for more details, though note that our definition differs from \cite[Def.~1.2]{RS19} by
the sign of the dot.
\end{rem}

We call the morphism $\begin{tikzpicture}[B,centerzero,scale=1.2]\draw (0,-0.2) -- (0,0.2);\closeddot{0,0};\end{tikzpicture}$ a \emph{dot}.
It is clear from the defining relations that
$\AB$ has symmetries
$\tT\colon\AB^\op \rightarrow \AB$ defined by reflecting string diagrams in a horizontal axis, and
$\tR\colon\AB^\rev \rightarrow \AB$ defined by reflecting string diagrams in a vertical axis then multiplying by $(-1)^d$ where $d$ is the total number of dots in the diagram.

For a polynomial $f(x) = \sum_{m=0}^n c_m x^m \in \kk[x]$, 
we use the notation
\begin{align*}
\begin{tikzpicture}[B,centerzero,scale=1.5]
\draw (0,-0.2) -- (0,0.2);
\multcloseddot{0,0}{east}{m};
\end{tikzpicture}
&\coloneq\left( 
\begin{tikzpicture}[B,centerzero,scale=1.5]
\draw (0,-0.2) -- (0,0.2);
\closeddot{0,0};
\end{tikzpicture}
\right)^{\circ m},&
\begin{tikzpicture}[B,centerzero,scale=1.5]
\draw (0,-0.2) -- (0,0.2);
\pin{0,0}{.5,0}{f(x)};
\end{tikzpicture}
=\begin{tikzpicture}[B,centerzero,scale=1.5]
\draw (0,-0.2) -- (0,0.2);
\pin{0,0}{-.5,0}{f(x)};
\end{tikzpicture}
&\coloneq
\sum_{m=0}^n c_m
\begin{tikzpicture}[B,centerzero,scale=1.5]
\draw (0,-0.2) -- (0,0.2);
\multcloseddot{0,0}{west}{m};
\end{tikzpicture}\ ,
\end{align*}
referring to the latter as {\em pinning} the polynomial $f(x)$ to the marked point on the string.
More generally, one can consider pins involving
polynomials $f(x) = \sum_{m=0}^n c_m x^m$ whose coefficients are elements of $\End_{\AB}(\one)$. These endomorphisms become bubbles placed in the 2-cell in which the pinhead is placed. 
For example,
\begin{align*}
\begin{tikzpicture}[B,centerzero,scale=1.5]
\draw (0,-0.2) -- (0,0.2);
\pin{0,0}{-.5,0}{f(x)};
\end{tikzpicture}
&\coloneq
\sum_{m=0}^n
\begin{tikzpicture}[B,centerzero,scale=1.5]
\draw (0,-0.2) -- (0,0.2);
\multcloseddot{0,0}{west}{m};
\node[rectangle,rounded corners,draw,pinhead,inner sep=2.5pt] at (-0.4,0) {$\color{black}c_m$};
\end{tikzpicture}\ ,&
\begin{tikzpicture}[B,centerzero,scale=1.5]
\draw (0,-0.2) -- (0,0.2);
\pin{0,0}{.4,0}{f(x)};
\end{tikzpicture}
&\coloneq
\sum_{m=0}^n
\begin{tikzpicture}[B,centerzero,scale=1.5]
\draw (0,-0.2) -- (0,0.2);
\multcloseddot{0,0}{east}{m};
\node[rectangle,rounded corners,draw,pinhead,inner sep=2.5pt] at (0.4,0) {$\color{black}c_m$};
\end{tikzpicture}\ .
\end{align*}
Now, the endomorphisms defined by the left pin and the right pin could be different!

Often, we will use power series in $\kk[x]\lround u^{-1} \rround$
pinned to dots in string diagrams to denote generating functions of morphisms in $\AB$.
We introduce the shorthand notation
\begin{align}\label{dotgenfuncs}
\begin{tikzpicture}[B,centerzero,scale=1.5]
\draw (0,-0.3) -- (0,0.3);
\circled{0,0}{u};
\end{tikzpicture}
&\coloneq
\begin{tikzpicture}[B,centerzero,scale=1.5]
\draw (0,-0.3) -- (0,0.3);
\pin{0,0}{0.5,0}{\frac{1}{u-x}};
\end{tikzpicture}\ ,&
\begin{tikzpicture}[B,centerzero,scale=1.5]
\draw (0,-0.3) -- (0,0.3);
\circledbar{0,0}{u};
\end{tikzpicture}
&\coloneq
\begin{tikzpicture}[B,centerzero,scale=1.5]
\draw (0,-0.3) -- (0,0.3);
\pin{0,0}{0.5,0}{\frac{1}{u+x}};
\end{tikzpicture}\ .
\end{align}
These are fixed by $\tT$ and interchanged by $\tR$.
Also replacing $u$ by $-u$ in one of them produces {\em minus} the other one.
We have that 
\begin{align}\label{whatfunB}
\begin{tikzpicture}[B,centerzero,scale=1]
\draw[-] (-.3,-.25)  to [out=90,in=90,looseness=3] (.3,-.25);
\circled{-.27,-.03}{u};
\end{tikzpicture}\ &=\ 
\begin{tikzpicture}[B,centerzero,scale=1]
\draw[-] (-.3,-.25)  to [out=90,in=90,looseness=3] (.3,-.25);
\circledbar{.27,-.03}{u};
\end{tikzpicture}\ ,&
\begin{tikzpicture}[B,centerzero,scale=1]
\draw[-] (-.3,-.25)  to [out=90,in=90,looseness=3] (.3,-.25);
\circledbar{-.27,-.03}{u};
\end{tikzpicture}\ &=\ 
\begin{tikzpicture}[B,centerzero,scale=1]
\draw[-] (-.3,-.25)  to [out=90,in=90,looseness=3] (.3,-.25);
\circled{.27,-.03}{u};
\end{tikzpicture}\ ,\\\label{gendotslideB}  
\begin{tikzpicture}[B,centerzero,scale=1.3]
\draw[-] (-0.3,-0.3)  -- (0.3,0.3);
\draw[-] (0.3,-0.3)  -- (-0.3,0.3);
\circled{-0.15,0.15}{u};
\end{tikzpicture}
\ -
\begin{tikzpicture}[B,centerzero,scale=1.3]
\draw[-] (-0.3,-0.3) -- (0.3,0.3);
\draw[-] (0.3,-0.3)  -- (-0.3,0.3);
\circled{0.15,-0.15}{u};
\end{tikzpicture}\ &=
 \ 
\begin{tikzpicture}[B,centerzero,scale=1.3]
\draw[-] (-0.2,-0.3)  -- (-0.2,0.3);
\draw[-] (0.2,-0.3)  -- (0.2,0.3);
\circled{-0.2,0}{u};
\circled{0.2,0}{u};
\end{tikzpicture} -
\ 
\begin{tikzpicture}[B,centerzero,scale=1.3]
\draw[-] (-0.2,-0.3)  to [looseness=2.3,out=90,in=90] (0.2,-0.3);
\draw[-] (0.2,0.3)  to [looseness=2.3,out=-90,in=-90] (-0.2,0.3);
\circled{0.16,-0.14}{u};
\circled{-0.16,0.14}{u};
\end{tikzpicture}\ ,&
\begin{tikzpicture}[B,centerzero,scale=1.3]
\draw[-] (-0.3,-0.3) -- (0.3,0.3);
\draw[-] (0.3,-0.3) -- (-0.3,0.3);
\circledbar{-0.15,0.15}{u};
\end{tikzpicture}
-
\begin{tikzpicture}[B,centerzero,scale=1.3]
\draw[-] (-0.3,-0.3) -- (0.3,0.3);
\draw[-] (0.3,-0.3)  -- (-0.3,0.3);
\circledbar{0.15,-0.15}{u};
\end{tikzpicture}\ &=
\ 
\begin{tikzpicture}[B,centerzero,scale=1.3]
\draw[-] (-0.2,-0.3)  to [looseness=2.3,out=90,in=90] (0.2,-0.3);
\draw[-] (0.2,0.3)  to [looseness=2.3,out=-90,in=-90] (-0.2,0.3);
\circledbar{0.16,-0.14}{u};
\circledbar{-0.16,0.14}{u};
\end{tikzpicture}\ -
\ 
\begin{tikzpicture}[B,centerzero,scale=1.3]
\draw[-] (-0.2,-0.3) -- (-0.2,0.3);
\draw[-] (0.2,-0.3) -- (0.2,0.3);
\circledbar{-0.2,0}{u};
\circledbar{0.2,0}{u};
\end{tikzpicture}\ .
\end{align}
Applying $\tT$ shows that
the mirror images of these relations in a horizontal axis also hold.

Recall that $h = \frac{1}{2}\in \kk$.
We will work with the following generating function for dotted bubbles:
\begin{align} \label{eatz}
\begin{tikzpicture}[B,centerzero]
\draw (-0.25,0) arc(180:-180:0.25);
\node at (.52,0) {\color{black}$(u)$};
\end{tikzpicture} &\coloneq \frac{u-h}{u}
\id_\one+\ 
\begin{tikzpicture}[B,centerzero,scale=1.35]
\draw (0,0)++(0,0.2) arc(90:-270:0.2);
\circled{0.2,0}{u};
\end{tikzpicture}
 \in \id_\one + u^{-1} \End_{\AB}(\one)\llbracket u^{-1}\rrbracket\ .
\end{align}
It is the generating function used by Gao, Rui and Song in \cite{GRS-2023b} multiplied by $u^{-1}$, and it is the generating function denoted 
$\begin{tikzpicture}[B,centerzero,scale=1.35]
\draw (0,0)++(0,0.2) arc(90:-270:0.2);
\node at (0,0) {\dotlabel{u}};
\filldraw (0,0)++(0.2,0) circle (1.5pt);
\end{tikzpicture}$ in \cite{SW2026} multiplied by $\frac{u-h}{u}$.
Since we have merely rescaled, the relations from 
\cite[Lem.~4.1--4.2]{GRS-2023b} and
\cite[Prop.~2.5]{SW2026} are easily translated to obtain the following relations:
\begin{align}\label{curlsup1}
\begin{tikzpicture}[B,centerzero]
\draw (-0.25,0) arc(180:-180:0.25);
\node at (.52,0) {\color{black}$(u)$};
\end{tikzpicture}
\begin{tikzpicture}[B,centerzero]
\draw (-0.25,0) arc(180:-180:0.25);
\node at (.67,0) {\color{black}$(-u)$};
\end{tikzpicture}&=\ 
\frac{(u+h)(u-h)}{u^2}\id_\one,\\\label{curlsup2}
\begin{tikzpicture}[B,centerzero]
\draw (-0.25,0) arc(180:-180:0.25);
\node at (.52,0) {\color{black}$(u)$};
\draw (1,-0.5) -- (1,0.5);
\pin{1,0}{2.6,0}{\frac{(u-x+1)(u-x-1)}{(u-x)^2}};
\end{tikzpicture}
\ &=
\begin{tikzpicture}[B,centerzero,scale=1]
\draw (0.5,-0.5) -- (0.5,0.5);
\pin{0.5,0}{-1.1,0}{\frac{(u+x+1)(u+x-1)}{(u+x)^2}};
\draw (0.8,0) arc(180:-180:0.25);
\node at (1.56,0) {\color{black}$(u)$};
\end{tikzpicture},\\\label{curlsup3}
\begin{tikzpicture}[B,anchorbase,scale=1.1]
\draw (0,-0.5) to[out=up,in=180] (0.3,0.2) to[out=0,in=up] (0.45,0) to[out=down,in=0] (0.3,-0.2) to[out=180,in=down] (0,0.5);
\circled{0.45,0}{u};
\end{tikzpicture}
&= 
\ \begin{tikzpicture}[B,centerzero,scale=1.1]
\draw (0,-0.5) -- (0,0.5);
\circledbar{0,0}{u};
\draw (0.25,0) arc(180:-180:0.2);
\node at (.89,0) {\color{black}$(u)$};
\end{tikzpicture}
- \frac{1}{2u}
\begin{tikzpicture}[B,centerzero,scale=1.1]
\draw (0,-0.5) -- (0,0.5);
\circled{0,0}{u};
\end{tikzpicture}
\ .
\end{align}
Note also that the
bubble generating function is invariant under both 
of the symmetries 
$\tT$ and $\tR$.

The fundamental theorem about $\AB$ proved by Rui and Song in \cite{RS19}
gives explicit bases for its morphism spaces.
To state this briefly,
let $Z \coloneq \End_{\AB}(\one)$. This is a commutative
algebra generated by the dotted bubbles 
\begin{equation}\label{brauerbubbles}
z_r \coloneq 
\begin{tikzpicture}[B,centerzero,scale=1.35]
\draw (0,0)++(0,0.2) arc(90:-270:0.2);
\multcloseddot{.2,0}{west}{r};
\end{tikzpicture}
\end{equation}
for all $r \geq 0$.
Now take $m,n \geq 0$.
We make $\Hom_{\AB}(B^{\otimes n}, B^{\otimes m})$
into a right $Z$-module so that $z \in Z$ acts on 
a morphism $\theta$ by 
$\theta z \coloneq \theta \otimes z$.
An {\em $m \times n$-shape} means a 
(magenta) string diagram $D$ for a morphism
in $\Hom_{\AB}(B^{\otimes n}, B^{\otimes m})$
with no dots and no bubbles. Such shapes exist if and only if
$m \equiv n \pmod{2}$.
We say that a shape is {\em reduced} if there are no zig-zags, bigons or loops, and
two shapes are {\em equivalent} if they define the same combinatorial matching between their boundary points. 
Let $\shape(m\times n)$ be a fixed choice of representatives for the equivalence classes of reduced $m \times n$ shapes.
Also, for $D \in \shape(m\times n)$ and $l\coloneq\frac{m+n}{2}$, the number of strings in $D$, we fix a choice of normally ordered points $p_1,\dots,p_l$ away from crossings and critical points, one on each string.
For 
$f \in \kk[x_1,\dots,x_l]$, we let
$D(f)\colon B^{\otimes n}\rightarrow B^{\otimes m}$
be the morphism
obtained from $D$ by pinning $f$ to the points $p_1,\dots,p_l$.

\begin{theo}[{\cite[Th.~B]{RS19}}]\label{RSbasis}
The algebra $Z$ is freely generated by
its even generators $z_{2r}\:(r \geq 0)$.
Moreover,
for $m, n \geq 0$,
the morphism space $\Hom_{\AB}(B^{\otimes n},
B^{\otimes m})$ is 0 unless
$m \equiv n \pmod{2}$, in which case it
is free as a right $Z$-module
with basis
\begin{equation}\label{RSexpansion}
\basis(m\times n)
\coloneq 
\left\{D(x_1^{e_1} \cdots x_l^{e_l})\ \:\big|\:D \in \shape(m \times n), 
e_1,\dots,e_l \geq 0
\right\}.
\end{equation}
\end{theo}

\begin{rem}\label{easypart}
The first assertion of \cref{RSbasis} implies that
a full set of relations for the algebra $Z$ 
may be obtained by equating coefficients in \cref{curlsup1}. The proof of the theorem has two parts. The easy part is a straightening algorithm
showing that $Z$ is spanned by monomials
in $z_{2r}\:(r \geq 0)$ and that
$\Hom_{\AB}(B^{\otimes n}, B^{\otimes m})$ is generated as a $Z$-module 
by $\basis(m\times n)$. The hard part is to prove linear independence.
Linear independence could also be reproved using the monoidal functor $\Psi\colon\AB \rightarrow \hatAB$ from \cref{passport}, whose target is built from the completed 2-iquantum group $\hatUU^\imath$:
given the easy part, the arguments in the proofs of \cref{psi0in,faithful} below show both that $Z$ is freely generated by the $z_{2r}$ and that $\Psi$ maps any non-trivial $Z$-linear combination of the elements of $\basis(m\times n)$ to a non-zero morphism.
However, that is a much more complicated argument than Rui and Song's approach, which exploits some natural actions of $\AB$ on the categories $\mathcal O$ for classical Lie algebras.
\end{rem}

\begin{rem}\label{grs-trick}
In \cite[Sec.~2]{GRS-2026}, Gao, Rui and Song also consider some variants of $\AB$ in which some of the relations are dropped. 
Let $\AB_{\operatorname{I}}$ be the strict $\kk$-linear monoidal category defined in the same way as $\AB$ above but omitting the Reidemeister I, II and III relations (we mean the relations in \cref{ri,rii-iii}).
Similarly, let $\AB_{\operatorname{II}}$ be defined in the same way but omitting Reidemeister II and III, and let
$\AB_{\operatorname{III}}$ be defined in the same way but only omitting Reidemeister III.
The following relations hold in $\AB_{\operatorname{I}}$:
\begin{align}
\begin{tikzpicture}[B,baseline=-2.9mm,scale=1.5]
\draw (-0.15,-0.4) to[out=60,in=-90] (0.15,0) arc(0:180:0.15) to[out=-90,in=120] (0.15,-0.4);
\closeddot{-.07,-.3};
\end{tikzpicture}\ 
&\ =\ -\ \begin{tikzpicture}[B,baseline=-2.9mm,scale=1.5]
\draw (-0.15,-0.4) to[out=60,in=-90] (0.15,0) arc(0:180:0.15) to[out=-90,in=120] (0.15,-0.4);
\closeddot{.07,-.3};
\end{tikzpicture}\ ,&
\begin{tikzpicture}[B,baseline=1.2mm,scale=1.5]
\draw (-0.15,0.4) to[out=-60,in=90] (0.15,0) arc(0:-180:0.15) to[out=90,in=-120] (0.15,0.4);
\closeddot{-.07,.3};
\end{tikzpicture}\ 
&\ =\ -\ \begin{tikzpicture}[B,baseline=1.2mm,scale=1.5]
\draw (-0.15,0.4) to[out=-60,in=90] (0.15,0) arc(0:-180:0.15) to[out=90,in=-120] (0.15,0.4);
\closeddot{.07,.3};\end{tikzpicture}\ .
\end{align}
We leave the easy proof of this as an exercise for the reader.
Less trivially, in \cite[Cor.~2.4]{GRS-2026}, it is shown that the following relations hold in $\AB_{\operatorname{II}}$:
\begin{align}\label{RIrels}
\begin{tikzpicture}[B,centerzero,scale=1.4]
\draw (0.2,-0.4) to[out=135,in=down] (-0.15,0) to[out=up,in=225] (0.2,0.4);
\draw (-0.2,-0.4) to[out=45,in=down] (0.15,0) to[out=up,in=-45] (-0.2,0.4);
\closeddot{-.08,.3};
\end{tikzpicture}\ 
&=\ \begin{tikzpicture}[B,centerzero,scale=1.4]
\draw (0.2,-0.4) to[out=135,in=down] (-0.15,0) to[out=up,in=225] (0.2,0.4);
\draw (-0.2,-0.4) to[out=45,in=down] (0.15,0) to[out=up,in=-45] (-0.2,0.4);
\closeddot{-.08,-.3};
\end{tikzpicture}\ ,&
\begin{tikzpicture}[B,centerzero,scale=1.4]
\draw (0.2,-0.4) to[out=135,in=down] (-0.15,0) to[out=up,in=225] (0.2,0.4);
\draw (-0.2,-0.4) to[out=45,in=down] (0.15,0) to[out=up,in=-45] (-0.2,0.4);
\closeddot{.08,.3};
\end{tikzpicture}\ 
&=\ \begin{tikzpicture}[B,centerzero,scale=1.4]
\draw (0.2,-0.4) to[out=135,in=down] (-0.15,0) to[out=up,in=225] (0.2,0.4);
\draw (-0.2,-0.4) to[out=45,in=down] (0.15,0) to[out=up,in=-45] (-0.2,0.4);
\closeddot{.08,-.3};
\end{tikzpicture}\ .
\end{align}
Also, in \cite[Cor.~2.7]{GRS-2026}, it is shown that the following relations hold in $\AB_{\operatorname{III}}$:
\begin{align}\label{RIIrelsa}
\begin{tikzpicture}[B,centerzero,scale=1.2]
\draw (0.4,-0.4) -- (-0.4,0.4);
\draw (0,-0.4) to[out=135,in=down] (-0.32,0) to[out=up,in=225] (0,0.4);
\draw (-0.4,-0.4) -- (0.4,0.4);
\closeddot{-.3,.3};
\end{tikzpicture}
-
\begin{tikzpicture}[B,centerzero,scale=1.2]
\draw (0.4,-0.4) -- (-0.4,0.4);
\draw (0,-0.4) to[out=135,in=down] (-0.32,0) to[out=up,in=225] (0,0.4);
\draw (-0.4,-0.4) -- (0.4,0.4);
\closeddot{.3,-.3};
\end{tikzpicture}
&= \begin{tikzpicture}[B,centerzero,scale=1.2]
\draw (0.4,-0.4) -- (-0.4,0.4);
\draw (0,-0.4) to[out=45,in=down] (0.32,0) to[out=up,in=-45] (0,0.4);
\draw (-0.4,-0.4) -- (0.4,0.4);
\closeddot{-.3,.3};
\end{tikzpicture}
-
\begin{tikzpicture}[B,centerzero,scale=1.2]
\draw (0.4,-0.4) -- (-0.4,0.4);
\draw (0,-0.4) to[out=45,in=down] (0.32,0) to[out=up,in=-45] (0,0.4);
\draw (-0.4,-0.4) -- (0.4,0.4);
\closeddot{.3,-.3};
\end{tikzpicture}\ ,\\\label{RIIrelsb}
\begin{tikzpicture}[B,centerzero,scale=1.2]
\draw (0.4,-0.4) -- (-0.4,0.4);
\draw (0,-0.4) to[out=135,in=down] (-0.32,0) to[out=up,in=225] (0,0.4);
\draw (-0.4,-0.4) -- (0.4,0.4);
\closeddot{-.12,.3};
\end{tikzpicture}
-
\begin{tikzpicture}[B,centerzero,scale=1.2]
\draw (0.4,-0.4) -- (-0.4,0.4);
\draw (0,-0.4) to[out=135,in=down] (-0.32,0) to[out=up,in=225] (0,0.4);
\draw (-0.4,-0.4) -- (0.4,0.4);
\closeddot{-.12,-.3};
\end{tikzpicture}
&= \begin{tikzpicture}[B,centerzero,scale=1.2]
\draw (0.4,-0.4) -- (-0.4,0.4);
\draw (0,-0.4) to[out=45,in=down] (0.32,0) to[out=up,in=-45] (0,0.4);
\draw (-0.4,-0.4) -- (0.4,0.4);
\closeddot{.12,.3};
\end{tikzpicture}
-
\begin{tikzpicture}[B,centerzero,scale=1.2]
\draw (0.4,-0.4) -- (-0.4,0.4);
\draw (0,-0.4) to[out=45,in=down] (0.32,0) to[out=up,in=-45] (0,0.4);
\draw (-0.4,-0.4) -- (0.4,0.4);
\closeddot{.12,-.3};
\end{tikzpicture}\ ,\\
\begin{tikzpicture}[B,centerzero,scale=1.2]
\draw (0.4,-0.4) -- (-0.4,0.4);
\draw (0,-0.4) to[out=135,in=down] (-0.32,0) to[out=up,in=225] (0,0.4);
\draw (-0.4,-0.4) -- (0.4,0.4);
\closeddot{.3,.3};
\end{tikzpicture}
-
\begin{tikzpicture}[B,centerzero,scale=1.2]
\draw (0.4,-0.4) -- (-0.4,0.4);
\draw (0,-0.4) to[out=135,in=down] (-0.32,0) to[out=up,in=225] (0,0.4);
\draw (-0.4,-0.4) -- (0.4,0.4);
\closeddot{-.3,-.3};
\end{tikzpicture}
&= \begin{tikzpicture}[B,centerzero,scale=1.2]
\draw (0.4,-0.4) -- (-0.4,0.4);
\draw (0,-0.4) to[out=45,in=down] (0.32,0) to[out=up,in=-45] (0,0.4);
\draw (-0.4,-0.4) -- (0.4,0.4);
\closeddot{.3,.3};
\end{tikzpicture}
-
\begin{tikzpicture}[B,centerzero,scale=1.2]
\draw (0.4,-0.4) -- (-0.4,0.4);
\draw (0,-0.4) to[out=45,in=down] (0.32,0) to[out=up,in=-45] (0,0.4);
\draw (-0.4,-0.4) -- (0.4,0.4);
\closeddot{-.3,-.3};
\end{tikzpicture}\ .\label{RIIrelsc}
\end{align}
The relations \cref{RIrels,RIIrelsa,RIIrelsb,RIIrelsc} were used to simplify the proofs of the
Reidemeister II and III relations in
\cite{GRS-2026}, and we will use them in exactly the same way in the proofs of \cref{lem:rii,lem:riii} below.
\end{rem}

\setcounter{section}{2}

\section{\texorpdfstring{The 2-$\mathrm i$quantum group}{The 2-iquantum group}}\label{iqg}

Next, we introduce the 2-iquantum group relevant for the constructions in this paper. It is a special case of the general definition in \cite{BWW}.
Let $I \coloneq \kk$ and $\tau\colon I \rightarrow I$ be the involution
with $\tau i=-i$. We usually prefer to write $-i$ instead 
of $\tau i$ but the latter is the notation used in the general setup of iquantum groups which we are following. 

Let $X$ be the free Abelian group $\bigoplus_{i \in I} \Z \varpi_i$ and
$h_i\colon X \rightarrow \Z$ be the homomorphism of Abelian groups 
with $h_i(\varpi_j) \coloneq \delta_{i,j}$. 
Let $\alpha_i \coloneq 2 \varpi_i - \varpi_{i-1}-\varpi_{i+1}$
so that
\begin{equation}\label{cartanmatrix}
a_{i,j} \coloneq h_i(\alpha_j) = \begin{dcases}
2&\text{if $i=j$}\\
-1&\text{if $i = j \pm 1$}\\
0&\text{otherwise}
\end{dcases}
\end{equation}
is a symmetric Cartan matrix
whose connected components are all of type $A_\infty$ if $p=0$
or $A_{p-1}^{(1)}$ if $p > 0$.
For $i \in I$, let
\begin{align}\label{firstchoice}
\del_i &\coloneq
\begin{dcases}
1&\text{if $i = h$}\\
-1&\text{if $i=0$}\\
0&\text{otherwise},
\end{dcases}
\end{align}
so that $\del_i + \del_{-i} = -a_{i,-i}$.
Let $t\colon X \rightarrow \Z/2$ and $p\colon X \rightarrow \Z/2$ be the homomorphisms of Abelian groups mapping $\lambda$ to 
\begin{align}\label{ambiguity}
t(\lambda) &\coloneq 
h_0(\lambda)\text{ mod 2},&
p(\lambda) &\coloneq 
\sum_{i \in I} h_i(\lambda)\text{ mod 2}.
\end{align}
We call $t(\lambda)$ the {\em type} and $p(\lambda)$
the {\em parity} of $\lambda$. Simple roots are even,
so $p(\lambda-\alpha_i)=p(\lambda)$. However,
$t(\lambda-\alpha_i)$ is different from $t(\lambda)$ 
when $i = \pm 1$.

Let $\tau\colon X \rightarrow X$ be the homomorphism with $\tau \varpi_i \coloneq \varpi_{-i}$.
Let $X^\imath$ be the quotient
of $X$ by the subgroup spanned by 
$\varpi_i+\varpi_{-i}\:(i \in I)$,
denoting the quotient map by
$\pi\colon X\twoheadrightarrow X^\imath$.
Elements of $X^\imath$ are called {\em iweights}.
For $\lambda \in X^\imath$ and $\mu \in X$, we allow ourselves to write expressions like 
$\lambda+\mu$ instead of
$\lambda+\pi(\mu)$.
The homomorphisms $t$ and $p$ factor through the quotient to induce
homomorphisms we also denote by $t, p\colon X^\imath \rightarrow \Z/2$.
The set $X^\imath$ decomposes as
\begin{equation}\label{evenodd}
X^\imath = X^\imath_{\0} \sqcup X^\imath_{\1},
\end{equation}
where $X^\imath_{\0}$ consists of the iweights with even parity, and $X^\imath_{\1}$ consists of the odd ones.
One more important notation is
\begin{align}
\label{sneaky}
\lambda_i &\coloneq 
h_i\big((\id-\tau)(\hat\lambda)\big)
=h_i(\hat\lambda)-h_{-i}(\hat\lambda)
\end{align}
for $\lambda \in X^\imath$ with pre-image $\hat\lambda \in X$.
This is well defined independent of the choice of 
$\hat\lambda$.
Note also that $p(\lambda) \equiv t(\lambda) + \sum_{i \in I'} \lambda_i\pmod{2}$ for any $I' \subset I$ chosen so that
$I = I' \sqcup \{0\}\sqcup (-I')$.

\begin{lem}\label{reallydumblemma}
For weights $\mu,\nu \in X$ of the same parity, 
we have that
$(\id-\tau)(\mu)=(\id-\tau)(\nu)$ in $X$
if and only if $\pi(\mu) = \pi(\nu)$ in $X^\imath$.
\end{lem}

\begin{proof}
An even weight $\lambda \in X$
is a $\tau$-fixed point if and only if
$\pi(\lambda)=0$ in $X^\imath$.
\end{proof}

We now have the data needed to define a quasi-split
iquantum group $\dot\U^\imath$ as in \cite[Sec.~3.1]{BWW}.
Since this will not play a major role, we will not repeat its generators-and-relations definition here, but note that it is a locally unital $\Q(q)$-algebra with distinguished idempotents
$\{1_\lambda\:|\:\lambda \in X^\imath\}$
and generators $b_i 1_\lambda = 1_{\lambda-\alpha_i} b_i$
for $i \in I$ and $\lambda \in X^\imath$.

We proceed to the main goal of the section: the definition of 
a graded 2-category
$\UU^\imath$ which is a categorification of an integral form of $\dot\U^\imath$. For this, we need some additional parameters, which involve some choices. 
Recall that $I = \kk$.
Let $\sim$ be the equivalence relation on $I$ generated by
the relations $i \sim -i$ and $i \sim i+1$.
We make a henceforth fixed choice $K$ of $\sim$-equivalence class
representatives.
The equivalence classes $[k]$ for $k \in K$ 
are the vertex sets of the Satake diagrams
in \cref{dynkintable}.
If $p = 0$, there are three types:
$[0]$, $[h]$, and all the others.
If $p > 2$, there are only two types: 
$[0] = [h]$, and all others.
We define two partitions
\begin{equation}\label{partitions}
I = I_+ \sqcup \{0\} \sqcup I_- = I_\lo \sqcup\{0\} \sqcup I_\hi
\end{equation}
of the set $I$
by letting $I_+$ and $I_-$ be the sets
indicated by the $+$ and $-$ labels
in the nodes in \cref{dynkintable}, and $I_\lo$ and $I_\hi$
be the sets on the bottom and top rows of these diagrams.
Important properties of these partitions 
are that $I_- = -I_+, I_\hi=-I_\lo$ and $h \in I_+ \cap  I_\lo$.
Note also that $1 \in I_\hi$, but it could either be in $I_+$ or $I_-$
depending on the characteristic $p$ (it is in $I_-$ if and only if 
$p\equiv 3\pmod{4}$).

For $i \in I$, we let $c_i\colon (X,+)\rightarrow  \kk^\times$
be the homomorphism of Abelian groups defined by
\begin{align}
c_i(\lambda)&\coloneq 
\begin{dcases}
1&\text{if $i \in I_+ \cup \{0\}$}\\
(-1)^{h_i(\lambda)}&\text{if $i \in I_-$.}
\end{dcases}
\end{align}
In \cite[Hyp.~3.3]{BWW}, these functions are 
called {\em normalization functions}, and they are required to satisfy the properties
$c_i(\alpha_i)=1$ for all $i \in I$, and
\begin{align}\label{ducks}
c_{-i}(\tau \lambda) &= (-1)^{h_i(\lambda)} c_i(\lambda)
\end{align}
for $i \in I-\{0\}$ and $\lambda \in X$.
Our choice for these functions clearly satisfies
both of these properties.
In particular, we have that
\begin{align}\label{dogs}
c_{i,j} \coloneq c_i(\alpha_j) &= 
\begin{dcases}
1&\text{if $i \in I_+\cup\{0\}$}\\
(-1)^{a_{i,j}}&\text{if $i \in I_-$}
\end{dcases}
\end{align}
for $i,j \in I$, and $c_{i,i}=1, c_{i,-i}=(-1)^{\delta_{i,-h}}$. Let
\begin{align}
t_{i,j}&\coloneq
\begin{dcases}
-1\hspace{18mm}&\text{if there is a single arrow
$\begin{tikzpicture}[iQ,centerzero,scale=1.5]\draw (0,0) to (.5,0);
\node at (.25,0) {$\scriptstyle <$}; 
\node at (0,0) {${\color{white}\bullet}\hspace{-5.5pt}\circ$};
\node at (0.5,0) {${\color{white}\bullet}\hspace{-5.5pt}\circ$};
\node at (0,0.15) {$\scriptstyle i$};
\node at (.5,0.15) {$\scriptstyle j$};
\end{tikzpicture}$ in \cref{dynkintable}}\\
1\hspace{18mm}&\text{if there is a single arrow
$\begin{tikzpicture}[iQ,centerzero,scale=1.5]\draw (0,0) to (.5,0);
\node at (.25,0) {$\scriptstyle >$}; 
\node at (0,0) {${\color{white}\bullet}\hspace{-5.5pt}\circ$};
\node at (0.5,0) {${\color{white}\bullet}\hspace{-5.5pt}\circ$};
\node at (0,0.15) {$\scriptstyle i$};
\node at (.5,0.15) {$\scriptstyle j$};
\end{tikzpicture}$ in \cref{dynkintable}}\\
-\sqrt{-1}\hspace{18mm}&\text{if there is a double arrow
$\begin{tikzpicture}[iQ,centerzero,scale=1.5]\draw (0,0) to (.5,0);
\node at (.25,0) {$\scriptstyle <\!\!<$}; 
\node at (0,0) {${\color{white}\bullet}\hspace{-5.5pt}\circ$};
\node at (0.5,0) {${\color{white}\bullet}\hspace{-5.5pt}\circ$};
\node at (0,0.15) {$\scriptstyle i$};
\node at (.5,0.15) {$\scriptstyle j$};
\end{tikzpicture}$ in \cref{dynkintable}}\\
\sqrt{-1}\hspace{18mm}&\text{if there is a double arrow
$\begin{tikzpicture}[iQ,centerzero,scale=1.5]\draw (0,0) to (.5,0);
\node at (.25,0) {$\scriptstyle >\!\!>$}; 
\node at (0,0) {${\color{white}\bullet}\hspace{-5.5pt}\circ$};
\node at (0.5,0) {${\color{white}\bullet}\hspace{-5.5pt}\circ$};
\node at (0,0.15) {$\scriptstyle i$};
\node at (.5,0.15) {$\scriptstyle j$};
\end{tikzpicture}$ in \cref{dynkintable}}\\
1&\text{in all other cases,}
\end{dcases}\end{align}
\begin{align}
Q_{i,j}(x,y)&\coloneq
\begin{dcases}
t_{i,j}(x-y)^{-a_{i,j}}&\text{if $i \neq j$}\\
0&\text{if $i=j$,}
\end{dcases}&
Q^\imath_{i,j}(x,y) &\coloneq (-1)^{\delta_{i,-j}} Q_{i,j}(x,y).
\label{typohere}
\end{align}
The last parameter choice to be made is
\begin{align}\label{shape}
\zeta_i &\coloneq
\begin{dcases}
1&\text{if $i \in I_+$}\\
-\tfrac{1}{2} &\text{if $i=0$}\\
-2&\text{if $i = -h$}\\
-1&\text{for all other $i \in I_-$.}
\end{dcases}\end{align}
Because
$\begin{tikzpicture}[iQ,centerzero,scale=1.5]\draw (0,0) to (.5,0);
\node at (.25,0) {$\scriptstyle <$}; 
\node at (0,0) {${\color{white}\bullet}\hspace{-5.5pt}\circ$};
\node at (0.5,0) {${\color{white}\bullet}\hspace{-5.5pt}\circ$};
\node at (0,0.15) {$\scriptstyle -h$};
\node at (.5,0.15) {$\scriptstyle h$};
\end{tikzpicture}$, $\del_h = 1$ and $\del_{-h}=0$,
these satisfy
\begin{align}
\zeta_{0}&=-\tfrac{1}{2},&
Q_{-i,i}(x,-x)&=(-1)^{\del_i+1}\zeta_i \zeta_{-i} x^{-a_{i,-i}}
\end{align}
for all $i \in I - \{0\}$,
which is another requirement from \cite[Hyp.~3.3]{BWW}.

There are several further pieces of notation dependent on the choices made.
Let $\gamma_i\colon(X^\imath,+) \rightarrow \kk^\times$
be the homomorphism of Abelian groups 
defined by $\gamma_i(\lambda) \coloneq (-1)^{h_{-i}(\hat\lambda)}c_i(\hat\lambda) / c_{-i}(\hat\lambda)$
for $\lambda \in X^\imath$ with pre-image $\hat\lambda \in X$.
Explicitly, we have that
\begin{align}
\gamma_i(\lambda) &= 
\begin{dcases}
1&\text{if $i \in I_+$}\\
(-1)^{t(\lambda)}&\text{if $i =0$}\\
(-1)^{\lambda_i}&\text{if $i \in I_-$.}
\end{dcases}\end{align}
Also let
\begin{align}
\label{bone}
\gamma_{i,j} \coloneq
\gamma_i(\pi(\alpha_j)) =
(-1)^{a_{-i,j}}c_{i,j} / c_{-i,j}&=
\begin{dcases}
1&\text{if $i \in I_+$}\\
(-1)^{a_{0,j}}&\text{if $i=0$}\\
(-1)^{a_{i,j}+a_{i,-j}}&\text{if $i \in I_-$.}
\end{dcases}\end{align}
In particular, $\gamma_{i,i}=\gamma_{i,-i} = (-1)^{\delta_{i,-h}}$.
We use the shorthand
\begin{align}\label{eitherway}
\zeta_i(\lambda) \coloneq 
\zeta_i \gamma_i(\lambda) &= 
\begin{dcases}
1&\text{for $i \in I_+$}\\
(-1)^{t(\lambda)+1}\tfrac{1}{2} &\text{if $i=0$}\\
(-1)^{\lambda_i+1}2 &\text{if $i =-h$}\\
(-1)^{\lambda_i+1} &\text{for all other $i \in I_-$}
\end{dcases}\end{align}
for $\lambda \in X^\imath$. Let
\begin{align}
r_{i,j}&\coloneq c_{i,j} / t_{i,j}.
\end{align}
We have that $r_{i,i}=r_{-i,i}=1$ for all $i \in I$
as required by \cite[Hyp.~3.3]{BWW}.
To see this for $i \in \{\pm h\}$, note that
$c_{h,-h}=t_{h,-h}=1$
and
$c_{-h,h}=t_{-h,h}=-1$.
The following lemma checks the final constraint on
parameters from \cite[Hyp.~3.3]{BWW}:

\begin{lem}\label{nasty}
For $i,j \in I$, we have that
$Q_{-i,-j}(x,y) = r_{i,j} r_{j,i} r_{i, -j} r_{j,-i} Q_{i,j}(-x,-y)$.
\end{lem}

\begin{proof}
It holds if $i=j$ as both sides are zero, so assume that $i \neq j$.
As $r_{i,j}=c_{i,j}/t_{i,j}$, $Q_{-i,-j}(x,y) = t_{-i,-j} (x-y)^{-a_{i,j}}$
and $Q_{i,j}(-x,-y) = t_{i,j}(-1)^{a_{i,j}}(x-y)^{-a_{i,j}}$,
the identity is equivalent to
$$
t_{-i,-j}t_{j,i}t_{i,-j}t_{j,-i}= (-1)^{a_{i,j}}c_{i,j}c_{j,i}c_{i,-j}c_{j,-i}.
$$
Now one needs to consider the various 
cases for $\{\pm i, \pm j\}$, i.e., it is a rank-two calculation.
Because $t_{i,j}=(-1)^{a_{i,j}}t_{j,i}$
and $(-1)^{a_{i,j}}t_{-j,-i}=t_{-i,-j}$,
we have that $t_{-j,-i}t_{i,j} = t_{-i,-j}t_{j,i}$.
From this, it follows that the identity holds for $(i,j)$ if and only if it holds for $(j,i)$, which cuts down the cases to be considered.
\begin{itemize}
\item 
If $i \notin \{- j, \pm j \pm 1\}$ the identity holds trivially. 
\item
If $i = -j$ it holds trivially except if
$i = \pm h$ and $j=\mp h$, and this case follows
by checking it for $i=h, j=-h$ using
$(-1)^{a_{h,-h}} c_{h,-h}c_{-h,h}=1$. 
\item
If $i = 0$ and $j = \pm 1$, or $j=0$ and $i = \pm 1$,
it suffices to check when $i=0$ and $j = \pm 1$.
The $j=-1$ case requires
$t_{0,1} t_{-1,0} t_{0,1} t_{-1,0} = - c_{0,-1}  c_{-1,0} c_{0,1}c_{-1,0}$
and the $j=1$ case requires
$t_{0,-1} t_{1,0} t_{0,-1} t_{1,0} = - c_{0,1}  c_{1,0}c_{0,-1} c_{1,0}$.
This is true---both sides of both equations equal $-1$.
\end{itemize}
We are now reduced to the situation that
$i,j \neq 0$, $i \neq -j$ and 
$i = \pm j \pm 1$.
\begin{itemize}
\item If $i,j \neq 0$, $i \neq -j$ and
$i \in \{-j \pm 1\}$ then $i \notin \{j \pm 1\}$, and
the identity to be checked reduces to
$t_{i,-j}t_{j,-i}= c_{i,-j}c_{j,-i}$.
If $i \in I_+$ and $j \in I_-$ then the right-hand side is $-1$,
as is the left-hand side because in this situation
the direction of the arrow connecting
$i$ to $-j$ is the same as the direction of the arrow connecting
$-i$ to $j$.
If $i, j \in I_+$ or $i,j \in I_-$ 
then the right-hand side is $1$, as is 
the left-hand side because in this situation
the direction of the arrow connecting
$i$ to $-j$ is the opposite of the direction of the arrow connecting
$-i$ to $j$.
\item
If $i,j \neq 0$, $i \neq -j$ and
$i \notin \{ -j \pm 1\}$ then $i \in \{j \pm 1\}$, and
the identity to be checked reduces to
$t_{-i,-j}t_{j,i}= -c_{i,j}c_{j,i}$.
If $i \in I_+$ and $j \in I_-$ then the right-hand side is $1$,
as is the left-hand side because in this situation
the direction of the arrow connecting
$i$ to $j$ is opposite to the direction of the arrow connecting
$-i$ to $-j$.
If $i, j \in I_+$ or $i,j \in I_-$ 
then the right-hand side is $-1$, as is 
the left-hand side because in this situation
the direction of the arrow connecting
$i$ to $j$ is the same as the direction of the arrow connecting
$-i$ to $-j$.\qedhere
\end{itemize}
\end{proof}

\begin{lem}\label{unexpectlytrue}
For $i,j \in I$, we have that $ r_{j,i}r_{-i,-j}r_{j,-i}r_{i,j}=\gamma_{i,j}\, t_{i,-j} / t_{i,j}$.
\end{lem}

\begin{proof}
Suppose first that $i=j$. Then the left-hand side is 1, as is the right-hand side because $\gamma_{i,i}=(-1)^{\delta_{i,-h}} = t_{i,-i}$
and $t_{i,i}=1$.

Now suppose that $i \neq j$. As $Q_{-i,-j}(x,y)=t_{-i,-j}(x-y)^{-a_{i,j}}$
and $Q_{i,j}(-x,-y)=(-1)^{a_{i,j}}t_{i,j}(x-y)^{-a_{i,j}}$, \cref{nasty} gives that
$r_{i,j}r_{j,i}r_{i,-j}r_{j,-i}=(-1)^{a_{i,j}}t_{-i,-j}/t_{i,j}$.
Multiplying by $r_{-i,-j}/r_{i,-j}=c_{-i,-j}t_{i,-j}/c_{i,-j}t_{-i,-j}$, we deduce that
$$
r_{j,i}r_{-i,-j}r_{j,-i}r_{i,j}=(-1)^{a_{i,j}}
\frac{c_{-i,-j}}{c_{i,-j}}\cdot\frac{t_{i,-j}}{t_{i,j}}.
$$
It remains to observe that $(-1)^{a_{i,j}}c_{-i,-j}/c_{i,-j}=\gamma_{i,j}$ by \cref{ducks,bone}.
\end{proof}

\begin{lem}\label{anothercrazyone}
For $i,j \in I$, we have that 
$r_{i,-j}r_{-i,j}r_{j,i}r_{-j,-i} = (-1)^{a_{i,j}}\gamma_{j,i}/\gamma_{i,j}$.
\end{lem}

\begin{proof}
\cref{unexpectlytrue} with $i$ and $j$ interchanged gives that
$r_{i,j}r_{-j,-i}r_{i,-j}r_{j,i} = \gamma_{j,i} t_{j,-i}/t_{j,i}$. Hence,
$$
r_{i,-j}r_{-i,j}r_{j,i}r_{-j,-i}
=\gamma_{j,i}\frac{t_{j,-i}}{t_{j,i}}\cdot\frac{r_{-i,j}}{r_{i,j}}
=\gamma_{j,i}\frac{t_{i,j}}{t_{j,i}}\cdot\frac{t_{j,-i}}{t_{-i,j}}\cdot\frac{c_{-i,j}}{c_{i,j}}
=\gamma_{j,i}(-1)^{a_{i,j}}(-1)^{a_{j,-i}}\cdot(-1)^{a_{-i,j}}\gamma_{i,j}^{-1},
$$
using $t_{i,j}=(-1)^{a_{i,j}}t_{j,i}$ (twice) and $\gamma_{i,j} = (-1)^{a_{-i,j}}c_{i,j}/c_{-i,j}$ from \cref{bone}.
As $a_{j,-i}=a_{-i,j}$, this is $(-1)^{a_{i,j}}\gamma_{j,i}/\gamma_{i,j}$.
\end{proof}

Now we have all of the required data in hand to be able to define the 2-iquantum group:

\begin{defin}[{\cite[Def.~3.5]{BWW}}]\label{iQdef}
The {\em 2-iquantum group}
$\UU^\imath$ is the graded 2-category
with object set $X^\imath$, 
generating 1-morphisms 
$B_i \one_\lambda = \one_{\lambda -\alpha_i} B_i
\colon \lambda \rightarrow \lambda-\alpha_i$
for $\lambda \in X^\imath$ and $i \in I$,
with identity 2-endomorphisms
denoted by unoriented strings
$\begin{tikzpicture}[iQ,anchorbase]
\draw[-] (0,-0.2)\botlabel{i} -- (0,0.2);
\region{0.2,0}{\lambda};
\region{-0.4,0}{\lambda-\alpha_i};
\end{tikzpicture}$,
and generating 2-morphisms listed in \cref{iQgens} for all $i,j \in I$ and $\lambda \in X^\imath$.
\begin{table}
\begin{align*}
\begin{array}{|l|c|}
\hline
\hspace{13mm}\text{Generator}&\text{Degree}\\
\hline
\,\ \ \begin{tikzpicture}[iQ,centerzero]
\draw[-] (0,-0.3) \botlabel{i} -- (0,0.3) \toplabel{i};
\closeddot{0,0};
\region{0.2,0}{\lambda};
\end{tikzpicture}\ \, 
\colon B_i \one_\lambda \Rightarrow B_i \one_\lambda&2 \\
\:\begin{tikzpicture}[iQ,centerzero]
\draw[-] (-0.25,-0.15) \botlabel{\ibar} to [out=90,in=90,looseness=3](0.25,-0.15) \botlabel{i};
\region{0.45,0.1}{\lambda};
\node at (0,.3) {$\phantom.$};
\node at (0,-.4) {$\phantom.$};
\end{tikzpicture}\hspace{-.03pt}
\colon B_{-i} B_{i} \one_\lambda \Rightarrow \one_\lambda
&1-\lambda_i+\del_i\\
\:\begin{tikzpicture}[iQ,centerzero]
\draw[-] (-0.25,0.15) \toplabel{\ibar} to[out=-90,in=-90,looseness=3] (0.25,0.15) \toplabel{i};
\region{0.45,-.1}{\lambda};
\node at (0,.2) {$\phantom.$};\node at (0,-.3) {$\phantom.$};
\end{tikzpicture}
\colon \one_\lambda \Rightarrow B_{-i} B_{i} \one_\lambda
&1-\lambda_i+\del_i\\
\!\!\!\begin{tikzpicture}[baseline=-1mm,iQ]
\draw[-] (-0.25,0) arc(180:-180:0.25);
\node at (-0.36,0) {\strandlabel{\ibar}};
\region{0.45,0}{\lambda};
\dottybubblelabel{0,0}{r};
\node at (0,-.3) {$\phantom{.}$};
\end{tikzpicture}
\colon \one_\lambda \Rightarrow\one_\lambda
\text{ for }0 \leq r \leq \del_i-\lambda_i
&2 r\\
\;\begin{tikzpicture}[iQ,centerzero,scale=.9]
\draw[-] (-0.3,-0.3) \botlabel{i} -- (0.3,0.3) \toplabel{i};
\draw[-] (0.3,-0.3) \botlabel{j} -- (-0.3,0.3) \toplabel{j};
\region{0.35,0}{\lambda};
\end{tikzpicture}
\;\,\colon B_{i} B_{j} \one_\lambda \Rightarrow B_j B_i \one_\lambda
&- a_{i,j} \\
\hline
\end{array}
\end{align*}
\caption{Generating 2-morphisms for the 2-iquantum group}\label{iQgens}
\end{table}
The degrees of the generating 2-morphisms are also listed in this table.
The generating 2-morphisms are subject to certain relations.
Before writing these down, we explain some further conventions.
\begin{itemize}
\item 
We usually label strings just at one end, but note that the string label 
changes sign across critical points of cups and caps.
\item
We usually only label one of the 2-cells by the iweight $\lambda$, and if we omit all labels we are writing something that is true for all possible labels.
\item
We use pins defined in exactly the same way as in the previous section; the coefficients of a polynomial labelling a pinhead drawn in a 2-cell labelled $\lambda$ should be elements of $\End_{\UU^\imath}(\one_\lambda)$.
We adopt analogous shorthands to \cref{dotgenfuncs}:
\begin{align}\label{qdotgenfuncs}
\ .
\end{align}
We extend the pin notation to pins with two or three dots, using additional variables $y$ and $z$ as needed. 
The alphabetical ordering of the variables $x,y,z$ corresponds to the lexicographical ordering of the Cartesian coordinates of the dots.
\item
Let
\begin{align}\label{ifakebubblepolynomial}
%
\ 
u^{\del_i-\lambda_i-r} - \frac{\delta_{i,0}}{2u} \id_{\one_\lambda},
\\\intertext{When $i =0$, the summation in \cref{ifakebubblepolynomial} 
is empty (zero), so it is simply $-\frac{1}{2u} \id_{\one_\lambda}
\in u^{-1}\End_{\UU^\imath}(\one_\lambda)$.
When $i \neq 0$, $%
$ is a polynomial 
in $\End_{\UU^\imath}(\one_\lambda)[u]$ which we call the {\em fake bubble polynomial}; it is zero if $\lambda_i > \del_i$.
Then we define the {\em bubble generating function}:
}
\label{ibubblegeneratingfunction}
%
\right)\in \End_{\UU^\imath}(\one_\lambda)\lround u^{-1} \rround.
\end{align}
\end{itemize}
The defining relations are as follows: 
\begin{align}
\label{ibubblerelnew}
%
\ .
\end{align}
\end{defin}

As usual, there are some further relations which are consequences of the defining relations. The following are analogous to \cref{whatfunB,gendotslideB}:
\begin{align}\label{whatfun}
%
\ .
\end{align}
The mirror images of these relations in a horizontal axis also hold.
Recall the Demazure operator $\partial_{x,y}$ from \cref{demazureoperator}.
We refer to \cref{idemazure} (and the analogous relation \cref{demazure} in $\AB$)
as the {\em Demazure relation} because it implies
\begin{align}\label{demazureslide}
%
\ .\label{dumbazureslider}
\end{align}
One way to deduce these identities is to use the first relation from 
\cref{genidemazure} and its mirror image plus \cref{trick,trick2}.
Also, if $i \neq 0$ then
\begin{align}\label{halfslider}
%
\ .
\end{align}
This implies that
\begin{equation}\label{symmetricslider}
%
\end{equation}
if $i \neq 0$ and $f(x,y)=f(y,x)$.
Also a direct calculation shows that
\begin{equation}
%
\label{zeroslider}
\end{equation}
for any odd $n \geq 1$;
see \cite[Sec.~3.4]{BWWa1} for a more general treatment of central elements in the $i=0$ case.
Note also the following relation, which is \cite[(3.43)]{BWW}:
\begin{align}\label{fancyicurl}
%
$
is simply equal to $0$ if $t(\lambda)=\0$ or $\id_{\one_\lambda}$ if $t(\lambda)=\1$.
For all non-zero $i \in I$,
the relation
\cref{ibubblerelnew} implies that
$\!%
$
for $0 < r \leq \del_i-\lambda_i$
are called {\em fake bubbles}.
When $\del_i - \lambda_i \geq 1$, it is necessarily the case that $\del_{-i}-\lambda_{-i} \leq 0$, hence,
the relation \cref{iinfgrass} implies that any fake bubble $%
$
can be expressed as a polynomial in the
genuinely dotted bubbles
$%
$.
This is similar to the situation for 2-quantum groups. It seems worth pointing out that in \cite{BWW} more general 2-iquantum groups are defined in which it is possible that
$\del_i + \del_{\tau i} \geq 2$, in which case
it might not be 
possible to express all fake bubbles as polynomials in genuine dotted bubbles.

We use the notation
\begin{align}\label{zedisdead}
Z_\lambda &\coloneq \End_{\UU^\imath}(\one_\lambda).
\end{align}
It will sometimes be useful to have a more algebraic notation available to denote the renormalized dotted bubbles which generate $Z_\lambda$:
for any $i \in I$, we
define
\begin{align}\label{likey}
O_{\lambda;i}(u) = 
\sum_{r \geq 0} z_{\lambda;i}^{(r)} u^{\del_i-\lambda_i-r} \coloneq 
\begin{tikzpicture}[baseline=-1mm,iQ]
\draw[-] (-0.25,0) arc(180:-180:0.25);
\node at (-0.36,0) {\strandlabel{\ibar}};
\region{0.95,0}{\lambda};
\node at (.53,0) {$(u)$};
\end{tikzpicture}.
\end{align}
More explicitly, $z_{\lambda;i}^{(0)} = \id_{\one_\lambda}$, and for $r > 0$ we have that
\begin{align}\label{flambda2}
z_{\lambda;i}^{(r)}=
\begin{dcases}
\zeta_i(\lambda)^{-1}
\begin{tikzpicture}[baseline=-1mm,iQ]
\draw[-] (-0.25,0) arc(180:-180:0.25);
\node at (-.38,0) {\strandlabel{\ibar}};
\region{0.5,0}{\lambda};
\dottybubblelabel{0,0}{r};
\end{tikzpicture}&\text{if $r \leq \del_i-\lambda_i$}\\
\zeta_i(\lambda)^{-1}
\begin{tikzpicture}[baseline=-1mm,iQ]
\draw[-] (-0.25,0) arc(180:-180:0.25);
\node at (-.38,0) {\strandlabel{\ibar}};
\region{2.2,0}{\lambda};
\multcloseddot{0.25,0}{west}{r+\lambda_i-\del_i-1};
\end{tikzpicture}&\text{if $r > \del_i-\lambda_i$.}
\end{dcases}
\end{align}
Note that $z_{\lambda;i}^{(r)}$ is of degree $2r$.
Let
\begin{equation}\label{basering}
R \coloneq \kk\big[z_i^{(r)}\:\big|\:i \in I,\ r \geq 1\big] \ \big/ \ 
J,
\end{equation}
where $J$ is the ideal generated by
the relations that are the coefficients of powers of $u$ in the identities 
$z_{-i}(u) z_{i}(-u) = 1$ for all $i \in I$,
where 
\begin{equation}\label{hounds}
z_i(u) \coloneq \sum_{r \geq 0} z_i^{(r)} u^{-r} \in R\llbracket u^{-1}\rrbracket,
\end{equation}
and $z_i^{(0)} \coloneq 1$ by convention.
We view $R$ as a graded algebra so that each
$z_i^{(r)}$ is in degree $2r$.
Comparing with \cref{iinfgrass}, 
there is a graded algebra homomorphism
\begin{align}\label{flambda}
\varphi_\lambda\colon R &\rightarrow Z_\lambda,
&z_i^{(r)}&\mapsto z_{\lambda;i}^{(r)},
&z_i(u) &\mapsto 
u^{\lambda_i-\del_i}
O_{\lambda;i}(u).
\end{align}
We note also that the ring $R$ is freely 
generated by the elements $z_0^{(2r-1)}$ and $z_i^{(r)}$ for all $r \geq 1$ and $i \in I' \subset I$ chosen so that
$I = I' \sqcup \{0\}\sqcup (-I')$.

The fundamental theorem about $\UU^\imath$ proved in \cite{BWW} is its non-degeneracy. This gives an explicit diagrammatic basis for each of its 2-morphism spaces, as follows.
Let $\langle I \rangle$ be the set of words
$\bi = i_1 \cdots i_s\:(s \geq 0, i_1,\dots,i_s \in I)$ in the alphabet $I$.
For such a word $\bi$,
we use the shorthand
$B_\bi \coloneq B_{i_1} \cdots B_{i_s}$.
Thus, the 1-morphisms in $\Hom_{\UU^\imath}(\lambda,\mu)$
are the monomials $B_\bi \one_\lambda$ for $\bi \in \langle I \rangle$ such that $\lambda-\alpha_\bi=\mu$,
where $\alpha_\bi\coloneq\alpha_{i_1}+\cdots+\alpha_{i_s}\in X$.
For two words $\bi, \bj \in \langle I \rangle$, we denote their concatenation simply by $\bi \bj$.
We define an equivalence relation $\approx$ on $\langle I \rangle$
by declaring that $\bi \approx \bj$ if 
\begin{itemize}
\item $\alpha_\bi$ and $\alpha_\bj$ have the same image in $X^\imath$;
\item the lengths of the words $\bi$ and $\bj$ are of the same parity.
\end{itemize}
The first property implies for any $\lambda \in X^\imath$
that $B_\bi \one_\lambda$ and $B_\bj \one_\lambda$ are objects of the same morphism category of $\UU^\imath$. In that case, the 2-morphism space
$\Hom_{\UU^\imath}(B_\bj\one_\lambda, B_\bi\one_\lambda)$
is 0 unless the second property holds.
We use the homomorphism $\varphi_\lambda\colon R \rightarrow Z_\lambda$ to make the 2-morphism space
$\Hom_{\UU^\imath}(B_\bj\one_\lambda, B_\bi\one_\lambda)$
into a graded right $R$-module so that
$z \in R$ acts by horizontal composition on the right with
$\varphi_\lambda(z)$.

Assuming that $\bi \approx \bj$, 
an {\em $\bi \times \bj$-shape} is a (black) string diagram $D$ with
no dots and no bubbles such that
$D\ \text{\regionlabel{\lambda}}\colon  B_\bj\one_\lambda \Rightarrow B_\bi\one_\lambda$ is a 2-morphism in $\UU^\imath$ for any $\lambda \in X^\imath$. 
We say that a shape is {\em reduced} if there are no zig-zags, bigons or loops, and
two shapes are {\em equivalent} if they define the same combinatorial matching between their boundary points. 
Let $\shape(\bi\times\bj)$ be a fixed choice of representatives for the equivalence classes of reduced $\bi \times \bj$ shapes.
Also, for $D \in \shape(\bi\times\bj)$ with $l$ strings ($l$ is the average of the lengths of the words $\bi$ and $\bj$), we fix a choice of normally ordered points $p_1,\dots,p_l$ away from crossings and critical points, one on each string.
For $f \in \kk[x_1,\dots,x_l]$, we let
\begin{equation}\label{clorific}
D(f)\ \text{\regionlabel{\lambda}}\colon B_\bj\one_\lambda\Rightarrow B_\bi\one_\lambda
\end{equation}
be the 2-morphism in $\UU^\imath$ 
obtained from $D$ by pinning $f$ to the points $p_1,\dots,p_l$.

\begin{theo}[Non-degeneracy]\label{nondegeneracy}
For $\lambda \in X^\imath$ and $\bi,\bj \in \langle I \rangle$ such that $\bi \approx \bj$, let
$l$ be the average of the lengths of $\bi$ and $\bj$.
The 2-morphism space $\Hom_{\UU^\imath}(B_\bj\one_\lambda, B_\bi\one_\lambda)$ is free as a right $R$-module with basis 
\begin{equation}\label{ibasis}
\basis_\lambda(\bi\times\bj)
\coloneq 
\left\{D(x_1^{e_1} \cdots x_l^{e_l})\ \text{\regionlabel{\lambda}}\:\big|\:D \in \shape(\bi\times\bj), 
e_1,\dots,e_l \geq 0
\right\}.
\end{equation}
In particular, $\varphi_\lambda\colon R \rightarrow Z_\lambda$ 
is an isomorphism for every $\lambda \in X^\imath$.
\end{theo}

\begin{proof}
This is \cite[Th.~5.10]{BWW}, which is applicable in the present situation because the Cartan matrix is symmetric and our parameter choices from \cref{typohere} are of the required geometric type. 
\end{proof}

\setcounter{section}{3}

\section{Brauer categorifications}\label{bcat}

For a $\kk$-linear category $\catR$, we denote the strict $\kk$-linear monoidal category
of $\kk$-linear endofunctors and natural transformations
by $\END_\kk(\catR)$.
Then
$\catR$ is a {\em strict left $\AB$-module category} if and only if
there is a strict $\kk$-linear monoidal functor $\AB \rightarrow \END_\kk(\catR)$.
We use the special terminology
{\em Abelian Brauer categorification} to mean a strict left $\AB$-module category
which is also a locally finite Abelian category, that is, it is Abelian, all objects are of finite length, and all morphism spaces are finite-dimensional.
These are the main objects of study in the remainder of the paper.

We say that a $\kk$-linear category $\catR$ is 
{\em locally finite Karoubian} if it is a 
Karoubian category, i.e., it is additive and idempotent complete, its object set is small,
and its morphism spaces are finite-dimensional vector spaces.
We denote the additive 
Karoubi envelope of a $\kk$-linear category $\catR$ by $\Kar(\catR)$; if $\catR$ is small with finite-dimensional morphism spaces then $\Kar(\catR)$ 
is a locally finite Karoubian category.
By a {\em Karoubian Brauer categorification}, we mean
a strict left $\AB$-module category which is also a
locally finite Karoubian category.
Abelian Brauer categorifications are Karoubian, but not vice versa.

Of course, we prefer the Abelian case because we like to think in terms of irreducible objects, composition series, and so forth.
This can be salvaged to some 
extent in the Karoubian case by passing from $\catR$ to the Abelian category $\Mod{\catR}$ of contravariant\footnote{This word means we are working with {\em right} modules!}
$\kk$-linear functors $V:\catR \rightarrow \Vec$.
We also refer to objects of $\Mod{\catR}$ 
as {\em locally finite-dimensional $\catR$-modules}. Since the morphism spaces of $\catR$ are finite-dimensional, the category $\Mod{\catR}$ is well behaved---it is what the authors call a {\em Schurian category}. Basic facts about such categories can be found in the expository article \cite{BD-actions}. 
In particular, all finitely generated modules are locally finite-dimensional 
and have projective covers in $\Mod{\catR}$.
There also is a natural notion of blocks. 
By Yoneda, the category $\catR$ embeds into 
$\Mod{\catR}$, 
with objects of $\catR$ being identified with the 
{\em distinguished} projective $\catR$-modules.
The indecomposable objects of $\catR$ are indecomposable projective $\catR$-modules.
Some caution is needed, however: a locally finite-dimensional $\catR$-module need not be finitely generated, and finitely generated $\catR$-modules need not be of finite length. 

We assume for the remainder of the section that
\begin{center}
\fbox{
\begin{minipage}{0.55\textwidth}
$\catR$ is an Abelian or a Karoubian Brauer categorification. 
\end{minipage}
}
\end{center}
The generating object $B$ of $\AB$
acts as an endofunctor $B\colon\catR \rightarrow \catR$, and morphisms in $\AB$
act as natural transformations.
Since $B$ is a self-dual object of $\AB$, the endofunctor $B\colon\catR \rightarrow \catR$ is self-adjoint. 
Hence, in the Abelian case, $B$ is exact.

In the Karoubian case, 
the action of $\AB$ on $\catR$ extends in a canonical way to make
$\Mod{\catR}$ into a strict
left $\AB$-module category; e.g., this standard construction is explained in \cite[Sec.~2.5]{BV-partition}.
The extended endofunctor $B$ 
is again exact, and it takes finitely generated $\catR$-modules to finitely generated $\catR$-modules. 
To avoid too much duplication as we are discussing the Abelian and Karoubian cases simultaneously subsequently, 
we use the shorthand notation $\catR^\fg$
to denote $\catR$ in the Abelian case, or the full subcategory of $\Mod{\catR}$ consisting of all finitely generated $\catR$-modules 
in the Karoubian case.
In the Karoubian case, morphism spaces in $\catR^\fg$ are finite-dimensional, indeed, the space of $\catR$-module homomorphisms from a finitely generated module to a locally finite-dimensional module is finite-dimensional. By an irreducible object $L \in \catR^\fg$, we mean an object which is irreducible as an $\catR$-module in the usual sense, i.e., an irreducible object of the Abelian category $\Mod{\catR}$.

\begin{rem}
In the Karoubian case, we do not assume that 
the $\kk$-linear 
category $\catR$ is locally Noetherian. Consequently, submodules of finitely generated $\catR$-modules need not be finitely generated.
This means that 
the category $\catR^\fg$ may not itself be Abelian, and we will never assume that this is the case below. It seems reasonable to us
to conjecture that $\catR$ is locally 
Noetherian in the important examples $\cB(A,f)$ 
constructed in \cref{brauercyclotomic}. One small piece of evidence for this:
in \cite[Th.~5.9]{SamSnowden-2022}, 
Sam and Snowden proved that the Brauer category $\catB_\delta$ 
over the field $\C$ is locally Noetherian for any $\delta \in \C$.
Since $\Kar(\catB_\delta)$ is equivalent to
the level one cyclotomic quotient
$\cB\big(u-\frac{1-\delta}{2}\big)$, the latter is locally Noetherian too.
\end{rem}

In string diagrams, we add a green string labelled by an object $V \in \catR^\fg$ on the right-hand boundary of a string diagram of a morphism in $\AB$ to indicate the morphism in $\catR^\fg$ obtained by evaluating the natural 
transformation on $V$. Morphisms in $\catR^\fg$ can be indicated by coupons on the green string.
An element $z$ of the center $Z(\catR^\fg) = \End(\id_{\catR^\fg})$ can be represented by a bubble labelled $z$ in the 2-cell immediately to the left of the green string, so that
$$
\begin{tikzpicture}[B,anchorbase]
\bubble{-.5,0};\bubblelabel{-.5,0}{z};
\draw[darkg](0,-.4)\botlabel{V} to (0,.4);
\end{tikzpicture}
=
\begin{tikzpicture}[B,anchorbase]
\draw[darkg](0,-.4)\botlabel{V} to (0,.4);
\catlabel{0,0}{z_V};
\end{tikzpicture}\ ,
$$
where $z_V$ is the evaluation of $z$ on $V \in \catR^\fg$.
This is the usual extension of string calculus for (strict) monoidal categories to (strict) module categories.

\begin{rem}
In the previous paragraph,
we restricted $V$ to being an object of $\catR^\fg$.
However, in the Karoubian case, the same string diagrams also
make sense with the green string labelled by any locally finite-dimensional $\catR$-module. 
Since any $\catR$-module is a union of finitely generated $\catR$-modules, and the endofunctor $B$ is exact and preserves finite generation, 
naturality implies that 
the resulting morphism is determined uniquely by its restriction to all finitely generated submodules.
\end{rem}

We also use the notation $Z_V$ to denote the center of the endomorphism algebra $\End_{\catR^\fg}(V)$. This 
is a finite-dimensional commutative $\kk$-algebra.
Let $m_V(u) \in \kk[u]$ be the (monic) minimal polynomial
of 
$\begin{tikzpicture}[B,centerzero,scale=1.2]\draw (0,-0.2) -- (0,0.2);\closeddot{0,0};
\draw[darkg](.25,-.2)\botlabel{V} to (.25,.2);
\end{tikzpicture} \in \End_{\catR^\fg}(BV)$.
This definition makes sense 
because $\End_{\catR^\fg}(B V)$ is finite-dimensional and $\kk[x]$ is a PID;
$m_V(u)$ is the unique monic polynomial which generates
the non-zero kernel of
the algebra homomorphism $\kk[u]\rightarrow \End_{\catR^\fg}(BV), u\mapsto\begin{tikzpicture}[B,centerzero,scale=1.2]\draw (0,-0.2) -- (0,0.2);\closeddot{0,0};
\draw[darkg](.25,-.2)\botlabel{V} to (.25,.2);
\end{tikzpicture}$.
Then we define the {\em $\beta$-degree} of $V$ by
\begin{equation}
\beta(V)\coloneq \deg m_V(u).
\end{equation}

The image of the bubble generating function 
$\begin{tikzpicture}[B,centerzero,scale=.9]
\draw (-0.25,0) arc(180:-180:0.25);
\node at (.57,0) {$\color{black}(u)$};
\end{tikzpicture} \in \End_{\AB}(\one)\llbracket u^{-1} \rrbracket$ under the monoidal functor
$\AB \rightarrow \END_\kk(\catR^\fg)$
is an element of $Z(\catR^\fg)\lround u^{-1}\rround$.
On any $V \in \catR^\fg$, 
this evaluates to a formal power series, which we denote by 
$O_V(u) \in Z_V\llbracket u^{-1} \rrbracket$. 
This formal power series encodes useful
information about the central character of $V$.
To explain what we mean, for a monic polynomial $f(u)$, we use the notation
\begin{equation}
\widetilde{f}(u)\coloneq
(-1)^{\deg f(u)} f(-u).
\end{equation}
Let
\begin{equation}\label{chif}
\chi_f(u) \coloneq \begin{dcases}
\frac{u-h}{u} \cdot \frac{\widetilde{f}(u)}{f(u)}
&\text{if $\deg f(u)$ is even}\\
\frac{u+h}{u} \cdot \frac{\widetilde{f}(u)}{f(u)}&\text{if $\deg f(u)$ is odd.}
\end{dcases}
\end{equation}

\begin{theo}\label{itsreallywednesday}
For any $V \in \catR^\fg$, there exists
a monic polynomial $f(u) \in Z_V[u]$ such that
\begin{align}\label{thur}
\begin{tikzpicture}[B,baseline = -2mm,scale=1.2]
\draw[-] (-0.2,-.4) to (-0.2,.4);
\draw[-,darkg] (0.5,.4) to (0.5,-.4)\botlabel{V};
\objpin{-.2,0}{.5,0}{f(x)};
\end{tikzpicture}&=0,&
O_V(u) &= \chi_f(u).
\end{align}
If $V$ is a brick, that is, $\End_{\catR^\fg}(V) = \kk$,
then $f(u) \coloneq m_V(u)$ has 
these properties.
\end{theo}

\begin{proof}
Suppose that $g(u)$ is a monic 
polynomial in $Z_V[u]$ such that
$$
\begin{tikzpicture}[B,baseline = -2mm,scale=1.2]
\draw[-] (-0.2,-.4) to (-0.2,.4);
\draw[-,darkg] (0.5,.4) to (0.5,-.4)\botlabel{V};
\objpin{-.2,0}{.5,0}{g(x)};
\end{tikzpicture}=0.
$$
For example, we could take $g(u) \coloneq m_V(u) \in \kk[u] \subseteq Z_V[u]$.
Let $g^*(u) \coloneq u O_V(-u) \widetilde{g}(u) \in Z_V \lround u^{-1}\rround$.
The following shows that $g^*(u)$ is a polynomial: for $r \geq 0$ we have that
\begin{align*}
\big[
g^*(u)\big]_{u:-r-1}&=
\big[ u O_{V}(-u) \widetilde{g}(u)\big]_{u:-r-1}
\stackrel{\cref{eatz}}{=}\left[
(u+h)\ 
\begin{tikzpicture}[B,anchorbase]
\draw[-,darkg] (1,.4) to (1,-.4)\botlabel{V};
\catlabel{1,0}{\widetilde{g}(u)};
\end{tikzpicture}\ -\ 
\begin{tikzpicture}[B,anchorbase]
\draw[-] (-0.15,0) arc(180:-180:0.25);
\circledbar{.35,0}{u};
\draw[-,darkg] (1.2,.4) to (1.2,-.4)\botlabel{V};
\catlabel{1.2,0}{u\widetilde{g}(u)};
\end{tikzpicture}
\right]_{u:-r-1}\\&=
-\left[
\begin{tikzpicture}[B,anchorbase]
\draw[-] (-0.15,0) arc(180:-180:0.25);
\circledbar{.35,0}{u};
\draw[-,darkg] (1.4,.4) to (1.4,-.4)\botlabel{V};
\catlabel{1.4,0}{u^{r+1}\widetilde{g}(u)};
\end{tikzpicture}
\right]_{u:-1}
\!\!\!\stackrel{\cref{trick}}{=}-\ 
\begin{tikzpicture}[B,anchorbase]
\draw[-] (-0.55,0) arc(180:-180:0.25);
\draw[-,darkg] (1.2,.4) to (1.2,-.4)\botlabel{V};
\objpin{-.05,0}{1.3,0}{(-x)^{r+1} \widetilde{g}(-x)};
\end{tikzpicture} = (-1)^{r+\deg g(u)}\ 
\begin{tikzpicture}[B,anchorbase]
\draw[-] (-0.25,0) arc(180:-180:0.25);
\draw[-,darkg] (1.2,.4) to (1.2,-.4)\botlabel{V};
\objpin{.25,0}{1.2,0}{x^{r+1} g(x)};
\end{tikzpicture}=0.
\end{align*}
It is clear that $g^*(u) \in Z_V[u]$ is monic, of degree one more than $g(u)$.
Also, we have that
$$
\begin{tikzpicture}[B,baseline = -2mm,scale=1.2]
\draw[-] (-0.2,-.4) to (-0.2,.4);
\draw[-,darkg] (0.5,.4) to (0.5,-.4)\botlabel{V};
\objpin{-.2,0}{.55,0}{g^*(x)};
\end{tikzpicture}=0.
$$
This follows from the following calculation:
\begin{align*}
0
&=
\begin{tikzpicture}[B,anchorbase,scale=1.3]
\draw[-] (0,-0.5) to[out=up,in=180] (0.3,0.2) to[out=0,in=up] (0.45,0) to[out=down,in=0] (0.3,-0.2) to[out=180,in=down] (0,0.5);
\draw[-,darkg] (1,.4) to (1,-.4) \botlabel{V};
\objpin{.45,0}{1,0}{xg(x)};
\end{tikzpicture}
\stackrel{\cref{curlsup3}}{=}
\left[
\ u\ 
\begin{tikzpicture}[anchorbase,scale=1.3,B]
\draw[-] (-0.7,-0.5) -- (-0.7,0.5);
\circledbar{-0.7,0}{u};
\draw (-0.4,0) arc(180:-180:0.2);
\node at (0.24,0) {$\color{black}(u)$};
\draw[-,darkg] (.9,.5) to (.9,-.5) \botlabel{V};
\catlabel{.9,0}{g(u)};
\end{tikzpicture}\ 
- h\ 
\begin{tikzpicture}[anchorbase,scale=1.3,B]
\draw[-] (0,-0.5) -- (0,0.5);
\circled{0,0}{u};
\draw[-,darkg] (.5,.5) to (.5,-.5) \botlabel{V};
\catlabel{.5,0}{g(u)};
\end{tikzpicture}
\right]_{u:-1}\\
&\stackrel{\cref{trick}}{=}
\left[(-1)^{\deg g(u)+1}\  
\begin{tikzpicture}[B,anchorbase,scale=1.3]
\draw[-] (0.3,-0.5) -- (0.3,0.5);
\circledbar{0.3,0}{u};
\draw[-,darkg] (1,.4) to (1,-.4) \botlabel{V};
\catlabel{1,0}{g^*(-u)};
\end{tikzpicture}\ \right]_{u:-1}\!\!\!\!
-h\ \begin{tikzpicture}[B,baseline = -2mm,scale=1.3]
\draw[-] (-0.1,-.4) to (-0.1,.4);
\draw[-,darkg] (0.5,.4) to (0.5,-.4)\botlabel{V};
\objpin{-.1,0}{.55,0}{g(x)};
\end{tikzpicture}
\stackrel{\cref{trick}}{=}
(-1)^{\deg g(u)+1}
\begin{tikzpicture}[B,baseline = -2mm,scale=1.3]
\draw[-] (-0.1,-.4) to (-0.1,.4);
\draw[-,darkg] (0.5,.4) to (0.5,-.4)\botlabel{V};
\objpin{-.1,0}{.55,0}{g^*(x)};
\end{tikzpicture}\ .
\end{align*}

Now we consider the linear combination
$$
f(u)\coloneq \frac{g^*(u) + \big(u +(-1)^{\deg g(u)} h\big) g(u)}{2}.
$$
This satisfies the first property from \cref{thur}.
To see that it satisfies the second property, note that
$$
g^{**}(u) = u^2 O_V(u) O_V(-u) g(u) = (u-h)(u+h) g(u)
$$ 
by \cref{curlsup1}. So
\begin{align*}
f^*(u) &= \frac{g^{**}(u) +\big(u-(-1)^{\deg g(u)} h\big) g^*(u)}{2}
=
\frac{(u-h)(u+h)g(u) +\big(u-(-1)^{\deg g(u)} h\big) g^*(u)}{2}\\
&= 
\big(u-(-1)^{\deg g(u)} h\big)
\frac{\big(u+(-1)^{\deg g(u)} h\big) g(u) + g^*(u)}{2}
=
\big(u+(-1)^{\deg f(u)} h\big) f(u).
\end{align*}
This
shows that $u O_V(-u) \widetilde{f}(u) = \big(u+(-1)^{\deg f(u)} h\big)f(u)$. 
Multiplying both sides by $O_V(u) / f(u)$ and using \cref{curlsup1},
it follows that 
$$
\frac{(u-h)(u+h)}{u}\cdot\frac{\widetilde{f}(u)}{f(u)} = \big(u+(-1)^{\deg f(u)} h\big)O_V(u),
$$
which implies the second property on cancelling a linear factor.

It remains to prove that $m_V(u)$ satisfies \cref{thur} 
in the case that $V$ is a brick. Of course, it satisfies the first property by definition. For the second property,
it is obvious if $BV = 0$ since then $m_V(u) = 1$ and
$O_V(u) = \frac{u-h}{u}$ by \cref{eatz}. So we may assume that $m_V(u)$ is of positive degree.
The construction in the 
first paragraph gives us a monic polynomial $m_V^*(u) \in \kk[u]$ of degree one more than $m_V(u)$ such that
$$
\begin{tikzpicture}[B,baseline = -2mm,scale=1.2]
\draw[-] (-0.2,-.4) to (-0.2,.4);
\draw[-,darkg] (0.5,.4) to (0.5,-.4)\botlabel{V};
\objpin{-.2,0}{.5,0}{m^*_V(x)};
\end{tikzpicture}=0.
$$
It follows that $m_V(u) \mid m^*_V(u)$, so we have that
$m^*_V(u) = m_V(u) (u-c)$ for some $c \in \kk$.
Also
\begin{align*}
0
&=
\begin{tikzpicture}[B,anchorbase,scale=1.3]
\draw[-] (0,-0.5) to[out=up,in=180] (0.3,0.2) to[out=0,in=up] (0.45,0) to[out=down,in=0] (0.3,-0.2) to[out=180,in=down] (0,0.5);
\draw[-,darkg] (1.1,.4) to (1.1,-.4) \botlabel{V};
\objpin{.45,0}{1.1,0}{m_V(x)};
\end{tikzpicture}
\stackrel{\cref{curlsup3}}{=}
\left[
\ 
\begin{tikzpicture}[anchorbase,scale=1.3,B]
\draw[-] (-0.7,-0.5) -- (-0.7,0.5);
\circledbar{-0.7,0}{u};
\draw (-0.4,0) arc(180:-180:0.2);
\node at (0.22,0) {$\color{black}(u)$};
\draw[-,darkg] (.9,.5) to (.9,-.5) \botlabel{V};
\catlabel{.9,0}{m_V(u)};
\end{tikzpicture}\ 
- \frac{h}{u}\ 
\begin{tikzpicture}[anchorbase,scale=1.3,B]
\draw[-] (0,-0.5) -- (0,0.5);
\circled{0,0}{u};
\draw[-,darkg] (.6,.5) to (.6,-.5) \botlabel{V};
\catlabel{.6,0}{m_V(u)};
\end{tikzpicture}
\right]_{u:-1}\\
&=
\left[\frac{(-1)^{\beta(V)+1}}{u}\  
\begin{tikzpicture}[B,anchorbase,scale=1.3]
\draw[-] (0.3,-0.5) -- (0.3,0.5);
\circledbar{0.3,0}{u};
\draw[-,darkg] (1,.4) to (1,-.4) \botlabel{V};
\catlabel{1,0}{m_V^*(-u)};
\end{tikzpicture}- \frac{h}{u}\ 
\begin{tikzpicture}[anchorbase,scale=1.3,B]
\draw[-] (0,-0.5) -- (0,0.5);
\circled{0,0}{u};
\draw[-,darkg] (.6,.5) to (.6,-.5) \botlabel{V};
\catlabel{.6,0}{m_V(u)};
\end{tikzpicture}\ \right]_{u:-1}
\stackrel{\cref{trick}}{=}
(-1)^{\beta(V)}\begin{tikzpicture}[B,baseline = -2mm,scale=1.3]
\draw[-] (-0.1,-.4) to (-0.1,.4);
\draw[-,darkg] (0.5,.4) to (0.5,-.4)\botlabel{V};
\objpin{-.1,0}{.55,0}{q(x)};
\end{tikzpicture}\ ,
\end{align*}
where $q(u) \in \kk[u]$ is the polynomial defined by
$$
q(u) \coloneq \left[\frac{m_V^*(u)-(-1)^{\beta(V)}hm_V(u)}{u}\right]_{u:\geq 0}
=
m_V(u)-
\left[\frac{\big(c+(-1)^{\beta(V)}h\big)m_V(u)}{u}\right]_{u:\geq 0}.
$$
We have that $m_V(u) \mid q(u)$. Since $m_V(u)$ is of positive degree,  
this implies that
$c +(-1)^{\beta(V)}h=0$.
Thus, we have shown that $m_V^*(u) = \big(u+(-1)^{\beta(V)} h\big) m_V(u)$.
From this, the second property from 
\cref{thur} for $f(u) \coloneq m_V(u)$ follows in the same way as at the end of 
the second paragraph of the proof.
\end{proof}

\begin{cor}\label{rationale}
For any $V \in \catR^\fg$, the formal Laurent series
$O_V(u) \in Z_V\lround u^{-1}\rround$
is a rational function.
\end{cor}

\begin{cor}[{\cite[Th.~4.2]{SW2026}}]\label{spec}
If $L \in \catR^\fg$ is an irreducible object 
then
$O_L(u) = \chi_{m_L}(u) \in \kk(u)$, i.e.,
\begin{equation}\label{restatement}
O_L(u) = 
\begin{dcases}
\frac{u-h}{u} \cdot \frac{\widetilde{m}_L(u)}{m_L(u)}
&\text{if $\beta(L)$ is even}\\
\frac{u+h}{u} \cdot \frac{\widetilde{m}_L(u)}{m_L(u)}&\text{if $\beta(L)$ is odd.}
\end{dcases}
\end{equation}
\end{cor}

Let $\Xi$ be the set of all rational functions
$\chi(u) \in \kk(u)$ such that the expansion of $\chi(u)$
at $u=\infty$ is in $1 + u^{-1} \kk\llbracket u^{-1}\rrbracket$
and $\chi(u)\chi(-u)=\frac{(u-h)(u+h)}{u^2}$.
In view of \cref{curlsup1,spec}, 
the rational functions $O_L(u)$
for irreducible $L \in \catR^\fg$ are of this form.
Recall that $I = \kk$. Let 
$X = \bigoplus_{i \in I} \Z \varpi_i$ 
be the weight lattice and
$X^\imath = X / \im(\id+\tau)$ be the Abelian group of iweights introduced in the previous section.
For a monic polynomial $f(u) = \prod_{i \in I} (u-i)^{\beta_i} \in \kk[u]$, 
let 
\begin{equation}\label{ahaha}
\wt^\imath(f) \coloneq \pi(\wt(f))
\qquad\text{ where }\qquad
\wt(f) \coloneq 
\begin{dcases}
\sum_{i \in I} \beta_i \varpi_i&\text{if $\deg f(u)$ is even}\\
\varpi_h + \sum_{i \in I} \beta_i \varpi_i&\text{if $\deg f(u)$ is odd.}
\end{dcases} 
\end{equation}
Note that $\wt(f) \in X$ is even, hence, $\wt^\imath(f)\in X^\imath$ is even.
For any $\lambda \in X^\imath$, let
\begin{equation}\label{maxideals}
\chi_\lambda(u) \coloneq 
\prod_{i \in I} (u-i)^{\del_i-\lambda_i}
\in \Xi.
\end{equation}
The following lemma also uses the notation
$\chi_f(u)$ from \cref{chif}:

\begin{lem}\label{lasers}
The function
$X^\imath_\0 \rightarrow \Xi$, $\lambda \mapsto \chi_\lambda(u)$ is a bijection. Moreover:
\begin{enumerate}
\item 
Every iweight $\lambda \in X_{\0}^\imath$ is of the form
$\wt^\imath(f)$ for some monic $f(u) \in \kk[u]$.
\item
Every rational function $\chi(u) \in \Xi$ is of the form
$\chi_f(u)$ for some monic $f(u) \in \kk[u]$.
\item 
For monic $f(u) \in \kk[u]$
and $\lambda \in X_\0^\imath$,
$\wt^\imath(f) = \lambda$
if and only if $\chi_f(u) = \chi_\lambda(u)$.
\end{enumerate}
\end{lem}

\begin{proof}
The opening statement follows easily from the other three parts, which are proved as follows:

\vspace{2mm}
\noindent
(i) Let $\lambda \in X^\imath$ be an even iweight.
It is the image of 
$\sum_{i \in I,\ \lambda_i > 0} \lambda_i \varpi_i$
if $t(\lambda)=\0$ or
$\varpi_0+\sum_{i \in I,\ \lambda_i > 0} \lambda_i \varpi_i$
if $t(\lambda)= \1$.
This is $\wt(f)$ for
the monic polynomial of even degree 
$$
f(u)
\coloneq \begin{dcases}
\prod_{i \in I,\ \lambda_i > 0} (u-i)^{\lambda_i} &\text{if $t(\lambda)=\0$}\\
u 
\prod_{i \in I,\ \lambda_i > 0} (u-i)^{\lambda_i}
&\text{if $t(\lambda) = \1$.}
\end{dcases}
$$

\vspace{2mm}
\noindent
(ii) Take $\chi(u) \in \Xi$ and 
let $\ord(\chi) \coloneq \sum_{i \in I} m_i \varpi_i$,
where $m_i$ is the order of $i$ as a zero or a pole of $\chi(u)$.
The assumption that $\chi(u)\chi(-u) = \frac{(u-h)(u+h)}{u^2}$ implies that $m_i + m_{-i} = 0$ for $i \in I - \{0,\pm h\}$, $m_0 = -1$, and $m_h + m_{-h} = 1$.
It follows that
$$
\chi(u)=
\begin{dcases}
\frac{u+h}{u}\cdot 
 \prod_{i \in I-\{0\},\ m_i \leq 0}
\left(\frac{u+i}{u-i}\right)^{-m_i}&\text{if $m_h \leq 0$}\\
\frac{u-h}{u}\cdot \prod_{i \in I-\{0\},\ m_i \leq 0}
\left(\frac{u+i}{u-i}\right)^{-m_i}&\text{if $m_{-h} \leq 0$.}
\end{dcases}
$$
This shows that there is a monic $f(u) \in \kk[u]$
such that $\chi(u)=\frac{u\pm h}{u} \cdot \frac{\widetilde{f}(u)}{f(u)}$. Then we either have that $\chi(u) = \chi_f(u)$
or $\chi(u) = \chi_{g}(u)$ where $g(u) \coloneq u f(u)$.

\vspace{2mm}
\noindent
(iii)
Take a monic polynomial $f(u) \in \kk[u]$
and $\lambda \in X^\imath_\0$
with pre-image $\hat\lambda \in X$.
One checks from \cref{sneaky,firstchoice}
and the definitions \cref{chif,maxideals}
that 
\begin{align}\label{padraig}
\ord(\chi_f) &=
\varpi_h - \varpi_0 - (\id - \tau)(\wt(f)),
&
\ord(\chi_\lambda) = \varpi_h-\varpi_0-(\id-\tau)(\hat\lambda). 
\end{align}
Hence, $\chi_f(u) = \chi_\lambda(u)
\Leftrightarrow \ord(\chi_f) = \ord(\chi_\lambda)
\Leftrightarrow (\id-\tau)(\wt(f)) = (\id - \tau)(\hat\lambda)$.
Since $\wt(f)$ and $\hat\lambda$ are both even,
\cref{reallydumblemma} implies that this holds if and only if
$\wt^\imath(f) = \lambda$.
\end{proof}

For any object $V\in\catR^\fg$ and $i \in I$, 
let $\beta_i(V)$ be the multiplicity of $i$ as a root of $m_V(u)$, so that
\begin{align}
m_V(u) &= \prod_{i \in I} (u-i)^{\beta_i(V)}.
\end{align}
We call the numbers $(\beta_i(V))_{i \in I}$ the {\em $\beta$-numbers} of $V$; their sum is the $\beta$-degree defined earlier. We define the {\em iweight} $\wt^\imath(L)$
of 
an irreducible object $L \in \catR^\fg$ to
be $\wt^\imath(m_L)$ as defined above, 
that is, the image in $X_\0^\imath$ of the even weight
\begin{equation}\label{whyitiseven}
\wt(L)\coloneq 
\begin{dcases}
\sum_{i \in I} \beta_i(L) \varpi_i&\text{if $\beta(L)$ is even}\\
\varpi_h+\sum_{i \in I} \beta_i(L) \varpi_i&\text{if $\beta(L)$ is odd.}
\end{dcases}
\end{equation}
\cref{spec,lasers} imply:

\begin{cor}\label{blockdec}
For irreducibles $K, L \in \catR^\fg$, $O_K(u) = O_L(u)$
if and only if 
$\wt^\imath(K)=\wt^\imath(L)$.
\end{cor}

From \cref{blockdec}, it follows that irreducible objects of $\catR^\fg$
of different iweight have different central character.
In the Abelian case, letting
$\catR_\lambda$ be the Serre subcategory
of $\catR$
generated by
all irreducible objects $L \in \catR$ with $\wt^\imath(L) = \lambda$,
we obtain the decomposition
\begin{equation}\label{blocks}
\catR = \bigoplus_{\lambda \in X_\0^\imath} \catR_\lambda
\end{equation}
of $\catR$ into ``blocks''.
We mean that there are no non-zero homomorphisms between different ``blocks''. We are calling $\catR_\lambda$ a ``block'' rather than a block because it may decompose further, and it could also be zero. 

There is also a ``block'' 
decomposition in the Karoubian case, which warrants a bit more discussion.
Given $\lambda \in X^\imath_\0$, 
it still makes sense to consider the Serre subcategory $\Mod{\catR}_\lambda$
of $\Mod{\catR}$ consisting of the locally finite-dimensional
$\catR$-modules
whose irreducible subquotients $L$ satisfy $\wt^\imath(L)=\lambda$.
Then $\catR_\lambda$ and $\catR^\fg_\lambda$ may be defined to be the full subcategories with objects that are all of the objects
of $\catR$ or $\catR^\fg$, respectively, 
that belong to $\Mod{\catR}_\lambda$.
Equivalently, $\catR_\lambda$ is the additive
subcategory of $\catR$ 
generated by the indecomposable objects $P \in \catR$
which, when viewed as indecomposable projective modules,
have unique irreducible quotients $L$ satisfying
$\wt^\imath(L)=\lambda$.
The Krull-Schmidt theorem implies that
\begin{equation}\label{blocksagain}
\catR = \bigoplus_{\lambda \in X_\0^\imath} \catR_\lambda.
\end{equation}
Since finitely generated modules have projective covers,
there is a similar decomposition of $\catR^\fg$.

\begin{rem}\label{moredirect}
There is another more direct 
way to determine whether an indecomposable 
object $V \in \catR^\fg$ belongs to the ``block'' $\catR^\fg_\lambda$. 
Indecomposability implies that $Z_V$ is a
finite-dimensional commutative local $\kk$-algebra with unique 
maximal ideal $\mathfrak{m}$ such that
$Z_V / \mathfrak{m} = \kk$.
For $f(u) \in Z_V[u]$, we write $\bar f(u)$ for the
image of $f(u)$ in $\kk[u]$ under the canonical quotient map.
Similarly, for a rational function $R(u) \in Z_V(u)$,
we write $\bar R(u)$ for its image in $\kk(u)$.
By \cref{itsreallywednesday}, 
$O_V(u) = \chi_f(u)$ for some monic $f(u) \in Z_V[u]$.
It follows that $\bar O_V(u) = \chi_{\bar f}(u)$.
Then we have that $V \in \catR^\fg_\lambda$
where $\lambda \in X^\imath_\0$ is 
$\wt^\imath(\bar f)$;
equivalently, by \cref{lasers}(iii),
$\bar O_V(u) = \chi_\lambda(u)$.
This agrees with the description in the previous paragraph
because $\bar O_V(u) = O_L(u)$ for any irreducible quotient $L$ of $V$. We leave the proof of this statement to the reader.
\end{rem}

Now we are going to play the same  
game as in \cite[Sec.~4]{HKM} (see also \cite[Sec.~4]{BSW-update}).
String diagrams in the remainder of the section
will be string diagrams for natural transformations
between endofunctors of $\catR^\fg$.
The endofunctor $B$ decomposes as
\begin{align}\label{eigenfunctors}
B &= \bigoplus_{i \in I} B_i,
\end{align}
where $B_i$ is the subfunctor of $B$ such that $B_i V$ is the generalized $i$-eigenspace of the endomorphism
$\begin{tikzpicture}[B,centerzero]
\draw[-] (0,-0.3) -- (0,0.3);
\closeddot{0,0};
\draw[darkg] (.28,-0.3) -- (.28,0.3);
\end{tikzpicture}\ \colon BV\rightarrow BV$
for each $V \in \catR^\fg$.
The identity endomorphism of $B_i$
will be represented using the labelled string
$\begin{tikzpicture}[B,centerzero]
\draw[-] (0,-0.2) \botlabel{i} -- (0,0.2);
\end{tikzpicture}$.
We use the 
string diagrams
\begin{align}\label{projectors}
\begin{tikzpicture}[B,centerzero]
\draw (0.08,-.3) \botlabel{i} to (.08,0);
\draw[B,-] (.08,0) to (0.08,.3);
\notch{.08,0};
\end{tikzpicture}
\colon B_i &\Rightarrow B,
&
\begin{tikzpicture}[B,centerzero]
\draw[B] (0.08,-.3) to (0.08,0);
\draw[-] (.08,0) to (0.08,.3)\toplabel{i};
\notch{.08,0};
\end{tikzpicture}
\colon B &\Rightarrow B_i
\end{align}
to denote the inclusions and projections.
We then have that
\begin{align}\label{orthogonality}
\begin{tikzpicture}[B,centerzero]
\draw[-] (0.08,-.3) \botlabel{i} to (0.08,.3)\toplabel{j};
\draw[B] (0.08,-.1) to (0.08,.1);
\notch{.08,.1};
\notch{.08,-.1};
\end{tikzpicture}
&= \delta_{i,j}
\begin{tikzpicture}[B,centerzero]
\draw[-] (0.08,-.3) \botlabel{i} to (0.08,.3);
\end{tikzpicture}\ ,&
\begin{tikzpicture}[B,centerzero]
\draw[-] (0.08,-.3) to (0.08,.3);
\draw[B] (0.08,-.1) to (0.08,.1);
\end{tikzpicture}
&= \sum_{i \in I}\ 
\begin{tikzpicture}[B,centerzero]
\draw[-] (0.08,-.3) to (0.08,.3);
\draw[B] (0.08,-.1) to (0.08,.1);
\notch{.08,.1};
\notch{.08,-.1};
\stringlabel{0.3,0}{i};
\end{tikzpicture}\ .
\end{align}
Note that the infinite sum on the right-hand side of \cref{orthogonality} makes sense as a natural transformation because all but finitely many terms are zero
after evaluation on any
object $V \in\catR^\fg$ (in the Karoubian case, this uses that such modules $V$ are finitely generated).

The restriction of the dot gives an
endomorphism
$\begin{tikzpicture}[B,centerzero]
\draw[-] (0.08,-.2) \botlabel{i} to (.08,.2);
\closeddot{0.08,0};
\end{tikzpicture}$
of the functor $B_i$
such that
\begin{align}\label{pan}
\begin{tikzpicture}[B,centerzero]
\draw (0.08,-.3) \botlabel{i} to (.08,0.1);
\draw[B] (.08,0.1) to (0.08,.3);
\notch{.08,.1};
\closeddot{.08,-.1};
\end{tikzpicture}
&=
\begin{tikzpicture}[B,centerzero]
\draw[B] (0.08,-.3) \botlabel{i} to (.08,-0.1);
\draw[-] (.08,-0.1) to (0.08,.3);
\notch{.08,-.1};
\closeddot{.08,.1};
\end{tikzpicture}\ ,&
\begin{tikzpicture}[B,centerzero]
\draw (0.08,-.3)  to (.08,0.1);
\draw[B,-] (.08,0.1) to (0.08,.3) \toplabel{i};
\notch{.08,.1};
\closeddot{.08,-.1};
\end{tikzpicture}
&=
\begin{tikzpicture}[B,centerzero]
\draw[B] (0.08,-.3)  to (.08,-0.1);
\draw[-] (.08,-0.1) to (0.08,.3)\toplabel{i};
\notch{0.08,-.1};
\closeddot{.08,.1};
\end{tikzpicture}\ .
\end{align}
We refer to these natural transformations as {\em projected dots}.
In \cref{sabcat}, we introduced a diagrammatic shorthand 
``pinning'' a polynomial $f(x)\in \kk[x]$ to a dot in a Brauer string diagram. 
This makes sense for projected dots as well.
In \cref{iqg}, we went a bit further, extending the convention to allow polynomials
in two variables $x,y$ to be pinned to two dots---double pins---and polynomials in $x,y,z$ to be pinned to three dots---triple pins. 
This was in the iquantum setting, but double and triple pins make
sense in the Brauer setting too.

Composing the cup and cap with projectors 
gives natural transformations
\begin{align}\label{flimsy}
\begin{tikzpicture}[B,anchorbase,scale=1.1]
\draw[-] (0.3,0.4)\toplabel{j} to (0.3,.2);
\draw[-] (-0.1,.2) to (-0.1,0.4)\toplabel{i};
\draw[B] (.3,.2) to[out=-90,in=-90,looseness=1.5] (-.1,.2);
\notch{-.1,.2};
\notch{.3,.2};
\end{tikzpicture}
\colon&\id_{\catR^\fg} \Rightarrow B_i \circ B_j,&
\begin{tikzpicture}[B,anchorbase,scale=1.1]
\draw[-] (0.3,-0.4)\botlabel{j} to (0.3,-.2);
\draw[-] (-0.1,-.2) to (-0.1,-0.4)\botlabel{i};
\draw[B] (.3,-.2) to[out=90,in=90,looseness=1.5] (-.1,-.2);
\notch{-.1,-.2};
\notch{.3,-.2};
\end{tikzpicture}\colon&B_i \circ B_j \Rightarrow \id_{\catR^\fg}.
\end{align}
These are zero if $i \neq -j$, as may be proved by a similar argument to the proof of \cite[Lem.~4.1]{HKM} using the dot slide relations \cref{dotslide}. 
It follows that
\begin{align}\label{icecreamtruck}
\begin{tikzpicture}[B,centerzero,scale=1.3]
\draw[-,B] (-0.3,.1) arc(180:0:0.15);
\draw[-,B] (0,-.1) arc(180:360:0.15);
\draw (0,-.1) to (0,.1);
\draw (-.3,0.1) to (-.3,-.4)\botlabel{\!-i\;};
\draw[-] (.3,-0.1) to (.3,.4)\toplabel{\!-i\;};
\notch{0,-.1};
\notch{.3,-.1};
\notch{-.3,.1};
\notch{0,.1};
\stringlabel{0.11,0.02}{i};
\end{tikzpicture}
&=
\sum_{j \in I} \begin{tikzpicture}[B,centerzero,scale=1.3]
\draw[-,B] (-0.3,.1) arc(180:0:0.15);
\draw[-,B] (0,-.1) arc(180:360:0.15);
\draw (0,-.1) to (0,.1);
\draw (-.3,0.1) to (-.3,-.4)\botlabel{\!-i\;};
\draw[-] (.3,-0.1) to (.3,.4)\toplabel{\!-i\;};
\notch{0,-.1};
\notch{0,.1};
\notch{.3,-.1};
\notch{-.3,.1};
\stringlabel{0.11,0.02}{j};
\end{tikzpicture}
=
\begin{tikzpicture}[B,centerzero,scale=1.3]
\draw[-,B] (-0.3,.1) arc(180:0:0.15);
\draw[-,B] (0,-.1) arc(180:360:0.15);
\draw[B] (0,-.1) to (0,.1);
\draw (-.3,0.1) to (-.3,-.4)\botlabel{\!-i\;};
\draw[-] (.3,-0.1) to (.3,.4)\toplabel{\!-i\;};
\notch{.3,-.1};
\notch{-.3,.1};
\end{tikzpicture}
=\begin{tikzpicture}[B,centerzero,scale=1.3]
\draw[-] (0,-0.4)\botlabel{\!-i\;} -- (0,0.4)\toplabel{\!-i\;};
\draw[B] (0,.15) to (0,-.15);
\notch{0,.15};
\notch{0,-.15};
\end{tikzpicture}
=
\begin{tikzpicture}[B,centerzero,scale=1.3]
\draw[-] (0,-0.4)\botlabel{\!-i\;} -- (0,0.4);
\end{tikzpicture}
\ ,&
\begin{tikzpicture}[B,centerzero,scale=1.3]
\draw[-,B] (-0.3,-.1) arc(180:360:0.15);
\draw[-,B] (0,.1) arc(180:0:0.15);
\draw (0,-.1) to (0,.1);
\draw (-.3,-0.1) to (-.3,.4)\toplabel{i};
\draw[-] (.3,0.1) to (.3,-.4)\botlabel{i};
\notch{0,.1};
\notch{0,-.1};
\notch{.3,.1};
\notch{-0.3,-.1};
\stringlabel{0.14,0.02}{-i};
\end{tikzpicture}
&=
\sum_{j \in I} \begin{tikzpicture}[B,centerzero,scale=1.3]
\draw[-,B] (-0.3,-.1) arc(180:360:0.15);
\draw[-,B] (0,.1) arc(180:0:0.15);
\draw (0,-.1) to (0,.1);
\draw (-.3,-0.1) to (-.3,.4)\toplabel{i};
\draw[-] (.3,0.1) to (.3,-.4)\botlabel{i};
\notch{0,.1};
\notch{0,-.1};
\notch{.3,.1};
\notch{-.3,-.1};
\stringlabel{0.11,0.02}{j};
\end{tikzpicture}
=
\begin{tikzpicture}[B,centerzero,scale=1.3]
\draw[-,B] (-0.3,-.1) arc(180:360:0.15);
\draw[-,B] (0,.1) arc(180:0:0.15);
\draw[B] (0,-.1) to (0,.1);
\draw (-.3,-0.1) to (-.3,.4)\toplabel{i};
\draw[-] (.3,0.1) to (.3,-.4)\botlabel{i};
\notch{.3,.1};
\notch{-.3,-.1};
\end{tikzpicture}
=
\begin{tikzpicture}[B,centerzero,scale=1.3]
\draw[-] (0,-0.4)\botlabel{i} -- (0,0.4)\toplabel{i};
\draw[B] (0,.15) to (0,-.15);
\notch{0,.15};
\notch{0,-.15};
\end{tikzpicture}
=
\begin{tikzpicture}[B,centerzero,scale=1.3]
\draw[-] (0,-0.4)\botlabel{i} -- (0,0.4);
\end{tikzpicture}
\ .
\end{align}
This shows that the {\em projected cups and caps}
\begin{align}\label{spam}
\begin{tikzpicture}[B,anchorbase,scale=1.4]
\draw[-] (0.3,0.4)\toplabel{\!-i\;} to (0.3,.2);
\draw[-] (-0.1,.2) to (-0.1,0.4)\toplabel{i};
\draw[B] (.3,.2) to[out=-90,in=-90,looseness=1.5] (-.1,.2);
\end{tikzpicture}&\coloneq\begin{tikzpicture}[B,anchorbase,scale=1.4]
\draw[-] (0.3,0.4)\toplabel{\!-i\;} to (0.3,.2);
\draw[-] (-0.1,.2) to (-0.1,0.4)\toplabel{i};
\draw[B] (.3,.2) to[out=-90,in=-90,looseness=1.5] (-.1,.2);
\notch{-.1,.2};
\notch{.3,.2};
\end{tikzpicture},& 
\begin{tikzpicture}[B,anchorbase,scale=1.4]
\draw[-] (0.3,-0.4)\botlabel{i} to (0.3,-.2);
\draw[-] (-0.1,-.2) to (-0.1,-0.4)\botlabel{\!-i\;};
\draw[B] (.3,-.2) to[out=90,in=90,looseness=1.5] (-.1,-.2);
\end{tikzpicture}&\coloneq\begin{tikzpicture}[B,anchorbase,scale=1.4]
\draw[-] (0.3,-0.4)\botlabel{i} to (0.3,-.2);
\draw[-] (-0.1,-.2) to (-0.1,-0.4)\botlabel{\!-i\;};
\draw[B] (.3,-.2) to[out=90,in=90,looseness=1.5] (-.1,-.2);
\notch{-.1,-.2};
\notch{.3,-.2};
\end{tikzpicture}
\end{align}
are the unit and counit of an adjunction making
$(B_{-i}, B_i)$ into an adjoint pair for each $i \in I$.

When pinned to a string of color $i$, $x - i$ is 
{\em locally nilpotent}, that is,
it is nilpotent on evaluation at any object of $\catR^\fg$
(although the nilpotency degree may not be bounded as 
the object varies).
This means that the 
pin notation makes sense not just for polynomials $f(x) \in \kk[x]$
pinned to a string of color $i$ but also
for any power series $f(x) \in \kk\llbracket x-i\rrbracket$.
This observation extends in the obvious way to double and triple pins too.

We have already introduced projected dots, caps and cups. There are also {\em projected crossings}
\begin{align}\label{polly}
\begin{tikzpicture}[centerzero,B]
\draw[-] (0.3,-.4) \botlabel{j} to (-0.3,.4)\toplabel{l};
\draw[-] (-0.3,-.4) \botlabel{i} to (0.3,.4)\toplabel{k};
\projcr{0,0};
\end{tikzpicture}
&\coloneq
\begin{tikzpicture}[centerzero,B]
\draw[-] (0.3,-.4) \botlabel{j} to (-0.3,.4)\toplabel{l};
\draw[-] (-0.3,-.4) \botlabel{i} to (0.3,.4)\toplabel{k};
\notch[40]{.11,-.14};
\notch[40]{-.11,.14};
\notch[-40]{.11,.14};
\notch[-40]{-.11,-.14};
\end{tikzpicture}\ .
\end{align}
The same argument as used to prove
\cite[Lem.~4.1]{HKM} 
shows that the projected crossing is zero unless at least one of the following holds:
\begin{itemize}
\item $i=k$ and $j=l$; 
\item $i = l$ and $j=k$;
\item $i=-j$ and $k=-l$.
\end{itemize}
The dot slide relation \cref{demazure} implies that
\begin{align}\label{moredemazure}
\begin{tikzpicture}[B,centerzero,scale=1]
\draw[-] (-0.3,-0.4)\botlabel{i}  -- (0.3,0.4)\toplabel{k};
\draw[-] (0.3,-0.4)\botlabel{j}  -- (-0.3,0.4)\toplabel{l};
\projcr{0,0};
\closeddot{-0.18,0.24};
\end{tikzpicture}
-
\begin{tikzpicture}[B,centerzero,scale=1]
\draw[-] (-0.3,-0.4)\botlabel{i} -- (0.3,0.4)\toplabel{k};
\draw[-] (0.3,-0.4)\botlabel{j}  -- (-0.3,0.4)\toplabel{l};
\projcr{0,0};
\closeddot{0.18,-0.24};
\end{tikzpicture}\ &= \ 
\delta_{i,l} \delta_{j,k}
\begin{tikzpicture}[B,centerzero,scale=1]
\draw[-] (-0.3,-0.4)\botlabel{i}  -- (-0.3,0.4);
\draw[-] (0.3,-0.4)\botlabel{k}  -- (0.3,0.4);
\end{tikzpicture} -
\delta_{i,-j}\delta_{k,-l}
\ 
\begin{tikzpicture}[B,centerzero,scale=1]
\draw[-] (-0.3,-0.4)  to [looseness=1.7,out=90,in=90] (0.3,-0.4) \botlabel{j};
\draw[-] (0.3,0.4) to [looseness=1.7,out=-90,in=-90] (-0.3,0.4) \toplabel{l};
\end{tikzpicture}\ =\ 
\begin{tikzpicture}[B,centerzero,scale=1]
\draw[-] (-0.3,-0.4)\botlabel{i}  -- (0.3,0.4)\toplabel{k};
\draw[-] (0.3,-0.4)\botlabel{j}  -- (-0.3,0.4)\toplabel{l};
\projcr{0,0};
\closeddot{-0.18,-0.24};
\end{tikzpicture}
-
\begin{tikzpicture}[B,centerzero,scale=1]
\draw[-] (-0.3,-0.4)\botlabel{i} -- (0.3,0.4)\toplabel{k};
\draw[-] (0.3,-0.4)\botlabel{j}  -- (-0.3,0.4)\toplabel{l};
\projcr{0,0};
\closeddot{0.18,0.24};
\end{tikzpicture}\ .
\end{align}
The following is the Brauer analog of \cite[Lem.~4.2]{HKM} and is proved in the same way: if $i \neq k$ and $j \neq l$ (required for the pins in the equation below to make sense) then
\begin{align}
\begin{tikzpicture}[centerzero,B,scale=1]
\draw[-] (0.3,-.4) \botlabel{j} to (-0.3,.4)\toplabel{l};
\draw[-] (-0.3,-.4) \botlabel{i} to (0.3,.4)\toplabel{k};
\projcr{0,0};
\end{tikzpicture}
&=
\delta_{i,l} \delta_{j,k}\ 
\begin{tikzpicture}[centerzero,B,scale=1]
\draw[-] (0.3,-.4) \botlabel{k} to (0.3,.4);
\draw[-] (-0.3,-.4) \botlabel{i} to (-0.3,.4);
\pinpin{-.3,0}{.3,0}{1,0}{\frac{1}{x-y}};
\end{tikzpicture}\ -
\delta_{i,-j} \delta_{k,-l}\ 
\begin{tikzpicture}[centerzero,B,scale=1]
\draw[-] (-0.3,-0.4) to [looseness=1.7,out=90,in=90] (0.3,-0.4) \botlabel{j};
\draw[-] (0.3,0.4) to [looseness=1.7,out=-90,in=-90] (-0.3,0.4) \toplabel{l};
\pinpin{.16,-.14}{-.16,.14}{-.85,.14}{\frac{1}{x-y}};
\end{tikzpicture}\label{eggs}\ .
\end{align}
We will not use this subsequently, but it explains the first two terms on the right-hand side of the key definition \cref{crossingdef} appearing later.

\begin{theo}\label{actions}
If $K$ is an irreducible subquotient of $B_i L$ for some irreducible $L \in \catR^\fg$ and $i \in I$ then $\wt^\imath(K) = \wt^\imath(L)-\alpha_i$.
\end{theo}

\begin{proof}
Recall the definition of 
$\ord(\chi)$ from the proof of \cref{lasers}(ii).
The bubble slide relation
\cref{curlsup2} implies that
\begin{equation}\label{helpful}
O_K(u) =  
\frac{(u-i)^2}{(u-i+1)(u-i-1)}
\cdot \frac{(u+i+1)(u+i-1)}{(u+i)^2}\cdot O_L(u).
\end{equation}
The order of the rational function appearing in front of $O_L(u)$ in this equation is
$(\id-\tau)(\alpha_i)$, so $\ord(O_K) = \ord(O_L)+(\id-\tau)(\alpha_i)$.
Since 
$\ord(O_L) = \varpi_h-\varpi_0-(\id - \tau) (\wt(L))$
by \cref{spec,padraig},
it follows that
$$
(\id-\tau)(\wt(K)) = (\id-\tau)(\wt(L))-(\id-\tau)(\alpha_i)
= (\id-\tau)(\wt(L)-\alpha_i).
$$
Since $\wt(K)$ and $\wt(L)-\alpha_i$ are even, 
\cref{reallydumblemma} implies finally that
$\wt^\imath(K)=\wt^\imath(L)-\alpha_i$.
\end{proof}

\begin{cor}\label{rightway}
The eigenfunctor 
$B_i\colon \catR^\fg \rightarrow \catR^\fg$ 
maps objects in $\catR^\fg_\lambda$ to objects in $\catR^\fg_{\lambda-\alpha_i}$.
\end{cor}

\begin{proof}
The Abelian case follows from exactness and \cref{actions}. 
In the Karoubian case, take $V\in\catR^\fg_\lambda$.
Since $V$ is finitely generated, so is $B_i V$.
If $B_i V$ has a non-zero component in $\catR^\fg_\mu$
for some $\mu \in X^\imath_\0$,
that component is finitely generated, hence it has an irreducible quotient $K$.
By adjunction
$0\neq\Hom_{\catR^\fg}(B_iV,K)
\cong\Hom_{\catR^\fg}(V,B_{-i}K).$
The image of a non-zero morphism 
$f:V \rightarrow B_{-i}K$ has an irreducible quotient $L$. 
Since $L$ is a quotient of $V$, 
we have $\wt^\imath(L)=\lambda$.
Since $L$ is a subquotient of $B_{-i}K$, \cref{actions}
implies that
$\lambda=\mu-\alpha_{-i}=\mu+\alpha_i$.
Thus $\mu=\lambda-\alpha_i$, as required.
\end{proof}

\begin{rem}
\cref{blockdec,rightway} are analogs of the results \cite{HKM} about Heisenberg categorifications, and the proofs are only slightly more difficult.
Both results are also proved in \cite{GRS-2026}. 
We proved them independently before \cite{GRS-2026} became available.
\end{rem}

By the {\em spectrum} of the categorification $\catR$, we mean
the set $I_\catR$ of roots of the minimal polynomials $m_L(u)$
for all irreducible objects $L \in \catR^\fg$. In other words,
$$
I_\catR = \{i \in I\:|\:\beta_i(L) > 0\text{ for some irreducible }L \in \catR^\fg\}.
$$
This is an invariant of $\catR$ which is interesting to determine in examples. Recall from the previous section that we denote the $\sim$-equivalence class of $k \in I$ by $[k]$; it is the vertex set of one of the Satake diagrams in \cref{dynkintable}.

\begin{theo}\label{spectrum}
The spectrum $I_\catR$ is the union of 
sets chosen from the $\sim$-equivalence classes
$[k]$ for $k \in I$, the exceptional 
set $\{0\}$, and the exceptional set
$[h] - \{\pm h\}$.
\end{theo}

\begin{proof}
The theorem follows from the following claims.

\vspace{2mm}
\noindent
\underline{Claim 1}:
$i \in I_\catR \Rightarrow -i \in I_\catR$.

\vspace{1mm}
\noindent
Proof:
There exists an irreducible $L \in \catR^\fg$ 
such that $\beta_i(L) >0$, hence, $B_i L \neq 0$.
Since $B_i L$ is finitely generated, it has an irreducible quotient $K$. Then
$$
0\neq\Hom_{\catR^\fg}(B_i L,K)
\cong\Hom_{\catR^\fg}(L,B_{-i}K),
$$
proving that $-i\in I_{\catR}$.

\vspace{2mm}
\noindent
\underline{Claim 2}: $i \in I_\catR, i \notin \{0,-1,-h,-h-1\} \Rightarrow i+1 \in I_\catR$.

\vspace{1mm}
\noindent
Proof:
Take an irreducible $L \in \catR^\fg$ with $\beta_i(L) > 0$.
This implies that $B_i L$ is non-zero.
Let $K$ be an irreducible quotient of $B_i L$.
The equality \cref{helpful}
plus the explicit formula from \cref{spec}
implies that
\begin{multline*}
\big(u-(-1)^{\beta(K)} h\big)(u-i+1)(u+i)^2(u-i-1)
\widetilde{m}_K(u) m_L(u)=\\
\big(u-(-1)^{\beta(L)} h\big)
(u+i+1)(u-i)^2(u+i-1)
\widetilde{m}_L(u) m_K(u).
\end{multline*}
Since $u-i-1$ divides the left-hand side,
and it does not equal $u \pm h$, $u+1+i$, $u-i$ or $u-1+i$
by the assumption on $i$,
it must divide 
$\widetilde{m}_L(u) m_K(u)$. If $u-i-1$ divides $\widetilde{m}_L(u)$ then $u+i+1$ divides $m_L(u)$, i.e.,
$\beta_{-i-1}(L) > 0$. Hence, $-i-1 \in I_\catR$ and $i+1 \in I_\catR$ by Claim 1. 
If $u-i-1$ divides $m_K(u)$ then $\beta_{i+1}(K) > 0$, so again $i+1 \in I_\catR$.

\vspace{2mm}
\noindent
\underline{Claim 3}: $1 \in I_\catR\Rightarrow 0 \in I_\catR$.

\vspace{1mm}
\noindent
Proof:
Let $L$ be an irreducible in $\catR^\fg$ with $B_1 L \neq 0$
and $K$ be an irreducible quotient of $B_1 L$.
\cref{actions} implies that $\wt^\imath(K) = \wt^\imath(L) - \alpha_1$. 
Recalling
\cref{ambiguity},
it follows that one of the weights $\wt^\imath(K)$ or $\wt^\imath(L)$ is of odd type, so either $\beta_0(K)$ or $\beta_0(L)$ is odd.
This shows that $0 \in I_\catR$.
\end{proof}

\begin{eg}\label{rseg}
As we mentioned already in the introduction, 
the main source of interesting 
examples of Abelian Brauer categorifications
comes from representations of 
the algebraic groups $\operatorname{Sp}_{2n}(\kk)\:(n \geq 1), \operatorname{SO}_n(\kk)\:(n \geq 2)$ or $\operatorname{Spin}_n(\kk)\:(n \geq 2)$, or their Lie algebras. In the Lie algebra case over the field $\C$, this is developed in detail in \cite{RS19}: the functor $B$ is defined by tensoring with the natural representation $V$, and the dot acts as the natural transformation $V \otimes - \Rightarrow V \otimes -$
defined by $\varepsilon\left(\Omega+\tfrac{\dim V-\varepsilon}{2}\right)$ where $\Omega$ is the Casimir tensor and 
$\varepsilon = -1$ in the symplectic case or $1$ in the orthogonal or spin cases. The spectra
for the categories of rational representations of these groups over $\C$ are as follows
(by ``spin representation'' we mean a rational representation of the spin group on which the non-trivial element of the kernel of the homomorphism to the special orthogonal group acts as $-1$):
$$
\begin{array}{|c|c|}
\hline
\text{Category}&\text{Spectrum}\\\hline
\text{Rational representations of
$\operatorname{Sp}_{2n}(\C)$}
&[h]-\{\pm h\}\\
\text{Rational representations of
$\operatorname{SO}_{2n}(\C)$}&[h]\\
\text{Rational representations of
$\operatorname{SO}_{2n+1}(\C)$}&[0]\\
\text{Spin representations 
of $\operatorname{Spin}_{2n}(\C)$}&[0]\\
\text{Spin representations of $\operatorname{Spin}_{2n+1}(\C)$}
&\{0\}\cup [h]-\{\pm h\}\\
\hline
\end{array}
$$
This follows from the results of \cite{RS19} using the classical lemma recalled in \cite[Lem.~4.8]{RS19}.
\end{eg}

\begin{eg}
A trivial example of a Brauer categorification:
the category $\SVec$ of finite-dimensional vector superspaces with morphisms that are even linear maps.
The endofunctor $B$ is the parity-switching functor, 
so $B^2 = \id_\SVec$.
The cup, the cap and the crossing act as identity natural transformations,
and $\begin{tikzpicture}[B,centerzero,scale=1.2]\draw (0,-0.2) -- (0,0.2);\closeddot{0,0};
\end{tikzpicture}\colon B\Rightarrow B$ is zero.
Up to isomorphism, there are only two irreducible objects 
$L_0 \coloneq \kk^{1|0}$ and $L_1 \coloneq \kk^{0|1}$, with minimal polynomials
$m_{L_0}(u) = m_{L_1}(u) = u$.
It follows that $\wt^\imath(L_0)=\wt^\imath(L_1) = \varpi_0+\varpi_h$, and
its spectrum is $\{0\}$.
\end{eg}

\begin{eg}
Another example in positive characteristic comes from 
the category $\operatorname{Ver}_p(\mathrm{SL}_2)$,
the semisimplification of the category of tilting modules for $\mathrm{Sp}_2(\kk) = \mathrm{SL}_2(\kk)$
in characteristic $p \geq 5$.
The endofunctor $B$ is induced by $V \otimes-$ for the natural module $V$,
and the dot acts as in \cref{rseg}.
The irreducible objects are
$L(0),\dots,L(p-2)$, and $V \otimes L(k) \cong L(k-1) 
\oplus L(k+1)$
with $L(-1) \coloneq 0$ and $L(p-1)$ being 
negligible, hence 0.  
The spectrum is $[h] - \{\pm h\}$.
\end{eg}

\setcounter{section}{4}

\section{\texorpdfstring{Iquantum categorifications}{Iquantum categorifications}}\label{iqcat}

Now we are going to look in the other direction for a while.
Let $\UU^\imath$ be the 2-iquantum group defined in \cref{iqg}.
We will usually ignore the grading, viewing it simply as a $\kk$-linear 2-category.
A {\em 2-representation} of $\UU^\imath$ is a strict $\kk$-linear 2-functor
from $\UU^\imath$ to the 2-category of $\kk$-linear categories.
Thus:
\begin{itemize}
\item For each $\lambda \in X^\imath$, there is given a $\kk$-linear category
$\catR_\lambda$.
\item For each $i \in I$ and $\lambda \in X^\imath$, there is given a $\kk$-linear functor
$B_i\one_\lambda = \one_{\lambda-\alpha_i} B_i\colon \catR_\lambda \rightarrow \catR_{\lambda-\alpha_i}$.
\item For each of the generating 2-morphisms of $\UU^\imath$ listed in \cref{iQgens}
there is given a natural transformation between these functors satisfying the defining relations from \cref{iQdef}.
\end{itemize}
We simply say that $(\catR_\lambda)_{\lambda \in X^\imath}$ is a 2-representation of $\UU^\imath$ when all of this data is in our hands.

We say that $(\catR_\lambda)_{\lambda \in X^\imath}$ is a {\em locally finite Abelian 2-representation} 
or a {\em locally finite Karoubian 2-representation}
of $\UU^\imath$ if in addition each of the categories
$\catR_\lambda$ is a locally finite Abelian or Karoubian category,
respectively.
Then it is convenient to set
\begin{equation}\label{hotsauce}
\catR \coloneq \bigoplus_{\lambda \in X^\imath} \catR_\lambda
\end{equation}
and extend the functors
arising from the action of $\UU^\imath$
so that the 1-morphism $\one_\lambda$ acts as the identity on $\catR_\lambda$ 
and as the zero functor on other $\catR_\mu$'s,
and $B_i \coloneq \bigoplus_{\lambda \in X^\imath} B_i \one_\lambda\colon\catR \rightarrow \catR$.

Since $p(\lambda -\alpha_i)=p(\lambda)$, the subcategories
\begin{align}
\catR_{\0} &\coloneq \bigoplus_{\substack{\lambda \in X^\imath
\\p(\lambda)=\0}} \catR_\lambda,&
\catR_{\1} &\coloneq \bigoplus_{\substack{\lambda \in X^\imath
\\p(\lambda)=\1}} \catR_\lambda
\end{align}
are stable under the action of the endofunctor
$B_i$. 
We say that $\catR$ is {\em even} 
if $\catR_{\1}$ is trivial
and {\em odd} if $\catR_{\0}$ is trivial.

\begin{rem}\label{whyeven}
There is a strict $\kk$-linear 2-isomorphism
$S^{-1}\circ \Omega^\imath\colon \UU^\imath \Iso \UU^\imath$
defined by composing the inverse of the symmetry $S$ from \cite[Lem.~3.7]{BWW} with the symmetry $\Omega^\imath$ from \cite[Sec.~3.3]{BWW}. 
This composition is a strict $\kk$-linear 2-functor taking the object
$\lambda \in X^\imath$ to $\varpi_h - \lambda$
and the generating 1-morphism $B_i \one_\lambda$ to
$B_{-i} \one_{\varpi_h-\lambda}$.
We omit the precise description of $S^{-1}\circ\Omega^\imath$
on 2-morphisms since it is not needed here.
The pull-back of an even (resp., odd) iquantum categorification
along this isomorphism is odd (resp., even). Thus, nothing is lost by restricting attention just to the study of even iquantum categorifications. We will do this later on, but there is no need to do it yet.
\end{rem}

\begin{defin}\label{betterwhopper}
We say that $\catR$ is an {\em Abelian iquantum categorification}
or a {\em Karoubian iquantum categorification}
if $\catR$ is as in \cref{hotsauce} for a locally finite Abelian or locally finite Karoubian 2-representation 
$(\catR_\lambda)_{\lambda \in X^\imath}$, respectively, 
such that in addition\footnote{This additional hypothesis was not present in \cite{HKM}. However, 
the correct formulation of
\cite[Th.~5.22]{HKM} requires it. We are grateful to 
Gao, Rui and Song whose paper \cite{GRS-2026} drew our attention to this oversight.}
$\big|\{i \in I\:|\:B_i V \neq 0\}\big|<\infty$ for all $\lambda \in X^\imath$ and $V \in \catR_\lambda$.
\end{defin}

Assume for the remainder of the section that
\begin{center}
\fbox{
\begin{minipage}{0.56\textwidth}
$\catR$ is an Abelian or Karoubian iquantum categorification. 
\end{minipage}
}
\end{center}
As the 1-morphism $B_i\one_\lambda$ is left and right dual to $\one_\lambda B_{-i}$,
the endofunctor $B_i$ is left and right adjoint to $B_{-i}$. In particular, in the Abelian case, 
all of the endofunctors $B_i$ are exact. 
In the Karoubian case,
we use the Yoneda embedding to 
identify each $\catR_\lambda$ 
with a full subcategory of the Schurian category 
$\Mod{\catR_\lambda}$,
hence, $\catR$ is identified with a full subcategory of
the Schurian category 
$\Mod{\catR}$.
The categorical action of $\UU^\imath$ extends in a canonical way to make $(\Mod{\catR_\lambda})_{\lambda \in X^\imath}$
into a 2-representation in its own right.
Each $B_i\colon\Mod{\catR_\lambda}\rightarrow \Mod{\catR_{\lambda-\alpha_i}}$ is exact, and it takes finitely generated modules to finitely generated. 

As in the previous section, to unify the exposition of the Abelian and Karoubian cases as much as possible from now on, 
we let $\catR^\fg=\bigoplus_{\lambda \in X^\imath} \catR^\fg_\lambda$ be defined so that $\catR^\fg=\catR$
and $\catR^\fg_\lambda=\catR_\lambda$ in the Abelian case,
while in the Karoubian case
we let $\catR^\fg$ and $\catR^\fg_\lambda$ be the full subcategories of $\Mod{\catR}$ and $\Mod{\catR_\lambda}$
consisting of all finitely generated modules.
The functor $B_i$ maps objects of $\catR^\fg_\lambda$
to objects of $\catR^\fg_{\lambda-\alpha_i}$.
Of course, we will also 
use string calculus to denote morphisms in $\catR^\fg$,
with green strings labelled by objects of $\catR^\fg$ on the right-hand boundary as in \cref{bcat}, and a label $\lambda$ in a 2-cell indicating the subcategory $\catR^\fg_\lambda$.

The next fundamental theorem is the iquantum counterpart of \cref{spec}. We prove it just for the case in hand, but there is a generalization to all 2-iquantum groups in the setup of \cite{BWW}. In diagonal type, it recovers the analogous result for 2-quantum groups, which was established already in \cite[Sec.~5.1]{HKM}.
The bubble generating function $\begin{tikzpicture}[baseline=-1mm,iQ,scale=.9]
\draw[-] (-0.25,0) arc(180:-180:0.25);
\node at (-0.36,0) {\strandlabel{\ibar}};
\region{0.95,0}{\lambda};
\node at (.56,0) {$(u)$};
\end{tikzpicture} \in \End_{\UU^\imath}(\one_\lambda)\lround u^{-1}\rround$
acts as a formal Laurent series
in $Z(\catR^\fg_\lambda)\lround u^{-1}\rround$.
In turn, the evaluation of this on any object
$V \in \catR^\fg_\lambda$ is a formal Laurent series denoted \begin{equation}
O_{V;i}(u) \in u^{\del_i-\lambda_i} \id_V + 
u^{\del_i - \lambda_i-1} Z_V\llbracket u^{-1}\rrbracket,
\end{equation}
where $Z_V$ denotes the center of $\End_{\catR^\fg_\lambda}(V)$ as before.
Also let $m_{V;i}(u) \in \kk[u]$ be the (monic) minimal polynomial of 
$\begin{tikzpicture}[iQ,baseline=-1mm]
\draw[-] (0,-0.2)\botlabel{i} -- (0,0.2);
\draw[darkg] (.3,-0.2)\botlabel{V} -- (.3,0.2);
\closeddot{0,0};
\end{tikzpicture}$.
We define the
{\em $\beta$-numbers} $(\beta_i(V))_{i \in I}$ by
$\beta_i(V) \coloneq \deg m_{V;i}(u)$.
The additional hypothesis in \cref{betterwhopper} ensures that
$\beta_i(V) = 0$ for all but finitely many $i \in I$, hence, it also makes sense to define
\begin{equation}\label{iqbetadegee}
\beta(V) \coloneq \sum_{i \in I} \beta_i(V).
\end{equation}
We call this the {\em $\beta$-degree} of $V$.
In the following lemma, we use an obvious variation of the pin notation
in which the pin label is a polynomial
in $Z_V[x]$, not merely $\kk[x]$.
We indicate this by drawing the pinhead as a coupon on the green string.

\begin{lem}\label{wednesday}
Suppose that $V \in \catR^\fg_\lambda$ 
for some $\lambda \in X^\imath$, and $0 \neq i \in I$.
Suppose that $f(u) \in Z_V[u]$ is a monic polynomial such that
$$
\begin{tikzpicture}[iQ,baseline = -2mm,scale=1.2]
\draw[-] (-0.2,-.4)\botlabel{i} to (-0.2,.4);
\draw[-,darkg] (0.5,.4) to (0.5,-.4)\botlabel{V};
\objpin{-.2,0}{.5,0}{f(x)};
\end{tikzpicture} = 0.
$$
Then $\deg f(u) \geq \lambda_i - \del_i$, and
$g(u) \coloneq f(u) O_{V;i}(u)$ is a monic polynomial in $Z_V[u]$
of degree $\deg f(u)+\del_i-\lambda_i$
such that
$$
O_{V;i}(u) = \frac{g(u)}{f(u)}\qquad\text{ and }\qquad
\begin{tikzpicture}[iQ,baseline = -1mm,scale=1.2]
\draw[-] (-0.3,-.4)\botlabel{\ibar} to (-.3,.4);
\draw[-,darkg] (0.5,.4) to (0.5,-.4)\botlabel{V};
\objpin{-.3,0}{.5,0}{g(-x)};
\end{tikzpicture} = 0.
$$
\end{lem}

\begin{proof}
Let 
$g(u) \coloneq f(u) O_{V;i}(u)$,
so that $O_{V;i}(u) = \frac{g(u)}{f(u)}$. 
By \cref{ibubblerelnew}, we have that 
$$
g(u) \in u^{\deg f(u)+\del_i-\lambda_i}\id_V + 
u^{\deg f(u)+\del_i-\lambda_i-1} Z_V\llbracket u^{-1}\rrbracket.
$$
Using a version of \cref{trick}, we have for $r \geq 0$ that
$$
\big[
g(u)\big]_{u:-r-1}
=
\big[
 O_{V;i}(u) f(u)\big]_{u:-r-1}
=
\left[\!\!
\begin{tikzpicture}[iQ,anchorbase]
\draw[-] (-0.25,0) arc(180:-180:0.25);
\node at (-0.36,0) {\strandlabel{\ibar}};
\node at (.52,0) {$(u)$};
\draw[-,darkg] (1.3,.4) to (1.3,-.4)\botlabel{V};
\catlabel{1.3,0}{f(u)};
\end{tikzpicture}\ 
\right]_{u:-r-1}\!\!\!=
\zeta_i(\lambda)^{-1}
\begin{tikzpicture}[iQ,anchorbase]
\draw[-] (-0.25,0) arc(180:-180:0.25);
\node at (-0.36,0) {\strandlabel{\ibar}};
\draw[-,darkg] (1.1,.4) to (1.1,-.4)\botlabel{V};
\objpin{.25,0}{1.1,0}{x^r f(x)};
\end{tikzpicture}\ = 0.
$$
Since the leading term of $g(u)$ is 
$u^{\deg f(u)+\del_i-\lambda_i}\id_V$, 
it follows that
$\deg f(u) \geq \lambda_i-\del_i$, and $g(u)$ is a monic polynomial in $u$ of this degree.
It remains to show that
$\begin{tikzpicture}[iQ,baseline = -1mm,scale=1.2]
\draw[-] (-0.2,-.4)\botlabel{\ibar} to (-0.2,.4);
\draw[-,darkg] (0.5,.4) to (0.5,-.4)\botlabel{V};
\objpin{-.2,0}{.55,0}{g(-x)};
\end{tikzpicture}
=0$.
This follows because
\begin{align*}
0
&=
\zeta_i(\lambda)^{-1}\begin{tikzpicture}[iQ,anchorbase,scale=1.3]
\draw[-] (0,-0.5)\botlabel{\ibar} to[out=up,in=180] (0.3,0.2) to[out=0,in=up] (0.45,0) to[out=down,in=0] (0.3,-0.2) to[out=180,in=down] (0,0.5);
\draw[-,darkg] (1.1,.4) to (1.1,-.4) \botlabel{V};
\objpin{.45,0}{1.1,0}{f(x)};
\end{tikzpicture}
\stackrel{\cref{fancyicurl}}{=}
\left[\ 
\begin{tikzpicture}[anchorbase,scale=1.3,iQ]
\draw[-] (-0.9,-0.5)\botlabel{\ibar} -- (-0.9,0.5);
\circledbar{-0.9,0}{u};
\draw (-0.4,0) arc(180:-180:0.2);
\node at (0.25,0) {$(u)$};
\node at (-0.5,0) {\strandlabel{\ibar}};
\draw[-,darkg] (.9,.5) to (.9,-.5) \botlabel{V};
\catlabel{.9,0}{f(u)};
\end{tikzpicture}\ 
\right]_{u:-1}
=
\left[\ 
\begin{tikzpicture}[iQ,anchorbase,scale=1.2]
\draw[-] (-0.2,-.4)\botlabel{\ibar} to (-0.2,.4);
\circledbar{-0.2,0}{u};
\draw[-,darkg] (.6,.4) to (0.6,-.4) \botlabel{V};
\catlabel{.6,0}{g(u)};
\end{tikzpicture}\ 
\right]_{u:-1}
\stackrel{\cref{trick}}{=}
\begin{tikzpicture}[iQ,anchorbase,scale=1.2]
\draw[-] (-0.3,-.4)\botlabel{\ibar} to (-0.3,.4);
\draw[-,darkg] (0.5,.4) to (0.5,-.4)\botlabel{V};
\objpin{-0.3,0}{.5,0}{g(-x)};
\end{tikzpicture}\ .
\end{align*}
\end{proof}

Recall that we write $\widetilde{f}(u)$ for 
$(-1)^{\deg f(u)} f(-u)$.

\begin{theo}\label{ben}
For $\lambda \in X^\imath$, and
an irreducible $L \in \catR^\fg_\lambda$, the following hold:
\begin{enumerate}
\item $O_{L;0}(u)=u^{-1}\dfrac{\widetilde{m}_{L;0}(u)}{m_{L;0}(u)}$ and 
$t(\lambda) \equiv \beta_0(L) \pmod{2}$.
\item
If $i \in I -\{0,\pm h\}$ then
$O_{L;i}(u)=\dfrac{\widetilde{m}_{L;-i}(u)}{m_{L;i}(u)}$
and $\lambda_i=\beta_i(L)-\beta_{-i}(L)$.
\item 
{\em Either} 
\begin{enumerate}
\item[(a)] $O_{L;h}(u) = \dfrac{\widetilde{m}_{L;-h}(u)}{m_{L;h}(u)}$,
$O_{L;-h}(u) = u\dfrac{\widetilde{m}_{L;h}(u)}{m_{L;-h}(u)}$,
and $\lambda_h = \beta_h(L)-\beta_{-h}(L)+1$; {\em or} 
\item[(b)]
$O_{L;h}(u) = u\dfrac{\widetilde{m}_{L;-h}(u)}{m_{L;h}(u)}$,
$O_{L;-h}(u) = \dfrac{\widetilde{m}_{L;h}(u)}{m_{L;-h}(u)}$
and 
$\lambda_h=\beta_h(L)-\beta_{-h}(L)$.
\end{enumerate}
\end{enumerate}
\end{theo}

\begin{proof}
Part (i) is a restatement of \cite[Cor.~3.11]{BWWa1}.
To prove (ii) and (iii),
take $i \in I - \{0\}$.
Applying \cref{wednesday} once with $f(u) = m_{L;i}(u)$ and again with
$f(u) = m_{L;-i}(u)$,
we get monic polynomials $g(u)$ and $q(u)$ such that
\begin{itemize}
\item
$\deg g(u) = \beta_i(L) + \del_i-\lambda_i$,
$O_{L;i}(u) = \frac{g(u)}{m_{L;i}(u)}$ and $m_{L;-i}(-u) \mid g(u)$;
\item
$\deg q(u) = \beta_{-i}(L) + \del_{-i} + \lambda_i$,
$O_{L;-i}(u) = \frac{q(u)}{m_{L;-i}(u)}$ and
$m_{L;i}(-u) \mid q(u)$.
\end{itemize}

If $i \neq \pm h$ then $\del_i = \del_{-i}=0$, and
$O_{L;i}(-u) O_{L;-i}(u)=
(-1)^{\lambda_i}$ by \cref{iinfgrass}.
It follows that
$\frac{g(-u)q(u)}{m_{L;i}(-u) m_{L;-i}(u)} = (-1)^{\lambda_i}$,
so $\deg g(u) + \deg q(u) = \deg m_{L;-i}(u) + \deg m_{L;i}(u)$.
Since $\widetilde{m}_{L;-i}(u) \mid g(u)$
and $\widetilde{m}_{L;i}(u) \mid q(u)$ and all are monic, it must be that $g(u)=\widetilde{m}_{L;-i}(u)$ and $q(u)=\widetilde{m}_{L;i}(u)$.
So we have that 
$O_{L;i}(u) = \frac{\widetilde{m}_{L;-i}(u)}{m_{L;i}(u)}$, and comparing degrees gives that
$\beta_i(L)-\beta_{-i}(L) = \lambda_i$.

Instead, if $i = h$ then $\del_h = 1, \del_{-h}=0$, and
\cref{iinfgrass} implies that
$O_{L;h}(-u) O_{L;-h}(u) =(-1)^{\lambda_h+1}u$.
So $(-1)^{\lambda_{h}+1} u = \frac{g(-u)q(u)}{m_{L;h}(-u) m_{L;-h}(u)}$.
Now $\deg g(u) + \deg q(u)= 
1+\deg m_{L;-h}(u) + \deg m_{L;h}(u)$.
So {\em either} 
(a) $g(u)=\widetilde{m}_{L;-h}(u)$, {\em or} 
(b) $q(u)=\widetilde{m}_{L;h}(u)$.
In the former case, $q(u) = u \widetilde{m}_{L;h}(u)$,
$O_{L;h}(u)=\frac{\widetilde{m}_{L;-h}(u)}{m_{L;h}(u)}$, 
$O_{L;-h}(u)=\frac{q(u)}{m_{L;-h}(u)} = u \frac{\widetilde{m}_{L;h}(u)}{m_{L;-h}(u)}$,
and comparing degrees gives that $\beta_h(L)-\beta_{-h}(L)+1 = \lambda_{h}$. 
In the latter case,
$g(u)=u\widetilde{m}_{L;-h}(u)$,
$O_{L;h}(u)=u\frac{\widetilde{m}_{L;-h}(u)}{m_{L;h}(u)}$, 
$O_{L;-h}(u) = \tfrac{\widetilde{m}_{L;h}(u)}{m_{L;-h}(u)}$,
and comparing degrees gives that $\beta_{h}(L)-\beta_{-h}(L)=\lambda_{h}$.
Thus, (iii) is proved.
\end{proof}

\begin{cor}\label{ovens}
For an irreducible $L \in \catR^\fg_\lambda$,
the iweight $\lambda$ is equal to 
the canonical image in $X^\imath$ of
\begin{align*}
\hat\lambda \coloneq
&\begin{dcases}
\sum_{i \in I} \beta_i(L) \varpi_i&\text{if $\lambda_h \equiv \beta_{h}(L)-\beta_{-h}(L)\pmod{2}$}\\
\varpi_h+\sum_{i \in I} \beta_i(L) \varpi_i&\text{if $\lambda_h \not\equiv \beta_{h}(L)-\beta_{-h}(L)\pmod{2}$}.
\end{dcases}
\end{align*}
\end{cor}

\begin{proof}
This follows from the properties relating $\lambda$ to the $\beta$-numbers
established in \cref{ben}.
\end{proof}

\begin{defin}\label{nilpdef}
We say that the Abelian or Karoubian iquantum 
categorification $\catR$ is {\em nilpotent} if 
the endomorphism
$\begin{tikzpicture}[iQ,baseline = -2mm,scale=1.2]
\draw[-] (-0.2,-.25)\botlabel{i} to (-0.2,.25);
\closeddot{-.2,0};
\draw[-,darkg] (0.1,.25) to (0.1,-.25)\botlabel{V};
\end{tikzpicture}:B_i V \rightarrow B_i V$ is nilpotent
for all $i \in I$,
$\lambda \in X^\imath$ and all $V \in \catR_\lambda$.
In the Karoubian case, it is then automatic that
$\begin{tikzpicture}[iQ,baseline = -2mm,scale=1.2]
\draw[-] (-0.2,-.25)\botlabel{i} to (-0.2,.25);
\closeddot{-.2,0};
\draw[-,darkg] (0.1,.25) to (0.1,-.25)\botlabel{V};
\end{tikzpicture}$ is nilpotent 
also for $V \in \catR^\fg_\lambda$.
\end{defin}

Nilpotency implies the following for
$i \in I$, $\lambda \in X^\imath$
and an irreducible object
$L \in \catR^\fg_\lambda$:
\begin{align}
m_{L;i}(u) &= u^{\beta_i(L)},&
O_{L;i}(u) &= u^{\del_i-\lambda_i}\id_L.
\end{align}

\begin{lem}\label{noodslem}
If $\catR$ is nilpotent then, for every $\lambda \in X^\imath$,
the evaluation of any dotted bubble of positive degree on any object $V \in \catR^\fg_\lambda$ is a nilpotent element of $Z_V$. In the Abelian case, its nilpotency degree is bounded above by the composition length of $V$.
\end{lem}

\begin{proof}
Let $z$ be a dotted bubble of positive degree. 
Since $O_{L;i}(u) = u^{\del_i-\lambda_i}\id_L$,
we have that
$z_L = 0$ for every irreducible $L \in \catR^\fg_\lambda$.
In the Abelian case, naturality now implies that 
$z_V$ is nilpotent with nilpotency degree bounded by the length of $V$.
In the Karoubian case, 
we may assume that $V$ is indecomposable. 
Then 
$A \coloneq \End_{\catR^\fg_\lambda}(V)$ is a finite-dimensional local $\kk$-algebra.
Choose an irreducible 
quotient $\pi:V\twoheadrightarrow L$. 
Naturality gives $\pi \circ z_V=z_L\circ \pi=0$,
so $z_V$ is not invertible. 
Hence, $z_V$ 
is in the Jacobson radical of $A$, which is nilpotent.
\end{proof}

Nilpotency of dots implies that 
it makes sense to pin a power series
$f(x) \in \kk\llbracket x \rrbracket$,
not merely a polynomial, to a dot on a string. 
More generally, one can pin any
$f(x) \in Z(\catR^\fg_\lambda)\llbracket x \rrbracket$
to any dot provided that 
the pinhead is placed in a 2-cell labelled by $\lambda$.
In the Karoubian case, all such diagrams make sense on evaluation on any finitely generated $\catR$-module. 
The diagram could also be interpreted as a natural transformation when evaluated on an arbitrary
locally finite-dimensional $\catR$-module since any such module is a direct limit of finitely generated objects.
However, we will never need to use this observation.

We end the section with a lemma which gives a replacement for \cref{wednesday} in the case $i=0$.

\begin{lem}\label{zerotrick}
For  $\lambda \in X^\imath$, and 
$V \in \catR^\fg_\lambda$,
there exists a monic polynomial $f(u) \in Z_V[u]$
with $\deg f(u) \equiv t(\lambda)\pmod{2}$
such that
$$
O_{V;0}(u) = u^{-1}
\frac{\widetilde{f}(u)}{f(u)}
\qquad\text{ and }\qquad
\begin{tikzpicture}[iQ,baseline = -2mm,scale=1.2]
\draw[-] (-0.2,-.4)\botlabel{0} to (-0.2,.4);
\draw[-,darkg] (0.5,.4) to (0.5,-.4)\botlabel{V};
\objpin{-.2,0}{.5,0}{f(x)};
\end{tikzpicture}=0.
$$
If $\catR$ is nilpotent then
$f(u)$ can be chosen so that
all of its 
coefficients apart from the coefficient of the top degree term are nilpotent elements of $Z_V$.
\end{lem}

\begin{proof}
We apply \cite[Th.~3.10]{BWWa1} taking the polynomial $f(u)$
there to be the minimal polynomial $m_{V;0}(u) \in \kk[u]$.
The theorem produces another polynomial $g(u)$ which satisfies exactly the requirements for $f(u)$ in 
the present lemma.
For the final assertion, in the nilpotent case, the minimal polynomial $m_{V;0}(u)$ is a monomial, 
and the $g(u)$ constructed from $f(u) \coloneq m_{V;0}(u)$
in \cite[Th.~3.10]{BWWa1} is monic with lower degree coefficients that are linear combinations of dotted bubbles of positive degree. 
So, by \cref{noodslem}, all of the lower degree coefficients of $g(u)$, which is the $f(u)$ in the present lemma, are nilpotent.
\end{proof}

\setcounter{section}{5}

\section{\texorpdfstring{Completions}{Completions}}\label{completediq}

Next, we introduce a completion $\hatUU^\imath$ of $\UU^\imath$ whose topological
2-representations capture the finiteness properties of nilpotent iquantum categorifications. 
Recall the 
graded ring $R = R_0 \oplus R_1 \oplus \cdots$ 
from \cref{basering}.
Let $R_{\geq d} \coloneq \bigoplus_{d' \geq d} R_{d'}$.
Also for $J \subseteq I$, let $U_J$ be the ideal of $R$
generated by $\{z_j^{(n)}\:|\:j \in J, n \geq 1\}$.
We put a linear topology on $R$, which we call the {\em GC topology}
(``graded and cofinite''),
by declaring that the ideals
\begin{equation}\label{justforclaude}
U_{d,J} \coloneq 
R_{\geq d} + U_J
\end{equation}
for all $d \in \Z$ 
and all cofinite subsets $J \subseteq I$ are a system of neighborhoods of 0.
Let $\hatR = \varprojlim_{d,J} R / U_{d,J}$ be the completion of $R$ with respect to this topology. 
The GC topology on $R$ is obviously separated, so the canonical algebra
homomorphism $R \rightarrow \hatR$ is injective.
Also 
\begin{equation}\label{dontknow}
\hatR / \overline{U}_{d,J} \cong R / U_{d,J},
\end{equation}
where $\overline{U}_{d,J}$ is the closure of $U_{d,J}$ in $\hatR$.
Similarly, $\overline{U}_J$ denotes the closure of $U_J$.

\begin{lem}\label{verysmall}
The neighborhoods $U_{d,J}$ for $d \geq 0$ and cofinite
$J \subseteq I$ with $J = -J$ already form a system of neighborhoods of 0 in $R$.
\end{lem}

\begin{proof}
Since $z_{-j}(u) = 
z_j(-u)^{-1}$, the ideal $U_{J}\unlhd R$ contains $z_{-j}^{(n)}$ for all $j\in J$ and $n \geq 1$. This shows that $U_{J} = U_{J\cup(-J)}$. 
The lemma follows using also that $R$ is positively graded.
\end{proof}

Let $M\subset R$ be the set of monomials in the generators 
$\{z_0^{(2n-1)}, z_i^{(n)}\:|\:i \in I_\hi, n \geq 1\}$ 
of $R$.
As noted earlier, this set is a linear basis for $R$.
It is also a topological linear basis for $\hatR$ in the strongest possible sense: any
element $z \in \hatR$ can be written 
as $z=\sum_{m \in M} c_m m$
for unique $c_m \in \kk$, and any such infinite linear combination
converges.
This follows because all but finitely many monomials in $M$
are in $U_{d,J}$
for any given 
$d \in \Z$ and cofinite $J \subseteq I$,
and for every finite family of monomials there exists
$d \in \Z$ and cofinite $J \subseteq I$ such that the images of the monomials in the 
family are linearly independent in $R / U_{d,J}$.
Thus, the function
\begin{align}\label{dumbanddumber}
\prod_{m \in M} \kk &\rightarrow \hatR,
&
(c_m)_{m \in M} &\mapsto \sum_{m \in M} c_m m 
\end{align}
is a homeomorphism if
we give $\kk$ the discrete topology and 
$\prod_{m \in M} \kk$ the product topology.

Now take 
$\lambda \in X^\imath$ and $\bi,\bj \in \langle I \rangle$
such that $\alpha_\bi$ and $\alpha_\bj$ have the same images in $X^\imath$. We will use the shorthand
\begin{equation}
H_\lambda(\bj,\bi) \coloneq 
\Hom_{\UU^\imath}(B_\bj \one_\lambda,B_\bi \one_\lambda).
\end{equation}
This is a free graded right $R$-module
with degree $d$ component $H_\lambda(\bj,\bi)_d$.
It is 0 unless $\bi \approx \bj$, i.e., the lengths of the words $\bi$ and $\bj$ have the same parity.
In the homogeneous basis 
$\basis_\lambda(\bi\times\bj)$
for $H_\lambda(\bj,\bi)$ 
as a free right $R$-module from \cref{ibasis}, 
the degrees of the basis vectors are bounded below, and
there are only finitely many basis vectors in each degree.
Let $H_\lambda(\bj,\bi)_{\geq d} \coloneq \sum_{d' \geq d} H_\lambda(\bj,\bi)_{d'}$.

The {\em GC topology} on $H_\lambda(\bj,\bi)$ is the linear topology with a system of neighborhoods of 0 given by the subspaces
\begin{equation}\label{justforjon}
H_\lambda(\bj,\bi)_{d,J} \coloneq H_\lambda(\bj,\bi)_{\geq d} + H_\lambda(\bj,\bi) U_J
\end{equation}
for all $d \in \Z$ and all cofinite subsets $J \subseteq I$.
We denote the completion 
$\varprojlim_{d,J} H_\lambda(\bj,\bi) / H_\lambda(\bj,\bi)_{d,J}$
with respect to this topology by
$\hatH_\lambda(\bj,\bi)$.
The action of $R$ on $H_\lambda(\bj,\bi)$ extends by continuity to make $\hatH_\lambda(\bj,\bi)$
into a topological right $\hatR$-module. 
The following lemma shows that $\basis_\lambda(\bi\times\bj)$,
the linear basis for $H_\lambda(\bj,\bi)$ as a right $R$-module,
is also a topological basis
for $\hatH_\lambda(\bj,\bi)$ as a topological right $\hatR$-module.

\begin{lem}\label{arewestupid}
In the above notation,
the linear map
\begin{align}\label{productmap}
\prod_{\theta \in \basis_\lambda(\bi\times\bj)} \hatR
&\rightarrow \hatH_\lambda(\bj,\bi),&
(z_\theta)_{\theta \in \basis_\lambda(\bi\times\bj)}
&\mapsto \sum_{\theta \in \basis_\lambda(\bi\times\bj)} \theta z_\theta
\end{align}
is a homeomorphism for the linear topology on the domain 
that is the product of the GC topologies, and the GC topology on the codomain.
\end{lem}

\begin{proof}
This follows from the following equivalent description of the GC topology on $H_\lambda(\bj,\bi)$: it is the linear topology with a system of neighborhoods of 0 given 
by the subspaces $\bigoplus_{\theta \in \basis_\lambda(\bi\times\bj)}
\theta U_{d(\theta),J(\theta)}$
for $d(\theta) \in \Z$ and cofinite subsets $J(\theta) \subseteq I$
such that $d(\theta) \leq 0$
(hence, $U_{d(\theta),J(\theta)}=R$)
for all but finitely many $\theta$.
For this new basis-dependent way of defining the GC topology on 
$H_\lambda(\bj,\bi)$,
it is clear that \cref{productmap} is a homeomorphism.

To prove that the new description of the GC topology is equivalent to the old one, we observe that
$H_\lambda(\bj,\bi)_{d, J} = \bigoplus_{\theta \in \basis_\lambda(\bi\times\bj)}
\theta U_{d-\deg(\theta),J}$. Only finitely many $\theta$
have $\deg(\theta) < d$, so this is one of the new neighborhoods.
Conversely, given one of the new neighborhoods
$N \coloneq \bigoplus_{\theta \in \basis_\lambda(\bi\times\bj)}
\theta U_{d(\theta),J(\theta)}$,
let
$F \coloneq \{\theta \in 
\basis_\lambda(\bi\times\bj)\:|\:d(\theta) > 0\}$, $d \coloneq \max(0, \deg(\theta)+d(\theta)\:|\:\theta \in F)$ and
$J \coloneq I \cap \bigcap_{\theta\in F} J(\theta)$.
Then $H_\lambda(\bj,\bi)_{d,J} \subseteq N$.
\end{proof}

The GC topology on $H_\lambda(\bj,\bi)$ is separated, so the canonical map $H_\lambda(\bj,\bi) \rightarrow \hatH_\lambda(\bj,\bi)$ 
is injective. This is also clear from \cref{arewestupid,nondegeneracy}.

\begin{defin}\label{completed2iq}
The {\em completed 2-iquantum group} $\hatUU^\imath$ is 
the $\kk$-linear 2-category with the following objects, 1-morphisms and 2-morphisms:
\begin{itemize}
\item Its object set is $X^\imath$, the same as $\UU^\imath$.
\item
For $\lambda,\mu \in X^\imath$, the objects of the morphism category
$\HOM_{\hatUU^\imath}(\lambda,\mu)$ 
are the formal direct sums
\begin{equation}\label{me}
B_\bI \one_\lambda = \one_\mu B_\bI
\coloneq 
\bigoplus_{\bi \in \bI} B_\bi \one_\lambda
\end{equation}
for (possibly infinite) sets $\bI$ such that $\bI \subseteq I^n$ for some $n \in \N$ and $\lambda-\alpha_\bi = \mu$ for all $\bi \in \bI$.
If $\bI = \{\varnothing\}$ (hence, $\lambda=\mu$) then
we denote $B_\bI \one_\lambda$ simply by $\one_\lambda$,
and if $\bI = \{i\}$ (hence, $\lambda-\alpha_i=\mu$) then we denote $B_\bI \one_\lambda$ simply by $B_i \one_\lambda$.
\item For $\bI, \bJ$ of this form,
the vector space
$\hatH_\lambda(\bJ, \bI) \coloneq \Hom_{\hatUU^\imath}(B_\bJ \one_\lambda, B_\bI \one_\lambda)$
consists of matrices 
$\theta = (\theta_{\bi,\bj})_{\bi\in \bI,\bj\in \bJ}$
such that
$\theta_{\bi,\bj} \in \hatH_\lambda(\bj,\bi)$ 
for each $\bi\in \bI$, $\bj\in \bJ$.
This space is 0 unless the lengths of 
the words in $\bI$ and the words in $\bJ$ are of the same parity.
\end{itemize}
Vertical composition is by matrix multiplication:
$$
\big(\theta_{\bi,\bj}\big)_{\bi \in \bI, \bj \in \bJ}
\circ \big(\theta_{\bj,\bk}\big)_{\bj \in \bJ,\bk \in \bK}
=\Big(\sum_{\bj \in \bJ} \theta_{\bi,\bj} \circ \theta_{\bj,\bk}\Big)_{\bi \in \bI, \bk \in \bK}.
$$
The infinite sums that are the matrix entries here make sense thanks to 
\cref{conv1} below. The horizontal composition of 1-morphisms
$B_{\bI'} \one_\mu\colon \mu\rightarrow \nu$ and 
$B_{\bI''} \one_\lambda\colon \lambda\rightarrow \mu$
is $B_{\bI} \one_\lambda\colon \lambda\rightarrow \nu$
where $\bI \coloneq \{\bi'\bi''\:|\:\bi' \in \bI', 
\bi'' \in \bI''\}$.
The horizontal composition of 2-morphisms represented by the matrices
$(\theta_{\bi', \bj'})_{\bi' \in \bI',\bj' \in \bJ'}$ and $(\theta_{\bi'',\bj''})_{\bi'' \in \bI'',\bj'' \in \bJ''}$
is the matrix with $(\bi'\bi'',\bj'\bj'')$-entry
$\theta_{\bi',\bj'}\theta_{\bi'',\bj''}$
as defined in \cref{conv2} below.
Finally, we equip each 2-morphism space 
\begin{equation}
\hatH_\lambda(\bJ,\bI)
= \prod_{\bi \in \bI, \bj \in \bJ}
\hatH_\lambda(\bj,\bi)
\end{equation}
with the topology that is the product of the GC topologies on each hom space, thereby making
$\hatUU^\imath$ into a 2-category enriched in topological vector spaces.
\end{defin}

We will not go through the checks of the 2-category axioms in detail. The only one that is slightly delicate is the associativity of vertical composition. This follows because the infinite sums defining vertical composition reduce to finite sums on projection to the quotient modulo a GC neighborhood of 0 like in the proof of \cref{conv1}(ii), and for those associativity is clear.

\begin{lem}\label{conv1}
Suppose given $\lambda \in X^\imath$,
$\bi,\bj \in \langle I\rangle$ 
and
$\bK \subseteq I^n$
such that $\mu\coloneq\lambda-\alpha_\bi =
\lambda-\alpha_\bj=\lambda-\alpha_\bk$ for all $\bk \in \bK$. 
\begin{enumerate}
\item For any $\bk \in \bK$, 
the vertical composition in $\UU^\imath$ extends by continuity 
to define a bilinear map
$$
\hatH_\lambda(\bk,\bi)
\times 
\hatH_\lambda(\bj,\bk)
\rightarrow
\hatH_\lambda(\bj,\bi),
\qquad
(f,g) \mapsto f \circ g.
$$
\item
Given
2-morphisms
$\theta_{\bi,\bk} \in \hatH_\lambda(\bk,\bi)$
and 
$\theta_{\bk,\bj} \in \hatH_\lambda(\bj,\bk)$
for $\bk \in \bK$, the infinite sum
$\sum_{\bk \in \bK} \theta_{\bi,\bk} \circ \theta_{\bk,\bj}$
converges in
$\hatH_\lambda(\bj,\bi)$.
\end{enumerate}
\end{lem}

\begin{proof}
(i) 
The result is trivial if
$H_\lambda(\bk,\bi)=0$ or $H_\lambda(\bj,\bk)=0$.
Assuming both are non-zero,
let $a$ be minimal such that $H_\lambda(\bj,\bk)_a \neq 0$
and $b$ be minimal such that $H_\lambda(\bk,\bi)_b \neq 0$.
For any $d \in \Z$ and cofinite $J \subseteq I$, 
we have that
\begin{align*}
H_\lambda(\bk,\bi)_{d-a,J} \circ H_\lambda(\bj,\bk)
&\subseteq H_\lambda(\bj,\bi)_{d,J},&
H_\lambda(\bk,\bi) \circ H_\lambda(\bj,\bk)_{d-b,J}
&\subseteq H_\lambda(\bj,\bi)_{d,J}.
\end{align*}
So the linear map $H_\lambda(\bk,\bi) \otimes H_\lambda(\bj,\bk) \rightarrow H_\lambda(\bj,\bi) / H_\lambda(\bj,\bi)_{d,J}$
induced by composition factors through
the quotient
$(H_\lambda(\bk,\bi) / H_\lambda(\bk,\bi)_{d-a,J}) \otimes 
(H_\lambda(\bj,\bk) / H_\lambda(\bj,\bk)_{d-b,J})$.
Taking the inverse limit over target neighborhoods 
then gives the required 
continuous extension.

\vspace{2mm}
\noindent
(ii) 
By (i), the composition
$\theta_{\bi,\bk} \circ \theta_{\bk,\bj}$ is a well-defined
element of $\hatH_\lambda(\bj,\bi)$ for each $\bk \in \bK$.
To prove that the sum $\sum_{\bk \in \bK} 
\theta_{\bi,\bk} \circ \theta_{\bk,\bj}$ converges, it
suffices to show for each GC neighborhood
$H_\lambda(\bj,\bi)_{d,J}$ 
that the canonical projection of 
$\theta_{\bi,\bk} \circ \theta_{\bk,\bj}$ to
$H_\lambda(\bj,\bi) / H_\lambda(\bj,\bi)_{d,J}$ 
is zero for all but finitely many 
$\bk \in \bK$.
Since $\bK \subseteq I^n$, this is a consequence of
the following:

\vspace{2mm}
\noindent
{\underline{Claim}:}
{\em For every cofinite $J \subseteq I$, 
there is a finite
subset $\overline{J} \subset I$ such that
$H_\lambda(\bk,\bi) \circ H_\lambda(\bj,\bk) \subseteq H_\lambda(\bj,\bi) U_J$
for every $\bk \in I^n - \overline{J}^n$ with $\lambda-\alpha_\bk=\mu$.}

\vspace{1mm}
\noindent
{Proof:}
Consider the Cartan graph $G$ 
with vertex set $I$ and edges
$i\ \text{---}\ i\!+\!1$ for each $i \in I$.
Let $S$ be a finite subset of $I$ closed under $\pm$
which contains
$I - J$, all letters of the words $\bi$ and $\bj$,
and all $i \in I$ such that $\del_i-\lambda_i \neq 0$.
Then define 
$\overline{J} \coloneq \{i \in I\:|\:d_G(i,S) \leq n\}$, where $d_G(i,S)$ is the distance in the graph $G$. Since $S$ is closed under $\pm$, $d_G(i,S) = d_G(-i,S)$, so $\overline{J}$ is also closed under $\pm$.

Now we take $\bk \in I^n - \overline{J}^n$ with $\lambda-\alpha_\bk=\mu$,
$\theta \in H_\lambda(\bk,\bi)$ and $\phi \in H_\lambda(\bj,\bk)$
and show that $\theta\circ \phi \in 
H_\lambda(\bj,\bi) U_J$.
We may assume that
$\theta$ and $\phi$ 
are string diagrams of reduced shapes 
multiplied on the right by some monomials $m_\theta$ and $m_\phi$ in dotted bubbles, so
$$
\theta \circ \phi = 
\begin{tikzpicture}[iQ,centerzero]
\draw (-.7,-.74) to (-.7,-.9)\botlabel{j_1};
\node at (0,1) {$\cdots$};
\draw (.7,-.74) to (.7,-.9)\botlabel{j_k};
\draw (-.7,.74) to (-.7,.9)\toplabel{i_1};
\node at (0,-1.05) {$\cdots$};
\draw (.7,.74) to (.7,.9)\toplabel{i_l};
\draw (-.5,.27) to (-.5,-.27);
\node at (0,0) {$\cdots$};
\draw (.5,.27) to (.5,-.27);
\strand{-.66,0}{k_1};
\strand{.68,0}{k_n};
\node[draw,fill=white,inner sep=1.3pt] at (0,0.5) {$\hspace{7mm}\theta\hspace{7mm}$};
\node[draw,fill=white,inner sep=1.3pt] at (0,-.5) {$\hspace{7mm}\phi\hspace{7mm}$};
\node[rectangle,rounded corners,draw,pinhead,inner sep=2.5pt]
at (1.5,.5) {$\color{black}m_\theta$};
\node[rectangle,rounded corners,draw,pinhead,inner sep=2.5pt]
at (1.55,-.5) {$\color{black}m_\phi$};
\region{1.5,0}{\lambda};
\end{tikzpicture}\ .
$$
At least one of the $n$ letters 
of the word $\bk$ is in $I - \overline{J}$. 
So at most $n-1$ 
of the letters of $\bk$ are of distance $\leq n$ from $S$.
Consequently,
there exists $1 \leq t\leq n$ such that no letter of $\bk$ is of distance exactly $t$ from $S$. 
Then we let $A$ (resp., $B$) 
be the set of $k \in I$ such that $\pm k$
is a letter of $\bk$ and $d_G(k,S) > t$ (resp., $d_G(k,S) < t$).
The union $D$ 
of connected components of $\theta \circ \phi$ labelled by elements of $A$ is non-empty because some letter of $\bk$ is in $A$, and closed because the letters of $\bi$ and $\bj$ are not in $A$.
All elements of $A$ are separated
from $S \cup B$ 
in the sense that they are not connected by an edge in the graph $G$. Hence, by the defining relations of $\UU^\imath$,
$D$ can be commuted past all other strings to move it to 
the right-hand boundary.

It remains to observe that any element of $H_\lambda(\bj,\bi)$ multiplied on the right by a non-empty
closed string diagram $D$ 
all of whose strings are labelled by elements of the set $A$ lies in $H_\lambda(\bj,\bi) U_J$. This follows
because $D$, which is 
located in a 2-cell of iweight $\lambda$,
can be straightened using the relations to obtain a linear combination of non-trivial monomials in non-zero dotted bubbles labelled by elements of $A$. Since $\del_k-\lambda_k = 0$ for $k \in A$, any dotted bubble of this form is of strictly positive degree, so it is a non-zero multiple of $z_{\lambda;k}^{(r)}$
for some $k \in A$ and $r \geq 1$.
Since $A \subseteq J$,
any linear combination of non-trivial 
monomials in such elements lies in $\varphi_\lambda(U_J)$
where $\varphi_\lambda$ is the isomorphism from \cref{flambda}.
\end{proof}

\begin{lem}\label{conv2}
Suppose given
$\bi = \bi'\bi''$ and $\bj = \bj'\bj''$
for $\bi' \in I^{m'}, \bi'' \in I^{m''}$ and
$\bj' \in I^{n'}, \bj'' \in I^{n''}$
such that $\mu = \lambda-\alpha_{\bi''} = \lambda-\alpha_{\bj''}$ and $\nu = \mu - \alpha_{\bi'}=\mu-\alpha_{\bj'}$ for some $\lambda,\mu,\nu \in X^\imath$.
The horizontal composition in $\UU^\imath$ extends by continuity to define a bilinear map
$$
\hatH_\mu(\bj',\bi')
\times \hatH_\lambda(\bj'',\bi'')
\rightarrow
\hatH_\lambda(\bj,\bi),\qquad
(f',f'') \mapsto f' f''.
$$
\end{lem}

\begin{proof}
If $H_\mu(\bj',\bi')$ or $H_\lambda(\bj'',\bi'')$ is 0 then the result is trivial, so assume that they are both non-zero.
Let $a$ be minimal such that $H_\lambda(\bj'',\bi'')_a \neq 0$
and $b$ be minimal such that $H_\mu(\bj',\bi')_b \neq 0$.
Also let $K \subseteq I$ be the set of 
all $i \in I$ 
such that $a_{i,j} =a_{i,-j} = 0$ 
for all letters $j$ appearing in $\bi''$ or $\bj''$.
The bubble slide relation \cref{ibubslide} 
implies that any dotted bubble 
labelled by $i \in K$ 
commutes freely past any string colored by one of the letters appearing in the words $\bi''$ and $\bj''$.

Now suppose that $J \subseteq I$ is a cofinite subset.
Let $J' \coloneq J \cap K$. It is a cofinite subset of $I$
contained in $J$.
Then we have that
$\varphi_\mu(U_{J'}) H_\lambda(\bj'',\bi'') \subseteq H_\lambda(\bj'',\bi'') U_{J'} \subseteq H_\lambda(\bj'',\bi'') U_{J}$.
From this, it follows that
\begin{align*}
H_\mu(\bj',\bi')_{d-a,J'}\  H_\lambda(\bj'',\bi'')
&\subseteq H_\lambda(\bj,\bi)_{d,J},&
H_\mu(\bj',\bi')\  H_\lambda(\bj'',\bi'')_{d-b,J}
&\subseteq H_\lambda(\bj,\bi)_{d,J}.
\end{align*}
The lemma follows from this like at the end of
the proof of \cref{conv1}(i).
\end{proof}

The following 1-morphisms in $\hatUU^\imath$ will be particularly important: 
for $\lambda,\mu \in X^\imath$ and $n \geq 0$, let
\begin{equation}\label{usefulB}
\one_\mu B^n \one_\lambda \coloneq
\bigoplus_{\substack{\bi \in I^n\\\lambda-\alpha_\bi=\mu}}
B_\bi \one_\lambda.
\end{equation}
This formal direct sum 
is the object $B_\bI \one_\lambda = \one_\mu B_\bI$ 
of the morphism category $\HOM_{\hatUU^\imath}(\lambda,\mu)$
from \cref{me} with $\bI\coloneq \{\bi \in I^n\:|\:\lambda-\alpha_\bi=\mu\}$.

\begin{defin}\label{iqmonoidal}
Let $\hatAB$
be the strict $\kk$-linear monoidal category
with objects $B^{\otimes n}\:(n \in \N)$ and morphisms
\begin{equation}
\Hom_{\hatAB}(B^{\otimes n}, B^{\otimes m})
\coloneq
\prod_{\lambda,\mu \in X_\0^\imath} \Hom_{\hatUU^\imath}(\one_\mu B^n\one_\lambda, \one_\mu B^m \one_\lambda).
\end{equation}
We denote a morphism in this space
by 
$\theta = (\theta^{\mu,\lambda})_{\lambda,\mu \in X_\0^\imath}$
for $\theta^{\mu,\lambda} \in \Hom_{\hatUU^\imath}(\one_\mu B^n \one_\lambda, \one_\mu B^m \one_\lambda)$.
The vertical composition of $\theta\colon B^{\otimes m}\rightarrow B^{\otimes l}$ and $\phi\colon B^{\otimes n} \rightarrow B^{\otimes m}$ is defined 
in terms of vertical composition in $\hatUU^\imath$ 
by $(\theta \circ \phi)^{\mu,\lambda} \coloneq \theta^{\mu,\lambda}\circ \phi^{\mu,\lambda}$.
Tensor product of objects is defined by $B^{\otimes n'} \otimes B^{\otimes n''} \coloneq B^{\otimes (n'+n'')}$.
The tensor product
$\theta \otimes \phi\colon B^{\otimes (n'+n'')}\rightarrow B^{\otimes(m'+m'')}$ 
of morphisms 
$\theta\colon B^{\otimes n'} \rightarrow B^{\otimes m'}$ and $\phi\colon B^{\otimes n''} \rightarrow B^{\otimes m''}$ is defined in terms of horizontal composition of 2-morphisms in $\hatUU^\imath$ by 
\begin{equation}
(\theta\otimes\phi)^{\mu,\lambda}_{\bi,\bj}
\coloneq
\begin{dcases}
\theta^{\mu,\nu}_{\bi',\bj'} \phi^{\nu,\lambda}_{\bi'',\bj''}
&\text{if $\nu \coloneq \lambda-\alpha_{\bi''} = \lambda-\alpha_{\bj''}$}\\
0&\text{otherwise,}
\end{dcases}
\end{equation}
for $\bi, \bj$ with
$\lambda-\alpha_\bi = \lambda-\alpha_\bj = \mu$
such that
$\bi = \bi'\bi''$ for $\bi' \in I^{m'}$,
$\bi'' \in I^{m''}$
and $\bj = \bj'\bj''$ 
for $\bj' \in I^{n'}$, $\bj'' \in I^{n''}$.
We equip each morphism space in $\hatAB$
with the topology that is the product of the GC topologies on each hom space. Then
$\hatAB$ is a strict monoidal category enriched in topological vector spaces.
Its identity object $\one$ is $B^{\otimes 0}$.
\end{defin}

\begin{rem}\label{tricky}
Unravelling definitions,
$\Hom_{\hatAB}(B^{\otimes n}, B^{\otimes m})$
is identified with
the completion of 
$$
\bigoplus_{\substack{\lambda \in X^\imath_\0,\ \bi \in I^m, \,\bj \in I^n\\\lambda-\alpha_\bi = \lambda-\alpha_\bj}}
H_\lambda(\bj,\bi)
$$
with respect to the linear topology
which has a system of neighborhoods of 0 given by
the subspaces 
\begin{equation*}
\bigoplus_{\substack{\lambda \in X^\imath_\0,\ \bi \in I^m,\, \bj \in I^n\\\lambda-\alpha_\bi = \lambda-\alpha_\bj}}
H_\lambda(\bj,\bi)_{d(\lambda;\bi,\bj),J(\lambda;\bi,\bj)}
\end{equation*}
for $d(\lambda;\bi,\bj) \in \Z$
and cofinite subsets $J(\lambda;\bi,\bj) \subseteq I$
such that
$H_\lambda(\bj,\bi)_{d(\lambda;\bi,\bj),J(\lambda;\bi,\bj)} = 
H_\lambda(\bj,\bi)$
for all but finitely many of the triples $(\lambda;\bi,\bj)$.
\end{rem}

Recall that $Z_\lambda = \End_{\UU^\imath}(\one_\lambda)$.
We denote the algebra $\End_{\hatUU^\imath}(\one_\lambda)
=\hatH_\lambda(\varnothing,\varnothing)$, which is the completion of $Z_\lambda$, by $\hatZ_\lambda$.
For each $\lambda$, 
the isomorphism $\varphi_\lambda$ from \cref{flambda}
induces an isomorphism
\begin{equation}
\hat\varphi_\lambda\colon\hatR\iso \hatZ_\lambda.
\end{equation}
We described the GC topology on $\hatR$ rather explicitly at the beginning of the section. Using the isomorphisms
$\varphi_\lambda$ and $\hat\varphi_\lambda$, 
this description can be transferred to $Z_\lambda$ and $\hatZ_\lambda$.
In particular, we will denote 
$\varphi_\lambda(U_J)$, 
$\varphi_\lambda(U_{d,J})$,
$\hat\varphi_\lambda(\overline{U}_J)$
and 
$\hat\varphi_\lambda(\overline{U}_{d,J})$
by 
$U_{\lambda;J}$, 
$U_{\lambda;d,J}$, 
$\overline{U}_{\lambda;J}$ and 
$\overline{U}_{\lambda;d,J}$, respectively.

We denote $\End_{\hatAB}(\one)$ by $\hatZ$.
By a special case of \cref{tricky}, the complete topological algebra $\hatZ$ is identified with the completion of
$\bigoplus_{\lambda \in X^\imath_\0}
Z_\lambda$
with respect to the linear topology defined by neighborhoods of 0 given by 
$\bigoplus_{\lambda \in X^\imath_\0} U_{\lambda;d(\lambda),J(\lambda)}$
for
$d(\lambda) \in \Z$ and cofinite $J(\lambda) \subseteq I$
such that 
$U_{\lambda;d(\lambda),J(\lambda)} = \End_{\UU^\imath}(\one_\lambda)$
for all but finitely many $\lambda$.
So we have that 
\begin{equation}\label{hatZ}
\hatZ=\prod_{\lambda \in X^\imath_\0} \hatZ_\lambda.
\end{equation}
The infinite product
$\prod_{i \in I} O_{\lambda;i}(u-i)$
makes sense in $\hatZ_\lambda\lround u^{-1}\rround$.
This expression will play an important role in the next section.

Also of particular interest are the algebras $\End_{\hatUU^\imath}(B_i \one_\lambda)$ and $\End_{\hatAB}(B)$. 
By a special case of \cref{tricky}, $\End_{\hatAB}(B)$ is the completion of
$\bigoplus_{i \in I,\lambda \in X^\imath_\0}
\End_{\UU^\imath}(B_i \one_\lambda)$
with respect to the topology defined by neighborhoods of 0
given by 
$\bigoplus_{i \in I,\lambda \in X^\imath_\0}
H_\lambda(i,i)_{d(\lambda;i),J(\lambda;i)}$
for $d(\lambda;i) \in \Z$ 
and cofinite $J(\lambda;i) \subseteq I$
such that 
$H_\lambda(i,i)_{d(\lambda;i),J(\lambda;i)}
=
\End_{\UU^\imath}(B_i \one_\lambda)$
for all but finitely many $(\lambda;i)$.
So we have that 
\begin{equation}\label{hatIZ}
\End_{\hatAB}(B)
=\prod_{\substack{i \in I\\\lambda \in X^\imath_\0}} \End_{\hatUU^\imath}(B_i \one_\lambda).
\end{equation}

\begin{theo}\label{extending}
The action of $\UU^\imath$ on any
nilpotent Abelian or Karoubian 
iquantum categorification $\catR$ 
extends by continuity 
to make $\catR$ into a 2-representation of $\hatUU^\imath$.
\end{theo}

\begin{proof}
We first prove two claims:

\vspace{2mm}
\noindent{\underline{Claim 1}:} 
{\em
For each object $V \in \catR_\lambda$,
the 
kernel of the homomorphism $Z_\lambda \rightarrow \End_{\catR}(V),
z \mapsto z_V$
contains a GC neighborhood of 0 in $Z_\lambda$.}

\vspace{1mm}
\noindent
{Proof:}
Let $K$ be a finite subset of $I$
closed under $\pm$ which contains all 
$i \in I$ with $B_i V \neq 0$ or $\del_i-\lambda_i \neq 0$.
Then $\big(z_{\lambda;j}^{(n)}\big)_V = 0$ for all $j \in J \coloneq I - K$ and $n \geq 1$.

For each $i$ and $n\geq 1$, \cref{noodslem} implies that
$\big(z_{\lambda;i}^{(n)}\big)_V$ is in the Jacobson radical 
$J(Z_V)$ of the finite-dimensional commutative algebra $Z_V$. 
Choose
$N\geq 1$ such that $J(Z_V)^N = 0$.
Also choose $M\geq 1$ so that $\big(z_{\lambda;k}^{(n)}\big)_V=0$ whenever $k\in K$ and $n>M$. Any monomial that can act nontrivially therefore has fewer than $N$ factors, each of degree at most $2M$. Consequently every monomial of degree at least
$d=2M(N-1)+1$
acts by zero, and $U_{\lambda;d,J}$ lies in the kernel
of the homomorphism
$Z_\lambda \rightarrow \End_\catR(V)$, as required.

\vspace{2mm}
\noindent{\underline{Claim 2}:} {\em
For each object $V \in \catR_\lambda$
and $\bi,\bj \in \langle I \rangle$
with $\lambda-\alpha_\bi = \lambda-\alpha_\bj$,
the evaluation map
$H_\lambda(\bj,\bi)
\rightarrow \Hom_\catR(B_\bj V, B_\bi V)$ is continuous 
for the GC topology on the first hom space and the discrete topology on the second one.}

\vspace{1mm}
\noindent
{Proof:}
Let $d$ and $J$ be as in Claim 1.
Consider
the basis $\basis_\lambda(\bi\times\bj)$ for 
$H_\lambda(\bj,\bi)$.
At each of the chosen dot positions on the finitely many
shapes in $\shape(\bi\times\bj)$,
the dot acts nilpotently on the relevant object obtained from $V$ by applying the functors $B_i$ for the strings to the right of that position.
Hence, a basis diagram acts by zero on $V$ 
providing it is of sufficiently high degree, say $\geq d'$.
Then any element of $H_\lambda(\bj,\bi)_{d+d',J}$
acts as zero on $V$.

\vspace{2mm}
Claim 2 implies that the evaluation map
$H_\lambda(\bj,\bi)
\rightarrow \Hom_\catR(B_\bj V, B_\bi V)$
extends to the completion
$\hatH_\lambda(\bj,\bi)$.
To complete the proof of the theorem, 
it remains to observe that
for a fixed $V \in \catR$ and a fixed length $n \geq 0$,
there are only finitely many words $\bi \in I^n$
such that $B_\bi V \neq 0$.
This follows by induction on $n$ because there are only finitely many $i \in I$ such that $B_i V \neq 0$.
Consequently,
the formal sum defining the 1-morphism $B_\bI$
of $\hatUU^\imath$ acts on $V$ as a finite direct sum,
and the $\bI \times \bJ$ matrices that are 2-morphisms in $\hatUU^\imath$ evaluate to finitely supported matrices.
\end{proof}

\begin{cor}\label{wish}
Let $\catR$ be an even nilpotent Abelian or Karoubian
iquantum categorification
viewed as a 2-representation of $\hatUU^\imath$ as in \cref{extending}.
Then $\catR$ can be made into a strict left $\hatAB$-module
category as follows:
\begin{enumerate}
\item 
The object $B^{\otimes n}$ of $\hatAB$, i.e.,
the formal direct sum \cref{usefulB} of 1-morphisms
of $\hatUU^\imath$, acts so that the summand $\one_\mu B^n \one_\lambda$ acts as the 
functor $\catR_\lambda\rightarrow\catR_\mu$ defining the action of this 1-morphism of $\hatUU^\imath$, and as 0 
on all other iweight subcategories.
\item
A morphism $\theta = (\theta^{\mu,\lambda})_{\mu,\lambda \in X^\imath_\0}$ in $\hatAB$, which by definition is a matrix of 2-morphisms in $\hatUU^\imath$, acts 
as the obvious matrix of natural transformations.
\end{enumerate}
\end{cor}

\begin{proof}
This follows from the definitions. The only point to note is that
only finitely many terms of \cref{usefulB} act in a non-zero way on a given object $V \in \catR$, so the evaluation of the matrix
$\theta = (\theta^{\mu,\lambda})_{\mu,\lambda \in X^\imath_\0}$ on $V$ has finite support.
\end{proof}

\setcounter{section}{6}

\section{Internal bubbles}\label{bubbles}

In this section, rather than working directly with a nilpotent iquantum categorification, we will work in the completed 2-iquantum group $\hatUU^\imath$ which was 
just defined. All convergence of formal Laurent series in $u^{-1}$
with coefficients that are 2-morphisms in $\hatUU^\imath$
is pointwise, that is, the coefficients of each 
power of $u$ converge in the GC topology.

For $\lambda \in X^\imath$, recall that $\hatZ_\lambda \coloneq 
\End_{\hatUU^\imath}(\one_\lambda)$
and 
$O_{\lambda;i}(u) =
\begin{tikzpicture}[baseline=-1mm,iQ,scale=.9]
\draw[-] (-0.25,0) arc(180:-180:0.25);
\node at (-0.36,0) {\strandlabel{\ibar}};
\region{1.05,0}{\lambda};
\node at (.56,0) {$(u)$};
\end{tikzpicture} \in \hatZ_\lambda \lround u^{-1} \rround$.
In the formal Laurent series $O_{\lambda;i}(u)$, we can replace $u$ by $u-i$ to obtain
$O_{\lambda;i}(u-i) \in \hatZ_\lambda\lround u^{-1}\rround$.
For fixed $\lambda$, 
we have that $O_{\lambda;i}\big(u-i\big) \in \id_{\one_\lambda} + u^{-1} \hatZ_\lambda\llbracket u^{-1}\rrbracket$ for all
but finitely many $i \in I$. 

\begin{lem}\label{unexpected}
For any $\lambda \in X^\imath$ and $i \in I$, we have that
\begin{align}
\prod_{j \in I - \{i\}}O_{\lambda;j}(u-j)O_{\lambda;-j}(j-u)
&=\begin{dcases}
\frac{1-4u^2}{4} \id_{\one_\lambda}&\text{if $i=0$}\\
(-1)^{\lambda_h}\frac{1+2u}{2u^2} \id_{\one_\lambda}&\text{if $i=h$}\\
(-1)^{\lambda_h}\frac{1-2u}{2u^2}\id_{\one_\lambda}&\text{if $i=-h$}\\
(-1)^{\lambda_i}\frac{4u^2-1}{4u^2}
\id_{\one_\lambda}
&\text{otherwise.}
\end{dcases}\label{munch}
\end{align}
Also, 
recalling the partition $I = I_\lo\cup\{0\}\cup I_\hi$ from \cref{partitions},
we have that
\begin{align} \label{kins}
\prod_{j \in I_\lo} O_{\lambda;j}(u-j)O_{\lambda;-j}(j-u)
&=
(-1)^{p(\lambda)} \zeta_0(\lambda)(1-2u) \id_{\one_\lambda}.
\end{align}
\end{lem}

\begin{proof}
For each $j \in I$, we have by \cref{iinfgrass} 
that $O_{\lambda;-j}(u)O_{\lambda;j}(-u)
= (-1)^{\del_j+\lambda_j} u^{-a_{j,-j}}\id_{\one_\lambda}$.
Equivalently, replacing $u$ by $j-u$, we have that
$O_{\lambda;j}(u-j) O_{\lambda;-j}(j-u) = 
(-1)^{\del_j+\lambda_j} (j-u)^{-a_{j,-j}}\id_{\one_\lambda}$.
Also, the only $j$ such that either $a_{j,-j}$ or $\del_j$ are non-zero are $j=0,\pm h$.
The identity \cref{munch} follows easily from these facts.
For \cref{kins}, the product in the statement of the lemma
simplifies to $(-1)^{1+\sum_{j \in I_\lo} \lambda_j} (h-u)\id_{\one_\lambda}
=(-1)^{p(\lambda)+t(\lambda)+1}(h-u)\id_{\one_\lambda}$,
then we use that
$\zeta_0(\lambda) =\frac{(-1)^{t(\lambda)+1}}{2}$.
\end{proof}

In order to state and prove \cref{badbunny},
we also need something we call {\em fancy notation}, which we used originally in the proof of \cite[Th.~7.1]{BSW-update}.
First, we introduce the following algebraic notation encoding the fake bubbles
from \cref{ifakebubblepolynomial}: let
\begin{align}\label{thorsten0}
P_{\lambda;i}(u) &\coloneq
\begin{dcases}
(-1)^{t(\lambda)}  u^{-1}
\id_{\one_\lambda}&\text{if $i=0$}\\
\zeta_i(\lambda)^{-1}
\sum_{n = 0}^{\del_i-\lambda_i}
\begin{tikzpicture}[baseline=-1mm,iQ]
\draw[-] (-0.25,0) arc(180:-180:0.25);
\node at (-0.36,0) {\strandlabel{\ibar}};
\region{0.4,0}{\lambda};
\dottybubblelabel{0,0}{n};
\end{tikzpicture}\ 
(u-i)^{\del_i-\lambda_i-n}
&\text{otherwise.}
\end{dcases}
\end{align}
This is an element of $u^{-1}\hatZ_\lambda[u]$; actually, it is a polynomial in $\hatZ_\lambda[u]$ unless $i=0$.
More generally, for any (possibly infinite)
subset $J \subseteq I$, let
\begin{equation}
P_{\lambda;J}(u) \coloneq \prod_{j \in J} P_{\lambda;j}(u) \in u^{-1}\hatZ_\lambda[u].
\end{equation}
The infinite product here makes sense because
$P_{\lambda;j}(u) = \id_{\one_\lambda}$
for all but finitely many $j \in I$. This would not be true if we had not included the normalization factor 
$\zeta_i(\lambda)^{-1}$ (usually just a sign) in the definition \cref{thorsten0}!

Next, let $K$ be a finite subset of $I$, and
consider the algebra $\hatZ_\lambda\llbracket x_k\:|\:k \in K \rrbracket$. 
The fancy notation is the function
\begin{equation}
Q_{\lambda;K}\colon 
\hatZ_\lambda\llbracket x_k\:|\:k \in K \rrbracket
\rightarrow
\hatZ_\lambda
\end{equation}
mapping $f(x_k\:|\:k \in K)$ 
to the element of $\hatZ_\lambda$ obtained by
pinning $f(x_k\:|\:k \in K)$ to the string diagram
$\prod_{k \in K} \zeta_k(\lambda)^{-1}\!
\begin{tikzpicture}[iQ,baseline=-1mm,scale=.9]
\draw[-] (-0.25,0) arc(180:-180:0.25);
\node at (-.4,0) {\strandlabel{\kbar}};
\closeddot{.25,0};
\region{.6,0}{\lambda};
\end{tikzpicture}$ with the variable $x_k$ corresponding to the dot on the $k$th bubble in the product. 
This map is a $\hatZ_\lambda$-module homomorphism but it is {\em not} an algebra homomorphism.
For example, in the most important special case that $K = \{i\}$, when we denote
$Q_{\lambda;K}$ simply by $Q_{\lambda;i}$,
we have that
\begin{equation}\label{thorsten1}
Q_{\lambda;i}(f(x_i)) = \zeta_i(\lambda)^{-1}
\begin{tikzpicture}[iQ,baseline=-1mm,scale=.9]
\draw[-] (-0.25,0) arc(180:-180:0.25);
\pin{.25,0}{1,0}{f(x)};
\region{1.9,0}{\lambda};
\strand{-.4,0}{\ibar};
\end{tikzpicture}
\end{equation}
for $f(x_i) \in \hatZ_\lambda\llbracket x_i\rrbracket$.
Extending $P_{\lambda;i}$ and $Q_{\lambda;i}$ 
to generating functions in the formal variable $u$ 
in the usual way, we have by \cref{ibubblegeneratingfunction} that
\begin{equation}\label{elias}
O_{\lambda;i}(u) = 
P_{\lambda;i}(u+i) + 
Q_{\lambda;i}\big(\tfrac{1}{u-x_i}\big)
\in \hatZ_\lambda\lround u^{-1}\rround.
\end{equation}
Another example of fancy notation:
\begin{equation}\label{why} 
Q_{\lambda;\{i,j\}}\big(\tfrac{1}{i-j+x_i-x_j}\big) = 
\zeta_i(\lambda)^{-1} \zeta_j(\lambda)^{-1}
\begin{tikzpicture}[iQ,baseline=-1mm,scale=1]
\draw[-] (-0.25,.35) arc(180:-180:0.25);
\node at (-.35,.35) {\strandlabel{\jbar}};
\draw[-] (-.25,-.35) arc(180:-180:0.25);
\node at (-.35,-.35) {\strandlabel{\ibar}};
\pinpinmid{.25,.3}{1.7,0}{.25,-.3}{(i-j+x-y)^{-1}};
\region{3.05,0}{\lambda};
\end{tikzpicture}\qquad\text{ for $i \neq j$ in $I$.}
\end{equation}

Now we come to the most important observation of all:
for any $c \in \kk^\times$ and an indeterminate $x$, it makes
sense to evaluate $O_{\lambda;j}(u)$ at $u=c+x$ to obtain a formal power series
$O_{\lambda;j}(c+x) \in 
\hatZ_\lambda\llbracket x \rrbracket$.
To justify this, by \cref{elias}, we have that
\begin{equation}\label{convention}
O_{\lambda;j}(c+x) = P_{\lambda;j}(c+x+j)+Q_{\lambda;j}\big(\tfrac{1}{c+x-x_j}\big)
\in \hatZ_\lambda\llbracket x\rrbracket,
\end{equation}
but we still need to 
explain why the two terms on the right-hand side here make sense.
Since $P_{\lambda;j}(u+j)$ is a polynomial in $\hatZ_\lambda[u]$
for $j \neq 0$, $P_{\lambda;0}(u) \in u^{-1} \hatZ_\lambda[u]$,
and $c+x$ is invertible in $\hatZ_\lambda\llbracket x\rrbracket$, it
makes sense for any $j$ to evaluate $u$ in $P_{\lambda;j}(u+j)$ at $c+x$ to
obtain $P_{\lambda;j}(c+x+j)$.
To understand the second term in \cref{convention}, we have by definition that
$Q_{\lambda;j}\big(\tfrac{1}{u-x_j}\big)
= \zeta_j(\lambda)^{-1}
\sum_{n \geq 0} u^{-n-1}
\begin{tikzpicture}[iQ,baseline=-1mm,scale=.9]
\draw[-] (-0.25,0) arc(180:-180:0.25);
\multcloseddot{.25,0}{west}{n};
\strand{-.4,0}{\jbar};
\end{tikzpicture}$,
so 
$$
Q_{\lambda;j}\big(\tfrac{1}{c+x-x_j}\big)
= \zeta_j(\lambda)^{-1}
\sum_{n \geq 0} (c+x)^{-n-1}
\begin{tikzpicture}[iQ,baseline=-1mm,scale=.9]
\draw[-] (-0.25,0) arc(180:-180:0.25);
\multcloseddot{.25,0}{west}{n};
\strand{-.4,0}{\jbar};
\end{tikzpicture}.
$$
Its $x^m$-coefficient is
$\zeta_j(\lambda)^{-1}
\sum_{n \geq 0}
c^{-m-n-1}\binom{-n-1}{m}
\begin{tikzpicture}[iQ,baseline=-1mm,scale=.9]
\draw[-] (-0.25,0) arc(180:-180:0.25);
\multcloseddot{.25,0}{west}{n};
\strand{-.4,0}{\jbar};
\end{tikzpicture}$,
which is a convergent series in $\hatZ_\lambda$ as the degree of the bubble increases with $n$.
The expressions $O_{\lambda;j}(i-j+x_i)$ on the right-hand side of
the identity in the following theorem are of this form; they
are elements of the ring $\hatZ_\lambda\llbracket x_i\rrbracket$:

\begin{theo}\label{badbunny}
For $\lambda \in X^\imath$, we have that
$$
\prod_{i \in I} O_{\lambda;i}(u-i)
=
\Big(1-\tfrac{(-1)^{p(\lambda)}}{2u}\Big)\id_{\one_\lambda}+\sum_{i \in I}
Q_{\lambda;i}\Big(\tfrac{\prod_{j \in I - \{i\}}O_{\lambda;j}(i-j+x_i)}{u-i-x_i}\Big).
$$
\end{theo}

\begin{proof}
We first observe that
$\prod_{i \in I}O_{\lambda;i}(u-i)$
is a power series in $\id_{\one_\lambda} + u^{-1} \hatZ_\lambda\llbracket u^{-1}\rrbracket$,
so the constant terms agree on both sides of the identity.
This follows because
\begin{align*}
O_{\lambda;i}(u-i)&\in u^{\del_i-\lambda_i}\id_{\one_\lambda} +
u^{\del_i-\lambda_i-1}\hatZ_\lambda\llbracket u^{-1}\rrbracket
\end{align*}
for each $i \in I$ by \cref{ibubblerelnew},
and $\sum_{i \in I} (\del_i-\lambda_i) = 0$ by \cref{firstchoice,sneaky}.
The proof is now reduced to showing that
\begin{equation}\label{pedro}
\left[\prod_{i \in I} O_{\lambda;i}(u-i)\right]_{u:<0}\!\!\!
=
\frac{(-1)^{p(\lambda)+1}}{2u}\id_{\one_\lambda}
+\sum_{i \in I}
Q_{\lambda;i}\Big(\tfrac{\prod_{j \in I - \{i\}}O_{\lambda;j}(i-j+x_i)}{u-i-x_i}\Big).
\end{equation}
The left-hand side of \cref{pedro}
expands to
\begin{align*}
\Bigg[\sum_{\substack{J, K\text{ with }|K|<\infty\\I=J \sqcup K}}
\Big(\prod_{j \in J} P_{\lambda;j}(u)\Big)
\Big(\prod_{k \in K} Q_{\lambda;k}\big(\tfrac{1}{u-k-x_k}\big)\Big)\Bigg]_{u:<0}.
\end{align*}
Using fancy notation, this can be rewritten more simply as
\begin{align}
\sum_{\substack{J, K\text{ with }|K|<\infty\\I=J \sqcup K}}Q_{\lambda;K} \left(\left[\frac{P_{\lambda;J}(u)}{\prod_{k \in K}(u\!-\!k\!-\!x_k)}\right]_{\!u:<0}\right).\label{night}
\end{align}
Let $y_i \coloneq i + x_i$.
By partial fractions, for any polynomial $f(u)$ in $u$ and 
finite $K \subseteq I$, we have that
\begin{align*}
\left[
\frac{f(u)}{\prod_{k \in K}(u-y_k)}
\right]_{u:<0} &\!=
\left[
\sum_{i \in K} \frac{f(u)}{(u-y_i)\prod_{k \in K - \{i\}}(y_i-y_k)}
\right]_{u:<0}
\!\!\!\!\stackrel{\cref{trick}}{=}\ 
\sum_{i \in K} \frac{f(y_i)}{(u-y_i)\prod_{k \in K - \{i\}}(y_i-y_k)},\\
\left[
\frac{f(u)}{u\prod_{k \in K}(u-y_k)}
\right]_{u:<0} &\!=\cdots=
\sum_{i \in K} \frac{f(y_i)}{y_i(u-y_i)\prod_{k\in K-\{i\}}(y_i-y_k)}
+\frac{f(0)}{u\prod_{k \in K}(-y_k)}.
\end{align*}
In \cref{night}, $P_{\lambda;J}(u)$ is a polynomial (so the first partial fractions identity can be applied taking $f(u)\coloneq P_{\lambda;J}(u)$) unless $0 \in J$, in which case
$P_{\lambda;J}(u) = \frac{(-1)^{t(\lambda)}P_{\lambda;J-\{0\}}(u)}{u}$
with $P_{\lambda;J-\{0\}}(u)$ being a polynomial (so the second partial fractions identity makes sense taking $f(u) \coloneq (-1)^{t(\lambda)}P_{\lambda;J-\{0\}}(u)$).
Using both partial fractions identities, we expand \cref{night}, simplifying slightly to obtain
\begin{equation}\label{todo2}
\sum_{\substack{J, K\text{ with }|K|<\infty\\
I=J \sqcup K}}\!\!\!\!
Q_{\lambda;K}
\bigg(\sum_{i \in K} 
\frac{P_{\lambda;J}(i+x_i)}{(u\!-\!i\!-\!x_i) \prod_{k \in K-\{i\}}(i\!-\!k\!+\!x_i\!-\!x_k)}
\bigg)
+\!\!
\!\!\!\sum_{\substack{J, K\text{ with }|K|<\infty\\
I-\{0\}=J \sqcup K}}
\!\!\!\!\!Q_{\lambda;K}
\bigg(
\frac{(-1)^{t(\lambda)}P_{\lambda;J}(0)}{u\prod_{k \in K}(-k\!-\!x_k)}
\bigg).
\end{equation}
Both terms in this expression make sense as formal power series in $u^{-1}$.
Rewriting the first term by moving the summation over $i \in K$ to the front produces
\begin{align*}
\sum_{i \in I}
\sum_{\substack{J, K\text{ with }|K|<\infty\\
I-\{i\}=J \sqcup K}}
&Q_{\lambda;\{i\}\cup K}
\bigg(
\frac{P_{\lambda;J}(i+x_i)}{(u\!-\!i\!-\!x_i) \prod_{k \in K}(i\!-\!k\!+\!x_i\!-\!x_k)}
\bigg)\\
&=
\sum_{i \in I}
Q_{\lambda;i}\Bigg(\frac{1}{u\!-\!i\!-\!x_i}
\sum_{\substack{J, K\text{ with }|K|<\infty\\
I-\{i\}=J \sqcup K}}
Q_{\lambda;K}\bigg(
\frac{P_{\lambda;J}(i+x_i)}{\prod_{k \in K}(i\!-\!k\!+\!x_i\!-\!x_k)}
\bigg)\Bigg)\\
&=
\sum_{i \in I}
Q_{\lambda;i}\Bigg(\frac{1}{u\!-\!i\!-\!x_i}
\prod_{j \in I - \{i\}}
\bigg(P_{\lambda;j}(i+x_i)+Q_{\lambda;j}\big(\tfrac{1}{i-j+x_i-x_j}\big)
\bigg)\Bigg)\\
&=\sum_{i \in I}
Q_{\lambda;i}\Big(\tfrac{\prod_{j \in I - \{i\}}O_{\lambda;j}(i-j+x_i)}{u-i-x_i}\Big),
\end{align*}
which is the main term on the right-hand side of \cref{pedro}.

It just remains to show
that the second term in \cref{todo2} is equal to 
$-\frac{(-1)^{p(\lambda)}}{2u}\id_{\one_\lambda}$.
It is
\begin{multline}
\frac{(-1)^{t(\lambda)}}{u}\!\!
\sum_{\substack{J, K\text{ with }|K|<\infty\\I - \{0\}=J \sqcup K }}\!\!
\Big({\textstyle\prod_{j \in J}}
P_{\lambda;j}(0)\Big)
\Big({\textstyle\prod_{k \in K}}
Q_{\lambda;k}\big(\tfrac{1}{-k-x_k}\big)
\Big)\\
=
\frac{(-1)^{t(\lambda)}}{u}\!\prod_{i \in I-\{0\}}\!\!
\Big(P_{\lambda;i}(0)+Q_{\lambda;i}\big(\tfrac{1}{-i-x_i}\big)\Big)
=\frac{(-1)^{t(\lambda)}}{u}\!\prod_{i \in I-\{0\}}
O_{\lambda;i}(-i).\label{firstday}
\end{multline}
The identity \cref{kins} implies that
$$
\prod_{i \in I_\lo} O_{\lambda;i}(x-i)O_{\lambda;-i}(i-x)
=
(-1)^{p(\lambda)} \zeta_0(\lambda)(1-2x) \id_{\one_\lambda}
$$
in $\hatZ_\lambda\llbracket x \rrbracket$.
Evaluating at $x=0$
gives that
$\prod_{i \in I-\{0\}} O_{\lambda;i}(-i)
=(-1)^{p(\lambda)} \zeta_0(\lambda) \id_{\one_\lambda}$, which is exactly what is needed to
deduce that \cref{firstday}
equals
$-\frac{(-1)^{p(\lambda)}}{2u}\id_{\one_\lambda}$.
\end{proof}

\begin{cor}\label{worsebunny}
For $\lambda \in X^\imath$ and $i \in I$, we have that
\begin{multline*}
P_{\lambda;i}(u)\!\!\!\! \prod_{j \in I - \{i\}} \!\!\!\!O_{\lambda;j}(u-j)
+Q_{\lambda;i}\bigg(\frac{\prod_{j \in I - \{i\}} O_{\lambda;j}(u-j) -\prod_{j \in I - \{i\}} O_{\lambda;j}(i-j+x_i)}{u-i-x_i}\bigg)\\
=\Big(1-\frac{(-1)^{p(\lambda)}}{2u}\Big)\id_{\one_\lambda}+ \sum_{j \in I-\{i\}}
Q_{\lambda;j}\bigg(\frac{\prod_{k \in I - \{j\}}O_{\lambda;k}(j-k+x_j)}{u-j-x_j}\bigg).
\end{multline*}
\end{cor}

\begin{proof}
We have that $P_{\lambda;i}(u) = O_{\lambda;i}(u-i)-Q_{\lambda;i}\big(\tfrac{1}{u-i-x_i}\big)$.
Hence,
$$
P_{\lambda;i}(u)\!\!\!\! \prod_{j \in I - \{i\}} \!\!\!\!O_{\lambda;j}(u-j)
= 
\prod_{j \in I} O_{\lambda;j}(u-j)
-
Q_{\lambda;i}\bigg(\frac{\prod_{j \in I - \{i\}} O_{\lambda;j}(u-j)}{u-i-x_i}\bigg).
$$
Using this, the left-hand side of the identity we are trying to prove becomes
$$
\prod_{j \in I} O_{\lambda;j}(u-j)
-Q_{\lambda;i}\bigg(\frac{\prod_{j \in I - \{i\}} O_{\lambda;j}(i-j+x_i)}{u-i-x_i}\bigg).
$$
Now we use \cref{badbunny} to see that this equals
$$
\Big(1-\frac{(-1)^{p(\lambda)}}{2u}\Big)\id_{\one_\lambda}+\sum_{j \in I}
Q_{\lambda;j}\Big(\tfrac{\prod_{k \in I - \{j\}}O_{\lambda;k}(j-k+x_j)}{u-j-x_j}\Big)
-Q_{\lambda;i}\bigg(\frac{\prod_{j \in I - \{i\}} O_{\lambda;j}(i-j+x_i)}{u-i-x_i}\bigg).
$$
The last term cancels with the $j=i$ term in the summation, and the corollary follows.
\end{proof}

To further simplify notation, we adopt the convention of omitting the subscript $\lambda$ when any of $O_{\lambda;i}, 
P_{\lambda;i}$ or $Q_{\lambda;i}$ are drawn in a pinhead in a string diagram, 
since the value of $\lambda$ can be inferred from the label of the ambient 2-cell.
Recall also that we omit 2-cell labels altogether when we are writing something that is true for all possible iweights.
\cref{badbunny} suggests considering 
the following for $i \neq j$:
\begin{align}
\begin{tikzpicture}[iQ,centerzero]
\draw[-] (0,-.35)\botlabel{i} to (0,.35);
\pin{0,0}{1.1,0}{O_{j}(i-j+x)};
\end{tikzpicture}\ ,
\qquad
\begin{tikzpicture}[iQ,centerzero]
\draw[-] (0,-.35)\botlabel{i} to (0,.35);
\pin{0,0}{1.2,0}{O_{-j}(j-i-x)};
\end{tikzpicture}\ ,\qquad
\begin{tikzpicture}[iQ,centerzero]
\draw[-] (0,-.35)\botlabel{i} to (0,.35);
\pin{0,0}{-1.2,0}{O_{j}(i-j+x)};
\end{tikzpicture}\ ,\qquad
\begin{tikzpicture}[iQ,centerzero]
\draw[-] (0,-.35)\botlabel{i} to (0,.35);
\pin{0,0}{-1.3,0}{O_{-j}(j-i-x)};
\end{tikzpicture}\ .\label{coughs}
\end{align}
When the right-hand 2-cell is labelled $\lambda$,
the first of these is the series
$O_{\lambda;j}(i-j+x) \in \hatZ_\lambda\llbracket x\rrbracket$ from \cref{convention}
pinned to the $i$-string, and
the others have an analogous meaning.
We refer to these elements 
of $\End_{\hatUU^\imath}(B_i \one_\lambda)$ collectively as {\em internal bubbles}. 

When pinned to the same string, it is obvious that all internal bubbles commute with each other and with dots.
From \cref{iinfgrass}, we get that
\begin{align}\label{inverses}
\begin{tikzpicture}[iQ,centerzero]
\draw[-] (0,-.6)\botlabel{i} to (0,.6);
\pin{0,.3}{1.05,.3}{O_{j}(i-j+x)};
\pin{0,-.3}{1.15,-.3}{O_{-j}(j-i-x)};
\region{2.3,0}{\lambda};
\end{tikzpicture}&=\begin{tikzpicture}[iQ,centerzero]
\draw[-] (0,-.6)\botlabel{i} to (0,.6);
\pin{0,0}{-1.35,0}{\frac{(-1)^{\del_{j}+\lambda_j}}{(j-i-x)^{a_{j,-j}}}};
\region{.3,0}{\lambda};
\end{tikzpicture}\ ,&
\begin{tikzpicture}[iQ,centerzero]
\draw[-] (0,-.6)\botlabel{i} to (0,.6);
\pin{0,.3}{-1.2,.3}{O_{j}(i-j+x)};
\pin{0,-.3}{-1.3,-.3}{O_{-j}(j-i-x)};
\region{.3,0}{\lambda};
\end{tikzpicture}&=\begin{tikzpicture}[iQ,centerzero]
\draw[-] (0,-.6)\botlabel{i} to (0,.6);
\pin{0,0}{-1.85,0}{\frac{(-1)^{\del_{j}+\lambda_j+a_{i,j}+a_{i,-j}}}{(j-i-x)^{a_{j,-j}}}};
\region{.3,0}{\lambda};
\end{tikzpicture}\ .
\end{align}
This is a place in which we cannot omit the 2-cell label $\lambda$ since the coefficients on the right-hand side depend on it!
Internal bubbles can be moved
across cups and caps, simply negating the sign of the variable $x$. For example:
\begin{align}\label{pivoting}
\begin{tikzpicture}[iQ,centerzero]
\draw[-] (-.5,.3) to (-.5,-.1) to [out=-90,in=-90,looseness=2] (0,-.1) to (0,.3)\toplabel{i};
\pin{0,0}{1.1,0}{O_{j}(i-j+x)};
\end{tikzpicture}&=
\begin{tikzpicture}[iQ,centerzero]
\draw[-] (-.5,.3)\toplabel{\ibar} to (-.5,-.1) to [out=-90,in=-90,looseness=2] (0,-.1) to (0,.3);
\pin{-0.5,0}{-1.75,0}{O_{j}(i-j-x)};
\end{tikzpicture}\ .
\end{align}
The next two
lemmas give {\em internal bubble slides}.

\begin{lem}
For $i, j, k \in I$ with $i \neq k$, we have that
\begin{align}\label{slideright1}
\begin{tikzpicture}[iQ,centerzero,scale=.6]
\draw(-.85,-1)\botlabel{i} to (-.85,1);
\draw(3.2,-1)\botlabel{j} to (3.2,1);
\pin{-.85,0}{1.3,0}{O_{k}(i-k+x)};
\end{tikzpicture}
&=
\begin{tikzpicture}[iQ,centerzero,scale=.6]
\draw(-.3,-1)\botlabel{i} to (-.3,1);
\draw(.6,-1)\botlabel{j} to (.6,1);
\pin{-.3,.76}{2.7,.76}{O_{k}(i-k+x)};
\pinpin{-.3,-0.45}{.6,-0.45}{3.15,-.45}{\frac{(i-k+x-y)^{a_{k,j}}}{(i-k+x+y)^{a_{k,-j}}}};
\end{tikzpicture}\ ,\\\label{slideright2}
\begin{tikzpicture}[iQ,centerzero,scale=.6]
\draw(-1,-1)\botlabel{i} to (-1,1);
\draw(3.4,-1)\botlabel{j} to (3.4,1);
\pin{-1,0}{1.35,0}{O_{-k}(k-i-x)};
\end{tikzpicture}
&=
\begin{tikzpicture}[iQ,centerzero,scale=.6]
\draw(-.3,-1)\botlabel{i} to (-.3,1);
\draw(.6,-1)\botlabel{j} to (.6,1);
\pin{-.3,.76}{2.7,.76}{O_{-k}(k-i-x)};
\pinpin{-.3,-0.45}{.6,-0.45}{3.15,-.45}{\frac{(k-i-x-y)^{a_{k,-j}}}{(k-i-x+y)^{a_{k,j}}}};
\end{tikzpicture}\ ,\\
\begin{tikzpicture}[iQ,centerzero,scale=.6]
\draw(.6,-1)\botlabel{i} to (.6,1);
\draw(-3.2,-1)\botlabel{j} to (-3.2,1);
\pin{.6,0}{-1.4,0}{O_{k}(i-k+x)};
\end{tikzpicture}
&=
\begin{tikzpicture}[iQ,centerzero,scale=.6]
\draw(.2,-1)\botlabel{i} to (.2,1);
\draw(-.7,-1)\botlabel{j} to (-.7,1);
\pin{.2,.76}{-2.6,.76}{O_{k}(i-k+x)};
\pinpin{.2,-0.45}{-.7,-0.45}{-3.1,-.45}{\frac{(i-k+x+y)^{a_{k,-j}}}{(i-k-x+y)^{a_{k,j}}}};
\end{tikzpicture}\ ,\\
\begin{tikzpicture}[iQ,centerzero,scale=.6]
\draw(.95,-1)\botlabel{i} to (.95,1);
\draw(-3.3,-1)\botlabel{j} to (-3.3,1);
\pin{.95,0}{-1.3,0}{O_{-k}(k-i-x)};
\end{tikzpicture}
&=
\begin{tikzpicture}[iQ,centerzero,scale=.6]
\draw(.1,-1)\botlabel{i} to (.1,1);
\draw(-.8,-1)\botlabel{j} to (-.8,1);
\pin{.1,.76}{-3,.76}{O_{-k}(k-i-x)};
\pinpin{.1,-0.45}{-.8,-0.45}{-3.15,-.45}{\frac{(k-i+x-y)^{a_{k,j}}}{(k-i-x-y)^{a_{k,-j}}}};
\end{tikzpicture}\ .
\end{align}
\end{lem}

\begin{proof}
These are all proved in a similar way, so we just explain the first identity. By \cref{ibubslide}, we have that
$$
\begin{tikzpicture}[anchorbase,iQ]
\draw[-] (-0.25,0) arc(180:-180:0.25);
\node at (-.36,0) {\strandlabel{\kbar}};
\node at (.51,0) {$(u)$};
\draw[-] (-.7,-0.5)\botlabel{i} to (-.7,0.5);
\draw[-] (1,-0.5)\botlabel{j} to (1,0.5);
\end{tikzpicture}=
\begin{tikzpicture}[anchorbase,iQ]
\draw[-] (-.8,-0.5)\botlabel{j} to (-.8,0.5);
\draw[-] (-3.4,-0.5)\botlabel{i} to (-3.4,0.5);
\draw[-] (-0.25,0) arc(180:-180:0.25);
\node at (-0.36,0) {\strandlabel{\kbar}};
\pin{-.8,0}{-2,0}{\frac{(u-x)^{a_{k,j}}}{(u+x)^{a_{k, -j}}}};
\node at (.51,0) {$(u)$};
\end{tikzpicture}\ .
$$
Then we replace $u$ by $i-k+x$ pinned to the string of color $i$, also replacing the $x$ pinned to the string of color $j$ by $y$ to take account of the pin changing to a double pin.
\end{proof}

\begin{lem}\label{naughtykitty}
For any $i \neq j$, we have that
\begin{align}\label{goal}
\begin{tikzpicture}[iQ,centerzero]
\draw[-] (0.2,-.5)\botlabel{i} to (0.2,.5);
\pin{0.2,0}{1.3,0}{O_{j}(i-j+x)};
\end{tikzpicture}
&=\begin{tikzpicture}[iQ,centerzero]
\draw (.9,-.5)\botlabel{i} to (.9,.5);
\pin{.9,-.35}{-.55,-.35}{\frac{(i-j+2x)^{a_{i,-j}}}{(i-j)^{a_{i,j}}}};
\pin{.9,.35}{-.4,.35}{O_{j}(i-j+x)};
\end{tikzpicture}\ ,&
\begin{tikzpicture}[iQ,centerzero]
\draw[-] (0,-.5)\botlabel{i} to (0,.5);
\pin{0,0}{-1.35,0}{O_{-j}(j-i-x)};
\end{tikzpicture}
&=
\begin{tikzpicture}[iQ,centerzero]
\draw[-] (0,-.5)\botlabel{i} to (0,.5);
\pin{0,.35}{1.26,.35}{O_{-j}(j-i-x)};
\pin{0,-.35}{1.4,-.35}{\frac{(j-i-2x)^{a_{i,-j}}}{(j-i)^{a_{i,j}}}};
\end{tikzpicture}
\ .
\end{align}
\end{lem}

\begin{proof}
It suffices to prove the first equality. Then the second can be deduced using \cref{inverses,bone} (or using \cref{izigzag,pivoting}). Here is the proof of the first equality:
\begin{align*}
\begin{tikzpicture}[iQ,centerzero]
\draw[-] (0,-.5)\botlabel{i} to (0,.5);
\pin{0,0}{1.05,0}{O_{j}(i-j+x)};
\end{tikzpicture}
&=
\begin{tikzpicture}[iQ,centerzero]
\draw[-] (0.2,.5) to (0.2,0) to [out=-90,in=-90,looseness=.7] (2.4,0) to [out=90,in=90,looseness=2] (2.9,0) to (2.9,-.5) \botlabel{i};
\pin{0.2,0}{1.25,0}{O_{j}(i-j+x)};
\end{tikzpicture}
\stackrel{\cref{slideright1}}{=}
\begin{tikzpicture}[iQ,centerzero]
\draw[-] (0,.5) to (0,0) to [out=-90,in=-90,looseness=2] (.5,0) to [out=90,in=90,looseness=.7] (2.8,0) to (2.8,-.5) \botlabel{i};
\pin{0,0}{1.7,0}{O_{j}(i-j+x)};
\pinpin{.44,-.2}{0.06,-.2}{-1.4,-.2}{\frac{(i-j+x-y)^{a_{j,-i}}}{(i-j+x+y)^{a_{j,i}}}};
\end{tikzpicture}\\
&=\begin{tikzpicture}[iQ,centerzero]
\draw[-] (0,.5) to (0,0) to [out=-90,in=-90,looseness=2] (.5,0) to [out=90,in=90,looseness=2] (1,0) to (1,-.5) \botlabel{i};
\pin{0,.4}{-1.2,0.4}{O_{j}(i-j+x)};
\pin{0.06,-.2}{-1.4,-.3}{\frac{(i-j+x+x)^{a_{i,-j}}}{(i-j+x-x)^{a_{i,j}}}};
\end{tikzpicture}=\begin{tikzpicture}[iQ,centerzero]
\draw[-] (0,.5) to (0,-.5)  \botlabel{i};
\pin{0,.35}{-1.2,0.35}{O_{j}(i-j+x)};
\pin{0,-.35}{-1.4,-.35}{\frac{(i-j+2x)^{a_{i,-j}}}{(i-j)^{a_{i,j}}}};
\end{tikzpicture}\ .\qedhere
\end{align*}
\end{proof}

From \cref{inverses,naughtykitty}, it follows that
\begin{align}\label{sleepydog}
\begin{tikzpicture}[iQ,centerzero]
\draw[-] (0,-.6)\botlabel{i} to (0,.6);
\pin{0,.2}{1.1,.2}{O_{j}(i-j+x)};
\pin{0,-.2}{-1.25,-.2}{O_{-j}(j-i-x)};
\region{.7,-.3}{\lambda};
\end{tikzpicture}&=(-1)^{\lambda_j}\ \begin{tikzpicture}[iQ,centerzero]
\draw[-] (0,-.6)\botlabel{i} to (0,.6);
\pin{0,0}{-.85,0}{s_{i,j}(x)};
\region{.3,0}{\lambda};
\end{tikzpicture}
\end{align}
for any $i \neq j$ in $I$, where
\begin{align}
s_{i,j}(x) &\coloneq
(-1)^{\del_j}
\frac{(j-i-2x)^{a_{i,-j}}}{(j-i)^{a_{i,j}}(j-i-x)^{a_{j,-j}}} \in \kk\llbracket x \rrbracket.
\end{align}
As $\del_j+\del_{-j} = -a_{j,-j}$,
we have that 
\begin{equation}\label{baldy}
s_{i,j}(x) = (-1)^{a_{i,j}+a_{i,-j}} s_{-i,-j}(-x).
\end{equation}

\begin{lem}\label{deadly}
For any $i \in I$, we have that 
$\displaystyle\prod_{j \in I-\{i\}} s_{i,j}(x) = \zeta_i \zeta_{-i} 
\gamma_{-i,i}$.
\end{lem}

\begin{proof}
This is a direct calculation.
The product $\prod_{j \in I-\{i\}} s_{i,j}(x)$ is the product of the following four terms:
\begin{align*}
A&\coloneq\!\!\!\prod_{j \in I-\{i\}}\!\!(-1)^{\del_j},&
B&\coloneq\!\!\!\prod_{j \in I-\{i\}}\!\!(j-i)^{-a_{i,j}},&
C&\coloneq\!\!\!\prod_{j \in I-\{i\}}\!\!(j-i-2x)^{a_{i,-j}},&
D&\coloneq\!\!\!\prod_{j \in I-\{i\}}\!\!(j-i-x)^{-a_{j,-j}}.
\end{align*}
Only $j = i \pm 1$ contribute to $B$, so $B = -1$ always. For the other three products,
only $j= 0,h$ contribute to $A$, only $j =-i,-i\pm 1$ contribute to $C$, and only $j =0,\pm h$ contribute to $D$.
When $i \notin \{0,\pm h\}$, we have that
$A=1,B=-1,C=\frac{4(i+x)^2}{4(i+x)^2-1}$
and $D=\frac{(i+x)^2-h^2}{(i+x)^2}$
so $ABCD=-1$, which equals 
$\zeta_i\zeta_{-i}\gamma_{-i,i}=-1$ in this case.
We leave the check of the exceptional cases
$i = 0$ and $i=\pm h$ to the reader.
\end{proof}

In the case $i=0$, we need the following more refined statement:

\begin{lem}\label{sin}
$\displaystyle\prod_{j \in I_\hi} s_{0,j}(x) = \zeta_0 (h+x)/(h-x)$
and
$\displaystyle\prod_{j \in I_\lo} s_{0,j}(x) = \zeta_0(h-x)/(h+x)$.
\end{lem}

\begin{proof}
This is a similar calculation to the one in \cref{deadly}. In the notation from that proof, when $i=0$, the only $j \neq 0$ that contribute are
$j=h \in I_\lo$ (to $A$ and $D$), $j=-h \in I_\hi$ (to $D$), and $j=\pm 1$ (to $B$ and $C$), remembering that
$1 \in I_\hi$ and $-1 \in I_\lo$.
\end{proof}

Now we are going to consider the following infinite products for $i \in I$ 
and $\lambda \in X^\imath$:
\begin{align}\label{quicka}
\Delta_{\lambda;i}(x)
&\coloneq
\begin{dcases}
(-1)^{p(\lambda)}\prod_{j \in I-\{i\}}O_{\lambda;j}(i-j+x)
&\text{if $i \in I_+$}\\
\frac{1}{h+x}
\prod_{j \in I_\hi} O_{\lambda;j}(-j+x)
&\text{if $i = 0$}\\
\id_{\one_\lambda}
\hspace{60mm}
&\text{if $i \in I_-$,}
\end{dcases}\\
\label{quickb}
\nabla_{\lambda;i}(x)
&\coloneq
\begin{dcases}
\id_{\one_\lambda}
&\text{if $i \in I_+$}\\
(-1)^{p(\lambda)}\zeta_0(\lambda)^{-1}
(h+x)\prod_{j \in I_\lo}O_{\lambda;j}(-j+x)&\text{if $i = 0$}\\
(-1)^{p(\lambda)}\zeta_i(\lambda)^{-1}
\prod_{j \in I-\{i\}}O_{\lambda;j}(i-j+x)
&\text{if $i \in I_-$.}
\end{dcases}
\end{align}
These are all elements of $\hatZ_\lambda\llbracket x\rrbracket$.
Also let
\begin{equation}\label{quickc}
\Deltanabla_{\lambda;i}(x) \coloneq
\Delta_{\lambda;i}(x)\nabla_{\lambda;i}(x) = (-1)^{p(\lambda)}\zeta_i(\lambda)^{-1}\prod_{j \in I - \{i\}}O_{\lambda;j}(i-j+x).
\end{equation}
From the definitions, \cref{munch,kins,eitherway}, we have that
\begin{align}\label{bibble}
\Delta_{\lambda;-i}(-x)\nabla_{\lambda;i}(x) &= 
\begin{dcases}
1 &\text{if $i \in I_+$}\\
2x+1&\text{if $i=0$}\\
\frac{x-1}{2(x-h)^2}&\text{if $i=-h$}\\
\frac{1}{4(i+x)^2}-1&\text{if $i \in I_--\{-h\}$.}
\end{dcases}
\end{align}
Hence, $\Delta_{\lambda;i}(x)$, $\nabla_{\lambda;i}(x)$ 
and $\Deltanabla_{\lambda;i}(x)$
are all invertible in 
$\hatZ_\lambda\llbracket x \rrbracket$. 
Moreover, setting $x=0$ in the $i=0$ case of \cref{bibble}, we have that
\begin{equation}\label{betterargument}
\Deltanabla_{\lambda;0}(0) = 1.
\end{equation}

Again we use the convention 
of 
omitting the subscript $\lambda$
when $\Delta_{\lambda;i}(x)$, $\nabla_{\lambda;i}(x)$ 
or $\Deltanabla_{\lambda;i}(x)$ is pinned 
to a dot in a string diagram, 
since the intended value of $\lambda$ 
can be inferred from the label of the 2-cell containing the pinhead.
We use the following additional shorthands, which we refer to as {\em glyphs}:
\begin{align}\label{longyear1a}
\begin{tikzpicture}[iQ,centerzero]
\draw[-] (0,-.5)\botlabel{i} to (0,.5);
\internalbubbleplus[180]{0,0};
\end{tikzpicture}
&\coloneq
\begin{tikzpicture}[iQ,centerzero]
\draw[-] (0,-.5)\botlabel{i} to (0,.5);
\pin{0,0}{.75,0}{\Delta_i(x)};
\end{tikzpicture}\ ,&
\begin{tikzpicture}[iQ,centerzero]
\draw[-] (0,-.5)\botlabel{i} to (0,.5);
\internalbubbleplus[0]{0,0};
\end{tikzpicture}
&\coloneq
\begin{tikzpicture}[iQ,centerzero]
\draw[-] (0,-.5)\botlabel{i} to (0,.5);
\pin{0,0}{.75,0}{\nabla_i(x)};
\end{tikzpicture}\ ,\\
\begin{tikzpicture}[iQ,centerzero]
\draw[-] (0,-.5)\botlabel{i} to (0,.5);
\internalbubbleminus[180]{0,0};
\end{tikzpicture}
&\coloneq
\begin{tikzpicture}[iQ,centerzero]
\draw[-] (0,-.5)\botlabel{i} to (0,.5);
\pin{0,0}{-1,0}{\Delta_{-i}(-x)};
\end{tikzpicture}\ ,
\label{longyear2a}&
\begin{tikzpicture}[iQ,centerzero]
\draw[-] (0,-.5)\botlabel{i} to (0,.5);
\internalbubbleminus[0]{0,0};
\end{tikzpicture}
&\coloneq 
\begin{tikzpicture}[iQ,centerzero]
\draw[-] (0,-.5)\botlabel{i} to (0,.5);
\pin{0,0}{-1,0}{\nabla_{-i}(-x)};
\end{tikzpicture}\ .
\end{align}
The convention here is that the glyphs labelled $+$ are shorthand for a pin in the 2-cell immediately to the right of the glyph, while the glyphs labelled $-$ are pins to the left.
There are also the {\em double glyphs}:
\begin{align}\label{doggynose}
\begin{tikzpicture}[iQ,centerzero]
\draw[-] (0,-.5)\botlabel{i} to (0,.5);
\fatbubbleplus{0,0};
\end{tikzpicture}
&\coloneq
\begin{tikzpicture}[iQ,centerzero]
\draw[-] (0,-.5)\botlabel{i} to (0,.5);
\internalbubbleplus[0]{0,-.18};
\internalbubbleplus[180]{0,.18};
\end{tikzpicture}
=\begin{tikzpicture}[iQ,centerzero]
\draw[-] (0,-.5)\botlabel{i} to (0,.5);
\pin{0,0}{.8,0}{\Deltanabla_i(x)};
\end{tikzpicture}\ ,&
\begin{tikzpicture}[iQ,centerzero]
\draw[-] (0,-.5)\botlabel{i} to (0,.5);
\fatbubbleminus{0,0};
\end{tikzpicture}
&\coloneq
\begin{tikzpicture}[iQ,centerzero]
\draw[-] (0,-.5)\botlabel{i} to (0,.5);
\internalbubbleminus[0]{0,-.18};
\internalbubbleminus[180]{0,.18};
\end{tikzpicture}
=\begin{tikzpicture}[iQ,centerzero]
\draw[-] (0,-.5)\botlabel{i} to (0,.5);
\pin{0,0}{-1.1,0}{\Deltanabla_{-i}(-x)};
\end{tikzpicture}\ .
\end{align} 
The following
is a restatement of \cref{badbunny} using the new glyph notation:

\begin{theo}\label{badbunny2}
For any $\lambda \in X^\imath$, we have that
\begin{equation}\label{whydidittakesolong}
\prod_{i \in I}
\begin{tikzpicture}[iQ,centerzero,scale=.9]
\draw (-0.25,0) arc(180:-180:0.25);
\node at (-0.37,0) {\strandlabel{\ibar}};
\node at (.88,0) {\color{black}$(u-i)$};
\region{1.7,0}{\lambda};
\end{tikzpicture}
=
\big(1-\tfrac{(-1)^{p(\lambda)}}{2u}\big)\id_{\one_\lambda}+
(-1)^{p(\lambda)}
\sum_{i \in I}
\begin{tikzpicture}[iQ,anchorbase,scale=2]
\draw (0.3,-.05) to (0.3,.05) to [out=90,in=90,looseness=2] (-.1,.05) to (-.1,-.05) to [out=-90,in=-90,looseness=2] (.3,-.05);
\fatbubbleplus{.3,0.1};
\region{.7,.15}{\lambda};
\pin{.3,-.1}{.7,-.1}{\frac{1}{u-i-x}};
\stringlabel{-.18,0}{\ibar};
\end{tikzpicture}\ .
\end{equation}
\end{theo}

\begin{cor}\label{worsebunny2}
If $i \neq 0$ then
\begin{equation}\label{phewy}
\zeta_i(\lambda)\ \begin{tikzpicture}[iQ,baseline=-3.9mm,scale=1.9]
\draw (-0.15,-0.4)\botlabel{i} to (-.15,0);
\pin{-.15,-.1}{-1.4,-.1}{P_i(i+x)\Deltanabla_i(x)+
Q_i\big(\partial_{x,x_i}\Deltanabla_i(x)\big)};
\region{-1.6,-.4}{\lambda};
\end{tikzpicture}=
(-1)^{p(\lambda)}\!\begin{tikzpicture}[iQ,baseline=-3.9mm,scale=1.9]
\draw (-0.15,-0.4)\botlabel{i} to (-.15,0);
\region{-.25,-.2}{\lambda};
\end{tikzpicture}
\ - \ 
\begin{tikzpicture}[iQ,baseline=-3.9mm,scale=1.9]
\draw (-0.15,-0.4)\botlabel{i} to (-.15,0);
\pin{-.15,-.1}{-.6,-.1}{\frac{1}{2(i+x)}};
\region{-.6,-.4}{\lambda};
\end{tikzpicture}+\!
\sum_{j \in I - \{i\}}\!\zeta_j(\lambda)\ 
\begin{tikzpicture}[iQ,baseline=-3.9mm,scale=1.9]
\draw (-0.15,-0.4)\botlabel{i} to (-.15,0);
\pin{-.15,-.1}{-.9,-.1}{Q_j\Big(\frac{\Deltanabla_j(x_j)}{i+x-j-x_j}\Big)};
\region{-.9,-.4}{\lambda};
\end{tikzpicture}\ .
\end{equation}
Also
\begin{equation}\label{louis}
\begin{tikzpicture}[iQ,baseline=-3.9mm,scale=1.9]
\draw (-0.15,-0.4)\botlabel{0} to (-.15,0);
\pin{-.15,-.1}{-.65,-.1}{
\frac{1-\Deltanabla_0(x)}{2x}};
\region{-.65,-.4}{\lambda};
\end{tikzpicture}+
\zeta_0(\lambda)\ \begin{tikzpicture}[iQ,baseline=-3.9mm,scale=1.9]
\draw (-0.15,-0.4)\botlabel{0} to (-.15,0);
\pin{-.15,-.1}{-.9,-.1}{
Q_0\big(\partial_{x,x_0}\Deltanabla_0(x)\big)};
\region{-.9,-.4}{\lambda};
\end{tikzpicture}=
(-1)^{p(\lambda)}\!\begin{tikzpicture}[iQ,baseline=-3.9mm,scale=1.9]
\draw (-0.15,-0.4)\botlabel{0} to (-.15,0);
\region{-.25,-.2}{\lambda};
\end{tikzpicture}+\!
\sum_{j \in I - \{0\}}\!\zeta_j(\lambda)\ 
\begin{tikzpicture}[iQ,baseline=-3.9mm,scale=1.9]
\draw (-0.15,-0.4)\botlabel{0} to (-.15,0);
\pin{-.15,-.1}{-.85,-.1}{Q_j\Big(\frac{\Deltanabla_j(x_j)}{x-j-x_j}\Big)};
\region{-.85,-.4}{\lambda};
\end{tikzpicture}\ .
\end{equation}
\end{cor}

\begin{proof}
The formula \cref{phewy} when $i\neq 0$ follows by taking the identity from \cref{worsebunny}, multiplying
by $(-1)^{p(\lambda)}$, then evaluating $u$ at $\begin{tikzpicture}[centerzero,iQ]
\draw(0,-.2)\botlabel{i} to (0,.2);
\pin{0,0}{-.65,0}{i+x};
\region{-1.2,0}{\lambda};
\end{tikzpicture}$.
If $i=0$, it does not make sense to evaluate $u$ at $x$ directly, so first we rearrange the identity from \cref{worsebunny} using 
that $P_{\lambda;0}(u) = \frac{(-1)^{t(\lambda)}}{u}
= -\frac{\zeta_0(\lambda)^{-1}}{2u}$
to obtain
\begin{multline*}
(-1)^{p(\lambda)}
\frac{\id_{\one_\lambda}-
(-1)^{p(\lambda)}\zeta_0(\lambda)^{-1}\prod_{j \in I - \{0\}} O_{\lambda;j}(u-j)}{2u}\\
+
Q_{\lambda;0}\bigg(\frac{\prod_{j \in I - \{0\}} O_{\lambda;j}(u-j) -\prod_{j \in I - \{0\}} O_{\lambda;j}(-j+x_0)}{u-x_0}\bigg)\\
=\id_{\one_\lambda}+ \sum_{j \in I-\{0\}}
Q_{\lambda;j}\bigg(\frac{\prod_{k \in I - \{j\}}O_{\lambda;k}(j-k+x_j)}{u-j-x_j}\bigg).
\end{multline*}
Then we multiply by $(-1)^{p(\lambda)}$ and evaluate $u$ at
$\begin{tikzpicture}[centerzero,iQ]
\draw(0,-.2)\botlabel{0} to (0,.2);
\closeddot{0,0};
\region{-.25,0}{\lambda};
\end{tikzpicture}$ as before. For the first term, this evaluation makes sense because
\cref{betterargument} implies that the numerator is divisible by the denominator (after evaluation). This proves \cref{louis}.
\end{proof}

In the rest of the section, we derive some further properties of the glyphs.
Using \cref{pivoting} and its variations, it follows that
\begin{align}\label{riveting}
\begin{tikzpicture}[iQ,centerzero]
\draw[-] (-.5,.3)\toplabel{i} to (-.5,-.1) to [out=-90,in=-90,looseness=2] (0,-.1) to (0,.3);
\internalbubbleplus[0]{0,.05};
\end{tikzpicture}&=
\begin{tikzpicture}[iQ,centerzero]
\draw[-] (-.5,.3)\toplabel{i} to (-.5,-.1) to [out=-90,in=-90,looseness=2] (0,-.1) to (0,.3);
\internalbubbleminus[0]{-.5,.05};
\end{tikzpicture}\ 
,&
\begin{tikzpicture}[iQ,centerzero]
\draw[-] (-.5,-.3)\botlabel{i} to (-.5,.1) to [out=90,in=90,looseness=2] (0,.1) to (0,-.3);
\internalbubbleplus[180]{0,0.0};
\end{tikzpicture}&=
\begin{tikzpicture}[iQ,centerzero]
\draw[-] (-.5,-.3)\botlabel{i} to (-.5,.1) to [out=90,in=90,looseness=2] (0,.1) to (0,-.3);
\internalbubbleminus[180]{-.5,0.0};
\end{tikzpicture}\ ,&
\begin{tikzpicture}[iQ,centerzero]
\draw[-] (-.5,.3)\toplabel{i} to (-.5,-.1) to [out=-90,in=-90,looseness=2] (0,-.1) to (0,.3);
\internalbubbleplus[180]{0,-0.05};
\end{tikzpicture}&=
\begin{tikzpicture}[iQ,centerzero]
\draw[-] (-.5,.3)\toplabel{i} to (-.5,-.1) to [out=-90,in=-90,looseness=2] (0,-.1) to (0,.3);
\internalbubbleminus[180]{-.5,-.05};
\end{tikzpicture}\ ,&
\begin{tikzpicture}[iQ,centerzero]
\draw[-] (-.5,-.3)\botlabel{i} to (-.5,.1) to [out=90,in=90,looseness=2] (0,.1) to (0,-.3);
\internalbubbleplus[0]{0,.05};
\end{tikzpicture}&=
\begin{tikzpicture}[iQ,centerzero]
\draw[-] (-.5,-.3)\botlabel{i} to (-.5,.1) to [out=90,in=90,looseness=2] (0,.1) to (0,-.3);
\internalbubbleminus[0]{-.5,0.05};
\end{tikzpicture}\ ,\\\label{riveting1}
\begin{tikzpicture}[iQ,centerzero]
\draw[-] (-.5,.3)\toplabel{i} to (-.5,-.1) to [out=-90,in=-90,looseness=2] (0,-.1) to (0,.3);
\internalbubbleminus[0]{0,.05};
\end{tikzpicture}&=
\begin{tikzpicture}[iQ,centerzero]
\draw[-] (-.5,.3)\toplabel{i} to (-.5,-.1) to [out=-90,in=-90,looseness=2] (0,-.1) to (0,.3);
\internalbubbleplus[0]{-.5,.05};
\end{tikzpicture}\ 
,&
\begin{tikzpicture}[iQ,centerzero]
\draw[-] (-.5,-.3)\botlabel{i} to (-.5,.1) to [out=90,in=90,looseness=2] (0,.1) to (0,-.3);
\internalbubbleminus[180]{0,0.0};
\end{tikzpicture}&=
\begin{tikzpicture}[iQ,centerzero]
\draw[-] (-.5,-.3)\botlabel{i} to (-.5,.1) to [out=90,in=90,looseness=2] (0,.1) to (0,-.3);
\internalbubbleplus[180]{-.5,0.0};
\end{tikzpicture}\ ,&
\begin{tikzpicture}[iQ,centerzero]
\draw[-] (-.5,.3)\toplabel{i} to (-.5,-.1) to [out=-90,in=-90,looseness=2] (0,-.1) to (0,.3);
\internalbubbleminus[180]{0,-0.05};
\end{tikzpicture}&=
\begin{tikzpicture}[iQ,centerzero]
\draw[-] (-.5,.3)\toplabel{i} to (-.5,-.1) to [out=-90,in=-90,looseness=2] (0,-.1) to (0,.3);
\internalbubbleplus[180]{-.5,-.05};
\end{tikzpicture}\ ,&
\begin{tikzpicture}[iQ,centerzero]
\draw[-] (-.5,-.3)\botlabel{i} to (-.5,.1) to [out=90,in=90,looseness=2] (0,.1) to (0,-.3);
\internalbubbleminus[0]{0,.05};
\end{tikzpicture}&=
\begin{tikzpicture}[iQ,centerzero]
\draw[-] (-.5,-.3)\botlabel{i} to (-.5,.1) to [out=90,in=90,looseness=2] (0,.1) to (0,-.3);
\internalbubbleplus[0]{-.5,0.05};
\end{tikzpicture}\ ,\\
\label{riveting2}
\begin{tikzpicture}[iQ,centerzero]
\draw[-] (-.5,.3)\toplabel{i} to (-.5,-.1) to [out=-90,in=-90,looseness=2] (0,-.1) to (0,.3);
\fatbubbleplus{0,-.05};
\end{tikzpicture}&=
\begin{tikzpicture}[iQ,centerzero]
\draw[-] (-.5,.3)\toplabel{i} to (-.5,-.1) to [out=-90,in=-90,looseness=2] (0,-.1) to (0,.3);
\fatbubbleminus{-.5,-.05};
\end{tikzpicture}\ 
,&
\begin{tikzpicture}[iQ,centerzero]
\draw[-] (-.5,-.3)\botlabel{i} to (-.5,.1) to [out=90,in=90,looseness=2] (0,.1) to (0,-.3);
\fatbubbleplus{0,0.05};
\end{tikzpicture}&=
\begin{tikzpicture}[iQ,centerzero]
\draw[-] (-.5,-.3)\botlabel{i} to (-.5,.1) to [out=90,in=90,looseness=2] (0,.1) to (0,-.3);
\fatbubbleminus{-.5,0.05};
\end{tikzpicture}\ ,&
\begin{tikzpicture}[iQ,centerzero]
\draw[-] (-.5,.3)\toplabel{i} to (-.5,-.1) to [out=-90,in=-90,looseness=2] (0,-.1) to (0,.3);
\fatbubbleminus{0,-0.05};
\end{tikzpicture}&=
\begin{tikzpicture}[iQ,centerzero]
\draw[-] (-.5,.3)\toplabel{i} to (-.5,-.1) to [out=-90,in=-90,looseness=2] (0,-.1) to (0,.3);
\fatbubbleplus{-.5,-.05};
\end{tikzpicture}\ ,&
\begin{tikzpicture}[iQ,centerzero]
\draw[-] (-.5,-.3)\botlabel{i} to (-.5,.1) to [out=90,in=90,looseness=2] (0,.1) to (0,-.3);
\fatbubbleminus{0,.05};
\end{tikzpicture}&=
\begin{tikzpicture}[iQ,centerzero]
\draw[-] (-.5,-.3)\botlabel{i} to (-.5,.1) to [out=90,in=90,looseness=2] (0,.1) to (0,-.3);
\fatbubbleplus{-.5,0.05};
\end{tikzpicture}\ .
\end{align}

\begin{lem}\label{bubblyinverses}
We have that
\begin{align}
\bigg(\ \begin{tikzpicture}[iQ,centerzero]
\draw[-] (0,-.5)\botlabel{i} to (0,.5);
\internalbubbleplus[0]{0,0};
\end{tikzpicture}\ \bigg)^{-1}&=
\begin{tikzpicture}[iQ,centerzero]
\draw[-] (0,-.5)\botlabel{i} to (0,.5);
\internalbubbleminus[180]{0,0};
\end{tikzpicture}\ ,& 
\bigg(\ \begin{tikzpicture}[iQ,centerzero]
\draw[-] (0,-.5)\botlabel{i} to (0,.5);
\internalbubbleplus[180]{0,0};
\end{tikzpicture}\ \bigg)^{-1}&=
\begin{tikzpicture}[iQ,centerzero]
\draw[-] (0,-.5)\botlabel{i} to (0,.5);
\internalbubbleminus[0]{0,0};
\end{tikzpicture}\ ,&
\bigg(\ \begin{tikzpicture}[iQ,centerzero]
\draw[-] (0,-.5)\botlabel{i} to (0,.5);
\fatbubbleplus{0,0};
\end{tikzpicture}\ \bigg)^{-1}&=
\begin{tikzpicture}[iQ,centerzero]
\draw[-] (0,-.5)\botlabel{i} to (0,.5);
\fatbubbleminus{0,0};
\end{tikzpicture}\ .\label{scorching}
\end{align}
\end{lem}

\begin{proof}
In view of \cref{riveting}, it suffices to prove the first equality, so we must show that
$$
\begin{tikzpicture}[iQ,centerzero]
\draw[-] (0,-.7)\botlabel{i} to (0,.7);
\internalbubbleplus[0]{0,.3};
\internalbubbleminus[180]{0,-0.3};
\region{.4,0}{\lambda};
\end{tikzpicture}
=
\ 
\begin{tikzpicture}[iQ,centerzero]
\draw[-] (0,-.7)\botlabel{i} to (0,.7);
\region{.3,0}{\lambda};
\end{tikzpicture}
$$
for every $\lambda \in X^\imath$.
This is immediate from the definitions 
if $i \in I_+$.
If $i=0$, the definitions and \cref{sleepydog} imply that
$$
\begin{tikzpicture}[iQ,centerzero]
\draw[-] (0,-.7)\botlabel{0} to (0,.7);
\internalbubbleplus[0]{0,.3};
\internalbubbleminus[180]{0,-0.3};
\region{.4,0}{\lambda};
\end{tikzpicture}=
\zeta_0^{-1} (-1)^{p(\lambda)+t(\lambda)+\sum_{j\in I_\lo} \lambda_j}
\begin{tikzpicture}[iQ,centerzero]
\draw[-] (0,-.7)\botlabel{0} to (0,.7);
\pin{0,.3}{0.7,.3}{\frac{h+x}{h-x}};
\pin{0,-.3}{1.3,-.3}{\prod_{j \in I_\lo} s_{0,j}(x)};
\region{1.4,.3}{\lambda};
\end{tikzpicture}\ ,
$$
which equals $\id_{B_0 \one_\lambda}$ thanks to \cref{sin}.
If $i \in I_-$, we instead get that
$$
\begin{tikzpicture}[iQ,centerzero]
\draw[-] (0,-.7)\botlabel{i} to (0,.7);
\internalbubbleplus[0]{0,.3};
\internalbubbleminus[180]{0,-0.3};
\region{.4,0}{\lambda};
\end{tikzpicture}=
\zeta_i^{-1} (-1)^{\lambda_i + \sum_{j\in I-\{i\}} \lambda_j}
\begin{tikzpicture}[iQ,centerzero]
\draw[-] (0,-.7)\botlabel{i} to (0,.7);
\pin{0,-.2}{1.5,-.2}{\prod_{j \in I-\{i\}} s_{i,j}(x)};
\region{1.4,.3}{\lambda};
\end{tikzpicture}\ .
$$
Since $\gamma_{-i,i}=\zeta_{-i}=1$, $\lambda_0 = 0$
and $\lambda_j + \lambda_{-j}=0$ for all $j \in I$,
\cref{deadly} implies that this 
equals $\id_{B_i\one_\lambda}$.
\end{proof}

\begin{lem}\label{premiracle}
For $i,j \in I$, we have that
$$
\frac{\gamma_{j,i}}{\gamma_{i,j}}
\bigg(\prod_{k \in I - \{i\}}\!\!\frac{(i-k+x+y)^{a_{k,-j}}}{(i-k-x+y)^{a_{k,j}}}\bigg)\bigg(\prod_{k \in I - \{j\}}\!\!\frac{(j-k+x-y)^{a_{k,i}}}{(j-k+x+y)^{a_{k,-i}}}\bigg)=r_{i,-j}r_{-i,j}r_{j,i}r_{-j,-i}.
$$
\end{lem}

\begin{proof}
Let $\phi(z) \coloneq \frac{z^2}{z^2-1}$. 
Some elementary calculation gives that
\begin{align*}
\prod_{k \in I - \{i\}}\!\!\frac{(i-k+x+y)^{a_{k,-j}}}{(i-k-x+y)^{a_{k,j}}}
&=\frac{\phi(i+j+x+y)}{\phi(i-j-x+y)}\cdot\frac{(y-x)^{a_{i,j}}}{(x+y)^{a_{i,-j}}},\\
\prod_{k \in I - \{j\}}\!\!\frac{(j-k+x-y)^{a_{k,i}}}{(j-k+x+y)^{a_{k,-i}}}
&=\frac{\phi(i-j-x+y)}{\phi(i+j+x+y)}\cdot\frac{(x+y)^{a_{j,-i}}}{(x-y)^{a_{j,i}}}.
\end{align*}
The product of these is $(-1)^{a_{i,j}}$.
The lemma now follows using \cref{anothercrazyone}.
\end{proof}

\begin{lem}\label{preglyphslide}
We have that
\begin{align}
\begin{tikzpicture}[centerzero,iQ,scale=1.4]
\draw[-] (0,-.3) \botlabel{j} to (0,.3);
\draw[-] (-0.3,-.3) \botlabel{i} to (-0.3,.3);
\pin{-.3,0}{.5,0}{\Deltanabla_i(x)};
\end{tikzpicture}
&=\gamma_{i,j}\ 
\begin{tikzpicture}[centerzero,iQ,scale=1.4]
\draw[-] (0,-.3) \botlabel{j} to (0,.3);
\draw[-] (-0.3,-.3) \botlabel{i} to (-0.3,.3);
\pinpin{-.3,.2}{0,.2}{1.5,.2}{\prod_{k \in I - \{i\}}\!\!\frac{(i+x-k+y)^{a_{k,-j}}}{(i-k+x-y)^{a_{k,j}}}};
\fatbubbleplus{-.3,-.1};
\end{tikzpicture}\ ,\label{glyphy1}\\
\begin{tikzpicture}[centerzero,iQ,scale=1.4]
\draw[-] (0,-.3) \botlabel{i} to (0,.3);
\draw[-] (-0.3,-.3) \botlabel{0} to (-0.3,.3);
\pin{-.3,0}{.5,0}{\Delta_0(x)};
\end{tikzpicture}
&=
\begin{tikzpicture}[centerzero,iQ,scale=1.4]
\draw[-] (0,-.3) \botlabel{i} to (0,.3);
\draw[-] (-0.3,-.3) \botlabel{0} to (-0.3,.3);
\pinpin{-.3,.2}{0,.2}{1.3,.2}{\prod_{j \in I_\hi}\!\!\frac{(-j+x+y)^{a_{j,-i}}}{(-j+x-y)^{a_{j,i}}}};
\internalbubbleplus[180]{-.3,-.1};
\end{tikzpicture}\ .\label{glyphy2}
\end{align}
\end{lem}

\begin{proof}
To prove the first identity,
we use \cref{slideright1} for each of the pins in the product of internal bubbles
from the definition \cref{doggynose}:
$$
\begin{tikzpicture}[centerzero,iQ,scale=1.4]
\draw[-] (0,-.3) \botlabel{j} to (0,.3);
\draw[-] (-0.3,-.3) \botlabel{i} to (-0.3,.3);
\pin{-.3,.1}{.5,.1}{\Deltanabla_i(x)};
\region{.5,-.25}{\lambda};
\end{tikzpicture}
=
\zeta_i(\lambda)\zeta_i(\lambda-\alpha_j)^{-1}
\begin{tikzpicture}[centerzero,iQ,scale=1.4]
\draw[-] (0,-.3) \botlabel{j} to (0,.3);
\draw[-] (-1,-.3) \botlabel{i} to (-1,.3);
\pinpin{-1,.2}{0,.2}{1.5,.2}{\prod_{k \in I - \{i\}}\!\!\frac{(i-k+x+y)^{a_{k,-j}}}{(i-k+x-y)^{a_{k,j}}}};
\pin{-1,-.2}{-.5,-.2}{\Deltanabla_i(x)};
\region{1.5,-.25}{\lambda};
\end{tikzpicture}
=
\gamma_{i,j}\ 
\begin{tikzpicture}[centerzero,iQ,scale=1.4]
\draw[-] (0,-.3) \botlabel{j} to (0,.3);
\draw[-] (-0.3,-.3) \botlabel{i} to (-0.3,.3);
\pinpin{-.3,.2}{0,.2}{1.5,.2}{\prod_{k \in I - \{i\}}\!\!\frac{(i-k+x+y)^{a_{k,-j}}}{(i-k+x-y)^{a_{k,j}}}};
\fatbubbleplus{-.3,-.1};
\region{1.5,-.25}{\lambda};
\end{tikzpicture}\ .
$$
The proof of the second identity is similar.
\end{proof}

\begin{lem}\label{glyphslide}
\begin{enumerate}
\item
For $i \neq \pm j$, we have that
\begin{align} 
\begin{tikzpicture}[centerzero,iQ,scale=1.4]
\draw[-] (0.3,-.4) \botlabel{j} to (-0.3,.4);
\draw[-] (-0.3,-.4) \botlabel{i} to (0.3,.4);
\fatbubbleplus{-.15,.2};
\end{tikzpicture}
&=\gamma_{j,i}^{-1}
\begin{tikzpicture}[centerzero,iQ,scale=1.4]
\draw[-] (0.3,-.4) \botlabel{j} to (-0.3,.4);
\draw[-] (-0.3,-.4) \botlabel{i} to (0.3,.4);
\pinpin{-.2,.25}{.2,.25}{1.7,.25}{\prod_{k \in I - \{j\}}\!\!\frac{(j-k+x-y)^{a_{k,i}}}{(j-k+x+y)^{a_{k,-i}}}};
\fatbubbleplus{.15,-.2};
\end{tikzpicture}\ ,&
\begin{tikzpicture}[centerzero,iQ,scale=1.4]
\draw[-] (0.3,-.4) \botlabel{j} to (-0.3,.4);
\draw[-] (-0.3,-.4) \botlabel{i} to (0.3,.4);
\fatbubbleplus{.2,.26};
\end{tikzpicture}
&=\gamma_{i,j}
\begin{tikzpicture}[centerzero,iQ,scale=1.4]
\draw[-] (0.3,-.4) \botlabel{j} to (-0.3,.4);
\draw[-] (-0.3,-.4) \botlabel{i} to (0.3,.4);
\pinpin{-.2,.25}{.2,.25}{1.7,.25}{\prod_{k \in I - \{i\}}\!\!\frac{(i-k+x+y)^{a_{k,-j}}}{(i-k-x+y)^{a_{k,j}}}};
\fatbubbleplus{-.15,-.2};
\end{tikzpicture}\ .
\end{align}
\item
If $i \neq 0$ then
\begin{align}
\begin{tikzpicture}[centerzero,iQ,scale=1.4]
\draw[-] (0.3,-.4) \botlabel{0} to (-0.3,.4);
\draw[-] (-0.3,-.4) \botlabel{i} to (0.3,.4);
\internalbubbleplus[180]{-.15,.2};
\end{tikzpicture}
&=
\begin{tikzpicture}[centerzero,iQ,scale=1.4]
\draw[-] (0.3,-.4) \botlabel{0} to (-0.3,.4);
\draw[-] (-0.3,-.4) \botlabel{i} to (0.3,.4);
\pinpin{-.2,.25}{.2,.25}{1.5,.25}{\prod_{j \in I_\hi}\!\!\frac{(-j+x-y)^{a_{j,i}}}{(-j+x+y)^{a_{j,-i}}}};
\internalbubbleplus[180]{.15,-.2};
\end{tikzpicture}\ ,&
\begin{tikzpicture}[centerzero,iQ,scale=1.4]
\draw[-] (0.3,-.4) \botlabel{i} to (-0.3,.4);
\draw[-] (-0.3,-.4) \botlabel{0} to (0.3,.4);
\internalbubbleplus[180]{.2,.26};
\end{tikzpicture}&=
\begin{tikzpicture}[centerzero,iQ,scale=1.4]
\draw[-] (0.3,-.4) \botlabel{i} to (-0.3,.4);
\draw[-] (-0.3,-.4) \botlabel{0} to (0.3,.4);
\pinpin{-.2,.25}{.2,.25}{1.5,.25}{\prod_{j \in I_\hi}\!\!\frac{(-j+x+y)^{a_{j,-i}}}{(-j-x+y)^{a_{j,i}}}};\internalbubbleplus[180]{-.15,-.2};
\end{tikzpicture}
\ .
\end{align}
\end{enumerate}
\end{lem}

\begin{proof}
(i) By \cref{preglyphslide}, we have that
\begin{align*}
\begin{tikzpicture}[centerzero,iQ,scale=1.6]
\draw[-] (0.3,-.4) \botlabel{j} to (-0.3,.4);
\draw[-] (-0.3,-.4) \botlabel{i} to (0.3,.4);
\fatbubbleplus{-.15,.2};
\end{tikzpicture}
&=\gamma_{j,i}^{-1}\begin{tikzpicture}[centerzero,iQ,scale=1.6]
\draw[-] (0.3,-.4) \botlabel{j} to (-0.3,.4);
\draw[-] (-0.3,-.4) \botlabel{i} to (0.3,.4);
\pinpin{-.16,.22}{.2,.25}{1.5,.35}{\prod_{k \in I - \{j\}}\!\!\frac{(j-k+x-y)^{a_{k,i}}}{(j-k+x+y)^{a_{k,-i}}}};
\pin{-.1,.14}{.7,-.2}{\Deltanabla_j(x)};
\end{tikzpicture}\\
&=
\gamma_{j,i}^{-1}\begin{tikzpicture}[centerzero,iQ,scale=1.6]
\draw[-] (0.3,-.4) \botlabel{j} to (-0.3,.4);
\draw[-] (-0.3,-.4) \botlabel{i} to (0.3,.4);
\pinpin{-.2,.25}{.2,.25}{1.6,.25}{\prod_{k \in I - \{j\}}\!\!\frac{(j-k+x-y)^{a_{k,i}}}{(j-k+x+y)^{a_{k,-i}}}};
\pin{.15,-.2}{.8,-.3}{\Deltanabla_j(x)};
\end{tikzpicture}=
\gamma_{j,i}^{-1}
\begin{tikzpicture}[centerzero,iQ,scale=1.6]
\draw[-] (0.3,-.4) \botlabel{j} to (-0.3,.4);
\draw[-] (-0.3,-.4) \botlabel{i} to (0.3,.4);
\pinpin{-.2,.25}{.2,.25}{1.6,.25}{\prod_{k \in I - \{j\}}\!\!\frac{(j-k+x-y)^{a_{k,i}}}{(j-k+x+y)^{a_{k,-i}}}};
\fatbubbleplus{.15,-.2};
\end{tikzpicture}\ .
\end{align*}
This proves the first equality.
The proof of the second equality is similar.

\vspace{2mm}
\noindent
(ii) Similar.
\end{proof}

\begin{lem}\label{anothermiracle}
For $i,j \in I$ with $(i,j) \neq (0,0)$, we have that
$$
\begin{tikzpicture}[centerzero,iQ,scale=1.6]
\draw[-] (0.3,-.4) \botlabel{j} to (-0.3,.4);
\draw[-] (-0.3,-.4) \botlabel{i} to (0.3,.4);
\fatbubbleplus{-.15,.2};
\fatbubbleplus{.15,.2};
\end{tikzpicture}
=
r_{i,-j}r_{-i,j}r_{j,i}r_{-j,-i}
\begin{tikzpicture}[centerzero,iQ,scale=1.6]
\draw[-] (0.3,-.4) \botlabel{j} to (-0.3,.4);
\draw[-] (-0.3,-.4) \botlabel{i} to (0.3,.4);
\fatbubbleplus{.15,-.2};
\fatbubbleplus{-.15,-.2};
\end{tikzpicture}\ .
$$
\end{lem}

\begin{proof}
We first prove this in the case $i \neq \pm j$.
By \cref{glyphslide}(i), we have that
$$
\begin{tikzpicture}[centerzero,iQ,scale=1.6]
\draw[-] (0.3,-.4) \botlabel{j} to (-0.3,.4);
\draw[-] (-0.3,-.4) \botlabel{i} to (0.3,.4);
\fatbubbleplus{-.15,.2};
\fatbubbleplus{.15,.2};
\end{tikzpicture}=
\frac{\gamma_{j,i}}{\gamma_{i,j}}\begin{tikzpicture}[centerzero,iQ,scale=1.6]
\draw[-] (0.3,-.4) \botlabel{j} to (-0.3,.4);
\draw[-] (-0.3,-.4) \botlabel{i} to (0.3,.4);
\fatbubbleplus{-.15,-.2};
\fatbubbleplus{.15,-.2};
\pinpin{-.15,.2}{.15,.2}{2.9,.2}{\left(\prod_{k \in I - \{i\}}\!\!\frac{(i-k+x+y)^{a_{k,-j}}}{(i-k-x+y)^{a_{k,j}}}\right)\left(\prod_{k \in I - \{j\}}\!\!\frac{(j-k+x-y)^{a_{k,i}}}{(j-k+x+y)^{a_{k,-i}}}\right)};
\end{tikzpicture}\ .
$$
It remains to apply \cref{premiracle}.

Next, we prove the lemma in the case $i=j \neq 0$. Expanding the definitions and using \cref{preglyphslide}
gives that
$$
\begin{tikzpicture}[centerzero,iQ,scale=1.6]
\draw[-] (0.3,-.4) \botlabel{i} to (-0.3,.4);
\draw[-] (-0.3,-.4) \botlabel{i} to (0.3,.4);
\fatbubbleplus{-.15,.2};
\fatbubbleplus{.15,.2};
\region{.3,0}{\lambda};
\end{tikzpicture}
=
\gamma_{i,i}^{-1}
\begin{tikzpicture}[centerzero,iQ,scale=1.6]
\draw[-] (0.3,-.4) \botlabel{i} to (-0.3,.4);
\draw[-] (-0.3,-.4) \botlabel{i} to (0.3,.4);
\pinpin{-.16,.22}{.2,.25}{1.49,.35}{\prod_{j \in I - \{i\}}\!\!\frac{(i-j+x-y)^{a_{j,i}}}{(i-j+x+y)^{a_{j,-i}}}};
\pinpin{-.1,.14}{.08,.1}{1,-.18}{\Deltanabla_i(x)\Deltanabla_i(y)};
\region{.2,-.05}{\lambda};
\end{tikzpicture}\ .
$$
Now we observe that both $\Deltanabla_{\lambda;i}(x)\Deltanabla_{\lambda;i}(y)$
and $\prod_{j \in I - \{i\}}\!\!\frac{(i-j+x-y)^{a_{j,i}}}{(i-j+x+y)^{a_{j,-i}}}$
are symmetric in $x$ and $y$.
To see this for the latter expression, the denominator is evidently symmetric. The numerator is symmetric because 
the only non-trivial factors arise from $j=i\pm 1$, these
being $$
(i-(i+1)+x-y)^{-1} (i-(i-1)+x-y)^{-1}
= (x-y-1)^{-1} (x-y+1)^{-1} = ((x-y)^2-1)^{-1},$$ 
which is symmetric in $x$ and $y$. In view of this symmetry, 
\cref{symmetricslider} implies 
that both of the double pins can be commuted past the crossing to obtain
$$
\gamma_{i,i}^{-1}
\begin{tikzpicture}[centerzero,iQ,scale=1.6]
\draw[-] (0.3,-.4) \botlabel{i} to (-0.3,.4);
\draw[-] (-0.3,-.4) \botlabel{i} to (0.3,.4);
\pinpin{-.16,-.22}{.2,-.25}{1.5,-.35}{\prod_{j \in I - \{i\}}\!\!\frac{(i-j+x-y)^{a_{j,i}}}{(i-j+x+y)^{a_{j,-i}}}};
\pinpin{-.1,-.14}{.08,-.1}{.9,.12}{\Deltanabla_i(x)\Deltanabla_i(y)};
\region{.2,.05}{\lambda};
\end{tikzpicture}
=\begin{tikzpicture}[centerzero,iQ,scale=1.6]
\draw[-] (0.3,-.4) \botlabel{i} to (-0.3,.4);
\draw[-] (-0.3,-.4) \botlabel{i} to (0.3,.4);
\fatbubbleplus{-.15,-.2};
\fatbubbleplus{.15,-.2};
\region{.3,0}{\lambda};
\end{tikzpicture}\ .
$$
Since $r_{i,-i}r_{-i,i}r_{i,i}r_{-i,-i}=1$, this completes the proof in this case.

Finally, to prove it when $i=-j \neq 0$, we take the $i=j \neq 0$ case just proved and rotate to get that
$$
\begin{tikzpicture}[centerzero,iQ,scale=1.6]
\draw[-] (0.3,-.4) \botlabel{\ibar} to (-0.3,.4);
\draw[-] (-0.3,-.4) \botlabel{i} to (0.3,.4);
\fatbubbleplus{.15,.2};
\fatbubbleminus{.15,-.2};
\end{tikzpicture}=
\begin{tikzpicture}[iQ,centerzero,scale=1.9]
\draw (-.5,.4) to (-.5,0) to [out=-90,in=-120,looseness=2] (0,0)
to [out=60,in=90,looseness=2] (.5,0) to (.5,-.4)\botlabel{\ibar};
\draw (-0.15,0.4) \braiddown (0.15,-0.4)\botlabel{i};
\fatbubbleplus{.12,.15};
\fatbubbleplus{-.12,.15};
\end{tikzpicture}
=
\begin{tikzpicture}[iQ,centerzero,scale=1.9]
\draw (-.5,.4) to (-.5,0) to [out=-90,in=-120,looseness=2] (0,0)
to [out=60,in=90,looseness=2] (.5,0) to (.5,-.4)\botlabel{\ibar};
\draw (-0.15,0.4) \braiddown (0.15,-0.4)\botlabel{i};
\fatbubbleplus{.12,-.15};
\fatbubbleplus{-.12,-.15};
\end{tikzpicture}
=\begin{tikzpicture}[anchorbase,iQ,scale=1.6]
\draw[-] (0.3,-.4) \botlabel{\ibar} to (-0.3,.4);
\draw[-] (-0.3,-.4) \botlabel{i} to (0.3,.4);
\fatbubbleplus{-.15,-.2};
\fatbubbleminus{-.15,.2};
\end{tikzpicture}\ .
$$
The desired relation follows easily from this using \cref{bubblyinverses}.
\end{proof}

To prove further results about the glyphs on a 0-string, 
we let
\begin{align}\label{DfromGPT}
D_\lambda(x,y) &\coloneq \partial_{x,y} \Delta_{\lambda;0}(x)
= \frac{\Delta_{\lambda;0}(x)-\Delta_{\lambda;0}(y)}{x-y},\\
\G_\lambda(x,y)&\coloneq 1+D_\lambda(x,y)\nabla_{\lambda;0}(-y).
\label{malaysia}
\end{align}
Both are elements of $\hatZ_\lambda\llbracket x,y\rrbracket$,
and $D_{\lambda}(x,y)=D_\lambda(y,x)$.
Again, we drop the subscript $\lambda$ when these are pinned into a string diagram.

\begin{lem}\label{scruffylem}
For $\lambda \in X^\imath$, we have that
\begin{align}\label{scruffyeq}
\ .
\label{scruffiereq}
\end{align}
\end{lem}

\begin{proof}
The identity \cref{scruffiereq} follows from \cref{scruffyeq}
by attaching a cap to the northwest of the diagram and a cup to the southeast then simplifying in an obvious way.
To prove \cref{scruffyeq}, 
rearranging \cref{DfromGPT} and using \cref{bibble} gives
\begin{equation}\label{scruffy}
\Delta_{\lambda;0}(x) \nabla_{\lambda;0}(-y) =1-x-y + (x-y) \G_\lambda(x,y).
\end{equation}
Since
$$
%
\ ,
$$
the identity \cref{scruffyeq} follows from \cref{scruffy}.
\end{proof}

Some more relations for the glyphs on a $0$-string:
\begin{align}\label{whoosh}
%
\ .
\end{align}
For the first of these, we applied \cref{bubblyinverses}.
The second follows from the definitions and 
\cref{slideright1}, noting that
the only $i \in I_\hi$ with $a_{i,0} \neq 0$ is $i=1$.
The third follows by first applying \cref{stomp} then using \cref{demazureslide}.
The fourth is proved in a similar way to the third.
These relations are needed to prove the following:

\begin{lem}\label{maybe}
For any $\lambda \in X^\imath$, we have that
\begin{equation}\label{stunning}
%
\ .
\end{equation}
\end{lem}

\begin{proof}
We first show that
\begin{align}\label{mainpoint}
%
\ .
\end{align}
To see this, we have by its definition that
$$
\G_\lambda(x,y)=1+\frac{\Delta_{\lambda;0}(x)\nabla_{\lambda;0}(-y)-\Delta_{\lambda;0}(y)\nabla_{\lambda;0}(-y)}{x-y}
\stackrel{\cref{bibble}}{=}\frac{\Delta_{\lambda;0}(x)\nabla_{\lambda;0}(-y)-(1-x-y)}{x-y}.
$$
Using this, \cref{whoosh,stomp},
it follows that
$$
%
\ .
$$
Composing these two equations gives
$$
%
\ ,
$$
which is \cref{mainpoint}.

Now consider the first term on the left-hand side of \cref{stunning}:
\begin{align*}
%
\ .
\end{align*}
Adding this equation to \cref{mainpoint} completes the proof of the lemma.
\end{proof}

\begin{lem}\label{easycurl}
Suppose that $\lambda \in X^\imath$ and $f(x) \in \hatZ_\lambda\llbracket x\rrbracket$.
\begin{enumerate}
\item For any $i \in I-\{0\}$, we have that
\begin{equation}\label{curlgen1}
%
\ .
\end{equation}
\item We have that
\begin{equation}\label{curlgen2}
%
\ .
\end{equation}
\end{enumerate}
\end{lem}

\begin{proof}
(i)
By \cref{genidemazure} (actually, its mirror image), we have that
$$
%
\ .
$$
Now we multiply this by $f(u)$
then take the $u^{-1}$-coefficients using \cref{trick,trick2}, extended continuously to formal power series.

\vspace{2mm}
\noindent
(ii)
We again start by sliding the dot generating function past the crossing. Now  
the undotted curl is 0 due to \cref{icurl}, but we get an extra error term:
\begin{align*}
%
\ .
\end{align*}
The final equality here used $\frac{1}{(u-x)(u+x)} = \frac{1}{2u}\big(\frac{1}{u-x}+\frac{1}{u+x}\big)$.
Then once again we multiply by $f(u)$ and take $u^{-1}$-coefficients.
\end{proof}

\begin{cor}\label{lastthingforrii}
For any $\lambda \in X^\imath$ and $p(x,y) \in \hatZ_\lambda\llbracket x,y\rrbracket$, we have that
\begin{align}\label{lastthing1}
%
\ .
\end{align}
\end{cor}

\begin{proof}
The first equality follows from \cref{easycurl}(ii) and \cref{ipivotal}.
The second follows by rotating the first through 180$^\circ$.
\end{proof}

\begin{cor}\label{carnomore}
Suppose that $\lambda \in X^\imath$.
\begin{enumerate}
\item 
For $i \in I- \{0\}$, we have that
\begin{equation}\label{bryan1}
%
\ .
\end{equation}
\item We have that
\begin{equation}\label{bryan2}
%
\ .
\end{equation}
\end{enumerate}
\end{cor}

\begin{proof}
(i)
We have that
\begin{align*}
%
\ .
\end{align*}

\vspace{1mm}
\noindent
(ii) We have that
\begin{align*}
%
\ .
\end{align*}
The second equality here makes sense because of \cref{betterargument}.
\end{proof}

\begin{lem}\label{easycross}
For $i \in I-\{0\}$, $\lambda \in X^\imath$ and $f(x) \in \hatZ_\lambda\llbracket x\rrbracket$, we have that
\begin{equation}\label{easiercross}
%
\ .
\end{equation}
\end{lem}

\begin{proof}
It suffices to prove this in the case that $f(x) \in 
\hatZ_\lambda[x]$, since it then extends by continuity to all power series.
We note to start with that
$$
%
\ ,
$$
as follows from \cref{iquadratic} using \cref{trick2}.
Using this for the second equality, it follows that
\begin{align*}
%
\ .
\end{align*}
Now we multiply this by $f(u)$ and take the $u^{-1}$-coefficients using 
\cref{trick}, \cref{trick2,easycurl} to deduce that
$$
%
\ .
$$
It remains to observe that 
$$
f(y) \big(\partial_{x,y} P_{\lambda;i}(i+x)\big) +\big(\partial_{x,y} f(x)\big) P_{\lambda;i}(i+x) = \partial_{x,y} \big(f(x) P_{\lambda;i}(i+x)\big)
$$
by the twisted Leibniz rule for Demazure operators, and
$Q_{\lambda;i}(\partial_{x,x_i} \partial_{x,y} f(x))
=
\partial_{x,y} Q_{\lambda;i}(\partial_{x,x_i}f(x))$.
\end{proof}

\begin{cor}\label{bingleysleeping}
For $i \in I - \{0\}$ and any $\lambda$,
we have that
\begin{equation}
%
\!.
\end{equation}
\end{cor}

\begin{proof}
We have that
\begin{align*}
%
.
\end{align*}
It remains to note that
\begin{align*}
\partial_{x,y}\bigg( \frac{1}{2(i+x)}\bigg)
&=
-\frac{1}{2(i+x)(i+y)},\\
\partial_{x,y}\bigg(
\frac{\Deltanabla_j(x_j)}{i-j+x-x_j}\bigg)
&=
-\frac{\Deltanabla_j(x_j)}{(i-j+x-x_j)(i-j+y-x_j)}.
\qedhere
\end{align*}
\end{proof}

\setcounter{section}{7}

\section{\texorpdfstring{Iquantum to Brauer}{Iquantum to Brauer}}\label{iQtoB}

In the rest of the article, we are going to show that the notions of ``Brauer categorification'' and ``even nilpotent iquantum categorification'' are equivalent in such a way that the following hold:
\begin{itemize}
\item 
The canonical 
decomposition $\catR = \bigoplus_{\lambda \in X^\imath} \catR_\lambda$ of an iquantum categorification 
is the same as the ``block'' decomposition of $\catR$ viewed as a Brauer categorification
from \cref{blocks,blocksagain}, 
interpreting $\catR_\lambda$ as zero for
$\lambda \in X^\imath_\1$.
\item
The action of the functor $B_i$ arising from the $\UU^\imath$-action (whose identity endomorphism is a black string labelled by $i$) is the same as the eigenfunctor
$B_i$ arising from the $\AB$-action 
(whose identity endomorphism is a magenta string labelled by $i$). 
\item 
The natural transformations defined by the dots on the two sides are related by
\begin{align}\label{veryeasy}
\begin{tikzpicture}[iQ,centerzero]
\draw[-] (0,-0.3)\botlabel{i} -- (0,0.3);
\closeddot{0,0};
\end{tikzpicture}=\ 
\begin{tikzpicture}[B,centerzero]
\draw[-] (0,-0.3)\botlabel{i} -- (0,0.3);
\pin{0,0}{-.7,0}{x-i};
\end{tikzpicture}
\qquad&\Leftrightarrow\qquad
\begin{tikzpicture}[B,centerzero]
\draw[-] (0,-0.3)\botlabel{i} -- (0,0.3);
\closeddot{0,0};
\end{tikzpicture}=\ \begin{tikzpicture}[iQ,centerzero]
\draw[-] (0,-0.3)\botlabel{i} -- (0,0.3);
\pin{0,0}{-.7,0}{x+i};
\end{tikzpicture}\ .
\end{align}
\end{itemize}
From this, it follows that the $\beta$-numbers and $\beta$-degree of an irreducible object of $\catR^\fg$
defined via the Brauer action
are the same as its $\beta$-numbers and $\beta$-degree defined via the iquantum action.

In this section, we start from the iquantum side, proving \cref{thma} from the introduction. 
In fact, we prove a stronger result formulated in terms of 
the topological monoidal category 
$\hatAB$ from \cref{iqmonoidal}:
we will show that there is a strict $\kk$-linear monoidal functor
$\Psi\colon\AB \rightarrow \hatAB$
which sends the generating object $B$ of $\AB$ to the object $B$ of
$\hatAB$.
Any even nilpotent iquantum categorification can be made into
a strict left $\hatAB$-module category in an obvious way, 
so \cref{thma} will then follow simply
by restricting the action of $\hatAB$ to $\AB$ along $\Psi$.

To prepare for the statement of the main theorem, we first introduce some notation 
which will be used to denote
the images of the generating morphisms of $\AB$ under the functor $\Psi$.
Recall $\G_\lambda(x,y) \in \hatZ_\lambda \llbracket x,y\rrbracket$ from \cref{malaysia} (which is denoted simply $\G(x,y)$ when drawn in a 2-cell labelled $\lambda$ in some string diagram).
For $i,j \in I$ and $\s,\t \in \{+,-\}$
such that ($\s i \in I_+\cup\{0\}$ and $\t j \in I_+$)
or $\s i = \t j = 0$, let
\begin{align}\label{trustfunds}
G^{\s\t}_{i,j}(x,y)&\coloneq
\begin{dcases}
1+\frac{1}{\s(i+x)-\t(j+y)}&\text{if $\s i \notin \{0, \t j, \t j-1\}$}\\
\frac{t^{-1}_{\s i,\t j}}{\s(i+x)-\t(j+y)}&\text{if $\s i=\t j-1 \neq 0$}\\
1+\s(i+x)-\t(j+y)&\text{if $\s i=\t j$}\\
\prod_{k \in \t I_\lo}
\big(\sqrt{-1}(\s x-\t(k+y))\big)^{-a_{k,j}}&\text{if $\s i=0 \neq \t j$,}
\end{dcases}
\end{align}
a power series in $\kk\llbracket x,y\rrbracket^\times$.
For $\lambda \in X^\imath_\0$ and $i,j,k,l \in I$
such that $\lambda-\alpha_i-\alpha_j=\lambda-\alpha_k-\alpha_l$, 
let
\begin{align}
\begin{tikzpicture}[iQ,centerzero,scale=1.2]
\draw[-] (0,-0.3) \botlabel{i} -- (0,0.3);
\magdot{0,0};
\region{.3,0}{\lambda};
\end{tikzpicture}&\coloneq
\begin{tikzpicture}[iQ,centerzero]
\draw[-] (0,-0.3) \botlabel{i} -- (0,0.3);
\pin{0,0}{-.7,0}{i+x};
\region{.3,0}{\lambda};
\end{tikzpicture},\label{dotdef}\\
\label{capdef}
\begin{tikzpicture}[iQ,anchorbase,scale=1.6]
\draw (0.3,-0.4) to (0.3,-.2);
\draw[-] (-0.1,-.2) to (-0.1,-0.4) \botlabel{i};
\draw (.3,-.2) to[out=90,in=90,looseness=1.5] (-.1,-.2);
\magup{0.1,-0.01};
\region{.5,-.2}{\lambda};
\end{tikzpicture}
&\coloneq
\begin{tikzpicture}[iQ,anchorbase,scale=1.6]
\draw (0.3,-0.4)  to (0.3,-.2);
\draw[-] (-0.1,-.2) to (-0.1,-0.4)\botlabel{i};
\draw (.3,-.2) to[out=90,in=90,looseness=2] (-.1,-.2);
\internalbubbleplus[180]{.3,-.2};
\region{.6,-.2}{\lambda};
\end{tikzpicture}=
\begin{tikzpicture}[iQ,anchorbase,scale=1.6]
\draw (0.3,-0.4)  to (0.3,-.2);
\draw[-] (-0.1,-.2) to (-0.1,-0.4)\botlabel{i};
\draw (.3,-.2) to[out=90,in=90,looseness=2] (-.1,-.2);
\internalbubbleminus[180]{-0.09,-.2};
\region{.5,-.2}{\lambda};
\end{tikzpicture}\ ,
\\\label{cupdef}
\begin{tikzpicture}[iQ,anchorbase,scale=1.6]
\draw[-] (0.3,0.4) to (0.3,.2);
\draw[-] (-0.1,.2) to (-0.1,0.4) \toplabel{i};
\draw (.3,.2) to[out=-90,in=-90,looseness=1.5] (-.1,.2);
\magdown{0.1,.02};
\region{.5,.2}{\lambda};
\end{tikzpicture}
&\coloneq 
\begin{tikzpicture}[iQ,anchorbase,scale=1.6]
\draw[-] (0.3,0.4)  to (0.3,.2);
\draw[-] (-0.1,.2) to (-0.1,0.4) \toplabel{i};
\draw (.3,.2) to[out=-90,in=-90,looseness=2] (-.1,.2);
\internalbubbleplus[0]{.3,0.2};
\region{.6,.2}{\lambda};
\end{tikzpicture}=
\begin{tikzpicture}[iQ,anchorbase,scale=1.6]
\draw[-] (0.3,0.4)  to (0.3,.2);
\draw[-] (-0.1,.2) to (-0.1,0.4) \toplabel{i};
\draw (.3,.2) to[out=-90,in=-90,looseness=2] (-.1,.2);
\internalbubbleminus[0]{-.09,0.2};
\region{.5,.2}{\lambda};
\end{tikzpicture}\ ,\\
\label{crossingdef}
\begin{tikzpicture}[centerzero,iQ,scale=1.6]
\draw[-] (0.3,-.4) \botlabel{j} to (-0.3,.4)\toplabel{l};
\draw[-] (-0.3,-.4) \botlabel{i} to (0.3,.4)\toplabel{k};
\magcr{0,0};
\region{.3,0}{\lambda};
\end{tikzpicture}
&\coloneq\ 
\underbrace{
\begin{tikzpicture}[centerzero,iQ,scale=1.6]
\draw[-] (0.2,-.4) \botlabel{j} to (0.2,.4);
\draw[-] (-0.2,-.4) \botlabel{i} to (-0.2,.4);
\pinpin{.2,0}{-.2,0}{-.8,0}{\frac{1}{i-j+x-y}};
\region{.45,0}{\lambda};
\end{tikzpicture}}_{\text{if }i = l,\,j = k\text{ and }i \neq j}
\ -\ \underbrace{\begin{tikzpicture}[iQ,anchorbase,scale=1.6]
\draw (.25,-.4)\botlabel{j} to[out=90,in=90,looseness=2.4] (-.25,-.4);
\draw (.25,.4)  to[out=-90,in=-90,looseness=2.4] (-.25,.4)\toplabel{l};
\internalbubbleminus[180]{-.2,-.25};
\internalbubbleplus[0]{0.2,0.25};
\pinpin{.18,-.15}{-.18,0.15}{-.91,.2}{\frac{1}{l-j+x-y}};
\region{.4,0}{\lambda};
\end{tikzpicture}}_{\text{if }i=-j,\,k=-l\text{ and }j\neq l}
\ +\underbrace{\ \begin{tikzpicture}[iQ,centerzero,scale=1.5]
\draw (-0.3,0.4)\toplabel{j} to(.3,-0.4)\botlabel{j};
\draw (0.3,0.4)\toplabel{i} to (-.3,-0.4)\botlabel{i};
\node[draw,fill=white,inner sep=1.3pt] at (0,0) {$\scriptstyle X_{i,j}$};
\region{.45,0}{\lambda};
\end{tikzpicture}\ }_{\text{if }i=k\text{ and }j=l}
\end{align}
omitting the terms on the right-hand side of \cref{crossingdef} if the indicated conditions are not satisfied, where
\begin{align}\label{hardercrossings}
\begin{tikzpicture}[iQ,centerzero,scale=1.5]
\draw (-0.3,0.4)\toplabel{j} to(.3,-0.4)\botlabel{j};
\draw (0.3,0.4)\toplabel{i} to (-.3,-0.4)\botlabel{i};
\node[draw,fill=white,inner sep=1.3pt] at (0,0) {$\scriptstyle X_{i,j}$};
\end{tikzpicture} \coloneq
\begin{dcases}
\begin{tikzpicture}[centerzero,iQ,scale=1.4]
\draw[-] (0.3,-.4) \botlabel{j} to (-0.3,.4);
\draw[-] (-0.3,-.4) \botlabel{i} to (0.3,.4);
\pinpinmid{.15,.2}{.9,0}{.15,-.2}{G^{++}_{i,j}(y,x)};
\end{tikzpicture}
+
\delta_{i,0}\delta_{j,0}\begin{tikzpicture}[iQ,centerzero,scale=1.4]
\draw[-] (0,.4)\toplabel{0} to[out=-90,in=-90,looseness=3] (.4,.4);
\draw[-] (0,-.4)\botlabel{0} to[out=90,in=90,looseness=3] (.4,-.4);
\pinpinmid{.37,.2}{1,0}{.37,-.2}{\G(x,-y)};
\end{tikzpicture}
&\text{if ($i \in I_+\cup\{0\}$ and $j \in I_+$) or $i=j=0$}\\
r_{j,i}\ 
\begin{tikzpicture}[centerzero,iQ,scale=1.6]
\draw[-] (0.3,-.4) \botlabel{j} to (-0.3,.4);
\draw[-] (-0.3,-.4) \botlabel{i} to (0.3,.4);
\internalbubbleminus[0]{-.15,.2};
\internalbubbleplus[180]{.19,-.27};
\pinpin{.07,-.1}{-.07,-.1}{-.8,-.1}{G^{-+}_{j,i}(y,x)};
\end{tikzpicture}&\text{if $i\in I_+$ and $j \in I_-\cup\{0\}$}\\
r_{i,-j}r_{j,i}\ 
\begin{tikzpicture}[centerzero,iQ,scale=1.6]
\draw[-] (0.3,-.4) \botlabel{j} to (-0.3,.4);
\draw[-] (-0.3,-.4) \botlabel{i} to (0.3,.4);
\internalbubbleminus[0]{.19,.27};
\internalbubbleplus[180]{-.19,-.27};
\pinpinmid{-.07,-.1}{-.8,0}{-.07,.1}{G^{--}_{i,j}(x,y)};
\end{tikzpicture}&\text{if $i \in I_-\cup\{0\}$ and $j \in I_-$}\\
r_{-j,-i}r_{i,-j}r_{j,i}
\begin{tikzpicture}[centerzero,iQ,scale=1.6]
\draw[-] (0.3,-.4) \botlabel{j} to (-0.3,.4);
\draw[-] (-0.3,-.4) \botlabel{i} to (0.3,.4);
\pinpin{-.07,.1}{.07,.1}{.8,.1}{G^{+-}_{j,i}(x,y)};
\fatbubbleminus{-.19,.27};
\fatbubbleplus{.15,-.2};
\end{tikzpicture}&\text{if $i\in I_-$ and $j \in I_+\cup\{0\}$.}
\end{dcases}
\end{align}
These are all 2-morphisms in $\hatUU^\imath$.
In \cref{dotdef,capdef,cupdef,crossingdef}, we have been careful to include the required iweight $\lambda$ in the rightmost 2-cell. We have not done that in \cref{hardercrossings}, but we are following our usual convention that when no 2-cell label is present we are writing an equation which is true for all possible labels.

\begin{rem}
We have written the definition of $X_{i,j}$ in the most case-free way that we can. 
When $i$ and $j$ are non-zero, some of the glyphs in the last three cases in \cref{hardercrossings} can be omitted since they are identities according to the definitions \cref{longyear1a,longyear2a}: we have that
\begin{align}
\begin{tikzpicture}[iQ,centerzero,scale=1.5]
\draw (-0.3,0.4)\toplabel{j} to(.3,-0.4)\botlabel{j};
\draw (0.3,0.4)\toplabel{i} to (-.3,-0.4)\botlabel{i};
\node[draw,fill=white,inner sep=1.3pt] at (0,0) {$\scriptstyle X_{i,j}$};
\end{tikzpicture} =
\begin{dcases}
r_{j,i}\ \begin{tikzpicture}[centerzero,iQ,scale=1.6]
\draw[-] (0.3,-.4) \botlabel{j} to (-0.3,.4);
\draw[-] (-0.3,-.4) \botlabel{i} to (0.3,.4);
\pinpin{.15,-.2}{-.15,-.2}{-.9,-.2}{G^{-+}_{j,i}(y,x)};
\end{tikzpicture}&\text{if $i\in I_+$ and $j \in I_-$}\\
r_{i,-j}r_{j,i}\ 
\begin{tikzpicture}[centerzero,iQ,scale=1.4]
\draw[-] (0.3,-.4) \botlabel{j} to (-0.3,.4);
\draw[-] (-0.3,-.4) \botlabel{i} to (0.3,.4);
\pinpinmid{-.15,-.2}{-.9,0}{-.15,.2}{G^{--}_{i,j}(x,y)};
\end{tikzpicture}&\text{if $i \in I_-$ and $j \in I_-$}\\
r_{-j,-i}r_{i,-j}r_{j,i}
\begin{tikzpicture}[centerzero,iQ,scale=1.6]
\draw[-] (0.3,-.4) \botlabel{j} to (-0.3,.4);
\draw[-] (-0.3,-.4) \botlabel{i} to (0.3,.4);
\pinpin{-.07,.1}{.07,.1}{.8,.1}{G^{+-}_{j,i}(x,y)};
\internalbubbleminus[0]{-.19,.27};
\internalbubbleplus[180]{.15,-.2};
\end{tikzpicture}&\text{if $i\in I_-$ and $j \in I_+$}.
\end{dcases}\label{simplercrossings}
\end{align}
\end{rem}

Recall from \cref{iqmonoidal} that a morphism
$\theta\colon B^{\otimes n} \rightarrow B^{\otimes m}$ in
$\hatAB$ is a matrix
$\theta=(\theta^{\mu,\lambda})_{\mu,\lambda\in X^\imath_\0}$
with $(\mu,\lambda)$-entry
$\theta^{\mu,\lambda}\in
\Hom_{\hatUU^\imath}(B_\bJ \one_\lambda,B_\bI \one_\lambda)$
for $\bI \coloneq \{\bi \in I^m\:|\: \lambda-\alpha_\bi = \mu\}$
and $\bJ\coloneq \{\bj \in I^n\:|\:\lambda-\alpha_\bj=\mu\}$.
Each $\theta^{\mu,\lambda}$ is itself a matrix
$\theta^{\mu,\lambda}=(\theta^{\mu,\lambda}_{\bi,\bj})_{\bi\in \bI, \bj \in \bJ}$
with $\theta^{\mu,\lambda}_{\bi, \bj}
\in \hatH_\lambda(\bj,\bi)$. In the following theorem, we will
denote such a matrix also by the formal sum
$\sum_{\bi \in \bI, \bj \in \bJ} \theta^{\mu,\lambda}_{\bi,\bj}$.

\begin{theo}\label{passport}
There is a strict $\kk$-linear monoidal
functor $\Psi\colon\AB \rightarrow \hatAB$
mapping the generating object $B$ of $\AB$ to the object $B$ of $\hatAB$, and defined on generating morphisms of $\AB$ so that
\begin{align}\label{gasp1}
\Psi\Big(\begin{tikzpicture}[B,centerzero]
\draw[-] (0,-0.3)-- (0,0.3);
\closeddot{0,0};
\end{tikzpicture}\Big)^{\mu,\lambda}
&\coloneq
\begin{dcases}
\begin{tikzpicture}[iQ,centerzero]
\draw[-] (0,-0.3)\botlabel{i} -- (0,0.3);
\magdot{0,0};
\region{0.4,0}{\lambda};
\end{tikzpicture}&\text{if $\mu = \lambda-\alpha_i$ for some $i \in I$}\\
0&\text{otherwise,}
\end{dcases}\\\label{gasp2}
\Psi\Big(\ \begin{tikzpicture}[B,anchorbase,scale=1.4]
\draw (0.3,-0.4) to (0.3,-.2);
\draw[-] (-0.1,-.2) to (-0.1,-0.4);
\draw (.3,-.2) to[out=90,in=90,looseness=1.5] (-.1,-.2);
\end{tikzpicture}\ \Big)^{\mu,\lambda}
&\coloneq
\delta_{\lambda,\mu}
\sum_{i\in I}
\begin{tikzpicture}[iQ,anchorbase,scale=1.4]
\draw (0.3,-0.4) to (0.3,-.2);
\draw[-] (-0.1,-.2) to (-0.1,-0.4) \botlabel{i};
\draw (.3,-.2) to[out=90,in=90,looseness=1.5] (-.1,-.2);
\magup{0.1,-.01};
\region{.5,-.2}{\lambda};
\end{tikzpicture},\\\label{gasp3}
\Psi\Big(\ 
\begin{tikzpicture}[B,anchorbase,scale=1.4]
\draw[-] (0.3,0.4) to (0.3,.2);
\draw[-] (-0.1,.2) to (-0.1,0.4);
\draw (.3,.2) to[out=-90,in=-90,looseness=1.4] (-.1,.2);
\end{tikzpicture}\ \Big)^{\mu,\lambda}
&\coloneq
\delta_{\lambda,\mu}
\sum_{i\in I}
\begin{tikzpicture}[iQ,anchorbase,scale=1.4]
\draw[-] (0.3,0.4) to (0.3,.2);
\draw[-] (-0.1,.2) to (-0.1,0.4) \toplabel{i};
\draw (.3,.2) to[out=-90,in=-90,looseness=1.5] (-.1,.2);
\magdown{0.1,0.02};
\region{.5,.2}{\lambda};
\end{tikzpicture},\\\label{gasp4}
\Psi\Big(\ 
\begin{tikzpicture}[centerzero,B,scale=1]
\draw[-] (0.3,-.4) to (-0.3,.4);
\draw[-] (-0.3,-.4) to (0.3,.4);
\end{tikzpicture}\ \Big)^{\mu,\lambda}
&\coloneq\!\!
\sum_{\substack{i,j,k,l \in I\\
\lambda-\alpha_i-\alpha_j=\mu\\\lambda-\alpha_k-\alpha_l=\mu}}
\!\!
\begin{tikzpicture}[centerzero,iQ,scale=1]
\draw[-] (0.3,-.4) \botlabel{j} to (-0.3,.4)\toplabel{l};
\draw[-] (-0.3,-.4) \botlabel{i} to (0.3,.4)\toplabel{k};
\magcr{0,0};
\region{.3,0}{\lambda};
\end{tikzpicture}
\end{align}
for $\lambda,\mu \in X^\imath_\0$.
Moreover,
\begin{equation}\label{firstthis}
\Psi\Big(\ 
\begin{tikzpicture}[B,centerzero,scale=.9]
\draw (-0.25,0) arc(180:-180:0.25);
\node at (.55,0) {\color{black}$(u)$};
\end{tikzpicture}\Big)^{\mu,\lambda}
=
\delta_{\lambda,\mu}
\prod_{i \in I}
\begin{tikzpicture}[iQ,centerzero,scale=.9]
\draw (-0.25,0) arc(180:-180:0.25);
\node at (-0.37,0) {\strandlabel{\ibar}};
\node at (.88,0) {\color{black}$(u-i)$};
\region{1.7,0}{\lambda};
\end{tikzpicture}.
\end{equation}
\end{theo}

\begin{proof}
We first explain how to deduce \cref{firstthis}.
By the definition \cref{eatz}, we have that
$\begin{tikzpicture}[B,centerzero,scale=.9]
\draw (-0.25,0) arc(180:-180:0.25);
\node at (.55,0) {\color{black}$(u)$};
\end{tikzpicture}=
\big(1-\tfrac{1}{2u}\big)\id_\one+\ 
\begin{tikzpicture}[B,anchorbase,scale=1.2]
\draw (-0.25,0) arc(180:-180:0.25);
\pin{.25,0}{.8,0}{\frac{1}{u-x}};
\end{tikzpicture}\ $,
so by \cref{dotdef,capdef,cupdef} we have that
$$
\Psi\Big(\ \begin{tikzpicture}[B,centerzero,scale=.9]
\draw (-0.25,0) arc(180:-180:0.25);
\node at (.55,0) {\color{black}$(u)$};
\end{tikzpicture}\Big)^{\mu,\lambda}
=
\delta_{\lambda,\mu} 
\big(1-\tfrac{1}{2u}\big)\id_{\one_\lambda}+
\delta_{\lambda,\mu}\ 
\sum_{i \in I}
\begin{tikzpicture}[iQ,anchorbase,scale=2]
\draw (0.3,-.05) to (0.3,.05) to [out=90,in=90,looseness=2] (-.1,.05) to (-.1,-.05) to [out=-90,in=-90,looseness=2] (.3,-.05);
\fatbubbleplus{.3,0.1};
\region{.7,.15}{\lambda};
\pin{.3,-.1}{.7,-.1}{\frac{1}{u-i-x}};
\stringlabel{-.18,0}{\ibar};
\end{tikzpicture}\ .
$$
Since $\lambda$ is even, the right-hand side here is the same as the right-hand side of
\cref{whydidittakesolong} multiplied by $\delta_{\lambda,\mu}$. So 
\cref{firstthis} follows by \cref{badbunny2}.
To complete the proof of \cref{passport},
it remains to show that all of the affine Brauer relations are satisfied.
This will be done in a series of lemmas in the remainder of the section.
\end{proof}

\begin{cor}\label{passportcor}
Any even nilpotent Abelian (resp., Karoubian) iquantum categorification $\catR$ can be made into an Abelian (resp., Karoubian) Brauer categorification in a canonical way.
\end{cor}

\begin{proof}
First apply the construction of \cref{wish}
to make $\catR$ into a 
strict left $\hatAB$-module category, then restrict the categorical action from 
$\hatAB$ to $\AB$ using the monoidal functor $\Psi$ from \cref{passport}.
\end{proof}

The remainder of the section will consist of a sequence of lemmas required to complete the proof of \cref{passport}.
We begin by establishing some properties of the power series
$G^{\s\t}_{i,j}(x,y)$.
Note that $G^{\s\t}_{i,j}(x,y)$ depends only on $\s i, \s x, \t j$ and $\t y$, hence, simultaneously
replacing $(\s,i,x)$ by $(-\s,-i,-x)$ or $(\t,j,y)$ by $(-\t,-j,-y)$
leaves $G^{\s\t}_{i,j}(x,y)$ unchanged. Also, if $\s \neq \t$ then $G^{\s\t}_{i,j}(x,y) = G^{\s\t}_{i,j}(y,x)$.

\begin{lem}\label{badone}
\begin{enumerate}
\item
For $i,j \in I_+$ with $i \neq j$, we have that
\begin{align*}
Q_{i,j}^\imath(x,y)G^{++}_{i,j}(x,y)G_{j,i}^{++}(y,x)&=1-\frac{1}{(i-j+x-y)^2}.
\end{align*}
\item 
For $i \in I_+$ and $j \in I_-$ with $i \neq -j$, we have that
\begin{align*}
r_{j,i}r_{-i,-j}r_{j,-i}r_{i,j}Q_{i,j}^\imath(x,y)
G_{i,j}^{+-}(x,y) G_{j,i}^{-+}(y,x) 
\!\!\!\prod_{k \in I-\{i\}}\!\!\frac{(i-k+x+y)^{a_{k,-j}}}{(i-k+x-y)^{a_{k,j}}}
&=1- \frac{1}{(i-j+x-y)^2}.
\end{align*}
\item For $i \in I_+$, we have that
\begin{align*}
Q_{i,-i}^\imath(x,y)
G^{+-}_{i,-i}(x,y)G^{-+}_{-i,i}(y,x)
\prod_{j \in I-\{i\}}\frac{(i-j+x+y)^{a_{j,i}}}{(i-j+x-y)^{a_{j,-i}}}=1- \frac{1}{(2i+x-y)^2}.
\end{align*}
\item
For $j \in I_+$, we have that
\begin{align*}
r_{0,j}Q_{0,j}^\imath(x,y)
G_{0,j}^{++}(x,y) G_{0,j}^{-+}(y,x)
\prod_{k \in I_\hi}\frac{(-k+x+y)^{a_{k,-j}}}{(-k+x-y)^{a_{k,j}}}
&=1- \frac{1}{(-j+x-y)^2}.
\end{align*}
\end{enumerate}
\end{lem}

\begin{proof}
(i) This follows on expanding the definition \cref{trustfunds}. It is easiest to consider the three cases  $i = j-1, i = j+1$, or $i \notin \{j \pm 1\}$ separately.

\vspace{2mm}
\noindent
(ii) The Cartan integer $a_{k,j}$ is zero unless $k\in \{j,j\pm 1\}$ and $a_{k,-j}$ is zero unless
$k \in \{-j,-j \pm 1\}$. Using this, the product term can be rewritten 
by first including the terms for all of these $k$ then excluding ones in which $k=i$ by multiplying by its inverse,
to show that
\begin{align*}
\prod_{k\in I-\{i\}}
\frac{(i-k+x+y)^{a_{k,-j}}}{(i-k+x-y)^{a_{k,j}}}
&=\Big(1-\frac{1}{(i-j+x-y)^2}\Big)
\Big(1-\frac{1}{(i+j+x+y)^2}\Big)^{-1}
\frac{(x-y)^{a_{i,j}}}{(x+y)^{a_{i,-j}}}.\\\intertext{Also \cref{trustfunds} gives that}
G^{+-}_{i,j}(x,y) G^{-+}_{j,i}(y,x)&=
t_{i,-j}^{-1}
\Big(1-\frac{1}{(i+j+x+y)^2}\Big) (x+y)^{a_{i,-j}}.\\\intertext{
Finally, by \cref{unexpectlytrue} and the definition \cref{typohere}, we have that}
r_{j,i}r_{-i,-j}r_{j,-i}r_{i,j}Q^\imath_{i,j}(x,y)
&= t_{i,-j} (x-y)^{-a_{i,j}}.
\end{align*}
The result follows on multiplying these three equations together. 

\vspace{2mm}
\noindent
(iii) Arguing as in (ii), the product term is
\begin{align*}
\prod_{j\in I-\{i\}}\frac{(i-j+x+y)^{a_{j,i}}}{(i-j+x-y)^{a_{j,-i}}}
&=\Big(1-\frac{1}{(2i+x-y)^2}\Big)\Big(1-\frac{1}{(x+y)^2}\Big)^{-1}\frac{(x-y)^{a_{i,-i}}}{(x+y)^2}.\\\intertext{Also \cref{trustfunds} gives that}
G^{+-}_{i,-i}(x,y)G^{-+}_{-i,i}(y,x)&=(1+x+y)(1-x-y)=-(x+y)^2\Big(1-\frac{1}{(x+y)^2}\Big).
\end{align*}
Finally, $Q^\imath_{i,-i}(x,y)=-t_{i,-i}(x-y)^{-a_{i,-i}}=-(x-y)^{-a_{i,-i}}$, since $t_{i,-i}=1$ for $i \in I_+$
(the only case in which $i$ and $-i$ are joined by an edge with $i \in I_+$ is $i=h$, and $t_{h,-h}=1$).
The result follows on multiplying these three equations together.

\vspace{2mm}
\noindent
(iv)
Since $I_\hi = -I_\lo$, we have that 
\begin{align*}
G^{++}_{0,j}(x,y)G^{-+}_{0,j}(y,x) &= 
\prod_{k \in I_\lo} \big(\sqrt{-1}(-k+x-y)\big)^{-a_{k,j}} \big(\sqrt{-1}(-k-x-y)\big)^{-a_{k,j}}\\
&= \prod_{k \in I_\lo} (-k+x-y)^{-a_{k,j}}
(k+x+y)^{-a_{k,j}}\\
&= \prod_{k \in I_\lo} (-k+x-y)^{-a_{k,j}}
\times \prod_{k \in I_\hi}
(-k+x+y)^{-a_{k,-j}}.
\end{align*}
Also, 
$r_{0,j}Q_{0,j}^\imath(x,y) = (x-y)^{-a_{0,j}}$
and $I = I_\hi \sqcup \{0\} \sqcup I_\lo$, so
\begin{align*}
r_{0,j}Q_{0,j}^\imath(x,y) &G^{++}_{0,j}(x,y)G^{-+}_{0,j}(y,x)
\prod_{k \in I_\hi}\!\!\frac{(-k+x+y)^{a_{k,-j}}}{(-k+x-y)^{a_{k,j}}}
= \prod_{k \in I} (-k+x-y)^{-a_{k,j}}\\
&=\frac{(-(j+1)+x-y)(-(j-1)+x-y)}{(-j+x-y)^2}=
1- \frac{1}{(-j+x-y)^2}.
\qedhere
\end{align*}
\end{proof}

\begin{lem}\label{lem:Bzigzag}
In the setup of \cref{passport},
the zig-zag identities hold:
$$
\ .
$$
\end{lem}

\begin{proof}
In view of the definitions \cref{gasp2,gasp3}, 
we must show that
$$
%
$$
for $i \in I$.
These identities are true by \cref{izigzag,bubblyinverses}.
\end{proof}

\begin{lem}
In the setup of \cref{passport},
the dot slide relations hold:
\begin{align*}
%
\ .
\end{align*}
\end{lem}

\begin{proof}
This follows using \cref{ipivotal}.
\end{proof}

\begin{lem}\label{lem:pitch}
In the setup of \cref{passport},
the pitchfork relations hold:
\begin{align*}
%
\ .
\end{align*}
\end{lem}

\begin{proof}
Since we have already checked the zig-zag relations, the proof reduces just to checking the first one of these relations, and this is equivalent to showing that
$$
%
\ .
$$
Using the definitions \cref{gasp2,gasp3,gasp4}, to verify this relation, we must show that
$$
%
$$
for all $\lambda \in X^\imath_\0$
and $i,j,k,l \in I$ such that 
$\lambda+\alpha_k-\alpha_j=\lambda-\alpha_l+\alpha_i$.
The first two terms on the right-hand side of \cref{crossingdef}
have the necessary cyclic symmetry, so the proof reduces to checking that
\begin{equation}\label{finse}
%
\end{equation}
for all $i,j \in I$.
It remains to check \cref{finse} for the various cases in the definition \cref{hardercrossings}:
\begin{itemize}
\item 
When $i \in I_+\cup\{0\}$ and $j \in I_+$,
using the pitchfork relation \cref{ipitchfork},
we have that
$$
%
\ .
$$
\item If $i \in I_+$ and $j \in I_- \cup \{0\}$
then both glyphs in \cref{finse} are identities, so we get
$$
%
\ .
\end{align*}
\item If $i \in I_-$ and $j \in I_+\cup\{0\}$
then the glyphs in \cref{finse} can both be replaced by double glyphs (the other half being trivial), so we get
\begin{align*}
%
\ .
\end{align*}
For the penultimate equality here, we used the identity
$$
r_{-j,-i}r_{i,-j}r_{j,i}
r_{-i,j}
%
\ ,
$$
which is true thanks to \cref{bubblyinverses,anothermiracle}.
\item 
Finally, for the exceptional case $i=j=0$, we have that 
$G^{++}_{0,0}(y,x)=1-x+y$, so
\begin{align*}
%
.
\qedhere
\end{align*}
\end{itemize}
\end{proof}

\begin{rem}\label{aha}
In the case $i=j=0$, the identity \cref{finse} 
established in the proof of \cref{lem:pitch} implies that there are {\em four} different ways to 
write $X_{0,0}$: we have that
\begin{align}
%
\ .\notag
\end{align}
\end{rem}

\begin{lem}\label{lem:demazure}
In the setup of \cref{passport},
the Demazure relations hold:
$$
%
\ 
.
$$
\end{lem}

\begin{proof}
The second equality follows from the first using the zig-zag and pitchfork relations already proved.
For the first, we must show that
$$
%
$$
for all $\lambda \in X^\imath_\0$ and $i,j,k,l \in I$
such that $\lambda-\alpha_i-\alpha_j=\lambda-\alpha_k-\alpha_l$.

If $i \neq k$ or $j \neq l$, by the definition of the dot and crossing, this reduces to showing that
$$
\underbrace{%
.
$$
This is clear. 

Now assume that $i=k$ and $j=l$.
Then we must show that
\begin{equation}\label{notbad}
%
.
\end{equation}
Actually, as already noted, we only need to prove the first equality for all $i,j\in I$, then the second follows by cyclic rotation. However, we will prove both equalities, preferring to use cyclic rotation to reduce to proving both equalities just in the case that 
both $i$ and $j$ are in $I_+ \cup\{0\}$. Assuming this from now on,
if $i \neq j$ then $i \neq -j$, and it is clear from the definition \cref{hardercrossings} and \cref{idemazure} that all three sides of \cref{notbad} are 0.

It remains to check \cref{notbad} when $i = j \in I_+\cup\{0\}$.
We have that
$G^{++}_{i,i}(y,x)=1-x+y$, so the left-hand side of \cref{notbad} is
\begin{align*}
\begin{tikzpicture}[centerzero,iQ,scale=1.4]
\draw[-] (-0.3,-.4) \botlabel{i} to (0.3,.4);
\draw[-] (0.3,-.4) \botlabel{i} to (-0.3,.4);
\pinpinmid{.06,.1}{.8,0}{.06,-.1}{1-x+y};
\closeddot{-.15,.2};
\end{tikzpicture}\ 
-
&\begin{tikzpicture}[centerzero,iQ,scale=1.4]
\draw[-] (-0.3,-.4) \botlabel{i} to (0.3,.4);
\draw[-] (0.3,-.4) \botlabel{i} to (-0.3,.4);
\pinpinmid{.06,.1}{.8,0}{.06,-.1}{1-x+y};
\closeddot{.15,-.2};
\end{tikzpicture}
-
\delta_{i,0} \begin{tikzpicture}[iQ,centerzero,scale=1.4]
\draw[-] (0,.4)\toplabel{0} to[out=-90,in=-90,looseness=3] (.4,.4);
\draw[-] (0,-.4)\botlabel{0} to[out=90,in=90,looseness=3] (.4,-.4);
\pinpinmid{.37,.2}{1.4,0}{.37,-.2}{(x+y)\G(x,-y)};
\end{tikzpicture}\\
&\stackrel{\cref{idemazure}}{=}
\begin{tikzpicture}[centerzero,iQ,scale=1.4]
\draw[-] (0.2,-.4) \botlabel{i} to (0.2,.4);
\draw[-] (-0.1,-.4) \botlabel{i} to (-0.1,.4);
\pinpinmid{.2,.15}{.9,0}{.2,-.15}{1-x+y};
\end{tikzpicture}
-
\delta_{i,0} \begin{tikzpicture}[iQ,centerzero,scale=1.4]
\draw[-] (0,.4)\toplabel{0} to[out=-90,in=-90,looseness=3] (.4,.4);
\draw[-] (0,-.4)\botlabel{0} to[out=90,in=90,looseness=3] (.4,-.4);
\pinpinmid{.37,.2}{1.8,0}{.37,-.2}{1-x+y+(x+y)\G(x,-y)};
\end{tikzpicture}
\stackrel{\cref{scruffiereq}}{=}
\begin{tikzpicture}[centerzero,iQ,scale=1.4]
\draw[-] (0.2,-.4) \botlabel{i} to (0.2,.4);
\draw[-] (-0.1,-.4) \botlabel{i} to (-0.1,.4);
\end{tikzpicture}-\delta_{i,0}\begin{tikzpicture}[iQ,centerzero,scale=1.2]
\draw[-] (0,.5)\toplabel{0} to[out=-90,in=-90,looseness=3] (.5,.5);
\draw[-] (0,-.5)\botlabel{0} to[out=90,in=90,looseness=3] (.5,-.5);
\internalbubbleminus[0]{.02,.28};
\internalbubbleplus[180]{.48,-.28};
\end{tikzpicture}.
\end{align*}
In view of \cref{riveting}, 
this proves the first equality in \cref{notbad}. The proof of the second equality is similar.
\end{proof}

\begin{lem}\label{lem:ri}
In the setup of \cref{passport},
the Reidemeister I relations hold:
\begin{align}
\begin{tikzpicture}[B,baseline=-3.9mm,scale=1.5]
\draw (-0.15,-0.4) to[out=60,in=-90] (0.15,0) arc(0:180:0.15) to[out=-90,in=120] (0.15,-0.4);
\end{tikzpicture}\ 
&\ =\  \begin{tikzpicture}[B,anchorbase,scale=1.8]
\draw[-] (-0.15,-0.15) -- (-0.15,0) arc(180:0:0.15) -- (0.15,-0.15);
\end{tikzpicture}\ ,
&
\raisebox{3mm}{\phantom{X}}\!\!
\begin{tikzpicture}[B,baseline=2.2mm,scale=1.5]
\draw (-0.15,0.4) to[out=-60,in=90] (0.15,0) arc(0:-180:0.15) to[out=90,in=-120] (0.15,0.4);
\end{tikzpicture}\ 
&\ =\ \begin{tikzpicture}[B,anchorbase,scale=1.8]
\draw[-] (-0.15,0.15) -- (-0.15,0) arc(-180:0:0.15) -- (0.15,0.15);
\end{tikzpicture}\ .
\end{align}
\end{lem}

\begin{proof}
It suffices to prove the first equality. The second one then follows using the zig-zag and pitchfork relations already established.
Applying \cref{gasp2,gasp4}, we need to show that
\begin{equation}\label{reduction}
\sum_{j \in I} 
\begin{tikzpicture}[iQ,baseline=-3.9mm,scale=2]
\draw (-0.15,-0.4)\botlabel{i} to[out=60,in=-90] (0.15,0) arc(0:180:0.15) to[out=-90,in=120] (0.15,-0.4);
\strand{.23,0}{j};
\magcr{0,-.24};
\magup{0,.16};
\end{tikzpicture}
\ =\  \begin{tikzpicture}[iQ,baseline=.8mm,scale=2]
\draw[-] (-0.15,-0.15)\botlabel{i} -- (-0.15,0) arc(180:0:0.15) -- (0.15,-0.15)\botlabel{\ibar};
\magup{0,.16};
\end{tikzpicture}
\end{equation}
for all $i \in I$. Now we consider two cases.

The following calculation proves \cref{reduction} when $i \neq 0$:
\begin{align*}
\sum_{j \in I} 
\begin{tikzpicture}[iQ,baseline=-3.9mm,scale=2]
\draw (-0.15,-0.4)\botlabel{i} to[out=60,in=-90] (0.15,0) arc(0:180:0.15) to[out=-90,in=120] (0.15,-0.4)\botlabel{\ibar};
\strand{.23,0}{j};
\magcr{0,-.24};
\magup{0,.16};
\end{tikzpicture}&=\begin{tikzpicture}[iQ,baseline=-3.9mm,scale=2]
\draw (-0.15,-0.4)\botlabel{i} to[out=60,in=-90] (0.15,0) arc(0:180:0.15) to[out=-90,in=120] (0.15,-0.4)\botlabel{\ibar};
\strand{.23,0}{i};
\magcr{0,-.24};
\magup{0,.16};
\end{tikzpicture}\!\!
+\!\!
\sum_{j \in I - \{i\}} 
\begin{tikzpicture}[iQ,baseline=-3.9mm,scale=2]
\draw (-0.15,-0.4)\botlabel{i} to[out=60,in=-90] (0.15,0) arc(0:180:0.15) to[out=-90,in=120] (0.15,-0.4)\botlabel{\ibar};
\strand{.23,0}{j};
\magcr{0,-.24};
\magup{0,.16};
\end{tikzpicture}
\stackrel{\cref{crossingdef}}{=}
\begin{tikzpicture}[iQ,baseline=-3.9mm,scale=2]
\draw (-0.15,-0.4)\botlabel{i} to[out=60,in=-90] (0.15,0) arc(0:180:0.15) to[out=-90,in=120] (0.15,-0.4)\botlabel{\ibar};
\strand{-.23,0}{\ibar};
\internalbubbleplus[180]{0.15,0.0};
\node[draw,fill=white,inner sep=1.3pt] at (0,-.24) {$\scriptstyle X_{i,-i}$};
\end{tikzpicture}
+
\begin{tikzpicture}[iQ,baseline=-3.9mm,scale=2]
\draw (-0.15,-0.4)\botlabel{i} to (-0.15,0) to [out=90,in=90,looseness=1.5]
(.15,0) to (.15,-.4);
\internalbubbleplus[180]{0.15,0.0};
\pinpin{.15,-.15}{-.15,-.15}{-.6,-.15}{\frac{1}{2i+x-y}};
\end{tikzpicture}
-
\sum_{j \in I - \{i\}} 
\begin{tikzpicture}[iQ,baseline=-3.9mm,scale=2]
\draw (-0.15,-0.4)\botlabel{i} to [out=90,in=90,looseness=2] (.15,-.4);
\strand{-.22,.05}{\jbar};
\pinpin{.12,-.29}{-.13,-.1}{-.7,-.1}{\frac{1}{i-j+x-y}};
\draw (0,0) circle (1.6mm);
\internalbubbleminus[180]{-.12,-.3};
\fatbubbleplus{0.14,0};
\end{tikzpicture}\\
&\!\!\!\stackrel{\cref{simplercrossings}}{=}
\begin{dcases}
\begin{tikzpicture}[iQ,baseline=-3.9mm,scale=2]
\draw (-0.15,-0.4)\botlabel{i} to[out=60,in=-90] (0.15,0) arc(0:180:0.15) to[out=-90,in=120] (0.15,-0.4);
\internalbubbleplus[180]{0.15,0.0};
\pinpin{.08,-.31}{-.08,-.31}{-.55,-.31}{1-x-y};
\end{tikzpicture}
+
\begin{tikzpicture}[iQ,baseline=-3.9mm,scale=2]
\draw (-0.15,-0.4)\botlabel{i} to (-0.15,0) to [out=90,in=90,looseness=1.5]
(.15,0) to (.15,-.4);
\internalbubbleplus[180]{0.15,-.15};
\pin{-.15,-.15}{-.6,-.15}{\frac{1}{2(i+x)}};
\end{tikzpicture}
-
\sum_{j \in I - \{i\}} 
\begin{tikzpicture}[iQ,baseline=-3.9mm,scale=2]
\draw (-0.15,-0.4)\botlabel{i} to [out=90,in=90,looseness=2] (.15,-.4);
\strand{-.23,0}{\jbar};
\draw (0,0) circle (1.6mm);
\pinpinmid{.13,-.1}{.7,-.22}{.13,-.32}{\frac{1}{i-j-x-y}};
\internalbubbleminus[180]{-.12,-.3};
\fatbubbleplus{0.12,0.1};
\end{tikzpicture}&\text{if $i \in I_+$}\\
\begin{tikzpicture}[iQ,baseline=-3.9mm,scale=2]
\draw (-0.15,-0.4)\botlabel{i} to[out=60,in=-90] (0.15,0) arc(0:180:0.15) to[out=-90,in=120] (0.15,-0.4);
\internalbubbleplus[180]{0.15,0.0};
\internalbubbleminus[0]{-0.15,0.0};
\internalbubbleplus[180]{0.11,-0.31};
\pinpin{.08,-.17}{-.08,-.17}{-.55,-.17}{1+x+y};
\end{tikzpicture} 
+
\begin{tikzpicture}[iQ,baseline=-3.9mm,scale=2]
\draw (-0.15,-0.4)\botlabel{i} to (-0.15,0) to [out=90,in=90,looseness=1.5]
(.15,0) to (.15,-.4);
\internalbubbleplus[180]{0.15,-.15};
\pin{-.15,-.15}{-.6,-.15}{\frac{1}{2(i+x)}};
\end{tikzpicture}
-
\sum_{j \in I - \{i\}} 
\begin{tikzpicture}[iQ,baseline=-3.9mm,scale=2]
\draw (-0.15,-0.4)\botlabel{i} to [out=90,in=90,looseness=2] (.15,-.4);
\strand{-.23,0}{\jbar};
\draw (0,0) circle (1.6mm);
\pinpinmid{.13,-.1}{.7,-.22}{.13,-.32}{\frac{1}{i-j-x-y}};
\internalbubbleminus[180]{-.12,-.3};
\fatbubbleplus{0.12,0.1};
\end{tikzpicture}&\text{if $i \in I_-$.}
\end{dcases}\\
&=\begin{tikzpicture}[iQ,baseline=-3.9mm,scale=2]
\draw (-0.15,-0.4)\botlabel{i} to[out=60,in=-90] (0.15,0) arc(0:180:0.15) to[out=-90,in=120] (0.15,-0.4);
\fatbubbleplus{0.15,0.0};
\internalbubbleplus[180]{0.11,-0.31};
\end{tikzpicture} 
+
\begin{tikzpicture}[iQ,baseline=-3.9mm,scale=2]
\draw (-0.15,-0.4)\botlabel{i} to (-0.15,0) to [out=90,in=90,looseness=1.5]
(.15,0) to (.15,-.4);
\internalbubbleplus[180]{0.15,-0.15};
\pin{-.15,-.15}{-.6,-.15}{\frac{1}{2(i+x)}};
\end{tikzpicture}
-
\sum_{j \in I - \{i\}} 
\begin{tikzpicture}[iQ,baseline=-3.9mm,scale=2]
\draw (-0.15,-0.4)\botlabel{i} to [out=90,in=90,looseness=2] (.15,-.4);
\strand{-.23,0}{\jbar};
\draw (0,0) circle (1.6mm);
\pinpin{.13,-.1}{-.13,-.32}{-.6,-.2}{\frac{1}{i-j+x-y}};
\internalbubbleplus[180]{.12,-.3};
\fatbubbleplus{0.12,0.1};
\end{tikzpicture} =
\begin{tikzpicture}[iQ,baseline=-3.9mm,scale=2]
\draw (-0.15,-0.4)\botlabel{i} to (-0.15,0) to [out=90,in=90,looseness=1.5]
(.15,0) to (.15,-.4);
\internalbubbleplus[180]{0.15,-0.15};
\end{tikzpicture}
=
\begin{tikzpicture}[iQ,baseline=.8mm,scale=2]
\draw[-] (-0.15,-0.15)\botlabel{i} -- (-0.15,0) arc(180:0:0.15) -- (0.15,-0.15);
\magup{0,.16};
\end{tikzpicture}\ .
\end{align*}
In the third equality here, we omitted some factors $r_{\pm i, \pm i}$ in the first term since they all equal 1. For the fourth equality, we simplified the first term, thereby reunifying the two cases: the factor $1-x-y$ commutes past the crossing, and both of $x+y$ and $-x-y$
are zero when at the top due to the dot slide relation, so the pins $1+x+y$ and $1-x-y$ can be removed. We also moved some glyphs using 
\cref{riveting}, and added a couple of extra glyphs which are identities.
For the penultimate equality, we applied the identity from \cref{bryan1} composed on the bottom right with a glyph labelled $+$, noting that $(-1)^{p(\lambda)}=1$ as $\lambda \in X^\imath_\0$, and the final equality is the definition \cref{capdef}.

In the $i=0$ case, we instead argue as follows:
\begin{align*}
\sum_{j \in I} 
\begin{tikzpicture}[iQ,baseline=-3.9mm,scale=2]
\draw (-0.15,-0.4)\botlabel{0} to[out=60,in=-90] (0.15,0) arc(0:180:0.15) to[out=-90,in=120] (0.15,-0.4)\botlabel{0};
\strand{.23,0}{j};
\magcr{0,-.24};
\magup{0,.16};
\end{tikzpicture}&=
\begin{tikzpicture}[iQ,baseline=-3.9mm,scale=2]
\draw (-0.15,-0.4)\botlabel{0} to[out=60,in=-90] (0.15,0) arc(0:180:0.15) to[out=-90,in=120] (0.15,-0.4)\botlabel{0};
\strand{.23,0}{0};
\magcr{0,-.24};
\magup{0,.16};
\end{tikzpicture}\!\!
+\!\!
\sum_{j \in I - \{0\}} 
\begin{tikzpicture}[iQ,baseline=-3.9mm,scale=2]
\draw (-0.15,-0.4)\botlabel{0} to[out=60,in=-90] (0.15,0) arc(0:180:0.15) to[out=-90,in=120] (0.15,-0.4)\botlabel{0};
\strand{.23,0}{j};
\magcr{0,-.24};
\magup{0,.16};
\end{tikzpicture}
\stackrel{\cref{crossingdef}}{=}
\begin{tikzpicture}[iQ,baseline=-3.9mm,scale=2]
\draw (-0.15,-0.4)\botlabel{0} to[out=60,in=-90] (0.15,0) arc(0:180:0.15) to[out=-90,in=120] (0.15,-0.4);
\strand{-.23,0}{0};
\internalbubbleplus[180]{0.15,0.0};
\node[draw,fill=white,inner sep=1.3pt] at (0,-.24) {$\scriptstyle X_{0,0}$};
\end{tikzpicture}
-
\sum_{j \in I - \{0\}} 
\begin{tikzpicture}[iQ,baseline=-3.9mm,scale=2]
\draw (-0.15,-0.4)\botlabel{0} to [out=90,in=90,looseness=2] (.15,-.4);
\strand{-.22,.05}{\jbar};
\pinpin{.12,-.29}{-.13,-.1}{-.7,-.1}{\frac{1}{-j+x-y}};
\draw (0,0) circle (1.6mm);
\internalbubbleminus[180]{-.12,-.3};
\fatbubbleplus{0.14,0};
\end{tikzpicture}\\
&\!\!\stackrel{\cref{twoways}}{=}
\begin{tikzpicture}[iQ,baseline=-3.9mm,scale=2]
\draw (-0.15,-0.4)\botlabel{0} to[out=60,in=-90] (0.15,0) arc(0:180:0.15) to[out=-90,in=120] (0.15,-0.4);
\internalbubbleplus[180]{0.15,0.0};
\internalbubbleminus[0]{-.15,0};
\internalbubbleplus[180]{.08,-.31};
\end{tikzpicture}
\ +\ 
\begin{tikzpicture}[iQ,centerzero,scale=1.4]
\draw[-] (0.5,-.4)\botlabel{0} to (0.5,.1) to [out=90,in=90,looseness=1] (1.5,.1) to (1.5,-.4);
\internalbubbleminus[0]{0.5,.23};
\internalbubbleplus[180]{1.5,-.28};
\internalbubbleplus[180]{1.5,.23};
\pinpinmid{0.5,-.02}{1,-.02}{1.5,-.02}{\G(x,y)};
\end{tikzpicture}
-
\sum_{j \in I - \{0\}} 
\begin{tikzpicture}[iQ,baseline=-3.9mm,scale=2]
\draw (-0.15,-0.4)\botlabel{0} to [out=90,in=90,looseness=2] (.15,-.4);
\strand{-.23,0}{\jbar};
\draw (0,0) circle (1.6mm);
\pinpinmid{.13,-.1}{.7,-.22}{.13,-.32}{\frac{1}{-j-x-y}};
\internalbubbleminus[180]{-.12,-.3};
\fatbubbleplus{0.12,0.1};
\end{tikzpicture}\\
&\!\!\substack{\cref{riveting}\\{\textstyle=}\\\cref{idemazure}}
\begin{tikzpicture}[iQ,baseline=-3.9mm,scale=2]
\draw (-0.15,-0.4)\botlabel{0} to[out=60,in=-90] (0.15,0) arc(0:180:0.15) to[out=-90,in=120] (0.15,-0.4);
\fatbubbleplus{0.15,0.0};
\internalbubbleplus[180]{.08,-.31};
\end{tikzpicture}
\ +\ 
\begin{tikzpicture}[iQ,centerzero,scale=1.4]
\draw[-] (0.5,-.4)\botlabel{0} to (0.5,.1) to [out=90,in=90,looseness=1] (1.8,.1) to (1.8,-.4);
\fatbubbleminus{0.5,.23};
\internalbubbleplus[180]{1.8,-.26};
\pinpinmid{0.5,-.22}{1.2,-.12}{.5,-.02}{\G(x,-y)};
\end{tikzpicture}
-
\sum_{j \in I - \{0\}} 
\begin{tikzpicture}[iQ,baseline=-3.9mm,scale=2]
\draw (-0.15,-0.4)\botlabel{0} to [out=90,in=90,looseness=2] (.15,-.4);
\strand{-.23,0}{\jbar};
\draw (0,0) circle (1.6mm);
\pinpin{.13,-.1}{-.13,-.32}{-.6,-.2}{\frac{1}{-j+x-y}};
\internalbubbleplus[180]{.12,-.3};
\fatbubbleplus{0.12,0.1};
\end{tikzpicture}\\&=
\begin{tikzpicture}[iQ,baseline=-3.9mm,scale=2]
\draw (-0.15,-0.4)\botlabel{0} to[out=60,in=-90] (0.15,0) arc(0:180:0.15) to[out=-90,in=120] (0.15,-0.4);
\fatbubbleplus{0.15,0.0};
\internalbubbleplus[180]{.08,-.31};
\end{tikzpicture}
\ +\ 
\begin{tikzpicture}[iQ,centerzero,scale=1.4]
\draw[-] (0.5,-.4)\botlabel{0} to (0.5,.1) to [out=90,in=90,looseness=1] (1.8,.1) to (1.8,-.4);
\pin{.5,.18}{-.2,.18}{\Deltanabla_0(-x)};
\internalbubbleplus[180]{1.8,-.26};
\pin{0.5,-.12}{1.2,-.12}{\G(x,-x)};
\end{tikzpicture}
-
\sum_{j \in I - \{0\}} 
\begin{tikzpicture}[iQ,baseline=-3.9mm,scale=2]
\draw (-0.15,-0.4)\botlabel{0} to [out=90,in=90,looseness=2] (.15,-.4);
\strand{-.23,0}{\jbar};
\draw (0,0) circle (1.6mm);
\pinpin{.13,-.1}{-.13,-.32}{-.6,-.2}{\frac{1}{-j+x-y}};
\internalbubbleplus[180]{.12,-.3};
\fatbubbleplus{0.12,0.1};
\end{tikzpicture}\ .
\end{align*}
Now we note that
$$
\G_\lambda(x,-x) = 1 + \frac{\Delta_{\lambda;0}(x)-\Delta_{\lambda;0}(-x)}{2x} \nabla_{\lambda;0}(x)
\stackrel{\cref{bibble}}{=}
1+\frac{\Deltanabla_{\lambda;0}(x)-1-2x}{2x}
=
\frac{\Deltanabla_{\lambda;0}(x)-1}{2x}
$$
for any $\lambda$.
As $\begin{tikzpicture}[iQ,centerzero]
\draw(0,-.25)\botlabel{0} to (0,.25);
\pin{0,.15}{-.9,.15}{\Deltanabla_0(-x)};
\pin{0,-.15}{.8,-.15}{\Deltanabla_0(x)};
\end{tikzpicture}=
\begin{tikzpicture}[iQ,centerzero]
\draw(0,-.25)\botlabel{0} to (0,.25);
\end{tikzpicture}$
by \cref{bubblyinverses}, 
it follows that
$$
\sum_{j \in I} 
\begin{tikzpicture}[iQ,baseline=-3.9mm,scale=2]
\draw (-0.15,-0.4)\botlabel{0} to[out=60,in=-90] (0.15,0) arc(0:180:0.15) to[out=-90,in=120] (0.15,-0.4)\botlabel{0};
\strand{.23,0}{j};
\magcr{0,-.24};
\magup{0,.16};
\end{tikzpicture}=\begin{tikzpicture}[iQ,baseline=-3.9mm,scale=2]
\draw (-0.15,-0.4)\botlabel{0} to[out=60,in=-90] (0.15,0) arc(0:180:0.15) to[out=-90,in=120] (0.15,-0.4);
\fatbubbleplus{0.15,0.0};
\internalbubbleplus[180]{.08,-.31};
\end{tikzpicture}
\ +\ 
\begin{tikzpicture}[iQ,baseline=-3.9mm,,scale=2]
\draw (-0.15,-0.4)\botlabel{0} to (-.15,-.2)  to[out=90,in=90,looseness=3] (0.15,-0.2) to (.15,-.4);
\pin{-0.15,-0.2}{-.7,-.2}{\frac{1-\Deltanabla_0(-x)}{2x}};
\internalbubbleplus[180]{.15,-0.2};
\end{tikzpicture}
-
\sum_{j \in I - \{0\}} 
\begin{tikzpicture}[iQ,baseline=-3.9mm,scale=2]
\draw (-0.15,-0.4)\botlabel{0} to [out=90,in=90,looseness=2] (.15,-.4);
\strand{-.23,0}{\jbar};
\draw (0,0) circle (1.6mm);
\pinpin{.13,-.1}{-.13,-.32}{-.6,-.2}{\frac{1}{-j+x-y}};
\internalbubbleplus[180]{.12,-.3};
\fatbubbleplus{0.12,0.1};
\end{tikzpicture}\ .
$$
It just remains to apply \cref{bryan2} like in the previous paragraph.
\end{proof}

\begin{lem}\label{lem:rii}
In the setup of \cref{passport},
the Reidemeister II relations hold:
$$
\begin{tikzpicture}[B,centerzero,scale=1.4]
\draw (0.2,-0.4) to[out=135,in=down] (-0.15,0) to[out=up,in=225] (0.2,0.4);
\draw (-0.2,-0.4) to[out=45,in=down] (0.15,0) to[out=up,in=-45] (-0.2,0.4);
\end{tikzpicture}\ 
=\ 
\begin{tikzpicture}[B,centerzero,scale=1.4]
\draw (-0.2,-0.4) -- (-0.2,0.4);
\draw (0.2,-0.4) -- (0.2,0.4);
\end{tikzpicture}\ .
$$
\end{lem}

\begin{proof}
Expanding the definition \cref{gasp4},
we need to prove that
\begin{equation}\label{todothis}
\sum_{\substack{m,n \in I\\ \lambda-\alpha_m-\alpha_n=\mu}}
\begin{tikzpicture}[iQ,centerzero]
\draw[-] (-0.2,-0.6) \botlabel{i} \braidup (0.2,0) \braidup (-0.2,0.6) \toplabel{l};
\draw[-] (0.2,-0.6) \botlabel{j} \braidup (-0.2,0) \braidup (0.2,0.6) \toplabel{k};
\magcr{0,-.31};\magcr{0,.31};
\strand{-.35,0}{m};\strand{.35,0}{n};
\region{.7,0}{\lambda};
\end{tikzpicture}
=\delta_{i,l}\delta_{j,k}
\begin{tikzpicture}[iQ,centerzero]
\draw[-] (-0.2,-0.6) \botlabel{i} -- (-0.2,0.6);
\draw[-] (0.2,-0.6) \botlabel{j} -- (0.2,0.6);
\region{.4,0}{\lambda};
\end{tikzpicture}
\end{equation}
for all $\lambda,\mu \in X^\imath_\0$
and 
$i,j,k,l \in I$
such that $\lambda-\alpha_i-\alpha_j=\mu=\lambda-\alpha_k-\alpha_l$.
Suppose first that $i \neq l$, 
so that the right-hand side of \cref{todothis} is 0.
We need to show that the left-hand side is 0 too. 
Since all of the defining relations of the category $\AB_{\operatorname{II}}$ from \cref{grs-trick}
have been verified, we can appeal to
the images in $\hatAB$ of the 
relations \cref{RIrels}. 
The first of these implies that
\begin{equation}
\sum_{\substack{m,n \in I\\ \lambda-\alpha_m-\alpha_n=\mu}}
\begin{tikzpicture}[iQ,centerzero]
\draw[-] (-0.2,-0.6) \botlabel{i} \braidup (0.2,0) \braidup (-0.2,0.6) \toplabel{l};
\draw[-] (0.2,-0.6) \botlabel{j} \braidup (-0.2,0) \braidup (0.2,0.6) \toplabel{k};
\magcr{0,-.31};\magcr{0,.31};
\strand{-.35,0}{m};\strand{.35,0}{n};
\region{.7,0}{\lambda};
\magdot{-.15,-.46};
\end{tikzpicture} 
=
\sum_{\substack{m,n \in I\\ \lambda-\alpha_m-\alpha_n=\mu}}
\begin{tikzpicture}[iQ,centerzero]
\draw[-] (-0.2,-0.6) \botlabel{i} \braidup (0.2,0) \braidup (-0.2,0.6) \toplabel{l};
\draw[-] (0.2,-0.6) \botlabel{j} \braidup (-0.2,0) \braidup (0.2,0.6) \toplabel{k};
\magcr{0,-.31};\magcr{0,.31};
\strand{-.35,0}{m};\strand{.35,0}{n};
\region{.7,0}{\lambda};
\magdot{-.15,.46};
\end{tikzpicture} 
\label{weirdargument}
\ .
\end{equation}
Hence, letting $b$
(resp., $t$)
be the endomorphism of 
$\Hom_{\hatUU^\imath}(B_i B_j \one_\lambda,
B_l B_k \one_\lambda)$
defined by 
vertical composition on the bottom with
$\begin{tikzpicture}[iQ,centerzero]
\draw[-] (0,-0.2)\botlabel{i} -- (0,0.2);
\draw[-] (0.3,-0.2)\botlabel{j} -- (0.3,0.2);
\closeddot{0,0};
\end{tikzpicture}$
(resp., vertical composition on top with 
$\begin{tikzpicture}[iQ,centerzero]
\draw[-] (0,-0.2)\botlabel{l} -- (0,0.2);
\draw[-] (0.3,-0.2)\botlabel{k} -- (0.3,0.2);
\closeddot{0,0};
\end{tikzpicture}$),
we have that $((l-i) \id + t-b)(\theta) = 0$
where
$$
\theta \coloneq
\sum_{\substack{m,n \in I\\ \lambda-\alpha_m-\alpha_n=\mu}}
\begin{tikzpicture}[iQ,centerzero]
\draw[-] (-0.2,-0.6) \botlabel{i} \braidup (0.2,0) \braidup (-0.2,0.6) \toplabel{l};
\draw[-] (0.2,-0.6) \botlabel{j} \braidup (-0.2,0) \braidup (0.2,0.6) \toplabel{k};
\magcr{0,-.31};\magcr{0,.31};
\strand{-.35,0}{m};\strand{.35,0}{n};
\region{.7,0}{\lambda};
\end{tikzpicture}.
$$
Since $t-b$ is of positive degree and $l \neq i$, 
$(l-i) \id + t-b$ 
is invertible, so this implies that $\theta = 0$.
This proves \cref{todothis} if $i \neq l$. A similar argument 
using the second relation from \cref{RIrels} proves it if $j \neq k$.
(The argument just explained is essentially the same as the argument 
in the proof of \cite[Lem.~7.22]{GRS-2026}.)

It remains to prove that
\begin{equation}\label{todothisreally}
\sum_{\substack{m,n \in I\\ \lambda-\alpha_m-\alpha_n=\mu}}
\begin{tikzpicture}[iQ,centerzero]
\draw[-] (-0.2,-0.6) \botlabel{i} \braidup (0.2,0) \braidup (-0.2,0.6) \toplabel{i};
\draw[-] (0.2,-0.6) \botlabel{j} \braidup (-0.2,0) \braidup (0.2,0.6) \toplabel{j};
\magcr{0,-.31};\magcr{0,.31};
\strand{-.35,0}{m};\strand{.35,0}{n};
\region{.7,0}{\lambda};
\end{tikzpicture} 
=
\begin{tikzpicture}[iQ,centerzero]
\draw[-] (-0.2,-0.6) \botlabel{i} -- (-0.2,0.6);
\draw[-] (0.2,-0.6) \botlabel{j} -- (0.2,0.6);
\region{.4,0}{\lambda};
\end{tikzpicture}
\end{equation}
for all $\lambda,\mu \in X^\imath_\0$
and 
$i,j \in I$ such that $\lambda-\alpha_i-\alpha_j=\mu$.
We prove this by considering various cases.
\begin{itemize}
\item
Suppose first that $i = j \neq 0$ (the easiest case).
The summation on the left-hand side of \cref{todothisreally}
has at most one non-zero term, the one in which $m=n=i$,
hence, the proof reduces to showing that
\begin{equation}\label{rii-b}
\begin{tikzpicture}[iQ,centerzero,scale=1.2]
\draw[-] (-0.2,-0.6) \botlabel{i} \braidup (0.2,0) \braidup (-0.2,0.6);
\draw[-] (0.2,-0.6) \botlabel{i} \braidup (-0.2,0) \braidup (0.2,0.6);
\node[draw,fill=white,inner sep=1.3pt] at (0,.3) {$\scriptstyle X_{i,i}$};
\node[draw,fill=white,inner sep=1.3pt] at (0,-.3) {$\scriptstyle X_{i,i}$};
\end{tikzpicture} =
\begin{tikzpicture}[iQ,centerzero,scale=1.2]
\draw[-] (-0.2,-0.6) \botlabel{i} to (-0.2,0.6);
\draw[-] (0.2,-0.6) \botlabel{i} to (0.2,0.6);
\end{tikzpicture}\ .
\end{equation}
Rotating through 180$^\circ$, using the relation \cref{finse} twice, shows that this is true for $i$ if and only if it is true for $-i$. 
So it suffices to prove \cref{rii-b} in the case that $i \in I_+$.
By the definitions \cref{hardercrossings,trustfunds}, we therefore need to show that
$$
\begin{tikzpicture}[iQ,centerzero,scale=1]
\draw[-] (-0.2,-0.6) \botlabel{i} \braidup (0.2,0) \braidup (-0.2,0.6);
\draw[-] (0.2,-0.6) \botlabel{i} \braidup (-0.2,0) \braidup (0.2,0.6);
\pinpinmid{.15,.17}{1,.3}{.15,.43}{1-x+y};
\pinpinmid{.15,-.17}{1,-.3}{.15,-.43}{1-x+y};
\end{tikzpicture} =
\begin{tikzpicture}[iQ,centerzero,scale=1.2]
\draw[-] (-0.2,-0.6) \botlabel{i} to (-0.2,0.6);
\draw[-] (0.2,-0.6) \botlabel{i} to (0.2,0.6);
\end{tikzpicture}
$$
for $i \in I_+$.
To prove this, note by \cref{idemazure} that
$$
\begin{tikzpicture}[iQ,centerzero,scale=1]
\draw[-] (-0.2,-0.3) \botlabel{i} -- (0.2,0.3);
\draw[-] (0.2,-0.3) \botlabel{i} -- (-0.2,0.3);
\pinpinmid{.12,.15}{1,0}{.12,-.15}{1-x+y};
\end{tikzpicture}=
\begin{tikzpicture}[iQ,centerzero,scale=1]
\draw[-] (-0.2,-0.3) \botlabel{i} -- (0.2,0.3);
\draw[-] (0.2,-0.3) \botlabel{i} -- (-0.2,0.3);
\pinpin{-.11,.17}{.11,.17}{1,.17}{1-x+y};
\end{tikzpicture}
+
\begin{tikzpicture}[iQ,centerzero,scale=1]
\draw[-] (-0.2,-0.3) \botlabel{i} -- (-0.2,0.3);
\draw[-] (0.2,-0.3) \botlabel{i} -- (0.2,0.3);
\end{tikzpicture}=
\begin{tikzpicture}[iQ,centerzero,scale=1]
\draw[-] (-0.2,-0.3) \botlabel{i} -- (0.2,0.3);
\draw[-] (0.2,-0.3) \botlabel{i} -- (-0.2,0.3);
\pinpin{-.11,-.12}{.11,-.12}{1,-.12}{1+x-y};
\end{tikzpicture}-
\begin{tikzpicture}[iQ,centerzero,scale=1]
\draw[-] (-0.2,-0.3) \botlabel{i} -- (-0.2,0.3);
\draw[-] (0.2,-0.3) \botlabel{i} -- (0.2,0.3);
\end{tikzpicture}\ .
$$
Since the square of a crossing of two strings of the same color 
is 0 by \cref{iquadratic}, it follows that
$$
\begin{tikzpicture}[iQ,centerzero,scale=1]
\draw[-] (-0.2,-0.6) \botlabel{i} \braidup (0.2,0) \braidup (-0.2,0.6);
\draw[-] (0.2,-0.6) \botlabel{i} \braidup (-0.2,0) \braidup (0.2,0.6);
\pinpinmid{.15,.17}{1,.3}{.15,.43}{1-x+y};
\pinpinmid{.15,-.17}{1,-.3}{.15,-.43}{1-x+y};
\end{tikzpicture} =\begin{tikzpicture}[iQ,centerzero,scale=1]
\draw[-] (-0.2,-0.6) \botlabel{i} to (0.2,0.6);
\draw[-] (0.2,-0.6) \botlabel{i} to (-0.2,0.6);
\pinpin{-.14,-.42}{.14,-.42}{1,-.42}{1+x-y};
\end{tikzpicture}-
\begin{tikzpicture}[iQ,centerzero,scale=1]
\draw[-] (-0.2,-0.6) \botlabel{i} to (0.2,0.6);
\draw[-] (0.2,-0.6) \botlabel{i} to (-0.2,0.6);
\pinpin{-.14,.42}{.14,.42}{1,.42}{1-x+y};
\end{tikzpicture}
-
\begin{tikzpicture}[iQ,centerzero,scale=1]
\draw[-] (-0.2,-0.6) \botlabel{i} to (-0.2,0.6);
\draw[-] (0.2,-0.6) \botlabel{i} to (0.2,0.6);
\end{tikzpicture} = \begin{tikzpicture}[iQ,centerzero,scale=1]
\draw[-] (-0.2,-0.6) \botlabel{i} to (-0.2,0.6);
\draw[-] (0.2,-0.6) \botlabel{i} to (0.2,0.6);
\end{tikzpicture}\ .
$$
\item
Suppose next that $i \neq \pm j$.
Then the left-hand side of \cref{todothisreally} has at most two non-zero terms, namely, the ones with $(m,n)=(i,j)$ or $(j,i)$. Expanding according to \cref{crossingdef}
and also moving the term with $(m,n)=(i,j)$ to the other side of the equality, we need to show that
\begin{equation}\label{rii-a}
\begin{tikzpicture}[iQ,centerzero,scale=1.2]
\draw[-] (-0.2,-0.6) \botlabel{i} \braidup (0.2,0) \braidup (-0.2,0.6);
\draw[-] (0.2,-0.6) \botlabel{j} \braidup (-0.2,0) \braidup (0.2,0.6);
\node[draw,fill=white,inner sep=1.3pt] at (0,.3) {$\scriptstyle X_{j,i}$};
\node[draw,fill=white,inner sep=1.3pt] at (0,-.3) {$\scriptstyle X_{i,j}$};
\end{tikzpicture}
=
\begin{tikzpicture}[iQ,centerzero,scale=1.2]
\draw[-] (-0.2,-0.6) \botlabel{i} to (-0.2,0.6);
\draw[-] (0.2,-0.6) \botlabel{j} to (0.2,0.6);
\pinpin{-.2,0}{.2,0}{1.4,0}{1-\frac{1}{(i-j+x-y)^2}};
\end{tikzpicture}\ .
\end{equation}
Rotating through 180$^\circ$ using \cref{finse} shows that this is true for $(i,j)$ if and only if it is true for $(-j,-i)$. Using this, the proof of \cref{rii-a} reduces to five sub-cases: $(i,j) \in I_+ \times I_+,
(i,j) \in I_+ \times I_-, (i,j) \in I_- \times I_+, 
(i,j) \in \{0\}\times I_+$ and $(i,j) \in I_+ \times \{0\}$. Now we go through these cases in turn, expanding the definitions of $X_{i,j}$ and $X_{j,i}$ 
using \cref{hardercrossings,simplercrossings} in each one:
\begin{itemize}
\item If $(i,j) \in I_+\times I_+$, we need to show that
$$
\begin{tikzpicture}[iQ,centerzero,scale=1.2]
\draw[-] (-0.2,-0.6) \botlabel{i} \braidup (0.2,0) \braidup (-0.2,0.6);
\draw[-] (0.2,-0.6) \botlabel{j} \braidup (-0.2,0) \braidup (0.2,0.6);
\pinpinmid{.15,.17}{1,.3}{.15,.43}{G^{++}_{j,i}(y,x)};
\pinpinmid{.15,-.17}{1,-.3}{.15,-.43}{G^{++}_{i,j}(y,x)};
\end{tikzpicture} =\begin{tikzpicture}[iQ,centerzero,scale=1.2]
\draw[-] (-0.2,-0.6) \botlabel{i} to (-0.2,0.6);
\draw[-] (0.2,-0.6) \botlabel{j} to (0.2,0.6);
\pinpin{-.2,0}{.2,0}{1.4,0}{1-\frac{1}{(i-j+x-y)^2}};
\end{tikzpicture}\ .
$$
Since $i \neq j$, dots commute past crossings, so we can
simplify the left-hand side by moving the pins to the bottom, then removing the double crossing using \cref{iquadratic}. In this way, the proof reduces to checking the algebraic identity
$$
Q_{i,j}(x,y)
G^{++}_{i,j}(x,y)
G^{++}_{j,i}(y,x) = 1- \frac{1}{(i-j+x-y)^2}.
$$
This is true by \cref{badone}(i).
\item 
If $(i,j) \in I_+ \times I_-$, we need to show
$$
r_{j,i}r_{-i,-j}r_{j,-i}r_{i,j}
\begin{tikzpicture}[iQ,centerzero,scale=1.2]
\draw[-] (-0.2,-0.6) \botlabel{i} \braidup (0.2,0) \braidup (-0.2,0.6);
\draw[-] (0.2,-0.6) \botlabel{j} \braidup (-0.2,0) \braidup (0.2,0.6);
\pinpin{-.18,-.5}{.18,-.5}{1.6,-.5}{G^{+-}_{i,j}(x,y)G^{-+}_{j,i}(y,x)};
\internalbubbleminus[0]{-.13,.45};
\internalbubbleplus[180]{.19,0.02};
\end{tikzpicture}=
\begin{tikzpicture}[iQ,centerzero,scale=1.2]
\draw[-] (-0.2,-0.6) \botlabel{i} to (-0.2,0.6);
\draw[-] (0.2,-0.6) \botlabel{j} to (0.2,0.6);
\pinpin{-.2,0}{.2,0}{1.4,0}{1-\frac{1}{(i-j+x-y)^2}};
\end{tikzpicture}\ .
$$
Here, we already moved the pins (which, in this case, are symmetric in $x$ and $y$) to the bottom.
Next we use \cref{glyphslide}(i) to slide the glyph labelled $+$ to the top, where it cancels with the glyph labelled $-$ thanks to \cref{bubblyinverses}, then expand the double crossing like in the previous sub-case. In this way, the proof reduces to showing that
$$
r_{j,i}r_{-i,-j}r_{j,-i}r_{i,j}Q_{i,j}(x,y)
G_{i,j}^{+-}(x,y)G_{j,i}^{-+}(y,x)  
\gamma_{i,j} \!\!
\prod_{k \in I-\{i\}}\!\frac{(i-k+x+y)^{a_{k,-j}}}{(i-k+x-y)^{a_{k,j}}} =1- \frac{1}{(i-j+x-y)^2}\ .
$$
Since $i \in I_+$, we have that $\gamma_{i,j}=1$, so this follows from \cref{badone}(ii).
\item 
The case $(i,j) \in I_- \times I_+$ is similar to the previous case. It reduces to checking the same identity but with $i$ and $j$ interchanged. Of course, this is also true by \cref{badone}(ii).
\item 
If $(i,j) \in \{0\}\times I_+$, we need to show
$$
r_{0,j}\ \begin{tikzpicture}[iQ,centerzero,scale=1.2]
\draw[-] (-0.2,-0.6) \botlabel{0} \braidup (0.2,0) \braidup (-0.2,0.6);
\draw[-] (0.2,-0.6) \botlabel{j} \braidup (-0.2,0) \braidup (0.2,0.6);
\pinpin{-.18,-.5}{.18,-.5}{1.6,-.5}{G_{0,j}^{++}(x,y) G_{0,j}^{-+}(y,x)};
\internalbubbleminus[0]{-.13,.43};
\internalbubbleplus[180]{.19,0.02};
\end{tikzpicture} =
\begin{tikzpicture}[iQ,centerzero,scale=1.2]
\draw[-] (-0.2,-0.6) \botlabel{0} to (-0.2,0.6);
\draw[-] (0.2,-0.6) \botlabel{j} to (0.2,0.6);
\pinpin{-.2,0}{.2,0}{1.3,0}{1-\frac{1}{(-j+x-y)^2}};
\end{tikzpicture}\ .
$$
Using \cref{glyphslide}(ii), this reduces to the identity
$$
r_{0,j}Q_{0,j}(x,y)
G_{0,j}^{++}(x,y) G_{0,j}^{-+}(y,x) 
\prod_{k \in I_\hi}\!\frac{(-k+x+y)^{a_{k,-j}}}{(-k+x-y)^{a_{k,j}}} =1- \frac{1}{(-j+x-y)^2}\ .
$$
This is \cref{badone}(iv).
\item 
The fifth sub-case 
$(i,j) \in I_+\times \{0\}$ is similar to the fourth sub-case.
\end{itemize}
\item
Finally suppose that $i=-j$ (possibly 0).
On the left-hand side of \cref{todothisreally}, we need to allow all terms in the summation 
with $m=-n$. Keeping the one with $n=i$ on the left-hand side and moving the others to the other side, we must show that
\begin{equation}
\begin{tikzpicture}[iQ,centerzero]
\draw[-] (-0.2,-0.6) \botlabel{i} \braidup (0.2,0) \braidup (-0.2,0.6) \toplabel{i};
\draw[-] (0.2,-0.6) \botlabel{\ibar} \braidup (-0.2,0) \braidup (0.2,0.6) \toplabel{\ibar};
\magcr{0,-0.3};
\magcr{0,0.3};
\strand{0.3,0}{i};
\strand{-0.3,0}{\ibar};
\end{tikzpicture}
=
\begin{tikzpicture}[iQ,centerzero]
\draw[-] (-0.2,-0.6) \botlabel{i} -- (-0.2,0.6);
\draw[-] (0.2,-0.6) \botlabel{\ibar} -- (0.2,0.6);
\end{tikzpicture}-
\sum_{j \in I - \{i\}}
\begin{tikzpicture}[iQ,centerzero]
\draw[-] (-0.2,-0.6) \botlabel{i} \braidup (0.2,0) \braidup (-0.2,0.6) \toplabel{i};
\draw[-] (0.2,-0.6) \botlabel{\ibar} \braidup (-0.2,0) \braidup (0.2,0.6) \toplabel{\ibar};
\magcr{0,-0.3};
\magcr{0,0.3};
\strand{0.3,0}{j};
\strand{-0.3,0}{\jbar};
\end{tikzpicture}
\ .\label{qual}
\end{equation}
We consider three sub-cases $i \in I_+$, $i \in I_-$ and $i = 0$:
\begin{itemize}
\item 
If $i \in I_+$ then, using \cref{crossingdef,simplercrossings} and rearranging slightly, 
the equation \cref{qual} expands to 
\begin{equation}
\begin{tikzpicture}[iQ,centerzero,scale=1.4]
\draw[-] (-0.2,-0.6) \botlabel{i} \braidup (0.2,0) \braidup (-0.2,0.6);
\draw[-] (0.2,-0.6) \botlabel{\ibar} \braidup (-0.2,0) \braidup (0.2,0.6);
\pinpin{-.18,-.5}{.18,-.5}{1.4,-.5}{G^{+-}_{i,-i}(x,y)G^{-+}_{-i,i}(y,x)};
\fatbubbleminus{-.13,.45};
\fatbubbleplus{.19,0.02};
\end{tikzpicture}\!\!\!
=
\begin{tikzpicture}[iQ,centerzero,scale=1.4]
\draw[-] (-0.2,-0.6) \botlabel{i} to (-0.2,0.6);
\draw[-] (0.2,-0.6) \botlabel{\ibar} to (0.2,0.6);
\pinpin{-.2,0}{.2,0}{1.1,0}{1-\frac{1}{(2i+x-y)^2}};
\end{tikzpicture}+
\begin{tikzpicture}[iQ,centerzero,scale=1.4]
\draw[-] (-0.2,-0.6) \botlabel{i} to (-.2,-.4) to[out=90,in=90,looseness=3] (.2,-.4) to (0.2,-0.6);
\draw[-] (-0.2,0.6) \toplabel{i}  to (-.2,.4) to[out=-90,in=-90,looseness=3] (.2,.4) to (0.2,0.6);
\fatbubbleminus{-.2,.4};
\pinpinmid{-.18,-.2}{-1.65,0}{-.17,.17}{\frac{1}{2i+x+y}\Big(\frac{1}{2i+2x}+\frac{1}{2i+2y}\Big)};
\end{tikzpicture}
\!-\! \sum_{j \in I - \{i\}}\!
\begin{tikzpicture}[iQ,centerzero,scale=1.4]
\draw[-] (-0.2,-0.6) \botlabel{i} [out=90,in=90,looseness=2.7] to (.2,-.6);
\draw[-] (-0.2,0.6) \toplabel{i} [out=-90,in=-90,looseness=2.7] to (.2,.6);
\draw (0,0) circle (.2);
\fatbubbleminus{-.2,.46};
\fatbubbleplus{.2,0};
\strand{-.31,0}{\jbar};
\pinpin{-.1,-.32}{.12,-.16}{.9,-.4}{\frac{1}{i-j+x-y}};
\pinpin{-.1,.32}{.12,.16}{.9,.4}{\frac{1}{i-j+x-y}};
\end{tikzpicture}\ .\label{quil}
\end{equation}
In this equation, we also added and removed some trivial glyphs so that we could draw double glyphs everywhere.
To prove it, we apply the identity 
from \cref{bingleysleeping} composed on the top left by a double glyph labelled $-$
and on the bottom by
$G^{+-}_{i,-i}(x,y)G^{-+}_{-i,i}(y,x)$, noting
that $\gamma_{i,-i}=1$ as $i \in I_+$. 
It follows for any $\lambda \in X^\imath_\0$ that
\begin{multline*}
\begin{tikzpicture}[iQ,centerzero,scale=1.4]
\draw[-] (-0.2,-0.6) \botlabel{i} \braidup (0.2,0) \braidup (-0.2,0.6);
\draw[-] (0.2,-0.6) \botlabel{\ibar} \braidup (-0.2,0) \braidup (0.2,0.6);
\pinpin{-.18,-.5}{.18,-.5}{1.4,-.5}{G^{+-}_{i,-i}(x,y)G^{-+}_{-i,i}(y,x)};
\fatbubbleminus{-.13,.45};
\fatbubbleplus{.19,0.02};
\region{1,0}{\lambda};
\end{tikzpicture}
=
\begin{tikzpicture}[iQ,centerzero,scale=1.4]
\draw[-] (-0.2,-0.6) \botlabel{i} to (-0.2,0.6);
\draw[-] (0.2,-0.6) \botlabel{\ibar} to (0.2,0.6);
\pinpin{-.2,-.05}{.2,-0.05}{.9,-0.05}{Q^\imath_{i,-i}(x,y)};
\pinpin{-.2,.45}{.2,.45}{1.6,.45}{\prod_{j \in I-\{i\}}\!\frac{(i-j+x+y)^{a_{j,i}}}{(i-j+x-y)^{a_{-j,i}}}};
\pinpin{-.18,-.5}{.18,-.5}{1.4,-.5}{G^{+-}_{i,-i}(x,y)G^{-+}_{-i,i}(y,x)};
\region{1.8,-.05}{\lambda};
\end{tikzpicture}
\ +\\
\begin{tikzpicture}[iQ,centerzero,scale=1.4]
\draw[-] (-0.2,-0.6) \botlabel{i} to (-.2,-.4) to[out=90,in=90,looseness=3] (.2,-.4) to (0.2,-0.6);
\draw[-] (-0.2,0.6) \toplabel{i}  to (-.2,.4) to[out=-90,in=-90,looseness=3] (.2,.4) to (0.2,0.6);
\fatbubbleminus{-.2,.4};
\pinpinmid{-.18,-.2}{-1,0}{-.16,.17}{\frac{1}{2(i+x)(i+y)}};
\region{.3,0}{\lambda};
\end{tikzpicture}
- \sum_{j \in I-\{i\}} \zeta_j(\lambda)\ \begin{tikzpicture}[anchorbase,scale=1.4,iQ]
\draw[-] (-0.2,-0.6) \botlabel{i} to (-.2,-.4) to[out=90,in=90,looseness=3] (.2,-.4) to (0.2,-0.6);
\draw[-] (-0.2,0.6) \toplabel{i}  to (-.2,.4) to[out=-90,in=-90,looseness=3] (.2,.4) to (0.2,0.6);
\fatbubbleminus{-.2,.4};
\pinpinmid{-.18,-.2}{-1.8,0}{-.16,.17}{Q_j\Big(\frac{\Deltanabla_j(x_j)}{(i-j+x-x_j)(i-j-x_j+y)}\Big)};
\region{.3,0}{\lambda};
\end{tikzpicture}.
\end{multline*}
In the second and third terms on the right-hand side here, we omitted 
the pin labelled $G^{+-}_{i,-i}(x,y)G^{-+}_{-i,i}(y,x)$
because this equals $1-(x+y)^2$, so it is the identity on the bottom cap due to the dot slide relation.
The first term is the first term on the right-hand side of \cref{quil} 
by \cref{badone}(iii),
the second term is easily seen to be the second term in \cref{quil} by adding fractions, and the third term is the desired third term too.
\item
The sub-case when $i \in I_-$ is similar.
\item In the final sub-case $i=0$, fix a choice of $\lambda \in X^\imath_\0$.
We begin by noting that 
\begin{equation}\label{newx00}
\begin{tikzpicture}[iQ,centerzero,scale=1.5]
\draw (-0.3,0.4)\toplabel{0} to(.3,-0.4)\botlabel{0};
\draw (0.3,0.4)\toplabel{0} to (-.3,-0.4)\botlabel{0};
\node[draw,fill=white,inner sep=1.3pt] at (0,0) {$\scriptstyle X_{0,0}$};
\region{.45,0}{\lambda};
\end{tikzpicture} =
\begin{tikzpicture}[centerzero,iQ,scale=1.4]
\draw[-] (0.3,-.4) \botlabel{0} to (-0.3,.4);
\draw[-] (-0.3,-.4) \botlabel{0} to (0.3,.4);
\pinpinmid{.15,-.2}{.7,0}{.15,.2}{1-x+y};
\region{.7,-.3}{\lambda};
\end{tikzpicture}
+
\begin{tikzpicture}[iQ,centerzero,scale=1.4]
\draw[-] (0,.4)\toplabel{0} to[out=-90,in=-90,looseness=3] (.4,.4);
\draw[-] (0,-.4)\botlabel{0} to[out=90,in=90,looseness=3] (.4,-.4);
\pinpinmid{.37,.2}{1,0}{.37,-.2}{\G(x,-y)};
\region{1,-.3}{\lambda};
\end{tikzpicture}
=
\Bigg(\ \begin{tikzpicture}[centerzero,iQ,scale=1.4]
\draw[-] (0.3,-.4) \botlabel{0} to (-0.3,.4);
\draw[-] (-0.3,-.4) \botlabel{0} to (0.3,.4);
\pinpinmid{-.15,.2}{.9,.2}{.15,.2}{1-x+y};
\region{.9,-.1}{\lambda};
\end{tikzpicture}
+\begin{tikzpicture}[centerzero,iQ,scale=1.4]
\draw[-] (0.1,-.4) \botlabel{0} to (0.1,.4);
\draw[-] (-0.3,-.4) \botlabel{0} to (-0.3,.4);
\region{.25,0}{\lambda};
\end{tikzpicture}\ \Bigg)
+
\begin{tikzpicture}[iQ,centerzero,scale=1.4]
\draw[-] (0,.4)\toplabel{0} to[out=-90,in=-90,looseness=3] (.4,.4);
\draw[-] (0,-.4)\botlabel{0} to[out=90,in=90,looseness=3] (.4,-.4);
\pinpinmid{.37,.2}{1.1,0}{.37,-.2}{\g(x,-y)};
\region{1,-.4}{\lambda};
\end{tikzpicture}\ ,
\end{equation}
where 
\begin{equation*}
\g(x,y) \coloneq \G(x,y)-1
\stackrel{\cref{bibble}}{=} \frac{\Delta_0(x)\nabla_0(-y)+2y-1}{x-y}.
\end{equation*}
According to \cref{qual}, we need to show that the square of this equals
\begin{align}\notag
\begin{tikzpicture}[iQ,centerzero,scale=1.2]
\draw[-] (-0.2,-0.6) \botlabel{0} to (-0.2,0.6);
\draw[-] (0.2,-0.6) \botlabel{0} to (0.2,0.6);
\region{.35,0}{\lambda};
\end{tikzpicture}
- \sum_{j \in I - \{0\}}
\begin{tikzpicture}[iQ,centerzero,scale=1.4]
\draw[-] (-0.2,-0.6) \botlabel{0} [out=90,in=90,looseness=2.7] to (.2,-.6);
\draw[-] (-0.2,0.6) \toplabel{0} [out=-90,in=-90,looseness=2.7] to (.2,.6);
\draw (0,0) circle (.2);
\internalbubbleminus[0]{-.2,.46};
\internalbubbleplus[180]{.2,-.46};
\fatbubbleplus{.2,0};
\strand{-.31,0}{\jbar};
\pinpin{-.1,-.32}{.12,-.16}{.9,-.33}{\frac{1}{-j+x-y}};
\pinpin{-.1,.32}{.12,.16}{.9,.33}{\frac{1}{-j+x-y}};
\region{.9,0}{\lambda};
\end{tikzpicture}\hspace{-10mm}&\hspace{10mm}=
\begin{tikzpicture}[iQ,centerzero,scale=1.2]
\draw[-] (-0.2,-0.6) \botlabel{0} to (-0.2,0.6);
\draw[-] (0.2,-0.6) \botlabel{0} to (0.2,0.6);
\region{.35,0}{\lambda};
\end{tikzpicture}
- \sum_{j \in I - \{0\}}
\zeta_j(\lambda)\ 
\begin{tikzpicture}[iQ,centerzero,scale=1.4]
\draw[-] (-0.2,-0.6) \botlabel{0} to (-.2,-.4) [out=90,in=90,looseness=2.7] to (.2,-.4) to (.2,-.6);
\draw[-] (-0.2,0.6) \toplabel{0} to (-.2,.4) [out=-90,in=-90,looseness=2.7] to (.2,.4) to (.2,.6);
\pinpinmid{-.19,.32}{-1.8,0}{-.19,-.33}{Q_j\Big(\frac{\Deltanabla_j(x_j)}{(-j+x-x_j)(-j-x_j+y)}\Big)};
\internalbubbleplus[0]{.2,.4};
\internalbubbleplus[180]{.2,-.4};
\region{.3,0}{\lambda};
\end{tikzpicture}\\
&=
\begin{tikzpicture}[iQ,centerzero,scale=1.2]
\draw[-] (-0.2,-0.6) \botlabel{0} to (-0.2,0.6);
\draw[-] (0.2,-0.6) \botlabel{0} to (0.2,0.6);
\region{.35,0}{\lambda};
\end{tikzpicture}
+ \sum_{j \in I - \{0\}}
\zeta_j(\lambda)\ 
\begin{tikzpicture}[iQ,centerzero,scale=1.4]
\draw[-] (-0.2,-0.6) \botlabel{0} to (-.2,-.4) [out=90,in=90,looseness=2.7] to (.2,-.4) to (.2,-.6);
\draw[-] (-0.2,0.6) \toplabel{0} to (-.2,.4) [out=-90,in=-90,looseness=2.7] to (.2,.4) to (.2,.6);
\pinpinmid{-.19,.32}{-1.6,0}{-.19,-.33}{ Q_j\Big(\partial_{x,y}\Big(\frac{\Deltanabla_j(x_j)}{-j+x-x_j}\Big)\Big)};
\internalbubbleplus[0]{.2,.4};
\internalbubbleplus[180]{.2,-.4};
\region{.3,0}{\lambda};
\end{tikzpicture}\notag\\
&\!\!\!\stackrel{\cref{louis}}{=}
\begin{tikzpicture}[iQ,centerzero,scale=1.2]
\draw[-] (-0.2,-0.6) \botlabel{0} to (-0.2,0.6);
\draw[-] (0.2,-0.6) \botlabel{0} to (0.2,0.6);
\region{.35,0}{\lambda};
\end{tikzpicture}
+
\begin{tikzpicture}[iQ,centerzero,scale=1.4]
\draw[-] (-0.2,-0.6) \botlabel{0} to (-.2,-.4) [out=90,in=90,looseness=2.7] to (.2,-.4) to (.2,-.6);
\draw[-] (-0.2,0.6) \toplabel{0} to (-.2,.4) [out=-90,in=-90,looseness=2.7] to (.2,.4) to (.2,.6);
\pinpinmid{-.19,.32}{-1.2,0}{-.19,-.33}{\partial_{x,y} \Big(\frac{1-\Deltanabla_0(x)}{2x}\Big)};
\internalbubbleplus[0]{.2,.4};
\internalbubbleplus[180]{.2,-.4};
\region{.3,0}{\lambda};
\end{tikzpicture}
\!\!+
\zeta_0(\lambda)\ 
\begin{tikzpicture}[iQ,centerzero,scale=1.4]
\draw[-] (-0.2,-0.6) \botlabel{0} to (-.2,-.4) [out=90,in=90,looseness=2.7] to (.2,-.4) to (.2,-.6);
\draw[-] (-0.2,0.6) \toplabel{0} to (-.2,.4) [out=-90,in=-90,looseness=2.7] to (.2,.4) to (.2,.6);
\pinpinmid{-.19,.32}{-1.4,0}{-.19,-.33}{ Q_0\left(\partial_{x,y}\left(\partial_{x,x_0}\Deltanabla_0(x)\right)\right)};
\internalbubbleplus[0]{.2,.4};
\internalbubbleplus[180]{.2,-.4};
\region{.3,0}{\lambda};
\end{tikzpicture}.\label{elephant}
\end{align}
Now we actually compute the square of \cref{newx00}.
It is easy to check that the bracketed term on its right-hand side squares to the identity, so we obtain
\begin{equation}\label{giraffe}
\begin{tikzpicture}[iQ,centerzero,scale=1.4]
\draw[-] (-0.2,-0.6) \botlabel{0} to (-0.2,0.6);
\draw[-] (0.2,-0.6) \botlabel{0} to (0.2,0.6);
\region{.35,0}{\lambda};
\end{tikzpicture}+
\begin{tikzpicture}[iQ,anchorbase,scale=1.4]
\draw(-.2,-.6)\botlabel{0} to [out=90,in=90,looseness=3] (.2,-.6);
\draw (-0.2,.6)\toplabel{0} to[out=-60,in=90] (0.15,0) arc(0:-180:0.15) to[out=90,in=-120] (0.2,.6);
\pinpinmid{.15,-.33}{.85,-.17}{.15,0}{\g(x,-y)};
\pinpin{-.1,.45}{.1,.45}{.7,.45}{1-x+y};
\region{.7,.15}{\lambda};
\end{tikzpicture}+
\begin{tikzpicture}[iQ,anchorbase,scale=1.4]
\draw(-.2,.6)\toplabel{0} to [out=-90,in=-90,looseness=3] (.2,.6);
\draw (-0.2,-.6)\botlabel{0} to[out=60,in=-90] (0.15,0) arc(0:180:0.15) to[out=-90,in=120] (0.2,-.6);
\pinpinmid{.15,.37}{.85,.2}{.15,0}{\g(x,-y)};
\pinpin{-.1,-.2}{.1,-.2}{.7,-.2}{1-x+y};
\region{.7,-.5}{\lambda};
\end{tikzpicture}
+2\begin{tikzpicture}[iQ,centerzero,scale=1.4]
\draw[-] (0,.6)\toplabel{0} to (0,.4) to[out=-90,in=-90,looseness=3] (.4,.4) to (.4,.6);
\draw[-] (0,-.6)\botlabel{0} to (0,-.4) to[out=90,in=90,looseness=3] (.4,-.4) to (.4,-.6);
\pinpinmid{.39,-.32}{1,0}{.39,.32}{\g(x,-y)};
\region{1,-.3}{\lambda};
\end{tikzpicture}\\+\ \begin{tikzpicture}[iQ,centerzero,scale=1.4]
\draw[-] (-0.2,-0.6) \botlabel{0} [out=90,in=90,looseness=2.7] to (.2,-.6);
\draw[-] (-0.2,0.6) \toplabel{0} [out=-90,in=-90,looseness=2.7] to (.2,.6);
\draw (0,0) circle (.2);
\strand{-.31,0}{0};
\pinpinmid{.15,.37}{.9,.22}{.15,.12}{\g(x,-y)};
\pinpinmid{.15,-.37}{.9,-.22}{.15,-.12}{\g(x,-y)};
\region{1.45,0}{\lambda};
\end{tikzpicture}\ .
\end{equation}
We need to show this is equal to \cref{elephant}.
The first terms are both identities. For the remaining four terms on the right-hand side of \cref{giraffe},
we use \cref{lastthing1,lastthing2} to eliminate the curls and simplify to see that the four terms are
\begin{align*}
\begin{tikzpicture}[iQ,anchorbase,scale=1.4]
\draw(-.2,-.6)\botlabel{0} to [out=90,in=90,looseness=3] (.2,-.6);
\draw (-0.2,.6)\toplabel{0} to[out=-60,in=90] (0.15,0) arc(0:-180:0.15) to[out=90,in=-120] (0.2,.6);
\pinpinmid{.15,-.33}{.9,-.17}{.15,0}{\g(x,-y)};
\pinpin{-.1,.45}{.1,.45}{.7,.45}{1-x+y};
\region{.7,.15}{\lambda};
\end{tikzpicture}&=\zeta_0(\lambda)\ 
\begin{tikzpicture}[iQ,centerzero,scale=1.4]
\draw[-] (-0.2,-0.6) \botlabel{0} to (-.2,-.4) [out=90,in=90,looseness=2.7] to (.2,-.4) to (.2,-.6);
\draw[-] (-0.2,0.6) \toplabel{0} to (-.2,.4) [out=-90,in=-90,looseness=2.7] to (.2,.4) to (.2,.6);
\pinpinmid{-.19,.32}{-1.75,0}{-.19,-.33}{Q_0\left((1-2y)\left(\partial_{y,x_0}\g(-x,-y)\right)\right)};
\region{.3,0}{\lambda};
\end{tikzpicture}-
\begin{tikzpicture}[iQ,centerzero,scale=1.4]
\draw[-] (-0.2,-0.6) \botlabel{0} to (-.2,-.4) [out=90,in=90,looseness=2.7] to (.2,-.4) to (.2,-.6);
\draw[-] (-0.2,0.6) \toplabel{0} to (-.2,.4) [out=-90,in=-90,looseness=2.7] to (.2,.4) to (.2,.6);
\pinpinmid{-.19,.32}{-1.6,0}{-.19,-.33}{(1-2y)\left(\partial_y
\g(-x,-y)\right)};
\region{.3,0}{\lambda};
\end{tikzpicture},\\
\begin{tikzpicture}[iQ,anchorbase,scale=1.4]
\draw(-.2,.6)\toplabel{0} to [out=-90,in=-90,looseness=3] (.2,.6);
\draw (-0.2,-.6)\botlabel{0} to[out=60,in=-90] (0.15,0) arc(0:180:0.15) to[out=-90,in=120] (0.2,-.6);
\pinpinmid{.15,.37}{.9,.2}{.15,0}{\g(x,-y)};
\pinpin{-.1,-.2}{.1,-.2}{.7,-.2}{1-x+y};
\region{.7,-.5}{\lambda};
\end{tikzpicture}&=\zeta_0(\lambda)\ 
\begin{tikzpicture}[iQ,centerzero,scale=1.4]
\draw[-] (-0.2,-0.6) \botlabel{0} to (-.2,-.4) [out=90,in=90,looseness=2.7] to (.2,-.4) to (.2,-.6);
\draw[-] (-0.2,0.6) \toplabel{0} to (-.2,.4) [out=-90,in=-90,looseness=2.7] to (.2,.4) to (.2,.6);
\pinpinmid{-.19,.32}{-1.7,0}{-.19,-.33}{
Q_0\left(\partial_{x,x_0}
\left((1+2x)\g(x,y)\right)\right)};
\region{.3,0}{\lambda};
\end{tikzpicture}-
\begin{tikzpicture}[iQ,centerzero,scale=1.4]
\draw[-] (-0.2,-0.6) \botlabel{0} to (-.2,-.4) [out=90,in=90,looseness=2.7] to (.2,-.4) to (.2,-.6);
\draw[-] (-0.2,0.6) \toplabel{0} to (-.2,.4) [out=-90,in=-90,looseness=2.7] to (.2,.4) to (.2,.6);
\pinpinmid{-.19,.32}{-1.4,0}{-.19,-.33}{\partial_x
\left((1+2x)\g(x,y)\right)};
\region{.3,0}{\lambda};
\end{tikzpicture},\\
2\begin{tikzpicture}[iQ,centerzero,scale=1.4]
\draw[-] (0,.6)\toplabel{0} to (0,.4) to[out=-90,in=-90,looseness=3] (.4,.4) to (.4,.6);
\draw[-] (0,-.6)\botlabel{0} to (0,-.4) to[out=90,in=90,looseness=3] (.4,-.4) to (.4,-.6);
\pinpinmid{.39,-.32}{.95,0}{.39,.32}{\g(x,-y)};
\region{.9,-.3}{\lambda};
\end{tikzpicture}&=2\ \begin{tikzpicture}[iQ,centerzero,scale=1.4]
\draw[-] (-0.2,-0.6) \botlabel{0} to (-.2,-.4) [out=90,in=90,looseness=2.7] to (.2,-.4) to (.2,-.6);
\draw[-] (-0.2,0.6) \toplabel{0} to (-.2,.4) [out=-90,in=-90,looseness=2.7] to (.2,.4) to (.2,.6);
\pinpinmid{-.19,.32}{-.8,0}{-.19,-.32}{\g(-x,y)};
\region{.3,0}{\lambda};
\end{tikzpicture},
\qquad\begin{tikzpicture}[iQ,centerzero,scale=1.4]
\draw[-] (-0.2,-0.6) \botlabel{0} [out=90,in=90,looseness=2.7] to (.2,-.6);
\draw[-] (-0.2,0.6) \toplabel{0} [out=-90,in=-90,looseness=2.7] to (.2,.6);
\draw (0,0) circle (.2);
\strand{-.31,0}{0};
\pinpinmid{.15,.37}{.9,.22}{.15,.12}{\g(x,-y)};
\pinpinmid{.15,-.37}{.9,-.22}{.15,-.12}{\g(x,-y)};
\region{1.45,0}{\lambda};
\end{tikzpicture}=
\zeta_0(\lambda)\ 
\begin{tikzpicture}[iQ,centerzero,scale=1.4]
\draw[-] (-0.2,-0.6) \botlabel{0} to (-.2,-.4) [out=90,in=90,looseness=2.7] to (.2,-.4) to (.2,-.6);
\draw[-] (-0.2,0.6) \toplabel{0} to (-.2,.4) [out=-90,in=-90,looseness=2.7] to (.2,.4) to (.2,.6);
\pinpinmid{-.19,.32}{-1.65,0}{-.19,-.33}{Q_0\big(\g(-x,-x_0)\g(x_0,y)\big)};
\region{.3,0}{\lambda};
\end{tikzpicture}.
\end{align*}
On the right-hand sides here, there are two types of term---ones with $Q_0$ and ones without. Comparing the two types with \cref{elephant}, 
it follows that the sum of these four terms is equal to the sum of the last two terms in \cref{elephant} providing
\begin{align*}
\Delta_0(-x)
\nabla_0(-y)\partial_{x,y} \Big(\tfrac{1-\Deltanabla_0(x)}{2x}\Big)&=-(1-2y)\left(\partial_y
\g(-x,-y)\right)-\partial_x
\left((1+2x)\g(x,y)\right)\\&\hspace{64mm}
+2\g(-x,y),\\
\Delta_0(-x)
\nabla_0(-y)
\partial_{x,y}\left(\partial_{x,x_0}
\Deltanabla_0(x)\right)&=
(1-2y)\left(\partial_{y,x_0}\g(-x,-y)\right)+\partial_{x,x_0}
\left((1+2x)\g(x,y)\right)\\&\hspace{64mm}
+\g(-x,-x_0)\g(x_0,y).
\end{align*}
Both of these equations 
may be checked by some explicit expansion using 
the identity $\Delta_0(-x)\nabla_0(x)=1+2x$ from \cref{bibble}.
\end{itemize}
\end{itemize}
\end{proof}

For the final Reidemeister III relation, we need one more preparation. Consider a morphism space $\hatH_\lambda(\bj,\bi) = \Hom_{\hatUU^\imath}(B_\bj \one_\lambda, B_\bi \one_\lambda)$ in the completed 2-iquantum group.
By \cref{arewestupid}, it has topological 
basis $\basis_\lambda(\bi\times\bj)$
as a topological right $\hatR$-module.
Let $n$ be the greatest number of crossings in any $D \in \shape(\bi\times\bj)$. Then there is a filtration
\begin{equation}\label{dotfiltration}
\{0\}=
F_{-1} \hatH_\lambda(\bj,\bi)
\subseteq 
F_0 \hatH_\lambda(\bj,\bi) \subseteq F_1 \hatH_\lambda(\bj,\bi)
\subseteq\cdots
\subseteq
F_n \hatH_\lambda(\bj,\bi) = \hatH_\lambda(\bj,\bi)
\end{equation}
such that $F_d \hatH_\lambda(\bj,\bi)$ is the closure of the $\hatR$-submodule spanned by the vectors in the basis for all reduced shapes $D \in \shape(\bi\times\bj)$ with at most $d$ crossings.
Let $\gr
\hatH_\lambda(\bj,\bi)$
be the associated graded $\hatR$-module with 
$\gr_d \hatH_\lambda(\bj,\bi)\coloneq
F_d \hatH_\lambda(\bj,\bi)/ F_{d-1} \hatH_\lambda(\bj,\bi)$.
We remind the reader that dots in string diagrams
can be commuted freely past crossings modulo terms with strictly fewer crossings. 
Moreover, as the proof of
the spanning part of \cref{nondegeneracy} 
in \cite[Lem.~5.9]{BWW}
proceeds by induction on the number of crossings,
any (not necessarily reduced) 
string diagram with at most $d$ crossings lies in $F_d H$.
In other words, 
in the associated graded, dots commute past crossings. 

\begin{lem}\label{modlower}
For any $i,j\in I$ and $\lambda \in X^\imath_\0$, we have that
\begin{equation}\label{riiilead}
\begin{tikzpicture}[iQ,centerzero,scale=1.2]
\draw (-.3,-.5)\botlabel{i} -- (.3,.5);
\draw (.3,-.5)\botlabel{j} -- (-.3,.5);
\node[draw,fill=white,inner sep=1.3pt] at (0,0)
{$\scriptstyle X_{i,j}$};
\region{.5,0}{\lambda};
\end{tikzpicture}
\ \equiv\
\begin{tikzpicture}[iQ,centerzero,scale=1.2]
\draw (-.3,-.5)\botlabel{i} -- (.3,.5)\toplabel{i};
\draw (.3,-.5)\botlabel{j} -- (-.3,.5)\toplabel{j};
\pinpin{-.2,-.35}{.2,-.35}{1.15,-.35}{q_{i,j}(x,y)};
\region{.4,0}{\lambda};
\end{tikzpicture}
\pmod{F_0 \hatH_\lambda(ij,ji)}
\end{equation}
for some $q_{i,j}(x,y) \in \kk\llbracket x,y \rrbracket^{\times}$.
\end{lem}

\begin{proof}
This follows from the definition \cref{hardercrossings}. In the last three cases, one needs to commute a glyph past the crossing, where it cancels with the other glyph. This may be done using the formulae in \cref{glyphslide}, which are valid also when $i=\pm j$ modulo a crossing-free term. Explicit formulae for $q_{i,j}(x,y)$ are not hard to work out, but we will not need to know them later. For example, if $i \in I_+ \cup \{0\}$ and 
$j \in I_+$ then
$q_{i,j}(x,y) = G_{i,j}^{++}(x,y)$, but it is more complicated
in other cases as there is an extra product term arising from the glyph slide.
\end{proof}

\begin{lem}\label{lem:riii}
In the setup of \cref{passport},
the Reidemeister III relations hold:
$$
\begin{tikzpicture}[B,centerzero,scale=1.4]
\draw (0.4,-0.4) -- (-0.4,0.4);
\draw (0,-0.4) to[out=135,in=down] (-0.32,0) to[out=up,in=225] (0,0.4);
\draw (-0.4,-0.4) -- (0.4,0.4);
\end{tikzpicture}
= \begin{tikzpicture}[B,centerzero,scale=1.4]
\draw (0.4,-0.4) -- (-0.4,0.4);
\draw (0,-0.4) to[out=45,in=down] (0.32,0) to[out=up,in=-45] (0,0.4);
\draw (-0.4,-0.4) -- (0.4,0.4);
\end{tikzpicture}\ .
$$
\end{lem}

\begin{proof}
Fix a choice of $\lambda \in X^\imath_\0$.
Expanding the definition \cref{gasp4},
we need to 
show that
\begin{equation}\label{crazybraid}
\sum_{\substack{l,l',m \in I\\
\lambda-\alpha_k-\alpha_l-\alpha_m=\mu\\
\lambda-\alpha_{i'}-\alpha_{l'}-\alpha_m=\mu}}
\begin{tikzpicture}[iQ,centerzero,scale=1.4]
\draw (0.4,-0.4)\botlabel{k} -- (-0.4,0.4)\toplabel{k'};
\draw (0,-0.4)\botlabel{j} to[out=135,in=down] (-0.32,0) to[out=up,in=225] (0,0.4)\toplabel{j'};
\draw (-0.4,-0.4)\botlabel{i} -- (0.4,0.4)\toplabel{i'};
\node at (-0.07,-0.18) {\strandlabel{l}};
\node at (-0.05,0.21) {\strandlabel{l'}};
\node at (-0.44,0) {\strandlabel{m}};
\magcr{-.21,-.22};
\magcr{-.21,.22};
\region{.4,0}{\lambda};
\magcr{0,0};
\end{tikzpicture}
= 
\sum_{\substack{l,l',m \in I\\
\lambda-\alpha_m-\alpha_l-\alpha_i=\mu\\
\lambda-\alpha_{m}-\alpha_{l'}-\alpha_{k'}=\mu}}\begin{tikzpicture}[iQ,centerzero,scale=1.4]
\draw (0.4,-0.4)\botlabel{k} -- (-0.4,0.4)\toplabel{k'};
\draw (0,-0.4)\botlabel{j} to[out=45,in=down] (0.32,0) to[out=up,in=-45] (0,0.4)\toplabel{j'};
\draw (-0.4,-0.4)\botlabel{i} -- (0.4,0.4)\toplabel{i'};
\node at (0.07,-0.18) {\strandlabel{l}};
\node at (0.05,0.21) {\strandlabel{l'}};
\node at (0.44,0) {\strandlabel{m}};
\magcr{.21,-.22};
\magcr{.21,.22};
\region{.7,0}{\lambda};
\magcr{0,0};
\end{tikzpicture}
\end{equation}
for all $\lambda,\mu \in X^\imath_\0$ and
$i,j,k,i',k',j' \in I$ such that
$\lambda-\alpha_i-\alpha_j-\alpha_k=\mu=\lambda-\alpha_{i'}-\alpha_{j'}-\alpha_{k'}$.
We have proved already that all of the defining relations of $\AB_{\operatorname{III}}$ are satisfied. 
So, like at the start of \cref{lem:rii}, we can use 
\cref{RIIrelsa,RIIrelsb,RIIrelsc} to prove the equality 
if $i \neq i', j \neq j'$ or $k \neq k'$; see also
\cite[Lem.~7.23]{GRS-2026}. Thus, the problem is reduced to showing 
for $\lambda,\mu \in X^\imath_\0$ 
and $i,j,k \in I$ with $\lambda-\alpha_i-\alpha_j-\alpha_k=\mu$
that
$$
\Theta_{i,j,k} \coloneq
\sum_{\substack{l,l',m \in I\\
\lambda-\alpha_k-\alpha_l-\alpha_m=\mu\\
\lambda-\alpha_{i}-\alpha_{l'}-\alpha_m=\mu}}
\begin{tikzpicture}[iQ,centerzero,scale=1.4]
\draw (0.4,-0.4)\botlabel{k} -- (-0.4,0.4)\toplabel{k};
\draw (0,-0.4)\botlabel{j} to[out=135,in=down] (-0.32,0) to[out=up,in=225] (0,0.4)\toplabel{j};
\draw (-0.4,-0.4)\botlabel{i} -- (0.4,0.4)\toplabel{i};
\node at (-0.07,-0.18) {\strandlabel{l}};
\node at (-0.05,0.21) {\strandlabel{l'}};
\node at (-0.44,0) {\strandlabel{m}};
\magcr{-.21,-.22};
\magcr{-.21,.22};
\region{.4,0}{\lambda};
\magcr{0,0};
\end{tikzpicture}
- \sum_{\substack{l,l',m \in I\\
\lambda-\alpha_m-\alpha_l-\alpha_i=\mu\\
\lambda-\alpha_{m}-\alpha_{l'}-\alpha_{k}=\mu}}\begin{tikzpicture}[iQ,centerzero,scale=1.4]
\draw (0.4,-0.4)\botlabel{k} -- (-0.4,0.4)\toplabel{k};
\draw (0,-0.4)\botlabel{j} to[out=45,in=down] (0.32,0) to[out=up,in=-45] (0,0.4)\toplabel{j};
\draw (-0.4,-0.4)\botlabel{i} -- (0.4,0.4)\toplabel{i};
\node at (0.07,-0.18) {\strandlabel{l}};
\node at (0.05,0.21) {\strandlabel{l'}};
\node at (0.44,0) {\strandlabel{m}};
\magcr{.21,-.22};
\magcr{.21,.22};
\region{.7,0}{\lambda};
\magcr{0,0};
\end{tikzpicture}
$$
equals 0.

Suppose for a contradiction that $\Theta_{i,j,k} \neq 0$ in
$\hatH_\lambda(ijk,kji)$.
We filter $\hatH_\lambda(ijk,kji)$ by number of crossings according to \cref{dotfiltration}.
Since dots commute with crossings modulo terms with fewer crossings, the definition \cref{crossingdef} and \cref{modlower} imply that
$$
\Theta_{i,j,k} \equiv
\begin{tikzpicture}[iQ,centerzero,scale=1.4]
\draw (0.4,-0.4)\botlabel{k} -- (-0.4,0.4)\toplabel{k};
\draw (0,-0.4)\botlabel{j} to[out=135,in=down] (-0.32,0) to[out=up,in=225] (0,0.4)\toplabel{j};
\draw (-0.4,-0.4)\botlabel{i} -- (0.4,0.4)\toplabel{i};
\pinpinpin{.3,-.3}{-.1,-.3}{-.3,-.3}{-1.8,-.3}{q_{j,k}(y,z)q_{i,k}(x,z)q_{i,j}(x,y)};
\region{.4,0}{\lambda};
\end{tikzpicture}
- \begin{tikzpicture}[iQ,centerzero,scale=1.4]
\draw (0.4,-0.4)\botlabel{k} -- (-0.4,0.4)\toplabel{k};
\draw (0,-0.4)\botlabel{j} to[out=45,in=down] (0.32,0) to[out=up,in=-45] (0,0.4)\toplabel{j};
\draw (-0.4,-0.4)\botlabel{i} -- (0.4,0.4)\toplabel{i};
\pinpinpin{.3,-.3}{.1,-.3}{-.3,-.3}{-1.8,-.3}{q_{i,j}(x,y)q_{i,k}(x,z)q_{j,k}(y,z)};
\region{.6,0}{\lambda};
\end{tikzpicture}\pmod{F_2 \hatH_\lambda(ijk,kji)}.
$$
The relation \cref{ibraid} implies that the braid relation holds modulo $F_2 \hatH_\lambda(ijk,kji)$, so this is 0.
Thus, we see that 
$\Theta_{i,j,k} \in F_2 \hatH_\lambda(ijk,kji)$, so it is an infinite
$\hatR$-linear combination
$$
\Theta_{i,j,k} = 
\sum_{D, e_1,e_2,e_3 \geq 0}  c_{D;e_1,e_2,e_3}
D(x_1^{e_1} x_2^{e_2} x_3^{e_3}) 
\begin{tikzpicture}[centerzero]
\region{0,0}{\lambda};
\end{tikzpicture}
$$ 
summing over reduced shapes $D \in \shape((k,j,i)\times (i,j,k))$ with at most two crossings.
Let $D \in \shape((k,j,i)\times (i,j,k))$ be a reduced shape with the greatest number of crossings such that
$c_{D;e_1,e_2,e_3} \neq 0$ for some
$e_1,e_2,e_3 \geq 0$.
Let $d$ be the number of crossings in $D$.
Since $d < 3$,
we can find $r \in \{1,2,3\}$
such that the $r$th point on the bottom boundary lies on a different string of $D$ to the $(4-r)$th point on the top boundary. 
Let $b$ be multiplication by a black dot at the $r$th point on the bottom boundary and $t$ be multiplication by a black
dot at the $(4-r)$th point on the top boundary.
The linear map $b-t\colon \hatH_\lambda(ijk,kji) \rightarrow
\hatH_\lambda(ijk,kji)$
induces a linear map
$$
\gr_d(b-t)\colon \gr_d \hatH_\lambda(ijk,kji) \rightarrow \gr_d \hatH_\lambda(ijk,kji)
$$
which preserves each shape component.
For the shape $D$, the operators $b$ and $t$ act up to signs by multiplication by two different string variables
$x_s$ and $x_{s'}$.
So $\gr_d(b-t)$ is multiplication by $\pm x_s \pm x_{s'}$
for some signs, which is injective
on the completed polynomial algebra $\hatR\llbracket x_1,x_2,x_3\rrbracket$ in variables corresponding to the strings of the shape $D$.
Hence the non-zero
$D$-component of $\Theta_{i,j,k}$ remains non-zero.
So we have that 
$(b-t)(\Theta_{i,j,k}) \neq 0$.
However, the relations \cref{RIIrelsa,RIIrelsb,RIIrelsc} imply
that $(b-t)(\Theta_{i,j,k}) = 0$; the scalars from \cref{dotdef} cancel since the label at the $r$th point on the bottom boundary equals the label at the $(4-r)$th point on the top boundary.
This is a contradiction.
\end{proof}

\setcounter{section}{8}

\section{Brauer to iquantum}\label{bccomplete}

In \cref{passport}, we constructed a strict $\kk$-linear monoidal functor
\begin{equation}\label{Psi}
\Psi\colon\AB \rightarrow \hatAB.
\end{equation}
In particular, this defines linear maps
\begin{equation}\label{Psimn}
\psi_{m,n}\colon\Hom_{\AB}(B^{\otimes n}, B^{\otimes m})
\rightarrow \Hom_{\hatAB}(B^{\otimes n},B^{\otimes m})
\end{equation}
for each $m,n \geq 0$.
In this section, we study the functor 
$\Psi$ 
further, first showing that it is faithful.
Since $\hatAB$ is enriched in topological vector spaces, 
there is an induced topology on the morphism spaces of 
$\AB$ with respect to which $\Psi$ is continuous, i.e., each $\psi_{m,n}$ is continuous.
Then we show that $\Psi$ is topologically dense, i.e.,
each $\psi_{m,n}$ has image that is a dense subset of the topological space
$\Hom_{\hatAB}(B^{\otimes n},B^{\otimes m})$. 
Hence, $\hatAB$ can be interpreted as a completion of $\AB$, justifying our notation. Finally, we use this general framework to prove \cref{thmb} from the introduction.

We begin with some notation that will be used throughout the section.
Recall that
\begin{align*}
Z&=\End_{\AB}(\one),
&
\hatZ &= \End_{\hatAB}(\one),
&
Z_\lambda &= \End_{\UU^\imath}(\one_\lambda),&
\hatZ_\lambda &= \End_{\hatUU^\imath}(\one_\lambda).
\end{align*}
Also, by \cref{hatZ}, we have that 
$\hatZ=\prod_{\lambda \in X^\imath_\0}
\widehat{Z}_\lambda$ with
the $\lambda$th projection
of $z \in \hatZ$ being 
$z^{\lambda,\lambda} \in \hatZ_\lambda$.
For $\lambda \in X^\imath_\0$, let
\begin{align}\label{psidef}
\psi_\lambda
\colon Z &\rightarrow \hatZ_\lambda,&
z &\mapsto \psi_{0,0}(z)^{\lambda,\lambda}.
\end{align}
We denote the Brauer bubble generating function
$\begin{tikzpicture}[B,centerzero,scale=.9]
\draw (-0.25,0) arc(180:-180:0.25);
\node at (.55,0) {\color{black}$(u)$};
\end{tikzpicture}\!\!$ from \cref{eatz}
by $O(u) \in \id_\one +u^{-1}Z\llbracket u^{-1}\rrbracket$. As in \cref{likey}, the iquantum bubble generating function is
$O_{\lambda;i}(u) \in \hatZ_\lambda\lround u^{-1}\rround$.
In this section, it will also be helpful to have available the renormalized
generating function
\begin{equation}\label{thiscrazyfunction}
z_{\lambda;i}(u) \coloneq \varphi_\lambda(z_i(u)) = 
u^{\lambda_i-\del_i}
O_{\lambda;i}(u)
\end{equation}
for $\varphi$ as in \cref{flambda}.
This has the property that
\begin{equation}\label{anotherwaytoseeit}
\prod_{i \in I} O_{\lambda;i}(u-i) = \chi_\lambda(u)
\prod_{i \in I} z_{\lambda;i}(u-i)
\end{equation}
where $\chi_\lambda(u)$ is the rational function from \cref{maxideals}.
Hence, using also
\cref{firstthis}, 
we have that
\begin{equation}\label{psiO}
\psi_\lambda(O(u)) = \prod_{i \in I} O_{\lambda;i}(u-i)
= \chi_\lambda(u) \prod_{i \in I} z_{\lambda;i}(u-i)
\end{equation}
in $\hatZ_\lambda\llbracket u^{-1}\rrbracket$,
where $\chi_\lambda(u) \in \kk(u)$ is the rational function from \cref{maxideals}.

Now fix $m,n \geq 0$ with $m \equiv n\pmod{2}$ (required for
$\Hom_{\AB}(B^{\otimes n}, B^{\otimes m})$ to be non-zero),
and set $l \coloneq \frac{m+n}{2}$.
Recall the set $\shape(m\times n)$ of reduced shapes
fixed before the statement of \cref{RSbasis}.
Also there, for $D \in \shape(m\times n)$,
we introduced the notation $D(f)$
for the morphism in $\Hom_{\AB}(B^{\otimes n}, B^{\otimes m})$
obtained by pinning $f \in \kk[x_1,\dots,x_l]$ 
to the distinguished points
$p_1,\dots,p_l$ of $D$. 
We extend this notation to define $D(f)$ for any 
$f \in Z[x_1,\dots,x_l]$ by declaring that the pinhead 
labelled $f$ is located in the rightmost 2-cell.
By \cref{RSbasis}, any $\theta \in \Hom_{\AB}(B^{\otimes n}, B^{\otimes m})$ can be written as
\begin{equation}\label{exp1}
\theta = \sum_{D \in \shape(m\times n)}
D(\pr_D(\theta))
\end{equation}
for unique $\pr_D(\theta) \in Z[x_1,\dots,x_l]$.
This defines linear maps
\begin{equation}\label{pr1}
\pr_D:\Hom_{\AB}(B^{\otimes n}, B^{\otimes m}) \rightarrow
Z[x_1,\dots,x_l]
\end{equation}
for each $D \in \shape(m\times n)$.

Similarly, on the iquantum side, we upgrade the notation
\cref{clorific} to allow any
$f \in \hatZ_\lambda\llbracket x_1,\dots,x_l\rrbracket$: for
$D \in \shape(\bi\times \bj)$ and $\lambda \in X^\imath$
as there, $D(f)\ \text{\regionlabel{\lambda}}$ means
the 2-morphism in $\hatH_\lambda(\bj,\bi)$
obtained by pinning $f$ to the distinguished points 
of $D\ \text{\regionlabel{\lambda}}$ in such a way that the pinhead is located in the rightmost 2-cell.
By \cref{arewestupid}, any $\theta \in \hatH_\lambda(\bj,\bi)$
can be written as
\begin{equation}\label{exp2}
\theta = \sum_{D\in \shape(\bi\times \bj)} D(\pr_{\lambda;D}(\theta))\ \text{\regionlabel{\lambda}}
\end{equation}
for unique $\pr_{\lambda;D}(\theta)
\in \hatZ_\lambda\llbracket x_1,\dots,x_l\rrbracket$.
This defines right
$\hatZ_\lambda$-module homomorphisms
\begin{equation}\label{pr2}
\cf_{\lambda;D}\colon \hatH_\lambda\big(\bj, \bi\big) \rightarrow \hatZ_\lambda\llbracket x_1,\dots,x_l\rrbracket.
\end{equation}
Moreover, by \cref{arewestupid,dumbanddumber},
$\cf_{\lambda;D}$ is continuous with respect to the 
GC topology on the domain, and the topology
on $\hatZ_\lambda\llbracket x_1,\dots,x_l\rrbracket = \prod_{e_1,\dots,e_l \geq 0} x_1^{e_1} \cdots x_l^{e_l}\hatZ_\lambda$ that is the product of the GC topologies on each factor.

Returning to the situation that 
$D \in \shape(m\times n)$ and $l \coloneq \frac{m+n}{2}$,
a {\em coloring} of $D$ simply means a tuple $\bc \in I^l$.
We write $D_\bc$ for the string diagram obtained from $D$ by changing all of its magenta strings to black strings, then labelling the strings by $c_1,\dots,c_l$ so that the string label at the $k$th distinguished point $p_k$ is $c_k$,
using our usual convention that string labels change sign at critical points.
Let $\bi^{D_\bc} \in I^m$ and $\bj^{D_\bc} \in I^n$ be the words
labelling the top and bottom boundary points of the strings in $D_\bc$. 
For any $\lambda \in X^\imath_\0$,
we have that 
\begin{equation}\label{nastymu}
\mu \coloneq 
\lambda-\alpha_{\bi^{D_\bc}} = \lambda-\alpha_{\bj^{D_\bc}}.
\end{equation}
The diagram
$D_\bc$ is a reduced $\bi^{D_\bc} \times \bj^{D_\bc}$-shape
in the sense explained just before \cref{clorific}.
Since \cref{nondegeneracy} holds for any choice of 
$\shape(\bi^{D_\bc}\times \bj^{D_\bc})$ and
any distinguished points, we can assume that 
$D_\bc$ is one of the fixed choices of reduced shapes in $\shape\big(\bi^{D_\bc}\times \bj^{D_\bc}\big)$, and that the distinguished points $p_1,\dots,p_l$ of $D_\bc$ on the iquantum side
are at exactly the same positions as the distinguished points of $D$ on the Brauer side.
Let $\#(D)$ denote the number of crossings in $D$,
and recall the filtration of
$\hatH_\lambda(\bj^{D_\bc}, \bi^{D_\bc})$ 
by number of crossings from \cref{dotfiltration}.

\begin{lem}\label{leading}
Let $D \in \shape(m\times n)$ and $l \coloneq \frac{m+n}{2}$, 
$\bc \in I^l$ be a coloring, and $\lambda,\mu \in X_\0^\imath$
be as in \cref{nastymu}.
For any $f \in Z[x_1,\dots,x_l]$,
we have that $\psi_{m,n}(D(f))^{\mu,\lambda}_{\bi^{D_\bc},\bj^{D_\bc}}
\in F_{\#(D)}
\hatH_\lambda(\bj^{D_\bc}, \bi^{D_\bc})$, and
$$
\cf_{\lambda;D_\bc}\big(\psi_{m,n}(D(f))^{\mu,\lambda}_{\bi^{D_\bc},\bj^{D_\bc}}\big) =
u \psi_\lambda\big(f(c_1+x_1,\dots,c_l+x_l)\big)
$$
for a unit
$u \in \hatZ_\lambda\llbracket x_1,\dots,x_l\rrbracket^\times$ depending 
only on $D, \bc$ and $\lambda$,
where $\psi_\lambda(f(c_1+x_1,\dots,c_l+x_l))$
means the element of $\hatZ_\lambda\llbracket x_1,\dots,x_l\rrbracket$ obtained from 
the shifted polynomial
$f(c_1+x_1,\dots,c_l+x_l)$
by applying $\psi_\lambda\colon Z \rightarrow\hatZ_\lambda$ to each of its coefficients.
Moreover,
$$
\cf_{\lambda;D_\bc}\big(\psi_{m,n}(D'(f'))^{\mu,\lambda}_{\bi^{D_\bc},\bj^{D_\bc}}\big) =
0
$$
for $f' \in Z[x_1,\dots,x_l]$ and
$D' \in \shape(m\times n)$
with $D' \neq D$ and $\#(D') \leq \#(D)$.
\end{lem}

\begin{proof}
By linearity, it 
suffices to consider $f = x_1^{e_1} \cdots x_l^{e_l} z$
for $e_1,\dots,e_l \geq 0$ and $z \in Z$.
Since $\Psi$ is monoidal, we have that $\psi_{m,n}\big(D(x_1^{e_1}\cdots x_l^{e_l} z)\big)= \psi_{m,n}\big(D(x_1^{e_1}\cdots x_l^{e_l})\big)\otimes 
\psi_{0,0}(z)$, so
$$
\psi_{m,n}\big(D(x_1^{e_1}\cdots x_l^{e_l} z)\big)^{\mu,\lambda}_{\bi^{D_\bc},\bj^{D_\bc}}
=
\psi_{m,n}\big(D(x_1^{e_1}\cdots x_l^{e_l} )\big)^{\mu,\lambda}_{\bi^{D_\bc},\bj^{D_\bc}} \psi_\lambda(z).
$$
In view of this, we may assume that $z=1$.

The morphism $D(x_1^{e_1}\cdots x_l^{e_l})$ can be written as a vertical composition of the generators---dots,  caps, cups and crossings---of $\AB$ tensored on the left and right by some identity morphisms.
To compute
$\psi_{m,n}\big(D(x_1^{e_1}\cdots x_l^{e_l})\big)^{\mu,\lambda}_{\bi^{D_\bc}, \bj^{D_\bc}}$,
we 
sum over all labellings of the segments of $D$ (the pieces of strings between consecutive generators) by elements of $I$ which are compatible with the boundary labellings $\bi^{D_\bc}$ at the top and $\bj^{D_\bc}$ at the bottom, with the generators themselves
replaced by their images under $\Psi$
from \cref{gasp1,gasp2,gasp3,gasp4} (see \cref{crazybraid} for an example of such an expansion). The secondary definitions \cref{gasp1,gasp2,gasp3,gasp4} are then further expanded using the primary definitions \cref{dotdef,capdef,cupdef,crossingdef}, the last producing a sum of at most two 
terms for each crossing.

Consider first the terms in this expansion in which the last term of \cref{crossingdef}, the one involving $X_{i,j}$, is chosen at every crossing.
The image of a cap or cup has the same label, up to sign, at its endpoints, and $X_{i,j}$ only appears when the labels at the two ends of each strand of the crossing agree.
So in such a term the labelling $\bc$ is constant along each string of $D$, up to sign changes at critical points; that is, 
the labelling of segments is determined completely by $\bc$, and there is exactly one such term.
In it, by \cref{dotdef}, 
the $e_k$ dots at the point $p_k$
have been replaced by the pin $(c_k+x)^{e_k}$; 
each cap and each cup carries the glyph from \cref{capdef} or \cref{cupdef}, which is a pin labelled by an invertible power series due to \cref{bibble}; 
and each crossing has been replaced by some $X_{i,j}$, 
which by \cref{modlower} is a black crossing with the invertible power series $q_{i,j}(x,y)$ pinned below the crossing, plus some terms without the crossing.

In all of the other terms, at least one crossing of $D$ has been removed, 
so these terms have at most $\#(D)-1$ crossings.
Since $F_{\#(D)-1}\hatH_\lambda\big(\bj^{D_\bc},\bi^{D_\bc}\big)$
is closed and vertical composition is continuous, we deduce that
$\psi_{m,n}\big(D(x_1^{e_1}\cdots x_l^{e_l})\big)^{\mu,\lambda}_{\bi^{D_\bc},\bj^{D_\bc}}$
lies in $F_{\#(D)}\hatH_\lambda\big(\bj^{D_\bc},\bi^{D_\bc}\big)$.
Moreover,
modulo lower terms in the crossing filtration, i.e., terms in
$F_{\#(D)-1}\hatH_\lambda\big(\bj^{D_\bc},\bi^{D_\bc}\big)$,
it is congruent to $D_\bc\ \text{\regionlabel{\lambda}}$
decorated with $(c_k+x)^{e_k}$ pinned at $p_k$, 
an invertible glyph near each cup and cap, and an invertible power series $q_{i,j}(x,y)$ near each crossing.

Now we move all glyph coefficients from interior 2-cells to the rightmost 2-cell using \cref{slideright1,slideright2}
and the definitions \cref{quicka,quickb}. 
This produces additional pins of invertible power series. 
Then we move the pointy ends of the pins 
to the distinguished points,
which produces some sign changes over critical points but, otherwise, does not change the diagram modulo lower terms
in the crossing filtration.
In this way, we obtain
$D_\bc\big(u (c_1+x_1)^{e_1}\cdots (c_l+x_l)^{e_l}\big)\ \text{\regionlabel{\lambda}}$ modulo lower terms in the crossing filtration
for some unit $u \in \hatZ_\lambda\llbracket x_1,\dots,x_l\rrbracket^\times$.
Applying $\cf_{\lambda;D_\bc}$ proves the first assertion.

For the second assertion, take $D' \in \shape(m\times n)$
with $D'\neq D$ and $\#(D') \leq \#(D)$
and $e'_1,\dots,e_l' \geq 0$.
Then we expand $\psi_{m,n}\big(D'(x_1^{e_1'}\cdots x_l^{e_l'})\big)^{\mu,\lambda}_{\bi^{D_\bc},\bj^{D_\bc}}$ 
in the same way as above.
All of the terms lie in $F_{\#(D)-1}\hatH_\lambda\big(\bj^{D_\bc}, \bi^{D_\bc}\big)$
except if $D'$ has exactly $\#(D)$ crossings, when there could be a term in which $X_{i,j}$ is chosen at every crossing.
This term is present only if there is a coloring $\bc'$ of $D'$
such that the boundary colors of $D'_{\bc'}$ agree with
$\bi^{D_\bc}$ at the top 
and $\bj^{D_\bc}$ at the bottom.
By the same sliding argument, if present,
this term is congruent to $D'_{\bc'}(g)\ \text{\regionlabel{\lambda}}$ modulo lower terms in the crossing filtration
for some $g \in \hatZ_\lambda\llbracket x_1,\dots,x_l\rrbracket$. The image of this 
under $\cf_{\lambda;D_\bc}$ is zero as $D'\neq D$.
\end{proof}

\begin{lem}\label{psi0in}
The homomorphism $\psi_0\colon Z \rightarrow \hatZ_0$ 
(that is, $\psi_\lambda\colon Z \rightarrow \hatZ_\lambda$
in the case $\lambda=0$) is injective.
\end{lem}

\begin{proof}
Let $A$ be the graded algebra that is the 
quotient of $Z_0$ by the homogeneous ideal generated by all $z_{0;i}^{(r)}$ for $i \neq 0$ and $r \geq 1$, and let
$a_r$ be the image of $z_{0;0}^{(r)}$ in $A$. Setting $a(u) \coloneq 1 + \sum_{r \geq 1} a_r u^{-r}$,
we have that $a(u)a(-u) = 1$.
Since $Z_0 \cong R$, $A$ is freely generated by $a_1,a_3,a_5,\dots$; cf. the sentence after \cref{flambda}.
Equating $u^{-2s}$-coefficients in $a(u)a(-u)=1$ gives
$2 a_{2s} = -\sum_{t=1}^{2s-1}(-1)^t a_t a_{2s-t}$. As $p \neq 2$, it follows by induction on $s$ that
$a_{2s} \in \kk[a_1,a_3,\dots,a_{2s-1}]$.

The quotient map $Z_0\twoheadrightarrow A$ maps the GC neighborhood $U_{0;d,I-\{0\}}$ into $A_{\geq d}$. 
Hence, it extends to a continuous homomorphism
$\pi\colon\hatZ_0\rightarrow
\widehat{A}$, where $\widehat{A}$ is the completion of $A$ with respect to the grading.
We have that
$\pi(O_{0;0}(u)) = u^{-1} a(u)$,
$\pi(O_{0;h}(u-h)) = u-h$,
and $\pi(O_{0;i}(u-i)) = 1$ for $i \in I - \{0,h\}$.
So \cref{psiO} gives that
$$
\pi\big(\psi_0(O(u))\big) = (1-hu^{-1})a(u)
= 1+ (a_1-h)u^{-1} + (a_2-h a_1) u^{-2}+\cdots.
$$
On the other hand, by \cref{eatz},
$$
O(u) = \frac{u-h}{u}+\sum_{r\geq 0}z_r u^{-r-1}
=1+ (z_0-h) u^{-1} + z_1 u^{-2}+z_2 u^{-3}+\cdots
$$
where
$z_r$ is the bubble with $r$ dots from \cref{brauerbubbles}.
Comparing coefficients, we deduce that
\begin{align*}
\pi(\psi_0(z_0)) &= a_1,&
\pi(\psi_0(z_r)) &= a_{r+1} - h a_r\quad(r \geq 1).
\end{align*}
Hence,
$\pi(\psi_0(z_{2s})) \in a_{2s+1} + \kk[a_1,a_3,\dots,a_{2s-1}]$
for all $s \geq 0$.
Since $Z$ is spanned by the monomials in $z_0,z_2,z_4,\dots$ (see \cref{easypart}),
this triangularity shows that $\pi\circ\psi_0$ maps these monomials to a basis of $A$. So it maps $Z$ isomorphically onto $A$. In particular, $\psi_0$ is injective.
\end{proof}

\begin{lem}\label{faithful}
The functor $\Psi$ is faithful.
\end{lem}

\begin{proof}
Take $m,n \geq 0$ with $m \equiv n \pmod{2}$
and set $l \coloneq \frac{m+n}{2}$.
Let $\theta \in \Hom_{\AB}(B^{\otimes n},B^{\otimes m})$ be non-zero. Using the easy part of \cref{RSbasis} (see \cref{easypart}),
we can write
$$
\theta = \sum_{D \in \shape(m\times n)} D(\theta_D(x_1,\dots,x_l))
$$ 
for some 
$\theta_D(x_1,\dots,x_l) \in Z[x_1,\dots,x_l]$.
Choose $D \in \shape(m\times n)$ with the greatest number of crossings 
among all shapes such that $\theta_D(x_1,\dots,x_l) \neq 0$.
Then we apply \cref{leading}, taking 
the coloring $\bc$ to be $\bzero \coloneq (0,\dots,0)$
and the iweight $\lambda$ to be 0.
As every $D' \neq D$ with $\theta_{D'}(x_1,\dots,x_l)\neq 0$ has at most $\#(D)$ crossings, we get that
$$
\cf_{0;D_\bzero}\Big(\psi_{m,n}(\theta)^{0,0}_{\bi^{D_\bzero},\bj^{D_\bzero}}\Big) =
u \psi_0\big(\theta_D(x_1,\dots,x_l)\big)
$$
for some unit $u \in \hatZ_0\llbracket x_1,\dots,x_l\rrbracket^\times$.
It remains to show that $\psi_0(\theta_D(x_1,\dots,x_l))\neq 0$.
As this 
is obtained from $\theta_D(x_1,\dots,x_l)$ 
simply by applying $\psi_0$ to its coefficients,
this follows from \cref{psi0in}.
\end{proof}

Henceforth, we identify
$\Hom_{\AB}(B^{\otimes n}, B^{\otimes m})$ with a subspace
of the complete topological vector space
$\Hom_{\hatAB}(B^{\otimes n},B^{\otimes m})$ using the injective 
map $\psi_{m,n}$.
In particular, this identifies $Z$ with a subalgebra of $\hatZ$.

\begin{lem}\label{zdense}
The subalgebra $Z$ is dense in the complete topological algebra $\widehat{Z}$.
\end{lem}

\begin{proof}
The topology on $\hatZ = \prod_{\lambda \in X^\imath_\0} \hatZ_\lambda$ is the product of the GC topologies on the factors.
Since $Z_\lambda \cong R$ as topological algebras, \cref{verysmall} implies that
the neighborhoods $U_{\lambda;d,J}$ with $J = -J$
form a system of neighborhoods of 0 in $Z_\lambda$.
In view of this, 
it suffices to show that the homomorphism
\begin{align}\label{wanttosurject}
\pi:Z &\rightarrow \prod_{\lambda\in \Lambda} \hatZ_\lambda/\overline{U}_{\lambda;d_\lambda,J_\lambda},&
z &\mapsto \big(\psi_\lambda(z) + \overline{U}_{\lambda;d_\lambda,J_\lambda}\big)_{\lambda\in \Lambda}
\end{align}
is surjective
for any finite set $\Lambda$ of even iweights, 
$d_\lambda\geq 1$ and 
cofinite subsets $J_\lambda = -J_\lambda \subseteq I$ for each $\lambda \in \Lambda$.
This follows by the two claims below:

\vspace{2mm}
\noindent
{\underline{Claim 1}:} {\em The homomorphism $\pi$ is surjective when $\Lambda = \{\lambda\}$ is a singleton.}

\vspace{1mm}
\noindent
Proof:
Write $d$ and $J$ for $d_\lambda$ and $J_\lambda$, and let $A \coloneq \hatZ_\lambda /\overline{U}_{\lambda;d,J}$. Let $K \coloneq I - J$, which is finite and closed under $\pm$, and choose $I' \subset I$ such that $I = I' \sqcup\{0\}\sqcup(-I')$. 
Let $K' \coloneq K \cap I'$.
As $Z_\lambda \cong R$, the description of $R$ after \cref{flambda} implies that $Z_\lambda / U_{\lambda;J}$
is a polynomial algebra 
freely generated by the images of $z_{\lambda;0}^{(2r-1)}$ (if $0 \in K$) and $z_{\lambda;i}^{(r)}$ ($i \in K'$) for $r \geq 1$, which are all of positive degree.
Hence, 
$A=\hatZ_\lambda/\overline{U}_{\lambda;d,J}
\cong Z_\lambda/U_{\lambda;d,J}$ 
 is finite-dimensional and local, its unique maximal ideal $\mathfrak{m}$ is the image of $Z_{\lambda;\geq 1}$, which is nilpotent,
and $\mathfrak{m}/\mathfrak{m}^2$ has a basis given by the images of the elements of the finite set
$$
\Omega \coloneq \big\{z_{\lambda;0}^{(2r-1)}\:\big|\:0 \in K, 2(2r-1)<d\big\}\cup
\big\{z_{\lambda;i}^{(r)}\:\big|\:i \in K', 2r<d\big\}.
$$
Let $S \subseteq A$ be the image of $Z$.
To show that $S = A$, it suffices to show that $S+\mathfrak{m}^2 \supseteq \mathfrak{m}$.
Indeed, assuming this, we get that $\mathfrak{m} \subseteq S + \mathfrak{m}^2 \subseteq S + (S\cap \mathfrak{m}+\mathfrak{m}^2)^2 \subseteq S + \mathfrak{m}^3 \subseteq\cdots$, and $\mathfrak{m}^r = 0$ for $r \gg 0$.

Let $z(u) \in A\lround u^{-1}\rround$ be the image of $\psi_\lambda(O(u))$.
Its coefficients belong to $S$.
By \cref{psiO}, $z(u)$ is the image of 
$\chi_\lambda(u)
\prod_{i\in K} z_{\lambda;i}(u-i)$, a finite product.
To compute this modulo $\mathfrak{m}^2$, note that $\prod_{i\in K}z_{\lambda;i}(u-i) \equiv 1 + \sum_{i\in K}\big(z_{\lambda;i}(u-i)-1\big)$, and
that $z_{\lambda;-i}(u) = z_{\lambda;i}(-u)^{-1} \equiv 1 - \big(z_{\lambda;i}(-u)-1\big)$; in particular, $2 z_{\lambda;0}^{(2r)} \in \mathfrak{m}^2$, hence $z_{\lambda;0}^{(2r)} \in \mathfrak{m}^2$ as $p\neq 2$.
Remembering that generators of degree at least $d$ are zero in $A$, we deduce that
$$
z(u) \equiv \chi_\lambda(u)\Big(1+\sum_{z \in \Omega} z\,R_z(u)\Big)\pmod{\mathfrak{m}^2\lround u^{-1}\rround},
$$
where $R_z(u) \coloneq u^{-r}$ if $z = z_{\lambda;0}^{(r)}$ and $R_z(u) \coloneq (u-i)^{-r} - (-1)^r(u+i)^{-r}$ if $z = z_{\lambda;i}^{(r)}$ for $i \in K'$.
The rational functions $R_z(u)$ for $z\in \Omega$ are linearly independent: as $K'$, $\{0\}$ and $(-K')$
are pairwise disjoint, the principal parts of 
$\sum_{z\in \Omega}c_z R_z(u)$ at $0$ and at $i\in K'$ are $\sum_{r} c_{z_{\lambda;0}^{(r)}}u^{-r}$ and $\sum_r c_{z_{\lambda;i}^{(r)}}(u-i)^{-r}$, respectively.
Hence the rational functions $\chi_\lambda(u)R_z(u)$ for $z \in \Omega$ are linearly independent too, and so are their expansions in $\kk\lround u^{-1}\rround$.
So there are integers $n_1,\dots,n_{|\Omega|}$ such that the matrix $\big([\chi_\lambda(u) R_z(u)]_{u:n_t}\big)_{z\in \Omega, 1 \leq t\leq|\Omega|}$ is invertible.
Since $[z(u)]_{u:n_t} - [\chi_\lambda(u)]_{u:n_t} \equiv \sum_{z\in \Omega}[\chi_\lambda(u)R_z(u)]_{u:n_t}\,z \pmod{\mathfrak{m}^2}$,
it follows that the elements $[z(u)]_{u:n_t}-[\chi_\lambda(u)]_{u:n_t}$ of $S$ span $\mathfrak{m}$ modulo $\mathfrak{m}^2$.

\vspace{2mm}
\noindent
{\underline{Claim 2}:} {\em The homomorphism $\pi$ is surjective in general.}

\vspace{1mm}
\noindent
Proof:
For $\lambda \in \Lambda$, let $K_\lambda$ be the kernel of the homomorphism $\pi_\lambda:Z \rightarrow A_\lambda \coloneq \hatZ_\lambda/\overline{U}_{\lambda;d_\lambda,J_\lambda}$ obtained by composing $\pi$ with the $\lambda$th projection. This
is surjective by Claim 1.
Let $M_\lambda \supseteq K_\lambda$ be the kernel of the composition of $\pi_\lambda$ with the projection of the local algebra $A_\lambda$ onto its residue field $\kk$.
As all $z_{\lambda;i}^{(n)}$ with $n\geq 1$ map to zero in $\kk$, \cref{psiO} shows that this composition maps $O(u)$ to $\chi_\lambda(u)$.
The rational functions $\chi_\lambda(u)$ for $\lambda\in\Lambda$ are distinct by \cref{lasers}, so the maximal ideals $M_\lambda$ ($\lambda \in \Lambda$) of $Z$ are distinct.
As the unique maximal ideal of $A_\lambda$ is nilpotent, we have that $M_\lambda^r \subseteq K_\lambda$ for $r\gg 0$,
so the ideals $K_\lambda$ are pairwise comaximal too.
The Chinese Remainder Theorem now gives that $Z \rightarrow \prod_{\lambda\in\Lambda} Z/K_\lambda \cong \prod_{\lambda\in\Lambda} A_\lambda$ is surjective.
\end{proof}

Since vertical and horizontal composition 
in $\hatAB$ are continuous (\cref{conv1,conv2}),
there is a monoidal subcategory
$\overline{\AB}$ of $\hatAB$
with objects $B^{\otimes n}$ for all $n \geq 0$,
and morphism spaces
$$
\Hom_{\overline{\AB}}(B^{\otimes n}, B^{\otimes m})
\coloneq
\overline{\Hom_{\AB}(B^{\otimes n},B^{\otimes m})}
$$
for $m,n \geq 0$, i.e., the closure
of $\Hom_{\AB}(B^{\otimes n},B^{\otimes m})$
in $\Hom_{\hatAB}(B^{\otimes n},B^{\otimes m})$.

\begin{theo}\label{btoq-thm}
$\overline{\AB}=\hatAB$.
\end{theo}

\begin{proof}
For $\lambda,\mu \in X^\imath_\0$, $\bi \in I^m, \bj \in I^n$ with $\mu\coloneq 
\lambda-\alpha_\bi = \lambda-\alpha_\bj$, and any $\theta \in \Hom_{\hatUU^\imath}(B_\bj\one_\lambda,B_\bi\one_\lambda)$,
let $\iota(\theta) \in \Hom_{\hatAB}(B^{\otimes n},B^{\otimes m})$ be the morphism with $\iota(\theta)^{\mu,\lambda}_{\bi,\bj}=\theta$,
all other components being 0.
It follows from \cref{iqmonoidal} that $\iota(\theta)\circ\iota(\phi) = \iota(\theta \circ \phi)$
when the vertical composition 
$\theta\circ\phi$ makes sense, and
$\iota(\theta)\otimes\iota(\phi) = \iota(\theta\phi)$ when 
the horizontal composition $\theta\phi$ makes sense.
We denote $\iota(\id_{\one_\lambda})\in\hatZ$ by $e_\lambda$, and $\iota(\id_{B_i\one_\lambda}) \in \End_{\hatAB}(B)$ by $e_{\lambda;i}$.
By \cref{tricky}, the subspace of $\Hom_{\hatAB}(B^{\otimes n},B^{\otimes m})$ spanned by $\iota(\theta)$ for all
$\theta \in \Hom_{\UU^\imath}(B_\bj\one_\lambda,B_\bi\one_\lambda)$
and all $\lambda,\mu, \bi,\bj$ is dense.
Also \cref{zdense} shows that
$\End_{\overline{\AB}}(\one) = \widehat{Z}$, 
hence, it contains 
$e_\lambda$ and $\iota(z)$ for all $z \in \hatZ_\lambda$.
These observations reduce the proof to showing that 
the images under $\iota$ of the generating 2-morphisms of $\UU^\imath$ belong to the appropriate morphism space of $\overline{\AB}$. We go through these shortly, after 
checking that $\overline{\AB}$ contains the idempotents
$e_{\lambda;i}$.

\vspace{2mm}
\noindent
{\underline{Idempotents}:} {\em 
For $i \in I$ and $\lambda \in X_\0^\imath$, 
$\End_{\overline{\AB}}(B)$ contains the idempotent
$e_{\lambda;i}$.}

\vspace{1mm}
\noindent
Proof:
By the definition of tensor product in \cref{iqmonoidal}, we have that $\id_B\otimes e_\lambda = \sum_{j\in I}e_{\lambda;j}$.
By \cref{hatIZ}, the closed subalgebra $(\id_B\otimes e_\lambda)\circ\End_{\hatAB}(B)\circ(\id_B\otimes e_\lambda)$
is identified with $\prod_{j\in I}\End_{\hatUU^\imath}(B_j\one_\lambda)$ with the product topology.
By \cref{gasp1,dotdef},
we have that 
\begin{equation}\label{prep}
\Psi\Big(\begin{tikzpicture}[B,centerzero,scale=.8]
\draw[-] (0,-0.3)-- (0,0.3);
\closeddot{0,0};
\end{tikzpicture}\Big)\circ(\id_B\otimes e_\lambda)=\Big(\;
\begin{tikzpicture}[iQ,centerzero,scale=.8]
\draw[-] (0,-0.3)\botlabel{j}-- (0,0.3);
\pin{0,0}{-.9,0}{j+x};
\region{.3,0}{\lambda};
\end{tikzpicture}
\Big)_{j \in I}.
\end{equation}
Now fix $i\in I$.
By the discussion before \cref{hatIZ},
any neighborhood of 0 in 
$\prod_{j\in I}\End_{\hatUU^\imath}(B_j\one_\lambda)$
contains
$$
N_{F,d} \coloneq
\prod_{j \in I} N_j
\qquad\text{where}\qquad 
N_j\coloneq \begin{cases}
\overline{\End_{\UU^\imath}(B_j \one_\lambda)_{\geq 2d}}&\text{if $j \in F$}\\
\End_{\hatUU^\imath}(B_j \one_\lambda)&\text{if $j \notin F$,}
\end{cases}
$$
for suitable finite $F \subset I$ and $d \geq 0$.
By the Chinese Remainder Theorem, there is a polynomial $f(x) \in \kk[x]$ with $f(x) \equiv \delta_{i,j}\pmod{(x-j)^d}$ 
for every $j \in F$.
Then $\Psi\Big(\begin{tikzpicture}[B,centerzero,scale=.8]
\draw[-] (0,-0.3)-- (0,0.3);
\pin{0,0}{.8,0}{f(x)};
\end{tikzpicture}\;\Big)\circ(\id_B\otimes e_\lambda)$
has $j$th component 
$\begin{tikzpicture}[iQ,centerzero,scale=.8]
\draw[-] (0,-0.3)\botlabel{j}-- (0,0.3);
\pin{0,0}{-1,0}{f(j+x)};
\region{.3,0}{\lambda};
\end{tikzpicture}
\equiv \delta_{i,j}\begin{tikzpicture}[iQ,centerzero,scale=.8]
\draw[-] (0,-0.3)\botlabel{j}-- (0,0.3);
\region{.3,0}{\lambda};
\end{tikzpicture} =\delta_{i,j} e_{\lambda;j} \pmod{N_j}$
for $j \in F$,
so
$\begin{tikzpicture}[iQ,centerzero,scale=.8]
\draw[-] (0,-0.3)\botlabel{j}-- (0,0.3);
\pin{0,0}{-1,0}{f(j+x)};
\region{.3,0}{\lambda};
\end{tikzpicture} \equiv e_{\lambda;i}\pmod{N_{F,d}}$.
As we have already established that
$e_\lambda \in \End_{\overline{\AB}}(\one)$,
this shows that $e_{\lambda;i}\in \End_{\overline{\AB}}(B)$. 

\vspace{2mm}
\noindent
{\underline{Dots}:} {\em 
For $i \in I$ and $\lambda \in X^\imath_\0$,
$\End_{\overline{\AB}}(B)$
contains
$\iota\Big(\,\begin{tikzpicture}[iQ,centerzero,scale=.8]
\draw[-] (0,-0.3)\botlabel{i}-- (0,0.3);
\closeddot{0,0};
\region{.35,0}{\lambda};
\end{tikzpicture}\!\Big)$.
More generally, it contains 
$\iota\Big(\ \begin{tikzpicture}[iQ,centerzero,scale=.8]
\draw[-] (0,-0.3)\botlabel{i}-- (0,0.3);
\pin{0,0}{.8,0}{f(x)};
\region{1.5,0}{\lambda};
\end{tikzpicture}\Big)$
and
$\iota\Big(\begin{tikzpicture}[iQ,centerzero,scale=.8]
\draw[-] (0,-0.3)\botlabel{i}-- (0,0.3);
\pin{0,0}{-.8,0}{f(x)};
\region{-1.5,0}{\lambda};
\end{tikzpicture}\;\Big)$
for any $f(x) \in \hatZ_\lambda\llbracket x \rrbracket$. Even more generally, the analogous statement is true for power series in more than one variable pinned to dots on several parallel strings.}

\vspace{1mm}
\noindent
Proof:
By \cref{prep}, we have that
\begin{equation}
\iota\Big(\ \begin{tikzpicture}[iQ,centerzero,scale=.8]
\draw[-] (0,-0.3)\botlabel{i}-- (0,0.3);
\closeddot{0,0};
\region{.35,0}{\lambda};
\end{tikzpicture}\Big) = 
\Psi\Big(\ \begin{tikzpicture}[B,centerzero,scale=.8]
\draw[-] (0,-0.3)-- (0,0.3);
\pin{0,0}{-.8,0}{x-i};
\end{tikzpicture}\ \Big)\circ e_{\lambda;i},
\end{equation}
which is in $\End_{\overline{\AB}}(B)$ as we have already shown
that $e_{\lambda;i}\in \End_{\overline{\AB}}(B)$.
The power series $f(x) = \sum_{r \geq 0} f_r x^r
\in \hatZ_\lambda\llbracket x\rrbracket$ pinned to the right of the string is in $\End_{\overline{\AB}}(B)$
because we have that
\begin{equation}
\iota\Big(\begin{tikzpicture}[iQ,centerzero,scale=.8]
\draw[-] (0,-0.3)\botlabel{i}-- (0,0.3);
\pin{0,0}{.8,0}{f(x)};
\region{1.5,0}{\lambda};
\end{tikzpicture}\Big) = \sum_{r\geq 0}\iota\Big(\begin{tikzpicture}[iQ,centerzero,scale=.8]
\draw[-] (0,-0.3)\botlabel{i}-- (0,0.3);
\multcloseddot{0,0}{east}{r};
\region{.35,0}{\lambda};
\end{tikzpicture}\Big)
\circ(\id_B\otimes\iota(f_r))\circ e_{\lambda;i},
\end{equation}
which converges because the dot with multiplicity $r$ is of degree $2r$. Pins with the pinhead on the left are treated similarly.

\vspace{2mm}
\noindent
{\underline{Caps and cups}:} {\em 
For $i \in I$ and $\lambda \in X^\imath_\0$,
$\Hom_{\overline{\AB}}(B^{\otimes 2}, \one)$
contains
$\iota\Big(\ \begin{tikzpicture}[iQ,anchorbase]
\draw[-] (0,0)\botlabel{i} to [out=90,in=90,looseness=3] (.6,0);
\region{.8,.3}{\lambda};
\end{tikzpicture}\Big)$ and
$\Hom_{\overline{\AB}}(\one,B^{\otimes 2})$ contains
$\iota\Big(\ \begin{tikzpicture}[iQ,anchorbase]
\draw[-] (0,0)\toplabel{i} to [out=-90,in=-90,looseness=3] (.6,0);
\region{.8,-.3}{\lambda};
\end{tikzpicture}\Big)$.}

\vspace{1mm}
\noindent
Proof:
By \cref{longyear1a,capdef,cupdef} and the invertibility established in \cref{bibble}, we have that
\begin{align}
\begin{tikzpicture}[iQ,baseline=3mm,scale=1.1]
\draw[-] (0,0)\botlabel{i} to [out=90,in=90,looseness=3] (.6,0);
\region{.8,.3}{\lambda};
\end{tikzpicture}
&=
\begin{tikzpicture}[iQ,anchorbase,scale=1.6]
\draw (0.3,-0.5) to (0.3,-.2);
\draw[-] (-0.1,-.2) to (-0.1,-0.5) \botlabel{i};
\draw (.3,-.2) to[out=90,in=90,looseness=1.5] (-.1,-.2);
\magup{0.1,-0.01};
\pin{.3,-.32}{1.1,-.32}{\Delta_{-i}(x)^{-1}};
\region{.5,-.1}{\lambda};
\end{tikzpicture}\ ,&
\begin{tikzpicture}[iQ,baseline=-4.5mm,scale=1.1]
\draw[-] (0,0)\toplabel{i} to [out=-90,in=-90,looseness=3] (.6,0);
\region{.75,-.2}{\lambda};
\end{tikzpicture}
&=
\begin{tikzpicture}[iQ,anchorbase,scale=1.6]
\draw[-] (0.3,0.5) to (0.3,.2);
\draw[-] (-0.1,.2) to (-0.1,0.5) \toplabel{i};
\draw (.3,.2) to[out=-90,in=-90,looseness=1.5] (-.1,.2);
\magdown{0.1,.02};
\pin{.3,.32}{1.1,.32}{\nabla_{-i}(x)^{-1}};
\region{.5,.13}{\lambda};
\end{tikzpicture}\ .
\end{align}
By \cref{gasp2}, $\iota\Big(\ \begin{tikzpicture}[iQ,anchorbase,scale=1.1]
\draw[-] (0,0)\botlabel{i} to [out=90,in=90,looseness=3] (.6,0);
\magup{.3,.53};
\region{.8,.35}{\lambda};
\end{tikzpicture}\Big)=e_\lambda\circ\Psi\Big(\ \begin{tikzpicture}[B,baseline=2mm,scale=1.1]
\draw[-] (0,0) to [out=90,in=90,looseness=3] (.6,0);
\end{tikzpicture}\ \Big)\circ\big(e_{\lambda-\alpha_{-i};i}\otimes e_{\lambda;-i}\big)$
which is in $\Hom_{\overline{\AB}}(B^{\otimes 2}, \one)$,
and
$\iota\Big(\ \begin{tikzpicture}[iQ,anchorbase,scale=1.6]
\draw (0.3,-0.5)\botlabel{\ibar} to (0.3,-.1);
\draw[-] (0,-.1) to (0,-0.5) \botlabel{i};
\pin{.3,-.32}{1,-.32}{\Delta_{-i}(x)^{-1}};
\region{1.6,-.3}{\lambda};
\end{tikzpicture}\Big)$ is in $\End_{\overline{\AB}}(B^{\otimes 2})$,
by the results 
established already.
The cup is treated in the same way using \cref{gasp3}.

\vspace{2mm}
\noindent
{\underline{Crossings}:} {\em 
For $i,j \in I$ and $\lambda \in X^\imath_\0$,
$\End_{\overline{\AB}}(B^{\otimes 2})$
contains
$\iota\Big(\ \begin{tikzpicture}[iQ,anchorbase,scale=1.7]
\draw[-] (0,0)\botlabel{i} to (.3,.4);
\draw[-] (.3,0)\botlabel{j} to (0,.4);
\region{.4,.2}{\lambda};
\end{tikzpicture}\Big)$.}

\vspace{1mm}
\noindent
Proof:
Take $i,j \in I$.
The only term on the right-hand side of \cref{crossingdef} which survives when the bottom labels are $i$ and $j$ and the top labels are $j$ and $i$
is the one involving $X_{i,j}$.
So, by \cref{gasp4}, 
\begin{equation}    
\iota\bigg(\ \begin{tikzpicture}[iQ,centerzero,scale=1.1]
\draw (-.3,-.4)\botlabel{i} -- (.3,.4);
\draw (.3,-.4)\botlabel{j} -- (-.3,.4);
\node[draw,fill=white,inner sep=1pt] at (0,0) {$\scriptscriptstyle X_{i,j}$};
\region{.5,0}{\lambda};
\end{tikzpicture}\bigg)
= \big(e_{\lambda-\alpha_i;j}\otimes e_{\lambda;i}\big)\circ
\Psi\Big(\ \begin{tikzpicture}[B,centerzero,scale=.8]
\draw[-] (0.3,-.4) to (-0.3,.4);
\draw[-] (-0.3,-.4) to (0.3,.4);
\end{tikzpicture}\ \Big)\circ\big(e_{\lambda-\alpha_j;i}\otimes e_{\lambda;j}\big),
\end{equation}
which is in $\End_{\overline{\AB}}(B^{\otimes 2})$ by the results established already.
By \cref{modlower}, we have that
$$
\begin{tikzpicture}[iQ,centerzero,scale=1.2]
\draw (-.3,-.4)\botlabel{i} -- (.3,.4);
\draw (.3,-.4)\botlabel{j} -- (-.3,.4);
\region{.45,0}{\lambda};
\end{tikzpicture}
\ \equiv\
\begin{tikzpicture}[iQ,centerzero,scale=1.2]
\draw (-.3,-.4)\botlabel{i} -- (.3,.4);
\draw (.3,-.4)\botlabel{j} -- (-.3,.4);
\node[draw,fill=white,inner sep=1.3pt] at (0,0)
{$\scriptstyle X_{i,j}$};
\pinpin{-.22,-.3}{.22,-.3}{1.25,-.3}{q_{i,j}(x,y)^{-1}};
\region{.5,0}{\lambda};
\end{tikzpicture}
\pmod{F_0 \Hom_{\hatUU^\imath}(B_i B_j \one_\lambda, B_j B_i \one_\lambda)}.
$$
The right-hand side is in $\End_{\overline{\AB}}(B^{\otimes 2})$ by the results established already.
Also $$
\iota\big(F_0\Hom_{\hatUU^\imath}(B_i B_j \one_\lambda, B_j B_i \one_\lambda)\big)\subseteq
\End_{\overline{\AB}}(B^{\otimes 2})
$$
because 
$F_0\Hom_{\hatUU^\imath}(B_i B_j \one_\lambda, B_j B_i \one_\lambda)$
is contained in 
the closure of the $\hatZ_\lambda$-span of crossing-free diagrams with dots, and 
the images under $\iota$ of all such diagrams are 
in the closed subspace $\End_{\overline{\AB}}(B^{\otimes 2})$.
\end{proof}

For the remainder of the section, we assume that
\begin{center}
\fbox{
\begin{minipage}{0.55\textwidth}
$\catR$ is an Abelian or Karoubian Brauer categorification. 
\end{minipage}
}
\end{center}
Recall the notation of \cref{bcat}.
In particular, for $V \in \catR$, we write $Z_V$ for the center of $\End_\catR(V)$, a finite-dimensional commutative algebra, and $z_V\in Z_V$ for the evaluation on $V$ of $z \in Z$. In order to prove \cref{thmb} from the 
introduction, the key additional ingredient is that
the action of $\AB$ on $\catR$ is continuous for the topology induced by $\Psi$ and the discrete topology on 
$\Hom_{\catR}(B^{\otimes n} V, B^{\otimes m} V)$
for each $V \in \catR$ and $m,n \geq 0$.
We prove this in the next few lemmas.

\begin{lem}\label{subsist}
Let $V \in \catR$ be an indecomposable object and
$\lambda \in X^\imath_\0$ be the unique even iweight such that
$V \in \catR_\lambda$.
Then there is a unique algebra homomorphism
$\omega\colon \hatZ_\lambda \rightarrow Z_V$ 
such that
\begin{enumerate}
\item 
$\omega(O_{\lambda;i}(u)) = O_{V;i}(u)$ for all $i \in I$,
where $O_{V;i}(u)$ is the local factor of $O_V(u)$ at $u=i$;
\item $\omega$ is continuous for the GC topology on $\hatZ_\lambda$
and the discrete topology on $Z_V$.
\end{enumerate}
Moreover:
\begin{enumerate}
\item[(iii)] $\omega(\psi_\lambda(z)) = z_V$ for any $z \in Z$.
\end{enumerate}
\end{lem}

\begin{proof}
The commutative algebra $Z_\lambda$ is generated by the
elements 
$z_{\lambda;i}^{(r)}$ subject only to the relations
$O_{\lambda;-i}(u)O_{\lambda;i}(-u) = (-1)^{\del_i+\lambda_i} u^{\del_i+\del_{-i}} \id_{\one_\lambda}$.
Thus, to show that there is an algebra homomorphism
$Z_\lambda \rightarrow Z_V$ satisfying (i), it suffices
to show that
\begin{equation}\label{light}
O_{V;-i}(u)O_{V;i}(-u) = (-1)^{\del_i+\lambda_i} u^{\del_i+\del_{-i}} \id_V.
\end{equation}
As $V$ is indecomposable, 
$Z_V$ is a finite-dimensional local commutative algebra. Let $\mathfrak{m}$ be its unique maximal ideal.
As explained in \cref{moredirect},
the image 
$\bar O_V(u)$
of $O_V(u)$ in $\kk(u) = (Z_V / \mathfrak{m})(u)$
is the rational function $\chi_\lambda(u)$ from \cref{maxideals}.
Hence, the local factors $O_{V;i}(u)$ 
of $O_V(u)$ satisfy
$$
O_{V;i}(u) \in u^{\del_i-\lambda_i}\id_V +u^{\del_i-\lambda_i-1}\mathfrak{m}[u^{-1}].
$$
Since $\sum_{i\in I}(\del_i-\lambda_i) = 0$, we have that
$$
O_V(-u) = \prod_{i \in I} O_{V;i}(-u-i) = 
\prod_{i \in I} \widetilde O_{V;-i}(u-i)
$$ 
where $\widetilde O_{V;i}(u)\coloneq(-1)^{\del_i+\lambda_i}
O_{V;i}(-u)$.
Also, the local factor of $\frac{(u-h)(u+h)}{u^2}$ at $u=i$ is $u^{-a_{i,-i}} = u^{\del_i+\del_{-i}}$.
By \cref{curlsup1}, $O_V(u) O_V(-u) = \frac{(u-h)(u+h)}{u^2}$ so, using the uniqueness of local factors from \cref{localfactors}, we get that
$O_{V;i}(u) \widetilde O_{V;-i}(u) = u^{\del_i+\del_{-i}}$.
The identity \cref{light} follows from this on replacing $i$ with $-i$ and using the definition of $\widetilde O_{V;i}(u)$.

To construct $\omega$, 
it remains to show that the homomorphism just constructed is continuous for the topology on $Z_\lambda$ induced by the GC topology on $\hatZ_\lambda$ and the discrete topology on $Z_V$, for then it extends uniquely to the desired continuous homomorphism $\omega$.
To prove this, \cref{localfactors} implies that the renormalized local factor $u^{\lambda_i-\del_i}O_{V;i}(u)$
is a polynomial in $\id_V + u^{-1} \mathfrak{m}[u^{-1}]$,
and it equals $\id_V$ for all but finitely many $i$.
Consequently, only finitely many generators $z_{\lambda;i}^{(r)}$ have nonzero images, and those images lie in the nilpotent ideal $\mathfrak{m}$. 
The degree argument used to prove Claim 1 in the proof of \cref{extending}
then shows that $\ker\omega$ contains a GC neighborhood.

Finally, (iii) holds from
the local factorization $O_V(u) = \prod_{i \in I} O_{V;i}(u-i)$
and \cref{psiO}.
\end{proof}

\begin{lem}\label{newattempt}
Suppose that $m,n \geq 0$ with $m \equiv n \pmod{2}$
and $l \coloneq \frac{m+n}{2}$.
Recalling \cref{pr1}, 
for any $D \in \shape(m\times n),
\bc \in I^l$ and $\lambda \in X^\imath_\0$, 
the function
\begin{align*}
\Hom_{\AB}(B^{\otimes n},B^{\otimes m})&\rightarrow\hatZ_\lambda\llbracket x_1,\dots,x_l\rrbracket,& 
\theta&\mapsto \psi_\lambda\big(\pr_D(\theta)(c_1+x_1,\dots,c_l+x_l)\big)
\end{align*}
is continuous for the subspace topology 
on the domain $\Hom_{\AB}(B^{\otimes n},B^{\otimes m})$
induced by the GC topology on $\Hom_{\hatAB}(B^{\otimes n},B^{\otimes m})$,
and the topology on $\hatZ_\lambda\llbracket x_1,\dots,x_l\rrbracket = \prod_{e_1,\dots,e_l \geq 0} x_1^{e_1} \cdots x_l^{e_l}\hatZ_\lambda$ that is the product of the GC topologies on each factor.
\end{lem}

\begin{proof}
We proceed by downward induction on $\#(D)$.
Let $\mu$ be as in \cref{nastymu}.
We apply the linear map
$\cf_{\lambda;D_\bc}\big(\psi_{m,n}(-)^{\mu,\lambda}_{\bi^{D_\bc},\bj^{D_\bc}}\big)$
to 
$\theta = \sum_{D' \in \shape(m\times n)} D'(\pr_{D'}(\theta)(x_1,\dots,x_l))$ and use
\cref{leading} to deduce that there
is a unit
$u \in \hatZ_\lambda\llbracket x_1,\dots,x_l\rrbracket^\times$ 
such that
$$
\cf_{\lambda;D_\bc}\!\big(\psi_{m,n}(\theta)^{\mu,\lambda}_{\bi^{D_\bc},\bj^{D_\bc}}\big)
= u \psi_\lambda\big(\pr_D(\theta)(c_1+x_1,\dots,c_l+x_l)\big)
+ \sum_{D'}\cf_{\lambda;D_\bc}\!\Big(\psi_{m,n}\big(D'(\pr_{D'}(\theta))\big)^{\mu,\lambda}_{\bi^{D_\bc},\bj^{D_\bc}}
\Big),
$$
summing over $D' \in \shape(m\times n)$ with  $\#(D') > \#(D)$. 
As noted after \cref{pr2}, $\cf_{\lambda;D_\bc}$ is continuous,
hence, the function on the left-hand side is a continuous function of $\theta$. Also multiplication by $u^{-1}$ is continuous.
Thus, we are reduced to
showing 
that the function
\begin{align*}
\Hom_{\AB}(B^{\otimes n}, B^{\otimes m})
&\rightarrow \hatH_\lambda\big(\bj^{D_\bc},\bi^{D_\bc}),
&
\theta &\mapsto
\psi_{m,n}\big(D'(\pr_{D'}(\theta))\big)^{\mu,\lambda}_{\bi^{D_\bc},\bj^{D_\bc}}
\end{align*}
is continuous
for $D' \in \shape(m\times n)$
with $\#(D') > \#(D)$.

As in the proof of \cref{leading}, for $f\in Z[x_1,\dots,x_l]$ we have that
$$
\psi_{m,n}(D'(f))^{\mu,\lambda}_{\bi^{D_\bc}, \bj^{D_\bc}} 
= \sum_\kappa E_\kappa\big(\psi_\lambda(f(\kappa_1+x_1,\dots,\kappa_l+x_l))\big).
$$
Here, the
sum is over all 
labellings $\kappa$ of the segments of $D'$ by elements of $I$ which are compatible with the boundary labels $\bi^{D_\bc}$ at the top and $\bj^{D_\bc}$ at the bottom,
with $\kappa_1,\dots,\kappa_l$ being the labels at the distinguished points.
To define $E_\kappa$, let
$D'_\kappa$
be obtained from $D'$ by replacing magenta strings with black strings, 
labelling segments according to $\kappa$, and
replacing crossings, cups and caps by the decorations in \cref{gasp2,gasp3,gasp4}.
Then 
$E_\kappa\colon\hatZ_\lambda\llbracket x_1,\dots,x_l\rrbracket\rightarrow \hatH_\lambda\big(\bj^{D_\bc},\bi^{D_\bc}\big)$ 
is the linear map 
taking $g\in \hatZ_\lambda\llbracket x_1,\dots,x_l\rrbracket$ 
to the morphism obtained
by pinning $g$ to $D'_\kappa$
so that its pinhead is in the rightmost 2-cell, labelling this 2-cell by $\lambda$.
Each $E_\kappa$ is continuous, since horizontal and vertical composition are continuous, as is
pinning a formal power series at the specified distinguished point. 

Let $N$ be a GC neighbourhood of zero in $\hatH_\lambda(\bj^{D_\bc},\bi^{D_\bc})$. 
Choose a cofinite subset $J\subseteq I$ such that
$\overline{H_\lambda(\bj^{D^\bc},\bi^{D^\bc})U_{\lambda;J}}
\subseteq N$, and finitely many
horizontal slices of the string diagram $D'$ so that every segment meets one of them. At each slice, we apply the claim in the proof of \cref{conv1}(ii), and take the union of the resulting finite sets of labels. This produces a finite $K \subset I$ such that $E_\kappa(g)\in
\overline{H_\lambda(\mathbf j^{D^c},\mathbf i^{D^c})U_{\lambda;J}}
\subseteq N$
for every $g\in\hatZ_\lambda\llbracket x_1,\ldots,x_l\rrbracket$ whenever $\kappa$ involves a label 
outside $K$. Thus, modulo $N$, the expansion reduces to the finite sum over labellings taking values in $K$. For each such labelling, the map
$$
\theta\mapsto
E_\kappa\!\left(
\psi_\lambda\bigl(
\pr_{D'}(\theta)(\kappa_1+x_1,\ldots,\kappa_l+x_l)
\bigr)\right)
$$
is continuous by the induction hypothesis and the continuity of $E_\kappa$. Consequently, the required map is continuous modulo every GC neighbourhood of zero, and hence is continuous. This completes the induction.
\end{proof}

\begin{lem}\label{evcontinuous}
For $V \in \catR$ and $m,n \geq 0$ with $m \equiv n \pmod{2}$, 
the evaluation map
$$
\Hom_{\AB}(B^{\otimes n},B^{\otimes m}) \rightarrow \Hom_\catR(B^n V,B^m V),\qquad \theta\mapsto \theta_V
$$
is continuous with respect to 
the topology on the domain induced by $\psi_{m,n}$ and the discrete topology on the codomain.
\end{lem}

\begin{proof}
We may assume that $V$ is indecomposable. 
Let $\lambda$ and
$\omega\colon\hatZ_\lambda\rightarrow Z_V$
be as in \cref{subsist}.
Set $U\coloneq\ker\omega$, which is 
an open subspace of $\hatZ_\lambda$.
Assume that the shapes 
$D\in \shape(m\times n)$ have been chosen so that they are composites of caps, cups and crossings, with the distinguished points at distinct
horizontal levels avoiding any of the generators. 
There is a finite set $F\subset I$ and $M\geq 1$ such that
for each distinguished point
the morphism $x$ obtained by evaluating the dot at that point on the object obtained by applying the appropriate power
of $B$ (i.e., the number of strings to the right)
to $V$ has spectrum contained in $F$, and its restriction to the generalized $i$-eigenspace of $x$
is annihilated by $(x-i)^M$.

For $D\in \shape(m\times n)$, $\bc\in F^l$, and $\bd\in\mathbb N^l$, define
$T_{D,\bc,\bd} \in \Hom_{\catR}(B^n V,B^m V)$
by evaluating $D$ on $V$, inserting the projection onto $B_{c_k}$ at the distinguished point $p_k$, then placing $(x-c_k)^{d_k}$ there, for each $1 \leq k \leq l$. Then
$T_{D,\bc,\bd}=0$ if $d_k\ge M$ for some $k$.
Define a linear map
$\Theta_{D,\bc}:\hatZ_\lambda\llbracket x_1,\ldots,x_l\rrbracket
\rightarrow\Hom_{\catR}(B^n V,B^m V)$
by the finite sum
$$
\Theta_{D,\bc}\colon\sum_\bd g_\bd x_1^{d_1} \cdots x_l^{d_l}
\mapsto
\sum_{\substack{\bd\in\mathbb N^l\\d_k<M\text{ for all }k}}
T_{D,\bc,\bd}\circ B^n\!\left(\omega(g_\bd)\right).
$$
This vanishes on the open subspace
$\left\{\sum_\bd g_\bd x_1^{d_1} \cdots x_l^{d_l}
\:\big|\:g_\bd\in U\text{ whenever }d_1,\dots,d_l<M\right\}$.
So $\Theta_{D,\bc}$ is continuous for the discrete topology on its codomain.
For every $f\in Z[x_1,\ldots,x_l]$, we have that
$$
D(f)_V
=
\sum_{\bc\in F^l}
\Theta_{D,\bc}\!\left(
\psi_\lambda\bigl(f(c_1+x_1,\ldots,c_l+x_l)\bigr)
\right).
$$
Consequently,
$$
\theta_V=
\sum_{D\in S(m\times n)}\ \sum_{\bc\in F^l}
\Theta_{D,\bc}\!\left(
\psi_\lambda\bigl(\pr_D(\theta)(c_1+x_1,\ldots,c_l+x_l)\bigr)
\right).
$$
Each summand is continuous by \cref{newattempt} 
and the continuity of $\Theta_{D,\bc}$. 
Both indexing sets are finite, proving the lemma.
\end{proof}

\begin{theo}\label{btoq-cor}
Any Brauer categorification $\catR$ can be made into an even nilpotent iquantum categorification via a construction which is the inverse of the one in \cref{passportcor}.
\end{theo}

\begin{proof}
Recall the ``block'' decomposition of $\catR$ from \cref{blocks,blocksagain}
and the eigenfunctor decomposition $B = \bigoplus_{i \in I} B_i$ of the endofunctor 
$B:\catR \rightarrow \catR$
from \cref{eigenfunctors}. We also define $\catR_\lambda$ to be a trivial (zero) $\kk$-linear category
in case $\lambda \in X^\imath_\1$,
so $\catR = \bigoplus_{\lambda \in X^\imath} \catR_\lambda$. The subcategories $\catR_\lambda$ are the required iweight categories in order to have 
an even iquantum categorification, and \cref{rightway} shows that  
the endofunctors $B_i$ take $\catR_\lambda$ to $\catR_{\lambda-\alpha_i}$.
To prove the theorem, the idea is to show that the action of $\AB$ on $\catR$ extends uniquely to a continuous action of $\hatAB$, then to pass from there to the action of $\UU^\imath$.
First, we establish two claims:

\vspace{2mm}
\noindent
{\underline{Claim 1}:} {\em The action of $\AB$ on $\catR$ extends uniquely to an action of $\hatAB$ which is continuous in the sense that the evaluation maps 
$\Hom_{\hatAB}(B^{\otimes n},B^{\otimes m})\rightarrow\Hom_\catR(B^n V,B^m V)$ are continuous with respect to 
the discrete topology on $\Hom_\catR(B^n V,B^m V)$.}

\vspace{1mm}
\noindent
Proof:
To extend the action of $\AB$ to $\hatAB$, we first need to define
$\theta_V:B^n V \rightarrow B^m V$ for any
$\theta\in\Hom_{\hatAB}(B^{\otimes n},B^{\otimes m})$ and $V\in\catR$,
then we need to check that $(\theta_V)_{V \in \catR}$ is a natural transformation, and that the assignments $B^{\otimes n}\mapsto B^n$ and
$\theta \mapsto (\theta_V)_{V \in \catR}$ define a strict $\kk$-linear monoidal functor $\hatAB \rightarrow \mathbf{End}_\kk(\catR)$.
To define $(\theta_V)_{V \in \catR}$,
take some object $V \in \catR$.
By \cref{evcontinuous}, there is an open subspace $N$ of $\Hom_{\hatAB}(B^{\otimes n},B^{\otimes m})$ such that all elements of 
$\Hom_{\AB}(B^{\otimes n},B^{\otimes m})\cap N$ evaluate to 0 on $V$.
By \cref{btoq-thm}, there is $\tilde\theta \in \Hom_{\AB}(B^{\otimes n},B^{\otimes m})$ with $\theta-\tilde\theta\in N$.
Then we let $\theta_V\coloneq \tilde\theta_V$ for this $\tilde\theta$;
this is well defined independent of the choice of $\tilde\theta$.
To prove naturality and compatibility with horizontal and vertical composition, any identity that needs to be checked
involves only finitely many objects of $\catR$. 
As horizontal and vertical composition 
in $\hatAB$ are continuous, the identity
follows from the corresponding identity for morphisms of $\AB$ approximating the morphisms of $\hatAB$ involved closely enough.
This makes $\catR$ into a strict left $\hatAB$-module category whose restriction along $\Psi$ is the given $\AB$-module structure.
Uniqueness is clear by density.

\vspace{2mm}
\noindent
{\underline{Claim 2}:} {\em The idempotents $e_\lambda=\iota(\id_{\one_\lambda})$ and $e_{\lambda;i}=\iota(\id_{B_i \one_\lambda})$ 
from the proof of \cref{btoq-thm} act as the projections onto the ``blocks'' $\catR_\lambda$ and the eigenfunctors $B_i$ from \cref{eigenfunctors}, respectively.}

\vspace{1mm}
\noindent
Proof:
For an indecomposable $V\in \catR_\mu$, consider the composition
$\hatZ\rightarrow \hatZ_\mu
\stackrel{\omega}{\rightarrow} Z_V$ for $\omega$ as in \cref{subsist}.
By \cref{subsist}(iii)
this continuous homomorphism maps $z \mapsto z_V$
for $z$ in the
dense subalgebra $Z$, hence, it is equal to the evaluation map from Claim 1. In particular,
$(e_\lambda)_V=\delta_{\lambda,\mu}\id_V$.
Taking finite direct sums proves that $e_\lambda$ acts as the projection onto $\catR_\lambda$.

Next, $\id_B = \sum_{i\in I}e_i$ where $e_i \coloneq \sum_{\lambda}e_{\lambda;i}$, and the discussion of idempotents in the proof of \cref{btoq-thm} shows that $e_i$ is the limit 
over all finite $F \subset I$ and $r \geq 0$
of $\psi_{1,1}\Big(\begin{tikzpicture}[B,centerzero,scale=.8]
\draw[-] (0,-0.3)-- (0,0.3);
\pin{0,0}{1,0}{f_{F,r}(x)};
\end{tikzpicture}\ \Big)$ for polynomials $f_{F,r}(x)$ chosen so that $f_{F,r}(x)\equiv\delta_{i,j}\pmod{(x-j)^r}$ for all $j \in F$.
Then $e_i$ acts on $BV$ as the projection onto $B_i V$.
Since $e_{\lambda;i} = e_i\circ(\id_B\otimes e_\lambda)$, it acts on $BV\:(V \in \catR_\lambda)$ as the projection onto $B_iV$, and it acts as 0 on $BV$ for $V\in\catR_\mu$ with $\mu\neq\lambda$.

\vspace{2mm}
Now we can complete the proof of the theorem.
For $\lambda\in X^\imath_\0$, let $B_i\one_\lambda$ act as the restriction $B_i|_{\catR_\lambda}\colon\catR_\lambda\rightarrow\catR_{\lambda-\alpha_i}$.
Any 2-morphism $\theta \in \Hom_{\UU^\imath}(B_\bj \one_\lambda,
B_\bi \one_\lambda)$ acts on $V\in\catR_\lambda$ as the restriction of $\iota(\theta)_V\colon B^n V\rightarrow B^m V$ to the summand $B_\bj V$ of $B^n V$; by Claim 2, this maps it into the summand $B_\bi V$ of $B^m V$.
As $\iota$ is compatible with vertical and horizontal composition and the action of $\hatAB$ on $\catR$ is a strict monoidal action, all of the relations of $\UU^\imath$ are satisfied.
Also $B_i V \neq 0$ for only finitely many $i\in I$.
So 
this makes $\catR = \bigoplus_{\lambda \in X^\imath}\catR_\lambda$ into an even iquantum categorification.
By \cref{gasp1,dotdef}, 
$\begin{tikzpicture}[iQ,centerzero,scale=.8]
\draw[-] (0,-0.3)\botlabel{i}-- (0,0.3);
\closeddot{0,0};
\region{.3,0}{\lambda};
\end{tikzpicture}$ acts as
$\begin{tikzpicture}[B,centerzero,scale=.8]
\draw[-] (0,-0.3)\botlabel{i}-- (0,0.3);
\pin{0,0}{-.8,0}{x-i};
\end{tikzpicture}$, as it should by \cref{veryeasy}, and this is nilpotent on $B_i V$ since $B_i V$ is a generalized eigenspace.
Hence, $\catR$ is nilpotent.

Finally, we indicate why this construction and the one from \cref{passportcor} are inverse to each other.
If we start from a Brauer categorification, apply the construction just explained, and then apply the construction of \cref{passportcor}, we obtain an action of $\hatAB$ on $\catR$ via \cref{extending}.
It is continuous, because the action of a morphism of $\hatAB$ on $V\in\catR$ only involves the finitely many entries of the matrix indexed by the iweights of $V$ and the words $\bj$ with $B_\bj V\neq 0$, and each of these acts continuously by Claim 2 in the proof of \cref{extending}.
It agrees with the action from Claim 1 on all $\iota(\theta)$ by construction, so the two actions agree by \cref{tricky}, and their restrictions along $\Psi$ are both the original action of $\AB$.
Conversely, if we start from an even nilpotent iquantum categorification, apply \cref{passportcor} and then the construction just explained, the action of $\hatAB$ obtained in Claim 1 is the continuous one coming from \cref{passportcor}, by the uniqueness in Claim 1. So $e_\lambda$ and $e_{\lambda;i}$ act as the projections onto $\catR_\lambda$ and $B_i\one_\lambda$, and $\iota(\theta)$ acts as $\theta$, i.e., we do indeed
recover the original iquantum categorification.
\end{proof}

\setcounter{section}{9}

\section{Cyclotomic quotients}\label{cyclotomic}

In this section, we prove that cyclotomic quotients of $\AB$
and cyclotomic quotients of $\UU^\imath$ are equivalent categories. 
The former are Karoubian Brauer categorifications, 
and the latter are even Karoubian iquantum categorifications. In the case of $\AB$, the definition of these cyclotomic quotients appeared already in \cite{RS19}; in fact, it can be traced back to \cite{AMR}. In the case of $\UU^\imath$, the definition was written down already in \cite{GRS-2026} in the AIII$_{2\infty}$ case. 
Recall for a monic polynomial $f(u)$
that $\widetilde{f}(u) = (-1)^{\deg f(u)} f(-u)$.

\begin{defin}\label{brauercyclotomic}
Let $A$ be a finite-dimensional 
commutative $\kk$-algebra and
$f(u)\in A[u]$ be a monic polynomial of degree $\beta \geq 0$.
Let $\AB \otimes_\kk A$ be the strict $A$-linear monoidal category obtained from $\AB$ by base change and
$\cI(A,f)$ be the $A$-linear
left tensor ideal of $\AB\otimes_\kk A$ generated by the relations
\begin{align}\label{cycrels1}
\begin{tikzpicture}[B,baseline = -2mm,scale=1.2]
\draw[-] (-0.2,-.4) to (-0.2,.4);
\pin{-.2,0}{.5,0}{f(x)};
\end{tikzpicture} &= 0,&
\begin{tikzpicture}[B,centerzero,scale=.9]
\draw (-0.25,0) arc(180:-180:0.25);
\node at (.57,0) {$\color{black}(u)$};
\end{tikzpicture}&= 
\begin{dcases}
\frac{u-h}{u} \cdot \frac{\widetilde{f}(u)}{f(u)} \id_\one
&\text{if $\beta$ is even}\\
\frac{u+h}{u} \cdot \frac{\widetilde{f}(u)}{f(u)} \id_\one&\text{if $\beta$ is odd.}
\end{dcases}
\end{align}
The {\em deformed cyclotomic Brauer category} of level $\beta$
associated to $A$ and $f$ is
$$
\cB(A,f) \coloneq \Kar\big(\AB \otimes_\kk A\ \big/\ \cI(A,f)\big).
$$
By definition, $\cB(A,f)$ 
is a strict $A$-linear left $\AB$-module category.
\cref{rsfd} below shows that it is a Karoubian Brauer categorification.
In the case that $A = \kk$, we denote $\cB(A,f)$ simply by
$\cB(f)$, calling it 
the {\em cyclotomic Brauer category} of level $\beta$ associated to $f$.
\end{defin}

\begin{lem}[Rui-Song]\label{rsfd}
Morphism spaces
in the deformed cyclotomic quotient
$\cB(A,f)$ are finite-dimensional for every
finite-dimensional commutative algebra $A$ and monic
$f \in A[x]$.
Hence, $\cB(A,f)$ is a Karoubian Brauer categorification.
\end{lem}

\begin{proof}
The proof of \cite[Prop.~3.7(b)--(c)]{RS19}
extends unchanged to the deformed setting to show that
$\Hom_{\cB(A,f)}(B^{\otimes n}, B^{\otimes m})$ is finite-dimensional for every $m,n \geq 0$.
\end{proof}

\begin{rem}
There is
an explicit basis theorem for the morphism spaces 
$\Hom_{\cB(A,f)}(B^{\otimes n}, B^{\otimes m})$.
For the cyclotomic Brauer category
$\cB(f)$, 
this follows immediately 
from \cite[Th.~C]{RS19}. 
Rui and Song prove their basis theorem working over an arbitrary ground ring, not just a field, so that the basis theorem in 
our deformed setting can be deduced from their result by 
making a flat base change from $\kk$ to a larger ring containing the bubble parameters, then applying \cite[Th.~C]{RS19} there.
In particular, $\End_{\cB(A,f)}(\one) \cong A$,
hence, $\cB(f)$ is non-zero, something which is by no means obvious from the definition.
\end{rem}

The following lemma is the universal property
of deformed cyclotomic Brauer categories:

\begin{lem}\label{lemma10-1}
Let $V$ be an indecomposable (hence, non-zero)
object of an Abelian or Karoubian 
Brauer categorification $\catR$, and $Z_V$ be the center of $\End_\catR(V)$ as usual.
There exists a monic polynomial $f(u) \in Z_V[u]$
of degree $\beta \geq 0$
such that 
\begin{align}\label{notchif}
\begin{tikzpicture}[B,baseline = -2mm,scale=1.2]
\draw[-] (-0.2,-.4) to (-0.2,.4);
\draw[-,darkg] (0.5,.4) to (0.5,-.4)\botlabel{V};
\objpin{-.2,0}{.5,0}{f(x)};
\end{tikzpicture}&=0,&
O_V(u) &= \begin{dcases}
\frac{u-h}{u} \cdot \frac{\widetilde{f}(u)}{f(u)}
&\text{if $\beta$ is even}\\
\frac{u+h}{u} \cdot \frac{\widetilde{f}(u)}{f(u)}&\text{if $\beta$ is odd.}
\end{dcases}
\end{align}
Moreover, 
there is a canonical morphism of strict left $\AB$-module categories
$\Pi\colon\cB(Z_V,f) \rightarrow \catR$ mapping the generating object $\one$ of $\cB(Z_V,f)$ to $V$.
\end{lem}

\begin{proof}
In view of \cref{chif}, the first assertion follows
immediately from \cref{itsreallywednesday}.
There is a
 canonical morphism of strict left $\AB$-module
categories $\AB\otimes_\kk Z_V 
\rightarrow \catR$ mapping $\one$ to $V$ 
and $\id_\one \otimes z \in \End_{\AB}(\one) \otimes_\kk Z_V$
to $z \in \End_{\catR}(V)$.
This sends morphisms in $\cI(Z_V,f)$ to $0$, as 
follows by comparing \cref{thur,cycrels1}, 
so it factors through the quotient to induce $\Pi$.
Finally, it extends to the additive Karoubi envelope 
because $\catR$ is additive and idempotent complete.
\end{proof}

Next, we formulate the definition of a cyclotomic quotient of $\UU^\imath$. 
For any $\lambda \in X^\imath$, the corresponding
{\em regular 2-representation}
of $\UU^\imath$ is
$(\one_\mu \UU^\imath \one_\lambda)_{\mu \in X^\imath}$
with $\mu$th iweight subcategory 
$$
\one_\mu \UU^\imath \one_\lambda \coloneq
\Hom_{\UU^\imath}(\lambda,\mu)
$$
for each $\mu \in X^\imath$, and categorical action defined by horizontal composition on the left.
It is an even 2-representation if $\lambda$ is even.

\begin{defin}\label{iqcyclotomic}
Let $A$ be a finite-dimensional 
commutative $\kk$-algebra and
$f_i(u) \in A[u]$ be monic polynomials 
for all $i \in I$ such that $f_i(u)=1$ for
all but finitely many $i\in I$.
Let $\beta_i\coloneq \deg f_i(u)$, so that $\beta_i = 0$
for all but finitely many $i \in I$, and $\beta \coloneq \sum_{i \in I} \beta_i$.
Let $\lambda \in X^\imath_\0$ be the canonical image of
the even weight
\begin{equation}\label{theweight}
\hat\lambda \coloneq \begin{dcases}
\sum_{i \in I} \beta_i \varpi_i
&\text{if $\beta$ is even}\\
\varpi_h+\sum_{i \in I} \beta_i \varpi_i&\text{if $\beta$ is odd}.
\end{dcases}
\end{equation}
Let $\cI(A,f_i\,|\,i \in I)$ be the sub-2-representation\footnote{We mean an $A$-linear ideal of the category which is closed under the left action of $\UU^\imath$.} of the $A$-linear 2-representation $\big(\one_\mu \UU^\imath \one_\lambda \otimes_\kk A\big)_{\mu \in X^\imath}$ generated by the relations
\begin{align}\label{nextstop}
\begin{tikzpicture}[iQ,baseline = -2mm,scale=1.2]
\draw[-] (-0.2,-.4)\botlabel{i} to (-0.2,.4);
\pin{-.2,0}{.45,0}{f_i(x)};
\region{1,0}{\lambda};
\end{tikzpicture} &= 0,&
\begin{tikzpicture}[iQ,centerzero,scale=.9]
\draw (-0.25,0) arc(180:-180:0.25);
\node at (.57,0) {$(u)$};
\node at (-0.36,0) {\strandlabel{\ibar}};
\region{1,0}{\lambda};
\end{tikzpicture}&= 
\begin{dcases}
u^{\delta_{i,h} -\delta_{i,0}}
\frac{\widetilde{f}_{-i}(u)}{f_i(u)}\id_{\one_\lambda}&\text{if  $\beta$ is even}\\
u^{\delta_{i,-h} -\delta_{i,0}}
\frac{\widetilde{f}_{-i}(u)}{f_i(u)}\id_{\one_\lambda}&\text{if  $\beta$ is odd.}
\end{dcases}
\end{align}
Then the {\em deformed cyclotomic iquantum category}
$\cC(A,f_i\,|\,i\in I)$
is the iquantum categorification with $\mu$th iweight subcategory
$$
\one_\mu \cC(A,f_i\,|\,i\in I)\coloneq
\Kar\big(\one_\mu \UU^\imath \one_\lambda \otimes_\kk A\ \big/\ \one_\mu \cI(A,f_i\,|\,i \in I)\big)
$$
for each $\mu \in X^\imath$. 
By
\cref{lemma10-6} below, it is a Karoubian iquantum categorification.
In the case that $A = \kk$, we denote $\cC(A,f_i\,|\,i\in I)$ simply by $\cC(f_i\,|\,i \in I)$, calling it the 
{\em cyclotomic iquantum category}.
\end{defin}

The following standard lemma referenced in \cref{iqcyclotomic}
extends \cite[Prop.~8.5]{GRS-2026}:

\begin{lem}\label{lemma10-6}
In the setup of \cref{iqcyclotomic},
$\cC(A,f_i\,|\,i\in I)$ is a Karoubian iquantum categorification.
Moreover, if the coefficients of each $f_i(u)$ apart from the leading coefficient 1 belong to the Jacobson radical of $A$,
then $\cC(A,f_i\,|\,i\in I)$ is nilpotent.
\end{lem}

\begin{proof}
Let $\cC \coloneq \cC(A,f_i\,|\,i\in I)$,
and $\lambda = \pi(\hat\lambda)$ as in \cref{theweight}, so that
the generating object of $\cC$ is $\one_\lambda$. 
We first show by induction on $n$
that $\Hom_{\cC}(B_\bj \one_\lambda,B_\bi \one_\lambda)$
is finite-dimensional for words $\bi\approx\bj$ 
of total length $2n$.
For the induction base $n=0$, \cref{nondegeneracy}
and the bubble relations in \cref{nextstop} 
show that $\End_{\cC}(\one_\lambda)$ is a quotient of $A$, so it is finite-dimensional.

Now consider the induction step.
By adjunction, we may assume that both of the words
$\bi$ and $\bj$ are of the same length $n > 0$.
Fix a finite subset $S\subset I$ closed under $\pm$ 
and containing the letters of the words $\bi$ and $\bj$.
In the additive envelope, let
\begin{align*}
T_n&\coloneq\bigoplus_{\bi\in S^n}
B_{\bi}\one_\lambda,&
E_n&\coloneq\End_{\cC}(T_n).
\end{align*}
Let $J_n$ be 
the two-sided 
ideal of $E_n$ consisting of sums of morphisms which factor through objects of the form
$B_{\bk}\one_\lambda$ with $\bk\in S^{n-2}$.
In particular, $J_1 = 0$.
For each such $\bk$, both of the spaces
$\Hom_{\cC}(T_n,B_{\bk}\one_\lambda)$
and
$\Hom_{\cC}(B_{\bk}\one_\lambda,T_n)$
are finite-dimensional by induction. 
Since there are only finitely many $\bk$, 
it follows that the ideal $J_n$ is finite-dimensional.

By \cref{nondegeneracy} (the easy spanning part), 
every morphism in $E_n$ is an $R$-linear combination of
string diagrams of reduced shapes decorated by some dots at the distinguished points. Elements of $R$ are bubbles in the rightmost 2-cell, which can be replaced by scalars from the algebra $A$ using
the second of the relations \cref{nextstop}.
We assume that the reduced shapes are chosen so that all caps are below all crossings and all cups are above them.
Then any reduced diagram with one or more cups (and, necessarily, the same number of caps) has $\leq (n-2)$ propagating strings,
so it factors through an 
object $B_\bk \one_\lambda$ for $\bk$ of length $n-2$, i.e.,
it belongs to $J_n$.
Consequently, the quotient 
$E_n/J_n$ is generated by dots, crossings, and idempotents. 
Moreover, 
modulo $J_n$, defining relations of $\UU^\imath$ simplify to ordinary quiver Hecke algebra relations in the deformed quiver Hecke algebras $QH_n \otimes_\kk A$ associated to the finite
Cartan matrix $(a_{i,j})_{i,j \in S}$
with parameters $Q^{\imath}_{i,j}(x,y)\:(i,j \in S)$, 
and the cyclotomic relation that is the first relation from \cref{nextstop} for dots on the rightmost string.
Thus, $E_n/J_n$ is a quotient of the corresponding deformed cyclotomic quiver Hecke algebra, which is well known to be finitely generated as a right $A$-module, an observation which goes back at least to \cite[Cor.~2.2]{BK-iso}.
Since $A$ is finite-dimensional, it follows that $E_n / J_n$ is finite-dimensional, and the induction is complete.

To complete the proof that $\cC$ is a Karoubian iquantum categorification, we still need to show that
$\big|\{j\in I\mid B_jV\neq0\}\big|<\infty$
for every object $V$. 
It suffices to prove this in the case that $V = 
B_{\bi} \one_\lambda$ for
$\bi=i_1\cdots i_n$. Let
$K\coloneq\{j\:|\:f_j(u)\neq 1\}\ \cup\
\bigcup_{r=1}^{n}\{\pm i_r,\ \pm i_r+1,\ \pm i_r-1\}$.
This set is finite. If $j\notin K$, the crossings moving the $j$-string past the letters of $\bi$ are invertible.
Thus
$B_jB_{\bi}\one_\lambda
\cong B_{\bi}B_j\one_\lambda=0$
since $f_j(u)=1$. 

It remains to prove the final assertion about nilpotency. 
This follows by modifying the above inductive argument.
Under the nilpotency assumption on the polynomials $f_i(u)$,
it is well known that dots are nilpotent in the
deformed cyclotomic quiver Hecke algebra mapping onto $E_n/J_n$.
Hence, in $E_n/J_n$, the dots are nilpotent.

In this paragraph, we prove by induction on $n$ that every dot on every object $B_{\bi}\one_\lambda$ with $\bi\in S^n$ is nilpotent. The assertion is vacuous for $n=0$, and follows from nilpotency modulo $J_1=0$ for n=1. Suppose that $n\geq 2$, and write $\bi=i_1\bi'$. For $\bk\in S^{n-2}$, adjunction gives
$$
\Hom_{\cC}(B_{\bk}\one_\lambda,B_{i_1}B_{\bi'}\one_\lambda)
\cong
\Hom_{\cC}(B_{-i_1}B_{\bk}\one_\lambda,B_{\bi'}\one_\lambda).
$$
Under this isomorphism, 
vertical composition on top
by a dot on the first string in the left-hand space
becomes vertical composition on the bottom by minus a dot on the first string in the right-hand space, and
vertical composition on top
by a dot on the $r$th string for $r>1$ becomes vertical composition on top by a dot on the $(r-1)$th string. 
Both words on the right have length $n-1$, so each of these operators is nilpotent by induction.

Let $X_r\in E_n$ be the sum of the dots on the $r$th strings of the summands of $T_n$. Since $S$ is finite, there is an exponent $a$ such that vertical composition on top with $X_r^a$ 
annihilates every $\Hom_{\cC}(B_{\bk}\one_\lambda,T_n)$ 
with $\bk\in S^{n-2}$. Hence $X_r^a J_n=0$. 
Nilpotency modulo $J_n$ gives $X_r^b\in J_n$ for some $b$, whence $X_r^{a+b}=0$. This completes the induction. Every word is contained in $S^n$ for some finite $S=-S$, and nilpotency passes to finite direct sums and direct summands, proving the assertion.
\end{proof}

The following lemma establishes the important universal property of $\cC(A,f_i\,|\,i\in I)$:

\begin{lem}\label{lemma10-5}
Let $\catR$ be any even nilpotent 
iquantum categorification, 
$\lambda \in X^\imath_\0$ and $V \in \catR_\lambda$ be an indecomposable object, so that $Z_V$ is a finite-dimensional local commutative $\kk$-algebra.
Also fix a choice of 
a parity $\sigma \in \{\0,\1\}$.
Then there exists a monic polynomial $f(u) \in Z_V[u]$
of degree $\beta \equiv \sigma \pmod{2}$
with local factors $f_i(u)$ defined as in \cref{hensel}
such that 
\begin{align}\label{somehow}
\begin{tikzpicture}[iQ,baseline = -2mm,scale=1.2]
\draw[-] (-0.2,-.4)\botlabel{i} to (-0.2,.4);
\draw[-,darkg] (0.5,.4) to (0.5,-.4)\botlabel{V};
\objpin{-.2,0}{.5,0}{f_i(x)};
\end{tikzpicture} &= 0,&
O_{V;i}(u)&= 
\begin{dcases}
u^{\delta_{i,h} -\delta_{i,0}}
\frac{\widetilde{f}_{-i}(u)}{f_i(u)}&\text{if $\sigma=\0$}\\
u^{\delta_{i,-h} -\delta_{i,0}}
\frac{\widetilde{f}_{-i}(u)}{f_i(u)}&\text{if $\sigma=\1$.}
\end{dcases}
\end{align}
Moreover, there is a canonical morphism of 2-representations
$\Pi\colon\cC(Z_V,f_i\,|\,i \in I) \rightarrow \catR$
mapping the generating object $\one_\lambda$
of $\cC(Z_V,f_i\,|\,i \in I)$ to $V$.
\end{lem}

\begin{proof}
Let $\mathfrak{m}$ be the unique maximal ideal of $Z_V$. 
Also let $I = I' \cup \{0\} \cup (-I')$ be a decomposition
chosen so that $-h \in I'$ if $\sigma = \0$ 
or $h \in I'$ if $\sigma = \1$.
Let $f_0(u) \in Z_V[u]$ be a monic polynomial chosen as 
in the final assertion of \cref{zerotrick}.
This lemma implies that \cref{somehow} holds in the case $i=0$.
Next, for each $i \in I'$, we 
let $f_i(u) \coloneq m_{V;i}(u) \in \kk[u] \subseteq Z_V[u]$,
then define $f_{-i}(u)$ so that 
$\widetilde{f}_{-i}(u) = f_i(u) O_{V;i}(u)$.
Both $f_i(u)$ and $f_{-i}(u)$ are then monic polynomials. \cref{wednesday} implies that both properties of \cref{somehow} hold for $i \in I'$, and that the first property in \cref{somehow} holds for $i \in -I'$. The second property in \cref{somehow} is also true for $i \in -I'$ due to 
\cref{iinfgrass}.

We have now constructed monic polynomials $f_i(u) \in Z_V[u]$ so that \cref{somehow} holds for all $i \in I$. 
Moreover, each $f_i(u)$ is monic and the coefficients of its lower degree terms are nilpotent. This is true for $i=0$ by \cref{zerotrick}, and it is clear for $i \in I'$ since those $f_i(u)$ are monomials. It follows for $i \in -I'$ since by the definition the lower degree terms of such $f_i(u)$ are linear combinations of bubbles of positive degree, which are nilpotent in $Z_V$ by \cref{noodslem}. 

We have that
$m_{V;i}(u)=1$ for all but finitely many $i$, 
and $\lambda_i=0$ for all but finitely many $i$. 
For $i\in I'$ avoiding these finite sets and $\{\pm h\}$, \cref{wednesday} 
shows that $f_{-i}(u)$ is monic of degree zero, hence is 1. Consequently, $f_i(u)=1$ for all but finitely many $i \in I$. Hence, it makes sense to 
set $f(u) \coloneq \prod_{i \in I} f_i(u-i)$.
Since the lower degree coefficients of all $f_i(u)$ are in the unique maximal ideal $\mathfrak{m}$,
$f_i(u)$ is the local factor of $f(u)$ at $u=i$.

Let $\beta_i \coloneq \deg f_i(u)$ and
$\beta \coloneq \deg f(u) = \sum_{i \in I} \beta_i$.
By \cref{ibubblerelnew} and \cref{somehow}, for $i \neq 0$, we have that
$\lambda_i = \beta_i-\beta_{-i}$ if $\sigma = \0$
and $\lambda_i = \beta_i-\beta_{-i} + \delta_{i,h}-\delta_{i,-h}$ 
if $\sigma = \1$. Also \cref{zerotrick} gives that $\beta_0 \equiv t(\lambda)\pmod{2}$.
Since $\lambda$ is even, it follows that $\beta \equiv \sigma \pmod{2}$, and that $\lambda = \pi(\hat\lambda)$
for $\hat\lambda$ defined by \cref{theweight}.

It just remains to prove the final assertion of the lemma.
By the universal property of the regular
2-representation $\UU^\imath \one_\lambda \otimes_\kk Z_V$, there is a canonical morphism of 2-representations $\UU^\imath \one_\lambda \otimes_\kk Z_V\rightarrow \catR$ mapping $\one_\lambda$ to $V$, and 
$\id_{\one_\lambda} \otimes 
z \in \End_\UU^\imath(\one_\lambda)\otimes_\kk Z_V$
to $z \in \End_{\catR}(V)$.
This maps the generators of
$\cI(Z_V,f_i\,|\,i \in I)$ to 0, as follows by comparing 
\cref{nextstop,somehow}, so it factors through the quotient
$\cC(Z_V,f_i\,|\,i \in I)$
to induce $\Pi$.
Finally, this extends to the additive Karoubi envelope
because $\catR$
is additive and idempotent complete.
\end{proof}

\begin{theo}\label{cyciso}
Suppose that $A$ is a finite-dimensional local commutative $\kk$-algebra and $f(u) \in A[u]$ is a monic polynomial of degree $\beta$.
Let
$f(u) = \prod_{i \in I} f_i(u-i)$ be
the factorization of $f(u)$ from \cref{hensel}, 
so $f_i(u)$ is the local factor of $f(u)$ at $u=i$.
Then there is an $A$-linear equivalence
\begin{equation*}
\Psi_{A,f}\colon\cB(A,f) \rightarrow \cC(A,f_i\,|\,i\in I).
\end{equation*}
\end{theo}

\begin{proof}
Let $\cB \coloneq \cB(A,f)$ and $\cC\coloneq \cC(A,f_i\,|\,i\in I)$.
Let $\lambda \in X^\imath$ be the image of the weight \cref{theweight}.
The local factors of $\chi_f(u)$ are
$$
R_i(u):=
\begin{cases}
u^{\delta_{i,h}-\delta_{i,0}}\,
\dfrac{\widetilde f_{-i}(u)}{f_i(u)}
&\text{if }\beta\text{ is even},\\
u^{\delta_{i,-h}-\delta_{i,0}}\,
\dfrac{\widetilde f_{-i}(u)}{f_i(u)}
&\text{if }\beta\text{ is odd.}
\end{cases}
$$
In particular, $\chi_f(u) = \prod_{i \in I} R_i(u-i)$.

To avoid ambiguity, let $V$ and $W$ denote the generating objects $\one$ and $\one_\lambda$ of $\cB$ and $\cC$, respectively. By \cref{lemma10-6,passportcor}, $\cC$ is a Karoubian Brauer categorification. On $B_i W$, the Brauer dot is $i+x$, where $x$ is the iquantum dot. Thus $f$ annihilates the Brauer dot on $BW$, since $f_i(x)=0$ on $B_i W$. Moreover, \cref{firstthis} and the defining bubble relations give
$O_W(u)=\prod_{i\in I}R_i(u-i)=\chi_f(u)$.
Consequently, the factorization argument in \cref{lemma10-1} yields an 
$A$-linear morphism of $\AB$-module categories
$F:\cB\rightarrow \cC,\ V\mapsto W$.

Conversely, equip $\cB$ with its even nilpotent iquantum categorification structure from \cref{btoq-cor}. 
Note that $\End_{\cB}(V)$ is a non-zero quotient of the local algebra $A$, so $V$ is indecomposable.
It has iweight $\lambda$ thanks to 
\cref{moredirect,lasers}. 
The factorization $f(x)=\prod_{i \in I} f_i(x-i)$
implies that the
iquantum dot on $B_iV$ is annihilated by $f_i$. Since $O_V(u)=\chi_f(u)$, uniqueness of local factorization and \cref{subsist} give
$O_{V;i}(u)=R_i(u)$.
Thus the factorization argument in \cref{lemma10-5} yields an $A$-linear morphism of 2-representations
$G:\cC\rightarrow \cB,\ W\longmapsto V$.

Under the
mutually inverse constructions of \cref{passportcor} and \cref{btoq-cor},
a morphism of $\AB$ categorifications intertwines the iquantum action constructed in \cref{btoq-cor},
and a morphism of iquantum categorifications 
intertwines the Brauer structures from \cref{passportcor}.
This follows from the uniqueness of continuous extension in the first case, and the
explicit formulae in the other.
Consequently, $G\circ F$ and $F\circ G$ are morphisms of Brauer and iquantum categorifications, respectively.
Both preserve the generating object and its prescribed $A$-action.
The universal properties used in \cref{lemma10-1,lemma10-5}
determine these functors, up to natural isomorphism, from the generating object and its prescribed $A$-action, so they
give natural isomorphisms
$G\circ F\cong\operatorname{id}_{\cB}$ and
$F\circ G\cong\operatorname{id}_{\cC}$.
Thus $F$ is the required $A$-linear equivalence $\Psi_{A,f}$.
\end{proof}

\bibliographystyle{alphaurl}
\bibliography{biq}
\end{document}